\documentclass[11pt]{amsart}
\usepackage{amsmath,amssymb,mathrsfs, xcolor, geometry, graphicx, subcaption, float}
\usepackage{hyperref}

\newtheorem{theorem}{Theorem}
\newtheorem{proposition}[theorem]{Proposition}
\newtheorem{corollary}[theorem]{Corollary}
\newtheorem{lemma}[theorem]{Lemma}

\theoremstyle{definition}
\newtheorem{definition}[theorem]{Definition}
\newtheorem{remark}[theorem]{Remark}

\makeatletter
\newcommand*\owedge{\mathpalette\@owedge\relax}
\newcommand*\@owedge[1]{%
  \mathbin{%
    \ooalign{%
      $#1\m@th\bigcirc$\cr
      \hidewidth$#1\m@th\wedge$\hidewidth\cr
    }%
  }%
}
\makeatother

\makeatletter 
\def\l@subsection{\@tocline{2}{0pt}{1pc}{5pc}{}} \def\l@subsection{\@tocline{2}{0pt}{2pc}{6pc}{}} \makeatother

\newcommand{\Ric}{\mathrm{Ric}}

\begin{document}

\title[Linear stability of the blowdown shrinker in 4D]{Linear stability of the blowdown Ricci shrinker in 4D}
\author{Keaton Naff \quad  \quad Tristan Ozuch}
\address{Lehigh University, Dept. of Math., 17 Memorial Dr E, Bethlehem, PA 18015}
\address{MIT, Dept. of Math., 77 Massachusetts Avenue, Cambridge, MA 02139}

\maketitle

\begin{abstract}
    We prove that the four-dimensional blowdown shrinking Ricci soliton constructed by Feldman-Ilmanen-Knopf is strictly linearly stable in the sense of Cao-Hamilton-Ilmanen. This provides the first known example of a non-cylindrical linearly stable shrinking Ricci soliton. This offers new insights into the topological behavior of generic solutions to the Ricci flow in four dimensions: on top of reversing connected sums and handle surgeries, they should also undo complex blow-ups.
    
    The proof starts from an explicit description of the metric and develops a tensor harmonic analysis, adapted to its weighted Lichnerowicz Laplacian and based on its $U(2)$-invariance. It further exploits the Kähler structure of the blowdown shrinking soliton and insights from four-dimensional selfduality. The main difficulty is that the weighted Lichnerowicz Laplacian of the soliton admits a $9$-dimensional set of eigentensors associated with nonnegative eigenvalues. We show that they correspond to the Ricci tensor and gauge transformations.
\end{abstract}

\setcounter{tocdepth}{1}

\tableofcontents

\section{Introduction} 

Ricci flow, introduced by Hamilton \cite{ham82}, has had profound implications in three-dimensional geometry and topology. The success of this program depended crucially on a thorough understanding of singularity formation in dimension $3$. For this understanding, three tools played a central role: the Hamilton-Ivey curvature pinching estimate \cite{ham82, ive93}, Hamilton's Harnack inequality \cite{Ham93a}, and Perelman's noncollapsing estimate \cite{p02}. Armed with these tools, Perelman completed Hamilton's program towards geometrization in \cite{p03}, which crucially relied on the classification of $3$-dimensional Ricci shrinking solitons and a qualitative description of all noncollapsed ancient flows through Perelman's canonical neighborhood theorem. Noncollapsed ancient solutions have since been completely classified \cite{Bre20, BDS21, BK21}.

In dimension $4$, even a classification of Ricci shrinking solitons seems out of reach. Indeed, we are far from a complete understanding of Einstein metrics with positive scalar curvature, which have been studied for decades. Additionally, while a version of the Hamilton-Ivey estimate is available for Ricci flows with positive isotropic curvature, no general analogue is known. Perelman's noncollapsing estimate still holds, but without strong curvature pinching, an abundance of singularity models is expected to appear for the Ricci flow in dimension $4$. It has been proposed that one should instead focus on the classification of \textit{stable} Ricci solitons, which are expected to be the only singularity models of a \textit{generic} Ricci flow, i.e. with initial condition in an open dense set. These stable Ricci solitons should correspond to the possible topological surgeries performed by a generic Ricci flow at its finite-time singularities. In this article, we show that the asymptotically conical blowdown soliton constructed by Feldman-Ilmanen-Knopf \cite{fik} is \textit{linearly stable}. Hence, the blowdown of $2$-spheres with self-intersection $\pm 1$ is expected to be a generic topological surgery of Ricci flow at its finite-time singularities. This is in stark contrast with codimension $1$ mean curvature flow where the only stable singularities are spheres and cylinders \cite{cm}.

\subsection{Background and motivation}

\subsubsection{Ricci flow in dimension $4$} 

Dimension $4$ is often considered a threshold dimension in topology, between `low-dimensional' and `high-dimensional' topology: it simultaneously displays some of the flexibility of high‐dimensional topology and some of the rigidity phenomena familiar in lower dimensions. In other words, it is flexible enough to allow the construction of complicated examples, but still lacks the tools required to reduce them to simpler situations. Many questions are still open in dimension $4$ and a long-term goal is to use Ricci flow to address some of them. 

Dimension 4 is a threshold dimension from both a geometric point of view and a PDE point view as well. From a geometric perspective, the fixed points of Ricci flow, Einstein metrics, are locally rigid in dimensions $2$ and $3$ while they exist in abundance in dimension $5$ and up. It is unclear if Einstein metrics should be considered abundant or not in dimension $4$. From a PDE perspective, in both Ricci flow and the Einstein equation, the PDE satisfied by the curvature features a quadratic nonlinearity which makes the equations critical in dimension $4$.

After the success of the Hamilton-Perelman program in dimension 3, and the subsequent development of a canonical Ricci flow through singularities by Bamler-Kleiner-Lott \cite{kl17, bk22}, dimension 4 is now the next horizon for Ricci flow. In this direction, Bamler's important work in higher dimension \cite{Bam21a,Bam21b,Bam21c,Bam23} provides a foundation. 
\\

Up to quotients, the only nontrivial shrinking Ricci solitons currently known to be \emph{stable} (i.e. generic) are:
\begin{itemize}
    \item $\mathbb{S}^3 \times \mathbb{R}$, which corresponds topologically to the inverse of a connected sum,
        \[
        \mathbb{S}^3 \times B^1 \longrightarrow  B^4 \times \mathbb{S}^0;
        \]
    \item $\mathbb{S}^2 \times \mathbb{R}^2$, which is believed to correspond to the handle surgery,
        \[
         \mathbb{S}^2 \times B^2 \longrightarrow B^3 \times  \mathbb{S}^1.
        \]
\end{itemize}
These happen to be the simplest topological operations, and it is striking that Ricci flow, a nonlinear PDE, recognizes them as the most basic and generic surgeries.

The next simplest topological surgery in dimension four is the \emph{complex blow-up}, i.e. the connected sum with either $\mathbb{CP}^2$ or $\overline{\mathbb{CP}}^2$, depending on orientation. By Freedman’s classification theorem, every simply-connected, non-spin $4$-manifold is homeomorphic to a connected sum of copies of $\mathbb{CP}^2$ and $\overline{\mathbb{CP}}^2$. One might have expected $\mathbb{R} \times \mathbb{S}^3$ and the Fubini-Study metric on $\mathbb{CP}^2$ or $\overline{\mathbb{CP}}^2$ to serve as the solitons associated with this connected sum operation. However, the Fubini-Study metric has been shown to be \textit{unstable} \cite{kro15, ks19}, and perturbations of it develop an FIK blowdown type I singularity (with Kähler structure in the orientation opposite to that of the initial Fubini-Study metric) in finite time \cite{gikw24}. 

The results of the present article indicate instead that Ricci flow generically performs this topological surgery as a single \emph{blowdown} operation, corresponding to the blowdown soliton constructed in \cite{fik}. In summary, along with the bullet points above, we have:
\begin{itemize}
    \item $\mathcal{O}(-1)$ (or $\mathcal{O}(1)$ in the opposite orientation), which corresponds to the blowdown surgery, 
    \[
    B^2(\mathcal{O}(- 1)) \longrightarrow B^4,
    \]
\end{itemize}
where $B^2(\mathcal{O}(- 1))$ is the disk bundle given by considering a tubular neighborhood of the exceptional divisor  $\mathbb{S}^2 \cong \mathbb{CP}^1$ in the tautological line bundle $\mathcal{O}(-1)$.

\subsubsection{Stable shrinking solitons and generic topological surgeries along Ricci flow}

As a geometric heat flow, Ricci flow locally homogenizes the geometry. Its global behavior can, however, be very different: it may develop finite-time singularities at which point the flow cannot be naively continued. The key obstruction to geometric and topological applications of Ricci flow is to understand these finite-time singularity models and develop a tractable geometric surgery procedure that lets one restart the flow. These geometric surgeries correspond to topological surgeries, and their classification led to our current understanding of $3$-dimensional geometry and topology. 

In dimension $3$, finite-time singularities are modeled by ancient $\kappa$-solutions. These are Ricci flows which have infinite back history, are noncollapsed, and have bounded nonnegative curvature. Among such flows, shrinking and steady gradient \textit{Ricci solitons} feature most prominently. Hamilton showed early on \cite{Ham95} how these solitons arise as special blow-up limits at singular times of Ricci flow. Later, Perelman's qualitative understanding of ancient $\kappa$-solutions in dimension 3 used a classification of Ricci shrinking solitons. Indeed, Perelman showed that an ancient $\kappa$-solution has associated with it an asymptotic Ricci shrinking soliton obtained via a blowdown procedure. The original blow-up and blow-down procedures of Hamilton and Perelman converged in a smooth sense, but these procedures are now known to hold quite generally in several weak senses, see \cite{nab10,emt11,Bam21c}.  

In dimension $4$, a full classification of ancient flows appears hopeless \cite{doorb} and the classification of Ricci solitons is distant, since even classifying Einstein $4$-manifolds seems completely out of reach. Since almost all known examples are unstable, it has been proposed to narrow the classification of singularity models to \textit{stable} ones which should intuitively appear generically: a Ricci flow starting from a generic perturbation of initial data should only encounter stable Ricci solitons. The question of the uniqueness of the flow through such singularities, valid up to dimension $3$ \cite{bk22}, but false from dimension $5$ and above \cite{ak22}, remains open in dimension $4$. 

A long-term goal would be a generic version of Perelman's canonical neighborhood theorem in dimension $4$. It would loosely say that, \textit{generically}, a region of \textit{high curvature} along Ricci flow must look like a region of a \textit{stable} shrinking or steady soliton. The optimistic hope is that there should only be a short list of such solitons. Additionally, the stability of solitons gives better analytical properties to the Ricci flow, as used in \cite{bk22}. 
\\

Our purpose in the present article is to prove that the well-known blowdown soliton constructed in \cite{fik} must belong to the list of stable shrinking Ricci solitons. From a topological point of view, this says that the \textit{blowdown}, which is the opposite of a complex blow-up is a \textit{generic topological operation} performed by Ricci flow. In the specific case of the blowdown soliton, a geometric surgery is proposed in \cite{fik}: there is a Ricci flow through a blowdown singularity that starts with the shrinker for negative times, reaches its asymptotic cone at $t=0$, and then, as proposed in \cite{fik}, continues using an expander asymptotic to the same cone. This is topologically a blowdown since a $(-1)$-sphere has been deleted from the manifold. Of course, with the opposite orientation, this corresponds to a $(+1)$-sphere. If, as expected, the limit at the singular time is a manifold with conical singularities, then the procedure is described in \cite{GS} since the asymptotic cone of the blowdown soliton has nonnegative curvature operator. A treatment of the stability of the associated expander remains an interesting  problem.

\subsubsection{The broad conjectural picture}
    In summary, a conjectural way of thinking about the (oriented) Ricci flow in dimension 4 is that (up to taking quotients) there are (thus far) generically four topological surgeries modeled on collapsing spheres: 
    \begin{itemize}
        \item One is given by a collapsing $\mathbb{S}^4$. This happens when a 4-sphere vanishes in a round point.
        \item One is given by a collapsing $\mathbb{S}^3$. This is modeled on the trivial line bundle $\mathbb{S}^3 \times \mathbb{R}$ over $\mathbb{S}^3$ and models the well-studied neckpinch. 
        \item Finally, two are given by collapsing $\mathbb{S}^2$. These surgeries are modeled on the (real) plane bundles $\mathcal O(-1)$ (or $\mathcal{O}(1)$) and  $\mathcal O(0) = \mathbb{S}^2 \times \mathbb{R}^2$ over $\mathbb{S}^2$. The trivial bundle $\mathcal{O}(0)$ models a bubble sheet singularity, while $\mathcal{O}(\pm 1)$ model blowdown surgeries (in opposite orientations).
    \end{itemize} 
    Of course, such a picture would require proving that $\mathbb{S}^4, \mathbb{S}^3 \times \mathbb{R}$, $\mathbb{S}^2 \times \mathbb{R}^2$ and $\mathcal{O}(-1)$, the Blowdown soliton, are the only stable (nonflat, simply-connected) gradient shrinking solitons in dimension 4. A complete classification remains a difficult and important open problem. 
    
    Such a classification would already be interesting for its analytical implications for Ricci flow. Of course, one might hope it would have further analytic and topological consequences for 4-manifolds in general. However, even with a classification of stable solitons in hand, many challenges remain and the extent of subsequent topological and analytic applications of the Ricci flow remains unclear, given the complexity of 4-manifolds. A natural next step in this conjectural program would be to study the Ricci flow of simply-connected 4-manifolds, perhaps with the aim of recovering Freedman's classification.

\subsubsection{The blowdown soliton}

Discovered by Feldman-Ilmanen-Knopf \cite{fik}, the $4$-dimensional \emph{blowdown soliton} $(M,g,f)$ is a gradient shrinking Ricci soliton on the complex blow-up of $\mathbb{C}^2$, $\mathrm{Bl}_1 \mathbb{C}^2$, or equivalently on $M=\mathbb{R}^4\#\overline{\mathbb{CP}}^2$.  
It solves the equation
\[
\Ric(g) + \nabla^{2}f=\tfrac12 g.
\]
This soliton is \emph{asymptotically conical}, converging under rescaling to a Kähler cone with positive scalar curvature. It models the blow‑down of an exceptional divisor along Kähler–Ricci flow and serves as the canonical Type I singularity model. It was first observed as a singularity model along a compact $4$-dimensional Kähler-Ricci flow in \cite{max14,sw11}, along a non-Kähler Ricci flow in \cite{gikw24} and in a noncompact setting in \cite{hug24}. It is one of the few known real gradient shrinking Ricci solitons and will be important in understanding the analytical, dynamical, and topological properties of Ricci flow in dimension $4$. It is the unique \textit{Kähler} Ricci soliton with its topology \cite{cds24,lw23}.

\subsubsection{Earlier work on stability of solitons}

Stability along Ricci flow can mean several things. 
\begin{enumerate}
\item \textit{Linear stability.} There are various notions of linear stability:
    \begin{enumerate}
    \item The second variation of one of Perelman's functionals is negative, i.e. up to second order perturbations, Perelman's functionals are indeed optimized at the fixed point. In the context of the blowdown soliton, this is our main Theorem \ref{thm:main}.
        
        This problem has a long history, for Einstein metrics \cite{koi78,you83,dww05,dww07,kro16,delay1,bo23,delay2}, and for Ricci solitons \cite{chi04,hall-murphy,cz12,kro15,ch15}.
\item The linearization of Ricci flow at a fixed point flows back to zero; see Corollary \ref{cor:Liouville thm}. These often take the form of Liouville theorems for the (rescaled) Ricci flow on the singularity models \cite{bk17,bk22,doorb}.
    \end{enumerate}
    \item \textit{Variational stability.} Perelman gave a variational structure to Ricci flow: it is the gradient flow of his geometric $\nu$-functional. The variational stability of a fixed point corresponds to locally optimizing the suitable Perelman's $\nu$-functional; see Corollary \ref{cor: decrease nu}. 
    \item \textit{Dynamical instability.} Ricci flow is a dynamical system, and Ricci flow should ``flow back'' to its stable fixed points. This is often the most difficult type of stability to prove, and has essentially only been proven for spherical metrics \cite{ham82,ham86,ham97}, see also \cite{CM} regarding cylinders.
\end{enumerate}

Until the present article, apart from the round sphere and round cylinders, there were no known examples of linearly stable shrinking Ricci solitons. By analogy with the mean curvature flow in codimension 1, one might have conjectured that there should not be any other examples.

To the authors' knowledge, all of the proofs of stability of Riemannian Einstein metrics or Ricci solitons rely on Weitzenböck formulae and Bochner's technique. This strategy is hopeless in the case of shrinking Ricci solitons whose Lichnerowicz Laplacians always have positive eigendirections: their Ricci tensors. In the case of the blowdown soliton, the difficulty is accentuated by the presence of several other positive eigendirections. The present proof necessarily considers and decomposes \textit{all} possible linear perturbations of the soliton.

In \cite{chi04}, the authors introduced the notion of \textit{central density} for gradient Ricci shrinking solitons. When $(M, g, f)$ is a Ricci shrinker, the central density is given by $\Theta(M) = e^{\nu(g)}$. The intuition from \cite{chi04} is that because $t \mapsto \nu(g_t)$ is nondecreasing along the Ricci flow,  the central density is an indirect measure of stability: the larger the value, the more stable the soliton should be. This was observed with all the known examples, but the stability of the blowdown soliton was unknown until now. By the main result of this article, the blowdown soliton is now the stable  soliton with the lowest entropy (currently known); see Table \ref{tab:page13}, and equation \eqref{eq:central-density} for the exact central density of the blowdown soliton.

In \cite{NO}, we show the recently discovered BCCD soliton \cite{bccd} - the last K\"ahler Ricci soliton in dimension four \cite{cds24, lw23} - is linearly unstable. We do so via the approach of \cite{hall-murphy}, used to show the instability of compact K\"ahler Ricci shrinking solitons, to the noncompact setting. 
\begin{table}[h]
    \centering
    \begin{tabular}{|c|c|c|c|}
        \hline
        \textbf{Soliton} & \textbf{Approximate Central Density} & \textbf{Stability} \\ 
        \hline
        $\mathbb{R}^4$ & 1 & {\color{teal}Stable} \\ 
        \hline
        $\mathbb{S}^4$ & 0.812 & {\color{teal}Stable} \\ 
        \hline
        $\mathbb{S}^3\times \mathbb{R}$ & 0.791 & {\color{teal}Stable} \\ 
        \hline
        $\mathbb{S}^2\times \mathbb{R}^2$ & 0.736 & {\color{teal}Stable} \\ 
        \hline
        \textbf{Blowdown} $\mathrm{Bl}_1\mathbb{C}^2$ & 0.672 & \textbf{{\color{teal}Stable}} \\ 
        \hline
        $\mathbb{CP}^2$ & 0.609 & {\color{purple}Unstable} \\ 
        \hline
        \textbf{BCCD} $\mathrm{Bl}_1(\mathbb{CP}^1 \times \mathbb{C})$& 0.562 & {\color{purple}\textbf{Unstable}} \\ 
        \hline
        $\mathbb{S}^2\times\mathbb{S}^2$ & 0.541 & {\color{purple}Unstable} \\ 
        \hline
        Other known compact ($\pi_1= \{1\}$) & $< 0.541$ & {\color{purple}Unstable} \\ 
        \hline
    \end{tabular}
    \caption{Central densities and stability of known smooth simply-connected solitons \cite{chi04}. Note that $\mathbb{CP}^2$ is actually weakly linearly stable, but dynamically unstable \cite{kro15, ks19}. In bold are the contributions of the authors, \cite{NO}. }
    \label{tab:page13}
\end{table}

\subsubsection{Comparison with the theory of generic mean curvature flow}

Let us briefly recall a key step of the proof that spheres and round cylinders are the only stable self-shrinkers in codimension 1 mean curvature flow (MCF).

In all codimension, self-shrinkers in MCF are critical points of the Gaussian area, $F_{x_0, t_0}(\Sigma)=(4\pi t_0)^{-\frac{n}{2}} \int_{\Sigma} e^{-|x-x_0|^2/(4t_0)}$. More generally, they are critical points of the entropy for MCF $\lambda(\Sigma) = \sup_{x_0, t_0} F_{x_0, t_0}(\Sigma)$, which is the cousin of Perelman's entropy $\nu$ studied in this work. As with $\nu$ along Ricci flow, $\lambda$ is monotone along the MCF (though $t\mapsto\lambda(\Sigma_t)$ decreases while $t\mapsto\nu(g_t)$ increases). A stable self-shrinker in MCF is typically understood to be a local minimizer of the entropy functional. However, Gaussian area stability, called $F$-stability, turns out to be easier to work with and implies entropy stability \cite{cm}.  

Now, the stability operator of the Gaussian area, $L$, acts on normal-valued vector fields. In any codimension, the mean curvature vector is an eigenvector of the stability operator, $L \vec H = \vec H$. This identity captures the fact that shrinking decreases the Gaussian area.  It is the cousin of an analogous identity, $L_f \Ric = \Ric$, in Ricci flow.

In codimension $1$, all normal-valued vector fields can be expressed as multiples of a choice of unit normal vector, $X = u \nu$. (This also has the added benefit of fixing the gauge, a much more difficult problem in Ricci flow.) In particular, via $\vec{H} = - \nu H$, the mean curvature is captured by a scalar quantity, and in this setting, the action of the stability operator $L$ reduces to its action on \textit{purely scalar} deformations. This opens up the opportunity to exploit the sign of $H$, which leads to the classification of stable self-shrinkers. If $H$ changes sign, it could not be associated with the first eigenvalue of $L$. Whenever this occurs, it turns out the actual first eigenfunction provides a destabilizing perturbation of the shrinker. Finally, mean-convex ($H \geq 0$) self-shrinkers have been shown to be the round sphere and the round cylinders \cite{hui84, cm}.

This route to classification of stable shrinkers is unavailable in higher codimension MCF and Ricci flow. Indeed, $\vec{H}$ has no sign and in Ricci flow, $L_f \Ric = \Ric$, is a tensor equation (an elliptic system), so one cannot conclude that $\Ric\geq 0$ even if $1$ is the first eigenvalue of $L_f$. In fact, it turns out that the Ricci curvature of the blowdown soliton is negative in some directions. Thus, the classification of stable singularity models of the Ricci flow will be much more involved than that of the codimension $1$ MCF. 

It is known that the sphere and the round cylinders are stable in higher codimension MCF. This article raises a natural question: \\

\textbf{Question:} \textit{Are there non-cylindrical stable self-shrinkers of the mean curvature flow?}\\

It is often said that Ricci flow in dimension $2n$ is a cousin of (higher codimension) MCF in dimension $n$. It would be very interesting to find a non-cylindrical stable self-shrinking surface (in any codimension). In sufficiently high dimension and codimension, the answer the above question is assuredly yes - although no examples are known.

\subsection{Statement of the main results}

In this article, we show that the blowdown soliton $(M, g, f)$ of \cite{fik} is linearly stable. This answers a question that has been open since its discovery and the notion of stability studied here was introduced \cite{chi04}. Some evidence towards the stability of this soliton among $U(2)$-invariant perturbations was given in \cite{iks19}. 

\begin{theorem}\label{thm:main}
The second variation of Perelman's $\nu$-functional at the $4$-dimensional FIK shrinking soliton $(M, g, f)$ is nonpositive under deformations that lie in $C^{\infty}(M) \cap H^1_f(M)$. Moreover, it is strictly negative under deformations orthogonal to both the Ricci tensor and the action of the diffeomorphism group. 
\end{theorem}

Here $h\in H^1_f(M)$ if $h, \nabla h\in L^2_f:= L^2(e^{-f}d\mu_g)$. 

\begin{remark}

    While proving Theorem \ref{thm:main}, we obtain information of independent interest about the bottom of the spectrum of the weighted Lichnerowicz Laplacian $L_f$ of the blowdown soliton. We find: 
    \begin{itemize}
        \item the eigentensor $\Ric$ is associated to the eigenvalue $1$;
        \item the four explicit eigentensors $S^1_{M,M'} = \nabla^2 (\hat{u} D^1_{M, M'})$, $M,M'\in\{-1,1\}$ defined in Section \ref{sec:nonnegative eigenvalues} are associated to the eigenvalue $1-\frac{1}{\sqrt{2}}$;
        \item the four explicit eigentensors $\nabla^2f$ and $S^2_{0,M'} = \nabla^2 (\hat{v} D^2_{0, M'})$ with $M'\in\{-2,0,2\}$ defined in Section \ref{sec:nonnegative eigenvalues} are associated to the eigenvalue $0$;
        \item all of the other eigenvalues of $L_f$ are \textit{strictly negative}, even if they are not in $\ker \operatorname{div}_f\cap\, \Ric^\perp$. See Remark \ref{rem:all-eigenvalues}.
    \end{itemize}
\end{remark}

Theorem \ref{thm:main} has direct dynamical consequences. First, a Liouville theorem for ancient  (modified) Ricci flows on the blowdown soliton. In this direction, another notion of linear stability comes from the dynamics of the linearization of the normalized Ricci flow
$$\partial_t g = -2\big(\Ric + \nabla^2 f - \tfrac12 g\big),$$
for which the FIK soliton $(M,g,f)$ is a stationary solution. After an infinitesimal reparametrization and rescaling, we can fix the gauge and project orthogonally to the Ricci direction for a perturbation $(h_t)_{t < 0}$ of the metric, so that $\operatorname{div}_f h_t=0$ and $h_t\perp_{L^2_f}\Ric$. As a result, the parabolic operator $\partial_t h_t = L_f h_t$ is the relevant linearized flow. We show that $\partial_t h_t = L_f h_t$ cannot escape the fixed point $h_t\equiv 0$. 

\begin{corollary}\label{cor:Liouville thm}
    Let $(h_t)_{t\in(-\infty,0] }$ be a family of symmetric $2$-tensors on the FIK shrinking soliton $(M, g, f)$ that is an ancient solution of the heat equation 
    $$ \partial_t h = L_fh.$$
    Assume $h_t$ satisfies the following properties for all $t$:
    \begin{enumerate}
        \item $\operatorname{div}_fh_t=0$,
        \item $h_t\perp_{L^2_f}\Ric$, and
        \item $\int_M |h_t|^2e^{-f}d\mu_g \leq 1$.
    \end{enumerate}
    Then $h_t$ vanishes identically. 
\end{corollary}

Such results have been crucial for questions of dynamical stability, uniqueness and gluing construction, see for instance \cite[Theorem 9.8]{bk22}, \cite[Proposition 5.2]{bk17} or \cite[Section 5]{doorb}.

Second, we additionally prove a nonlinear result: sufficiently small gauged perturbations of the blowdown soliton decrease Perelman's $\nu$-entropy.
\begin{corollary}\label{cor: decrease nu}
    Let $h$ be a symmetric $2$-tensor on the FIK shrinking soliton $(M,g, f)$ with uniformly bounded first two derivatives and satisfying $\operatorname{div}_fh=0$. Then, for all $\varepsilon>0$ sufficiently small,
    \[
    \nu(g+\varepsilon h)\leq \nu(g).
    \]
\end{corollary}

\subsection{Strategy of proof}

Our proof of Theorem \ref{thm:main} relies on a number of new results, ideas and techniques. The main difficulty is that the weighted Lichnerowicz operator $L_f = \Delta_f + 2R$ has several nonnegative eigendirections. We identify them as the Ricci tensor and Lie derivatives.

\subsubsection{Explicit parametrization of the blowdown soliton}

Our article starts with a totally explicit description of the blowdown soliton in dimension $4$. In the spirit of \cite{nw} (where the authors were first introduced to the metric ansatz used here), the blowdown metric is expressed in coordinates with rational coefficients from which all of the relevant geometric quantities can be readily computed.

\subsubsection{Radially symmetric deformations}

We then first focus on \textit{radially symmetric} deformations of the 4D blowdown soliton, which in particular include $U(2)$-invariant deformations. The subtle point is that since $\Ric$ satisfies $L_f\Ric = \Ric$ and $\mathrm{div}_f\Ric = 0$, no usual Bochner argument can work, and we must concern ourselves with the \textit{second} eigenvalue of $L_f$ among gauged $2$-tensors. Among non-gauged $2$-tensors, one also needs to deal with the radial gauge perturbation $\nabla^2f$ satisfying $L_f\nabla^2f=0$. Our approach is to carefully parametrize all gauged radial $2$-tensors and a conduct a fine analysis of the resulting Sturm-Liouville operators. The second eigenvalue is proven to be negative using Barta's technique.  

\subsubsection{High frequency modes}

In order to deal with nonradial perturbations of the metric, we use Fourier series on the Berger sphere cross sections of the soliton, exploiting the $U(2)$-invariance of the metric. We decompose arbitrary tensors in a suitable basis of $2$-tensors, and then further decompose their coefficient functions into Fourier series, specifically using Wigner functions, for each radial value. We thoroughly determine the subspaces preserved by the operator $L_f$, showing such subspaces are always parametrized by at most four radially symmetric functions. On each such subspace $L_f = \Delta_f + P$, where $\Delta_f$ is the radial weighted Laplacian, and $P$ a zeroth-order Hermitian operator. 

Using coarse estimates, we show that on the subspaces of high enough frequency, $P$ is negative, verifying the intuition that the higher the frequency, the more stable the deformation. For a small finite number of lower frequency modes, we study $P$ more directly. Proving that $P$ is negative is straightforward, but somewhat computational -- see Remark \ref{rem:sign-justification}.

\subsubsection{Remaining low frequency modes}

Unfortunately, the above method is not enough for all lower frequencies. This is a genuine difficulty: on top of the radially symmetric deformations $\Ric$ and $\nabla^2 f$, the weighted Lichnerowicz Laplacian $L_f$ turns out to have $7$ more nonnegative eigendirections. We identify this $7$-dimensional set as pure gauge perturbations. We finally prove that the next eigenvalue of $L_f$ is negative by comparison with Sturm-Liouville operators whose second eigenvalues can be estimated through Barta's technique again.

\subsection{Organization of the article}

The article is organized as follows. We start by presenting the blowdown soliton in coordinates and discussing its stability and algebraic properties in Section \ref{sec: soliton and stab}. We then restrict our attention to stability in the direction of radial perturbations in Part \ref{part:radial stability}. We then prove the stability of higher frequency perturbations in Part \ref{part:higher stability}. The computations of Christoffel symbols, curvatures, actions of $L_f$, $\mathrm{div}_f$, and background on Wigner functions are delegated to the Appendices in Part \ref{part:appendix}.

\subsection{Acknowledgements}

During this project, TO was partially supported by the National Science Foundation under grant DMS-2405328 and KN was partially supported by the National Science Foundation under grant DMS-2103265. The authors thank Louis Yudowitz for pointing out to us that the asymptotic cone of the blowdown soliton had positive curvature operator, and that the results of \cite{GS} could be applied to manifolds with associated conical singularities.

\section{The Blowdown soliton metric and the linear stability problem}\label{sec: soliton and stab}

In this section, we introduce the four dimensional blowdown soliton discovered by Feldman-Ilmanen-Knopf and also known as FIK Ricci soliton metric \cite{fik}. We set up the linear stability problem and we introduce a natural basis of (symmetric) $2$-tensors that we will use in establishing the linear stability of the FIK shrinking soliton.

\subsection{The blowdown Ricci shrinking soliton}
Let $M = \mathcal{O}(-1)$ denote the tautological line bundle over $\mathbb{CP}^1 \cong \mathbb{S}^2$. Let $R > 0$ be a scale parameter with units of distance. By thinking of $M$ as the blow-up of $\mathbb{C}^2$ at one point, we may identify $M$ with $ \mathbb{S}^2 \sqcup (\mathbb{S}^3 \times (R, \infty))$ and introduce a radial coordinate $r : M \to [R, \infty)$ by projection. We let $\mathring{M} = M \setminus \Sigma$ and $\theta : \mathring{M} \to \mathbb{S}^3$ denote variable on the $\mathbb{S}^3$, where $\Sigma := r^{-1}(R) = \mathbb{S}^2$ denotes the exceptional divisor. We let $X_1, X_2, X_3$ denote the usual left-invariant frame on $\mathbb{S}^3$ and let $\eta_1,\eta_2, \eta_3$ denote the left-invariant dual coframe, with ordering chosen so that $X_1$ is tangent to the Hopf fiber of $\mathbb{S}^3$. See the Appendix Section \ref{app:3} for a more precise description of these frames. 

We define radially symmetric functions $C, F, f : M \to \mathbb{R}$ by 
\begin{align}
\nonumber C &:= 4\frac{r^2}{R^2}, \\
F &:= \frac{1}{\sqrt{2}} - (\sqrt{2} -1) \frac{R^2}{r^2} - \frac{1}{\sqrt{2}}(\sqrt{2} -1) \frac{R^4}{r^4},  \label{eq:FIK-functions}\\
\nonumber f & := \sqrt{2} \Big(\frac{r^2}{R^2}-1\Big) - \log(2(\sqrt{2}-1)). 
\end{align}
In the remainder of this article, we will take the scale parameter $R = 1$. The FIK metric on $M\setminus \Sigma$ is given by
\begin{equation}\label{eq:FIK-metric}
g := C \Big( \frac{1}{F r^2} dr^2 + F \eta_1^2 + \eta_2^2 + \eta_3^2\Big).
\end{equation}
\begin{remark}
This sort of parametrization of the FIK metric seems to have been acknowledged previously in the literature, but has not (to our knowledge) been used to study the FIK shrinking soliton in detail. In particular, the FIK shrinking soliton metric has usually been studied through its K\"ahler potential, which satisfies a second order ODE and is harder to write down. Our own discovery of this parametrization came through \cite{nw}, where an ansatz of this form was used to study special $U(2)$-invariant metrics.  
\end{remark}
In our setup, we think of $F$ as dimensionless, while $C$ and $f$ have units of distance squared.  Going forward, we label an often appearing constant 
\begin{equation}\label{eq:c0}
c_0 := \sqrt{2}-1,
\end{equation}
With this, we observe that 
\[
F(1) = 0, \qquad  F'(1) = 2,  \qquad F''(1) = -14 + 4 \sqrt{2}, \qquad e^{-f(1)} =2c_0.
\]
The conditions on $F$ ensure the metric $g$ extends smoothly to all of $M$ (see Section \ref{sec:extending-FIK-to-tip}). The function $F$ is concave and plotted in Figure \ref{fig:1}. 
\begin{figure}
\includegraphics[scale=0.7]{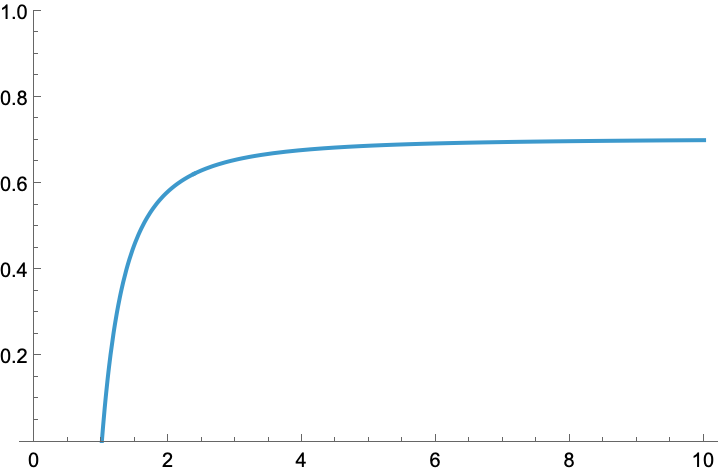}
\caption{Plot of $F(r)$ with $R = 1$.}\label{fig:1}
\end{figure}
Since they will appear in many computations, we note here the identities 
\begin{align}
    F &= \frac{1}{\sqrt{2}} -\frac{c_0}{r^2} - \frac{c_0}{\sqrt{2} r^4} = \frac{(r^2-1)(r^2 + c_0)}{\sqrt{2}r^4}, \label{eq:Fder1}\\
    F' &= \frac{2c_0}{r^3} + \frac{2\sqrt{2}c_0}{r^5}, \label{eq:Fder2}\\
    F'' &= -\frac{6c_0}{r^4} - \frac{10\sqrt{2}c_0}{r^6}\label{eq:Fder3},
\end{align}
We also introduce a function $s$ by the identity
\begin{equation}\label{eq:s}
s^2 := Fr^2 =\frac{r^2}{\sqrt{2}} - c_0   - \frac{c_0}{\sqrt{2}} \frac{1}{r^2}.
\end{equation}

Our definitions of $C, F, f$ are normalized so that 
\[
\mathrm{Ric} + \nabla^2 f = \frac{1}{2} g, \qquad \qquad  \frac{1}{(4\pi)^2}\int_M e^{-f} d\mu_g = 1,
\]
and
\[
\mathrm{scal} + |\nabla f|^2 - f = \log(2c_0) +\sqrt{2} c_0.
\]
The first two equations are standard in the study of Ricci shrinking solitons, while the last is a little unconventional (often the right-hand side is taken to be zero). Here our choice is to fix the volume constraint. The first and the third of these equations follow readily from Corollary \ref{cor:ricci-scalar} and Proposition \ref{prop:hessian-f} in the Appendices. The volume constraint is a consequence of \eqref{eq:volume-density} below and our definition of $f$:
\[
\int_M e^{-f} d\mu_g = 64c_0\pi^2 \int_1^\infty r^3 e^{-\sqrt{2}(r^2-1)}\, dr = (4\pi)^2.
\]
We point out here that by Lemma \ref{lem:lap-radial}, we have  
\begin{equation}\label{eq:lap-radial-coordinate}
    \Delta_f r^2 = 2-r^2,
\end{equation}
which is essentially a combination of the identity for  $\mathrm{scal} + |\nabla f|^2 - f$ and the trace of the soliton equation, $\Delta f + \mathrm{scal} = 2$.  

On $\mathring{M} = M \setminus \Sigma$, we have a global oriented orthonormal frame 
\begin{equation}\label{eq:frame}
e_0 := \frac{s}{2r} \partial_r,\quad  e_1 := \frac{1}{2s} X_1, \quad e_2 := \frac{1}{2r}X_2,\qquad e_3 := \frac{1}{2r}X_3,
\end{equation}
and a global oriented orthonormal coframe
\begin{equation}\label{eq:coframe}
e^0 := \frac{2r}{s} dr,\quad  e^1 := 2 s \,\eta_1, \quad e^2 := 2r\, \eta_2,\qquad e^3 := 2r \,\eta_3. 
\end{equation}
Then, of course
\[
g = e^0 \otimes e^0 + e^1 \otimes e^1 + e^2 \otimes e^2 + e^3 \otimes e^3.
\]
Because of the $U(2)$ invariance, we will find it useful to introduce 
\begin{equation}
e_-:= \frac{e_2 - e_3}{\sqrt{2}}, \qquad e_+ := \frac{e_2 + e_3}{\sqrt{2}},
\end{equation}
and let $e^+, e^-$ denote the coframe duals of these vector fields. Of course $\{e_0, e_1, e_-, e_+\}$ is an oriented orthonormal frame with dual $\{e^0, e^1, e^-, e^+\}$. 

Recall that the Hodge star in dimension 4 acting on $\Lambda^2$ has two eigenbundles with eigenvalues $\pm1$, called the selfdual and anti-selfdual bundles of $2$-forms, respectively. The spaces of selfdual and anti-selfdual 2-forms in the decomposition $\Lambda^2 = \Lambda^+_g \oplus \Lambda^-_g$ trivialize over $\mathring{M}$ with respect to the orthonormal bases
\begin{align}
\nonumber \omega_1^{\pm} := e^0 \wedge e^1 \pm e^2 \wedge e^3, \\
\omega_2^{\pm} := e^0 \wedge e^2 \pm e^3 \wedge e^1, \label{eq:duality-frame} \\
\nonumber \omega_3^{\pm} := e^0 \wedge e^3 \pm e^1 \wedge e^2.
\end{align}
Our convention is that $e^i \wedge e^j = \frac{1}{2}(e^i \otimes e^j - e^j \otimes e^i)$ and that $g(\alpha, \beta) = \alpha_{ij} \beta_{ij}$ in an orthonormal basis for $2$-forms $\alpha, \beta$, where here $\alpha_{ij} = \alpha(e_i, e_j)$. In particular, this implies $|\omega_a^\pm| = 1$.
The volume form for our metric is 
\[
\mu= \mu_g := e^0 \wedge e^1 \wedge e^2 \wedge e^3, 
\]
and we denote by $d\mu_g$ the volume density. We note that 
\begin{equation}\label{eq:volume-density}
e^{-f} d\mu_g =\frac{C^2}{r} e^{-f} \, dr \, d\mu_{\mathbb{S}^3}= 16 \, r^3 e^{-f} \, dr \, d\mu_{\mathbb{S}^3},
\end{equation}
and recall that $|\mathbb{S}^3| = 2\pi^2$.
The almost complex structure in our setting is 
\[
J_1^+ := e_1 \otimes e^0 - e_0 \otimes e^1 + e_3 \otimes e^2 - e_2 \otimes e^3.
\]
It can be checked that $g(J_1^+\cdot, \cdot) = 2 \omega_1^+$, so that the K\"ahler form corresponding to the FIK metric is $2\omega_1^+$. 

For the study of stability, we will need to compute the connection and curvatures of the FIK metric. For these computations, which are more or less straightforward in our explicit description of the metric, we refer the reader to Appendix \ref{app:1}.

\subsection{Setup for the stability of the blowdown soliton}

\subsubsection{The second variation of Perelman's entropy and stability} 

In this section, we introduce Perelman's entropy functional and the notion of linear stability studied in the paper. Many of the identities that follow have been established for compact shrinking solitons \cite{cz12,cz24}, but readily carry over to the FIK shrinking soliton. 

In what follows, we use the shorthand $L^2_f = L^2_f(M) := L^2(e^{-f} d\mu_g)$ and $H^k_f  = H^k_f(M):= W^{k,2}(e^{-f} d\mu_g)$ to denote the natural weighted Sobolev spaces on $(M, g, f)$. That is, given a tensor $u$ on $M$, 
\[
\|u\|_{L^2_f}:= \left(\int_M |u|^2_g \, e^{-f} d\mu_g \right)^{\frac{1}{2}} \quad \text{ and }\quad
\|u\|_{H^k_f}:= \left( ||u||_{L^2_f}^2 + \cdots + ||\nabla^k u||_{L^2_f}^2 \right)^{\frac{1}{2}}.\]
We often abuse notation and conflate Sobolev spaces (and $C^{\infty}$ spaces) of tensors of different types, writing $u \in C^{\infty}(M) \cap H^k_f(M)$ as a shorthand to mean $u$ is smooth and $\|u\|_{H^k_f} < \infty$.

Perelman's stability operator involves the drift Laplacian $\Delta_f := \Delta - \nabla_{\nabla f}$ and the weighted divergence operator $\mathrm{div}_f(\cdot) := e^f \mathrm{div}( e^{-f}\cdot)$. The following lemma is meant to give an elementary justification to integration by parts for these operators in our setting. 

\begin{lemma}\label{lem:integration-by-parts}
Suppose $u \in C^{\infty}(M) \cap H^2_f(M)$ and $v \in C^{\infty}(M) \cap H^1_f(M)$, then the integration by parts formula 
\[
\int_M v \Delta_f u \, e^{-f} \, d\mu_g =- \int_M \langle \nabla v, \nabla u\rangle\, e^{-f} d\mu_g 
\]
holds. More generally, if $X \in C^{\infty}(M) \cap H^1_f(M)$ is a vector field on $M$, then 
\[
\int_M  \mathrm{div}_f(X)e^{-f} d\mu_g = 0.
\]
\end{lemma}
\begin{proof}
Recall that $\mathrm{div}_f(v \nabla u) = v \Delta_f u + \langle \nabla v, \nabla u \rangle$. For $r_0 > 1$, the divergence theorem and Cauchy's inequality implies
\begin{equation*}
    \int_{\{r\leq r_0\}} (v \Delta_f u + \langle \nabla u, \nabla v \rangle ) e^{-f} \, d\mu_g = \int_{\{r = r_0\}} v \,\langle \nabla u, \frac{\nabla r}{|\nabla r|}\rangle   \, e^{-f} d\sigma_g  \leq \frac{1}{2} \int_{\{r = r_0\}} (v^2+|\nabla u|^2) \, e^{-f} d\sigma_g,
\end{equation*}
where $d\sigma_g = d\sigma_g(r_0)$ is the volume measure of the Berger sphere $\{r = r_0\} \subset M$. In particular, from our expression for the metric, we have
\[
d\mu_g = \frac{2r}{s} d\sigma_g(r)dr . 
\]
On the other hand, because $u \in H^2_r(M)$ and $v \in H^1_f(M)$, we have 
\[
\int_M (v^2 + |\nabla u|^2) \, e^{-f} d\mu_g = 2\int_1^\infty \frac{r}{s} B(r) \, dr < \infty
\]
where
\[
B(r_0) := \int_{\{r = r_0\}} (v^2 + |\nabla u|^2)\, e^{-f} d\sigma_g(r_0) \geq 0. 
\]
Because for $r$ large we have $\frac{r}{s}\approx 2^{1/4}$, it follows that there exists a sequence of radii $r_j \to \infty$ such that $B(r_j) \to 0$. Since above we have shown that
\[
\int_{\{r\leq r_j\}} (v \Delta_f u +\langle \nabla v, \nabla u \rangle ) e^{-f} \, d\mu_g  \leq \frac{1}{2} B(r_j),
\]
by sending $j \to \infty$, we conclude $\int_M (v \Delta_fu  + \langle \nabla v, \nabla u \rangle)e^{-f} d\mu_g = 0. $ This proves second identity in the special case $X= v \nabla u$, but same argument gives the general case. 
\end{proof}

It follows from work of Bakry-\'Emery, Morgan, and Hein-Naber \cite{BE, M, HN} that given a weighted Riemannian manifold $(M, g, f)$, when $(M,g)$ is compact or complete, $\int_M e^{-f} d\mu_g < \infty$, and $(M, g, f)$ has a log-Sobolev inequality, then the drift Laplacian $\Delta_f$ has a discrete spectrum
\[
0 = \lambda_0 < \lambda_1 \leq \cdots \to \infty.
\]
Moreover, if $\mathrm{Ric}_f = \mathrm{Ric} + \nabla^2 f \geq \frac{a}{2} g$, then $\lambda_1 \geq \frac{a}{2}$. The rigidity case was considered by Cheng-Zhou \cite{cz17} where they showed $\lambda_1 > \frac{1}{2}$ unless $(M, g)$ splits a Euclidean factor of dimension at least 1 (this step uses the weighted Bochner formula). Colding and Minicozzi also established the rigidity result among a wider class of evolving manifolds \cite{CM}.

These results apply in our study of the FIK shrinker and readily give the following lemma. 

\begin{lemma}\label{lem:vh}
    Suppose $f_0 \in C^{\infty}(M) \cap L^2_f(M)$ on the FIK shrinking soliton and satisfies $\int_M f_0 \, e^{-f} d\mu_g = 0$. There there exists a unique function $v \in C^{\infty}(M) \cap H^2_f(M)$ such that 
    \[
    \Delta_f v + \frac{1}{2} v = f_0, \qquad \int_M v e^{-f} d\mu_g = 0. 
    \]
\end{lemma}

\begin{proof}
Since the FIK shrinking soliton does not split a line, if $v \in C^{\infty}(M) \cap H^1_f(M)$ and $\int_M v \, e^{-f}d\mu_g = 0$, then we have the Poincar\'e inequality
\[
\int_M |\nabla v|^2 e^{-f} d\mu_g\geq \lambda_1  \int_M  v^2e^{-f} d\mu_g, 
\]
with $\lambda_1 = \lambda_1(\Delta_f)> \frac{1}{2}$.
Now the energy functional 
\[
E(v) := \int_M (|\nabla v|^2 - \frac{1}{2}v^2 + 2f_0 v) \, e^{-f} d\mu_g
\]
satisfies
\[
E(v) \geq  \epsilon \|\nabla v\|_{L^2_f}^2 + (1-\epsilon) \lambda_1 \|v\|_{L^2_f}^2 - \frac{1+\epsilon}{2} \|v\|_{L^2_f}^2 -  \frac{1}{2\epsilon} \|f_0\|_{L^2_f}^2
\]
for $\epsilon > 0$ and $v \in H^1_f(M)$.  Taking $\epsilon> 0$ sufficiently small so that $\lambda_1 > \frac{1}{2}\frac{1+\epsilon}{1-\epsilon}$ and minimizing $E$ over $v \in H^1_f(M)$ produces a weak solution to the desired problem. Now standard regularity theory shows that $v \in C^{\infty}(M) \cap H^2_f(M)$ and gives the a priori estimate $\|v\|_{H^2_f} \leq C \|f_0\|_{L^2_f}$. (See, for instance, Lemma \ref{lem:gauge fixing orthogonal decomp} below, where this is discussed in more detail). 
\end{proof}

We now proceed with the discussion of linear stability and refer the reader to \cite{cz12} for additional details. We begin by recalling that Perelman's entropy is given by 
\begin{equation}
\nu(g) := \inf\Big\{ \mathcal{W}(g, f, \tau) : \tau > 0, f \in C^{\infty}_0(M) \;\;\text{satisfying} \;\;(4\pi \tau)^{-\frac{n}{2}}\int_M e^{-f} d\mu_g = 1\Big\}
\end{equation}
where 
\begin{equation}
    \mathcal{W}(g, f, \tau) := (4\pi\tau)^{-\frac{n}{2}}\int_M [\tau(\mathrm{scal} + |\nabla f|^2) + f -n]  e^{-f} \, d\mu_g. 
\end{equation}
When $(M, g, f)$ is the FIK shrinker, one can compute that (see  Table \ref{tab:page13}) the central density of the shrinker is given by
\begin{equation}\label{eq:central-density}
    e^{\nu(g)} = e^{\mathcal{W}(g, f, 1)} = \frac{1}{2c_0} e^{-\sqrt{2} c_0} \approx 0.672.
\end{equation}

Consider now a symmetric $2$-tensor $h \in C^{\infty}(M)$. The stability operator associated to Perelman's entropy acting on deformations $h$ is given by 
\begin{equation}\label{eq:Nf}
N_f h := \frac{1}{2} L_f h + \mathrm{div}^\ast_f \mathrm{div}_f h + \frac{1}{2} \nabla^2 v_h - \Xi(h) \,\mathrm{Ric}
\end{equation}
where $\Delta_f := \Delta - \nabla_{\nabla f}$ and $\mathrm{div}_f(h) := \mathrm{div}(h)-h(\nabla f)$ are as introduced above; 
\begin{equation}
\mathrm{div}_f^\ast(X) :=-\frac{1}{2} \mathcal{L}_{X} g
\end{equation}
is the $L^2_f$-adjoint of the weighed divergence ($\mathcal{L}_Xg$ is the Lie derivative); $L_f$ is the weighted Lichnerowicz operator
\begin{equation}
L_fh := \Delta_f h + 2 R(h);
\end{equation}
$v_h$ is the (unique) solution (in $C^{\infty}(M) \cap H^2_f(M)$ if $h \in C^{\infty}(M) \cap H^2_f(M)$) of 
\begin{equation}\label{eq:def vh}
\Delta_f v_h + \frac{1}{2} v_h = \mathrm{div}_f \mathrm{div}_f h, \qquad \int_M v_h e^{-f}\, d\mu_g  = 0;
\end{equation}
and, finally, $\Xi(h)$ is the integration given by 
\begin{equation}
\Xi(h) := \frac{\int_M g(\mathrm{Ric}, h)\, e^{-f} d\mu_g}{\int_M \mathrm{scal} \, e^{-f} d\mu_g}.
\end{equation}

Given $h \in C^\infty(M) \cap H^2_f(M)$ and letting $g_s = g + s h$ for $s \in (-\delta, \delta)$ small, then first and second variations of Perelman's entropy in our setting are given by
\[
\delta\nu_g(h):= \frac{d}{ds} \Big|_{s =0} \nu(g_s) = -\frac{1}{16\pi^2}\int_M g\left(\mathrm{Ric} + \nabla^2 f - \frac{1}{2} g, h \right) \, e^{-f} d\mu_g = 0, 
\]
and 
\begin{equation}\label{eq:second-variation-H2}
\frac{d^2}{ds^2}\Big|_{s = 0} \nu(g_s) = \frac{1}{16\pi^2} \int_M g(N_f h, h) \, e^{-f} d\mu_g.
\end{equation}
Although the right hand side of the latter expression could rightly be used to define linear stability, an integration by parts allows one to define linear stability on a larger class of symmetric $2$-tensors as deformations. For $h \in C^{\infty}(M) \cap H^1_f(M)$, let us define 
\begin{align}\label{eq:second-variation-H1}
\delta^2 \nu_g(h) &:=  \frac{1}{16\pi^2} \int_M \left(R(h, h) -\frac{1}{2} |\nabla h|^2+ |\mathrm{div}_f h|^2 + \frac{1}{4} (v_h^2 - 2|\nabla v_h|^2) \right) e^{-f} d\mu_g \\
& \qquad - \frac{1}{16\pi^2}\Xi(h)^2 \int_M \mathrm{scal} \; e^{-f} d\mu_g.\nonumber 
\end{align}
Note that this quadratic form defines a natural bilinear form $\delta^2 \nu(h, \tilde{h})$ as in \eqref{eq:bilinear-d2nu} below. Observe that when integration by parts holds (for instance if $h$ is smooth and compactly-supported), then $\delta^2 \nu_g(h) = \frac{d^2}{ds^2}\big|_{s =0} \nu(g + sh)$. 
Obtaining this formula uses an integration by parts on the function $v_h$ (see Lemma \ref{lem:integration-by-parts}) that gives
\begin{align*}
    \int_M g(\nabla^2 v_h, h ) \; e^{-f} d\mu_g 
   & = \int_M v_h \;\mathrm{div}_f \mathrm{div}_f h \; e^{-f} d\mu_g 
    =  \int_M \left(\frac{1}{2} v_h^2 - |\nabla v_h|^2\right) \; e^{-f} d\mu_g.
\end{align*}
Because equation \eqref{eq:second-variation-H1} is more general than \eqref{eq:second-variation-H2} in the sense that it only requires $h \in H^1_f(M)$ instead of $H^2_f(M)$, we take it as our definition of the second variation $\delta^2 \nu_g(h)$, but use equality with \eqref{eq:second-variation-H2} whenever the integration by parts is justified.

Note that if we assume $h\in C^{\infty}(M) \cap H^1_f(M)$, then in \eqref{eq:second-variation-H1}, the function $v_h$ should be understood to be a weak solution in $C^{\infty}(M) \cap H^1_f(M)$ of $\Delta_f v_h + \frac{1}{2} v_h = \mathrm{div}_f\mathrm{div}_f h$, since now the right hand side only lies in $C^{\infty}(M)\cap H^{-1}_f(M)$. Recall the $H^{-1}_f$ norm is given by 
\[
\|\mathrm{div}_f(X)\|_{H^{-1}_f} = \sup_{\substack{\tilde{u} \in H^1_f \\\|\tilde{u}\|_{H^1_f} = 1}}\left|\int \langle X, \nabla \tilde{u} \rangle e^{-f} d\mu_g \right|\leq \|X\|_{L^2_f},
\]
and to say $v_h \in H^1_f(M)$ is a weak solution of this problem means that for all $u \in C^{\infty}_0(M)$ one has 
\[
\int_M \big(\langle \nabla v_h, \nabla u \rangle - \frac{1}{2}v_h u - \langle  \mathrm{div}_f h, \nabla u \rangle\big) \, e^{-f} d\mu_g = 0. 
\]
For $h \in C^{\infty}(M)\cap H^1_f(M)$, a result similar to Lemma \ref{lem:vh} asserts that $v_h \in C^{\infty}(M) \cap H^1_f(M)$ with $\|v\|_{H^1_f} \leq C \|\mathrm{div}_f\mathrm{div}_f h\|_{H^{-1}_f} \leq C \|\mathrm{div}_fh\|_{L^2_f}\leq C\|h\|_{H^1_f}$. 

\begin{definition}[Linear stability of a shrinking Ricci soliton]
    A gradient shrinking Ricci soliton $(M, g, f)$ is \textit{linearly stable} if for all $h\in H^1_f(M)$, 
\begin{equation}\label{eq:linear-stability-ineq}
\delta^2\nu_g(h) \leq 0. 
\end{equation}
\end{definition}

The sign for linear stability ($\frac{d^2}{ds^2}|_{s=0}\nu(g_s) \leq 0$) appears to be the opposite of what is usually expected in stability problems. This is because if $t \mapsto g_t$ is a Ricci flow then $t \mapsto \nu(g_t)$ is monotone \textit{increasing}. In other words, the flow tends back towards maximums of $\nu(g)$. The goal of this paper is to prove \eqref{eq:linear-stability-ineq} for all $h \in C^{\infty}(M) \cap H^1_f(M)$ on the FIK shrinking soliton. 

\subsubsection{Gauge-fixing}

The complexity of the expression for $N_fh$ is due to the action of the diffeomorphism group and the invariance of the Ricci shrinker equation under diffeomorphisms. Imposing the gauge condition $\mathrm{div}_f(h) = 0$ greatly simplifies the analysis of stability. We now outline the steps that we will take to show that the FIK shrinking soliton is linearly stable.
\begin{enumerate}   
    \item Firstly, we note that we only need to show $\delta^2 \nu_g(h) \leq 0$ assuming $h$ is a smooth and compactly supported symmetric $2$-tensor on the FIK shrinking soliton in Lemma \ref{lem:cpt supp enough}. The FIK shrinking soliton has bounded geometry (of all orders) in the sense that the injectivity radius is bounded below, $\mathrm{inj}(M,g) >0$, and the curvature and its derivatives are uniformly bounded above, $|\nabla^k R| \leq C_k$ for all $k$. By standard results for Sobolev spaces on such manifolds (see e.g. \cite[Proposition 3.2]{JJ}), compactly supported $2$-tensors are dense among $2$-tensors in $C^{\infty}(M) \cap H^1_f(M)$, and so the desired result will follow for all $h \in C^{\infty}(M) \cap H^1_f(M)$ by approximation. 
    \item Next, in Lemma \ref{lem:gauge fixing orthogonal decomp}, we show that every smooth, compactly-supported $2$-tensor $h$ can be expressed as an $L^2_f$-orthogonal sum 
    \[
        h = \hat{h} + \frac{\Xi(h)}{\Xi(\Ric)}\, \mathrm{Ric}+ \mathrm{div}_f^\ast Y
    \]
    of a $2$-tensor $\hat{h}$ in $C^{\infty}(M)\cap H^1_f(M)$ that satisfies $\mathrm{div}_f(\hat{h}) = 0$ and $\Xi(\hat{h}) = 0$, and a vector field $Y \in C^{\infty}(M) \cap H^2_f(M)$. 
    \item Then, in Lemma \ref{lem:stability-in-gauge}, we use that under such an orthogonal decomposition above, the second variation of Perelman's entropy in the direction $h$ is given by 
    \[
   \delta^2 \nu_g(h) = \delta^2\nu_g(\hat{h}) =  \frac{1}{32\pi^2} \int_M (2 R(\hat{h}, \hat{h}) - |\nabla \hat{h}|^2) \, e^{-f} d\mu_g.
    \] 
    \item Finally, the rest of the article will be devoted to showing that when $h \in C^{\infty}(M) \cap H^1_f(M)$ satisfies $\mathrm{div}_f(h) = 0$ and $\Xi(h) = 0$,  then 
    \[
    \int_M (2R({h}, {h}) - |\nabla {h}|^2)\, e^{-f} d\mu_g \leq 0.
    \]
\end{enumerate}

We remark that in Step (4), it is indeed sufficient to show that $\int_M (2R({h}, {h}) - |\nabla {h}|^2)\, e^{-f} d\mu_g \leq 0$ for \textit{any} $2$-tensor $h$. However, this inequality cannot hold for $h = \Ric$, so it is necessary to impose orthogonality to Ricci. We will see that the gauge equation is also needed in Section \ref{sec:nonnegative eigenvalues}. 

Step (4) above is the most important, longest, and most technical step in the proof of linear stability of the FIK shrinking soliton. It will rely on a fine decomposition of the tensor $h$ and the integrand $2R({h}, {h}) - |\nabla {h}|^2$ and the gauge condition $\operatorname{div}_fh=0$ according to tensor harmonics. We will use different techniques to deal with different harmonics. Steps (2) and (3) are well-known and standard in the study of linear stability of compact gradient Ricci shrinking solitons. In \cite{CM}, Colding and Minicozzi recently  studied Step (2) in the general noncompact setting, in their proof of rigidity of shrinking cylinders among gradient Ricci solitons in any dimension. Taking advantage of their quite general results, we review Steps (1), (2), and (3) for the FIK shrinking soliton in the remainder of this section. In Parts \ref{part:radial stability} and \ref{part:higher stability}, we will prove Step (4) by decomposing the problem into radial and nonradial deformations. 

\begin{lemma}\label{lem:cpt supp enough}
If $\delta^2 \nu_g(h) \leq 0$ for all compactly supported $h$ on the FIK shrinking soliton, then $\delta^2 \nu_g(h) \leq 0$ for all $h \in C^{\infty}(M) \cap H^1_f(M)$. 
\end{lemma}

\begin{proof}
Suppose $h \in C^{\infty}(M) \cap H^1_f(M)$. By density, we can find $h_k \in C^{\infty}_0(M)$ such that $h_k \to h$ with respect to the $H^1_f$-norm. The lemma follows if we can show continuity of the stability functional, $\delta^2 \nu_g(h_k) \to \delta^2 \nu_g(h)$. This is probably well-known, but for convenience we give a proof here. 

Towards that end, it is straightforward to see that the bilinear form 
\begin{align*}
    Q_1(h, \tilde{h}):= \int_M (R(h, \tilde{h}) -\langle \nabla h, \nabla \tilde{h} \rangle + \langle \mathrm{div}_fh, \mathrm{div}_f\tilde{h} \rangle)e^{-f} d\mu_g
\end{align*}
satisfies 
\[
Q_1(h, \tilde{h}) \leq C \|h\|_{H^1_f} \|\tilde{h}\|_{H^1_f}.
\]
On the other hand, if $v_h \in H^1_f(M)$ is a weak solution of $\Delta_f v_h + \frac{1}{2}v_h = \mathrm{div}_f \mathrm{div}_f h$, we have the estimate 
\[
\|v_h\|_{H^1_f} \leq C \|\mathrm{div}_f h\|_{L^2_f} \leq C \|h\|_{H^1_f}. 
\]
It follows that
\[
Q_2(h,\tilde{h}) := \frac{1}{4}\int_M \left( v_h v_{\tilde{h}} -2 \langle \nabla v_h, \nabla v_{\tilde{h}}\rangle \right)\, e^{-f} d\mu_g 
\]
satisfies 
\[
Q_2(h, \tilde{h}) \leq C \|v_h\|_{H^1_f} \|v_{\tilde{h}}\|_{H^1_f} \leq C \|h\|_{H^1_f} \|\tilde h\|_{H^1_f}. 
\]
Finally, we note that 
\[
\Xi(h) \leq C\|h\|_{L^2_f}.
\]
Putting these estimates together, implies that the bilinear form 
\begin{equation}\label{eq:bilinear-d2nu}
\delta^2 \nu_g(h,\tilde{h}) := \frac{1}{16\pi^2}\Big(Q_1(h, \tilde{h}) + Q_2(h, \tilde{h}) - \Xi(h) \Xi(\tilde{h})\int_M \mathrm{scal} \, e^{-f} d\mu_g\Big). 
\end{equation}
is bounded, hence continuous on $H^1_f$. In particular, since $\delta^2 \nu_g(h)= \delta^2 \nu_g(h,h)$, we conclude that if $h_k \to h$ in $H^1_f$, then $\delta^2 \nu_g(h_k) \to \delta^2 \nu_g(h)$. 
\end{proof}

\begin{lemma}\label{lem:gauge fixing orthogonal decomp}
    Consider a symmetric $2$-tensor $h \in C_0^{\infty}(M)$ on the FIK shrinking soliton. Then there exist a symmetric $2$-tensor $\hat{h} \in C^{\infty}(M) \cap H^1_f(M)$ and a vector field $Y \in C^{\infty}(M) \cap H^2_f(M)$ such that
    \begin{equation}\label{eq:decomposition h}
    h = \hat{h}+ \frac{\Xi(h)}{\Xi(\Ric)} \mathrm{Ric} + \mathrm{div}_f^\ast Y     
    \end{equation}
    and $\hat{h}$ satisfies $\mathrm{div}_f(\hat{h}) = \Xi(\hat{h}) = 0$. This decomposition is $L^2_f(M)$-orthogonal, and this orthogonality is preserved by the bilinear form $\delta^2\nu_g$. The decomposition is additionally preserved by $L_f$ and $N_f$.
\end{lemma}

\begin{proof}
Supposing that $h \in C_0^{\infty}(M)$ and $\Xi(h) = 0$, by Theorem 4.15 of \cite{CM}, there exists a constant $C$ depending on $(M,g,f)$, and a smooth vector field $Y$ that is $L^2_f$-orthogonal to the set of Killing vector fields on $(M,g,f)$ so that 
\[
\mathrm{div}_f(h - \mathrm{div}_f^*Y) = 0,
\]
and
\[ \|Y\|_{H^1_f}+\|\mathrm{div}_fY\|_{H^1_f}+\|\Delta_fY\|_{L^2_f} \leq C \|\mathrm{div}_fh\|_{L^2_f}.
\]
With this, our aim is to improve the above regularity of $Y$ in order for
\[
h':= \pi_{\ker \mathrm{div}_f} h:= h - \mathrm{div}_f^*Y, 
\]
the $L^2_f$-projection of $h$ onto the kernel of $\mathrm{div}_f$, to lie in $H^1_f(M)$. The goal is to show that $\pi_{\ker \mathrm{div}_f}: H^1_f\to H^1_f$ is continuous. 

Before proceeding, we note that vector field $Y$ can be constructed more concretely using, as proven in Section 4.2 of \cite{CM}, the existence of $V_i\in H^1_f$ for $i\in \mathbb{N}$ an $L^2_f$-orthonormal basis of eigentensors for the operator $\mathcal{P} := \mathrm{div}_f\mathrm{div}_f^*$. (Note the standard way to construct such bases can be found in \cite[Proposition 2.8, Chapter 5]{Taylor II} and \cite[Theorem 10.20]{Gri}.)  Within this basis, 0-eigenvectors (i.e. the kernel of $\mathcal{P}$) are precisely Killing vector fields. Now if $h\in H^1_f$, then $\mathrm{div}_f h \in L^2_f$ and there exists a decomposition
\[
\mathrm{div}_f h = \sum_i a_i V_i,
\]
with $\sum_i a_i^2 = \|\mathrm{div}_f h\|_{L^2_f}<\infty$. Additionally, for any $i$ so that $V_i$ is a Killing vector field, one has $a_i = 0$ since 
\[
a_i = \int_M \langle \mathrm{div}_f h, V_i \rangle\, e^{-f} d\mu_g = \int_M \langle h, \mathrm{div}_f^\ast V_i \rangle e^{-f} d\mu_g.
\]
This means that for any $i$ so that $a_i\neq 0$, $\mathcal{P}V_i= \lambda_i V_i$ for $\lambda_i>0$. Consequently, the vector field 
\[
Y := \sum_i \frac{a_i}{\lambda_i}V_i
\]
solves $\mathcal{P}Y = \mathrm{div}_fh$, and hence $\mathrm{div}_f(h-\mathrm{div}_f^\ast Y)=0$.

Returning to our previous goal, we now improve the regularity of $Y$ from $H^1_f$ to $H^2_f$ and show that the projection $\pi_{\ker\mathrm{div}_f}$ is continuous in $H^1_f$. In particular $h-\mathrm{div}_f^*Y\in H^1_f$. This is done in \cite[Proposition 3.10]{lz23} for Ricci solitons with bounded curvature, such as the blowdown soliton, and we give a quick argument below. As explained in Section 4.2 of \cite{CM}, we have $\Delta_f Y\in L^2_f$. By elliptic theory applied to the elliptic operator $\Delta_f$, noting that all balls of size $2$ have uniformly bounded geometry in $(M,g)$, this implies that for any $x\in M$, one has
\begin{equation}\label{eq: elliptic local vfields}
    \|\nabla^2Y\|_{L^2(B_1(x))}+\|\nabla Y\|_{L^2(B_1(x))}+\|Y\|_{L^2(B_1(x))} \leq C ( \|\Delta_f Y\|_{L^2(B_2 (x))} + \|Y\|_{L^2(B_2 (x))}).
\end{equation}
The constant $C$ above can be taken independent of $x$. A Vitali covering of $(M,g)$, which has Ricci curvature bounded below and a conical end, provides a set of $x_i$ so that the set of balls $B_1(x_i)$ provides a covering of $M$ and the balls $B_2(x_i)$ have a uniformly bounded intersection number. Then, multiplying \eqref{eq: elliptic local vfields} on both sides by $e^{-f(x_i)}$ (uniformly equivalent to $e^{-f}$ on $B_2(x_i)$) and summing over the $x_i$ gives the bound:
\begin{equation}
    \|Y\|_{H^2_f}\leq C ( \|\Delta_f Y\|_{L^2_f} + \|Y\|_{L^2_f})\leq C\|\mathrm{div}_f h \|_{L^2_f} \leq C\|h\|_{H^1_f}. 
\end{equation}

Thus we have justified defining $h' := h - \mathrm{div}_f^\ast Y$ with $h' \in H^1_f(M)$ and $Y \in H^2_f(M)$. Finally, to obtain orthogonality to Ricci, we define 
\[
\hat{h} := h' - \frac{\Xi(h)}{\Xi(\Ric)} \Ric.
\]
Then $\hat{h} \in H^1_f(M)$ and using that  $\mathrm{div}_f \Ric = 0$ on a shrinking soliton, we have 
\[
\Xi(\hat{h}) = \Xi(h') - \Xi(h) = \Xi(\mathrm{div}_f^\ast Y) = 0. 
\]
The remaining assertions on orthogonality well-known (see e.g. \cite{cz12} or \cite{cz24} and Lemma \ref{lem:stability-in-gauge}). This completes the proof. 
\end{proof}

\begin{lemma}\label{lem:stability-in-gauge}
    Suppose $h$ is a compactly supported $2$-tensor on the FIK shrinking soliton given as an $L^2_f$-orthogonal sum 
    \[
    h = \hat{h} + \Xi(h)\, \mathrm{Ric} + \mathrm{div}_f^\ast Y
    \]
    where $\hat{h} \in H^1_f(M)$ satisfies $\mathrm{div}_f(\hat{h}) = 0$ and $\Xi(\hat{h}) = 0$ and $Y \in H^2_f(M)$. Then
    \[
    \delta^2 \nu_g(h) = \delta^2 \nu_g(\hat{h})= \frac{1}{32\pi^2} \int_M (2R(\hat{h}, \hat{h}) - |\nabla \hat{h}|^2) \, e^{-f} d\mu_g. 
    \]
\end{lemma}

\begin{proof}
With the given decomposition, we have 
\begin{align*}
\delta^2 \nu_g (h) &= \delta^2 \nu_g(\hat{h}) +\frac{\Xi(h)^2}{\Xi(\Ric)^2} \delta^2 \nu_g (\Ric) + \delta^2 \nu_g(\mathrm{div}_f^\ast Y)  \\
& \qquad + 2\big( \delta^2 \nu_g( \hat{h}, \Ric) +\delta^2 \nu_g(\hat{h}, \mathrm{div}_f^\ast Y) +  \delta^2 \nu_g(\mathrm{div}_f^\ast Y, \Ric) \big).
\end{align*}
As in \cite{cz24}, we can directly (pointwise) show that $\delta^2 \nu_g(\Ric) = \delta^2 \nu_g(\mathrm{div}_f^\ast Y) =\delta^2 \nu_g(\mathrm{div}_f^\ast Y, \Ric) = 0$. Using that $\Xi(\hat{h}) = 0$ and $\mathrm{div}_f(\hat{h}) =0$, we similarly have $\delta^2 \nu_g( \hat{h}, \Ric) =\delta^2 \nu_g(\hat{h}, \mathrm{div}_f^\ast Y) = 0$. The last identity now follows from \eqref{eq:second-variation-H1} using once more that $\Xi(\hat{h}) = \mathrm{div}_f(\hat{h}) = 0$ as well as $v_{\hat{h}} = 0$ in view of Lemma \ref{lem:vh}.
\end{proof}

\subsubsection{Consequences of the linear stability}

We now prove Corollaries \ref{cor:Liouville thm} and \ref{cor: decrease nu} assuming Theorem \ref{thm:main}.

\begin{proof}[Proof of Corollary \ref{cor:Liouville thm}]
    Recall the background metric, $g$, is fixed and $h$ satisfies $\partial_t h = L_f h$ for $t < 0$. Following \cite[Proposition 5.2]{bk17}, we
    \[
    \frac12\frac{d}{dt}\Big(\int_M |h_t|^2e^{-f}d\mu_g \Big) = \int_M \langle L_f h, h\rangle e^{-f}d\mu_g.
    \]
    Now, by Theorem \ref{thm:main}, if for any $t$, $h_t\perp_{L^2_f}\Ric$ and $\operatorname{div}_fh_t=0$, then 
    \[
    \int_M \langle L_f h, h\rangle e^{-f}d\mu_g < \lambda_{\max}\int_M |h_t|^2e^{-f}d\mu_g
    \]
    for $\lambda_{\max}<0$. Consequently, the function $\Psi(t):=\int_M |h_t|^2e^{-f}d\mu_g $ satisfies
    \[
    \Psi'(t)< 2\lambda_{\max} \Psi(t),
    \]
    Thus, $\Psi(t)$ is monotone nonincreasing and for any $t>0$, $\Psi(0) < \Psi(-t) e^{2\lambda_{\max}t}$. In particular, since we assumed $\Psi(-t)\leq 1$, then $\Psi(0) = 0$, and similarly, $\Psi(-s) = 0$ for any $s>0$, that is $h_t$ vanishes identically. 

    \begin{remark}
        We could weaken the assumption on the bound on $h_t$ to $\Psi(-t) e^{2\lambda_{\max}t}\to 0$ as $t\to +\infty$.
    \end{remark}
\end{proof}

\begin{proof}[Proof of Corollary \ref{cor: decrease nu}]
    Since by definition, we have 
    \[
    \nu(g+\varepsilon h) = \inf_{\tau>0}\inf_{\phi \in C^{\infty}_0(M)}\left\{\mathcal{W}(g+\varepsilon h, \phi, \tau):  \int_M\frac{e^{-\phi}}{(4\pi\tau)^2} d\mu_{g + \varepsilon h} = 1\right\},
    \]
    we want to find some $\tau > 0$ and a test function $\phi_{\varepsilon} = \phi_{\varepsilon, h}$ such that $\int_M\frac{e^{-\phi_\varepsilon}}{(4\pi\tau)^2} d\mu_{g + \varepsilon h}= 1$ and
    \[
    \mathcal{W}(g+\varepsilon h,\phi_{\varepsilon},\tau)\leq \nu(g).
    \]
    Let $\tau = 1$ be the scale of the FIK shrinking soliton, then since we have the normalization $\int_M e^{-f} d\mu_g = 1$ it is well known that (e.g. \cite{CN}) $\nu(g) = \mathcal{W}(g, f,1)$. 

    We seek a function $\phi_\varepsilon$ such that $\|\phi_{\varepsilon} - f\|_{C^3}\leq C\varepsilon$ for some uniform $C>0$ and normalized so that $\int_M\frac{e^{-\phi_{\varepsilon}}}{(4\pi\tau)^2} d\mu_{g + \varepsilon h} = 1 $, then we have
    \[
    \mathcal{W}\Big(g+\varepsilon h,\phi_{\varepsilon},1\Big) = \nu(g)+O(\varepsilon^2),
    \]
    since a Ricci soliton is a critical point of $\nu$ for variation of the metric (regardless of the variation of the function $\varepsilon\mapsto \phi_\varepsilon$ satisfying the volume constraint).
    
    As in \cite[Lemma 3.8]{hall-murphy}, in order to optimize this second variation, define $\phi_{\varepsilon}:= f+\frac{1}{2}\varepsilon(\mathrm{tr}_g h-v_h)+ c_\varepsilon$. In fact, we can take $v_h=0$, since this uniquely solves \eqref{eq:def vh} when $\operatorname{div}_fh=0$. The constant $c_\varepsilon$ is chosen to preserve the constraint $\int_M\frac{e^{-\phi_{\varepsilon}}}{(4\pi\tau)^2} d\mu_{g + \varepsilon h}= 1$. Since we have $d\mu_{g+\varepsilon h} = (1+\varepsilon\frac{tr_g h}{2}+O(\varepsilon^2))d\mu_g$, we have $c_\varepsilon=O(\varepsilon^2)$.
    
    Finally, as in the proof of \cite[Proposition 3.9]{Lopez-Ozuch}, expanding to order $2$ in $\varepsilon$ with this choice of $\phi_\varepsilon$, we then find 
    \[
     \mathcal{W}\Big(g+\varepsilon h,\phi_{\varepsilon},1\Big) = \nu(g)+ \varepsilon^2\delta^2\nu_g(h) + O(\varepsilon^3),
     \]
    where $O(\varepsilon^3)$ is bounded by a constant (independent on $h$) times $\varepsilon^3\int_M (|\nabla h|^2 + |h|\,|\nabla^2 h|)|h|e^{-f}d\mu_g$ if $h$ is uniformly bounded. Now since $\delta^2 \nu_g(h) \leq 0$ by Theorem \ref{thm:main}, we conclude for $\varepsilon$ small enough that 
    \[
    \nu(g+\varepsilon h) \leq \mathcal{W}(g + \varepsilon h,  \phi_{\varepsilon},1) \leq \nu(g). 
    \]
\end{proof}

\begin{remark}
    A weakness of the statement of Corollary \ref{cor: decrease nu} is that given $h$ uniformly bounded, it is not obvious that its $L^2_f$-projection on $\ker\operatorname{div}_fh$ will remain uniformly bounded. The localized and controlled gauge-fixing scheme proposed in \cite{cm} is much better adapted.
\end{remark}

\subsection{A basis for symmetric $2$-tensors}

Our study of the stability of the FIK metric depends upon a good choice of (a $C^{\infty}(\mathring{M})$) basis for symmetric $2$-tensors on $\mathring{M}$. 

We begin by recalling that the trace map $\mathrm{tr} : \Lambda^2 \otimes \Lambda^2 \to \mathrm{Sym}^2(T^\ast M)$ which maps $S_{ijkl} \mapsto \mathrm{tr}(S)_{ik} = \frac{1}{2}(S_{ipkp} + S_{kpip})$ is an isomorphism when restricted to $\Lambda_g^- \otimes \Lambda_g^+ \to \mathrm{Sym}^2_0(T^\ast M)$. Given two 2-forms $\alpha,\beta$ we let $\alpha\circ \beta = \mathrm{tr}(\alpha \otimes \beta)$. Introducing the notation 
\[
e^{ij} := e^i \otimes e^j,
\]
recalling our convention that $e^i \wedge e^j = \frac{1}{2}(e^{ij} - e^{ji})$, one can readily verify that 
\begin{align*}
\omega_1^- \circ \omega_1^+ & = \frac{1}{4}(e^{00} + e^{11} - e^{22} - e^{33}), \\
\omega_2^- \circ \omega_2^+ & = \frac{1}{4}(e^{00} + e^{22} - e^{33} - e^{11}),  \\
\omega_3^- \circ \omega_3^+ & = \frac{1}{4}(e^{00} + e^{33} - e^{11} - e^{22}).  
\end{align*}
For the off diagonal terms, we have 
\begin{align*}
\omega_1^- \circ \omega_2^+ & = \frac{1}{4}(e^{12} + e^{21}+e^{03} +e^{30}), & \omega_2^- \circ \omega_1^+ &= \frac{1}{4}(e^{12} + e^{21}  - e^{03} - e^{30}),  \\
 \omega_3^- \circ \omega_1^+ &= \frac{1}{4}(e^{31}+e^{13} + e^{02}+ e^{20}),   & \omega_1^- \circ \omega_3^+ & = \frac{1}{4}(e^{31}  + e^{13} - e^{02} - e^{20}), \\
 \omega_2^- \circ \omega_3^+ & = \frac{1}{4}(e^{23} + e^{32}+ e^{01}  + e^{10}),& \omega_3^- \circ \omega_2^+ & = \frac{1}{4}(e^{23}  + e^{32}- e^{01} - e^{10})
\end{align*}
It can be checked that $\omega_a^+ \circ \omega_b^+ = \omega_a^- \circ \omega_b^- = 0$ for $a\neq b$.
On the other hand, for any fixed $a \in \{1,2,3\}$,  we have $\omega_a^\pm \circ \omega_a^\pm = \frac{1}{4} g$. 

In what follows, let us define a $C^{\infty}(\mathring{M})$-basis of  symmetric $2$-tensors, $\mathbf{B}$, consisting of the $2$-tensors: 
\begin{align*}
 \mathbf{b}_0 &= \omega_1^+\circ \omega_1^+, & \mathbf{b}_1 &=\omega_1^- \circ \omega_1^+, &\mathbf{b}_2 &= \frac{1}{\sqrt{2}}(\omega_2^- + \omega_3^-)\circ \omega_1^+,  &   \\
 \mathbf{b}_3 &= \frac{1}{\sqrt{2}}(\omega_2^- - \omega_3^-)\circ \omega_1^+, & \mathbf{b}_4 &= \frac{1}{\sqrt{2}}\omega_1^- \circ (\omega_2^+ + \omega_3^+), &  \mathbf{b}_5 &=\frac{1}{\sqrt{2}}\omega_1^- \circ (\omega_2^+ - \omega_3^+) \\
\mathbf{b}_6 &= \frac{1}{\sqrt{2}}(\omega_2^- \circ \omega_2^+ + \omega_3^- \circ \omega_3^+),& \mathbf{b}_7 &= \frac{1}{\sqrt{2}}(\omega_3^- \circ \omega_2^+-\omega_2^- \circ \omega_3^+ ),& \\
\mathbf{b}_8 &= \frac{1}{\sqrt{2}}(\omega_3^-\circ \omega_2^++ \omega_2^-\circ\omega_3^+ ),   & \mathbf{b}_9 &=\frac{1}{\sqrt{2}}(\omega_2^- \circ \omega_2^+ - \omega_3^- \circ \omega_3^+).
\end{align*}
These basis elements are also expressed in terms of the $e^{ij}$ basis elements for reference in Section \ref{app:2}. Among the basis, all elements are traceless, except $\mathbf{b}_0= \frac{1}{4} g$. Distinct basis elements are pointwise orthogonal to one another and for each $\mathbf{b} \in \mathbf{B}$, we have $|\mathbf{b}|= \frac{1}{2}$. We have chosen the above basis of $2$-tensors based upon the actions of the weighted Laplacian $h \mapsto \Delta_f h$ and the curvature $h \mapsto R(h)$ on $2$-tensors. In what follows, the action of $L_f$ on basis elements is important and can be found in Corollary \ref{cor:Lf-action-on-basis} and Proposition \ref{prop:Lambda-function-comps}. 

The basis $2$-tensors in general exists only on $\mathring{M} = M \setminus \Sigma$ and some do not converge smoothly as $r \to 1$. Here we note that both $\mathbf{b}_0$ and $\mathbf{b}_1$ do converge smoothly to (plus and minus) the round metric on $\Sigma$ as $r \to 1$. Additionally, we note on any set $\{r \geq r_0 > 1\}$, we have uniform bounds for derivatives of any $\mathbf{b}\in \mathbf{B}$. Importantly, it follows from Proposition \ref{lem:levi-civita} that $\nabla_{e_0} \mathbf{b} = 0$ for all $\mathbf{b} \in \mathbf{B}$.  

It can be seen from the formulas above that a general deformation of the FIK metric can be expressed as a $C^{\infty}(\mathring{M})$-linear combination on $\mathring{M}$ as
\begin{align}\label{eq:h-in-basis}
    h = \sum_{p =0}^{9} h_p \mathbf{b}_p. 
\end{align}
for smooth component functions $h_0,\dots, h_9$ on $\mathring{M}$. 

\subsection{$J_1^+$-invariance}
Recall a $2$-tensor is $J_1^+$-invariant if $h(J_1^+\cdot, J_1^+\cdot) = h(\cdot,\cdot)$ and $J_1^+$-anti-invariant if the sign flips. In terms of our frame $\{e^0, e^1, e^2, e^3\}$, $J^1_+$-(anti)-invariance corresponds to (anti-)-invariance under the mapping $(e^0, e^1, e^2, e^3) \to (-e^1, e^0, -e^3, e^2)$. It can be readily seen from the formulas in Appendix \ref{app:2} expressing the $\mathbf{b}_p$ in terms of the $e^i$ that $\mathbf{b}_p$ is $J^1_+$-invariant for $p \in \{0, 1, 2,3\}$ and is $J^1_+$-anti-invariant for $p \in \{4,5,6,7,8,9\}$. If $h$ is expressed as in \eqref{eq:h-in-basis} above,  then we let 
\begin{align}
    h_I &:=  h_0 \mathbf{b}_0 + h_1 \mathbf{b}_1 + h_2 \mathbf{b}_2 + h_3 \mathbf{b}_3, \\
    h_A &:= h_4 \mathbf{b}_4 + h_5 \mathbf{b}_5 + h_6 \mathbf{b}_6 + h_7 \mathbf{b}_7 + h_8 \mathbf{b}_8 + h_9 \mathbf{b}_9,
\end{align}
denote the $J_1^+$-invariant and $J_1^+$-anti-invariant pieces of $h$.

We end this section noting the important fact that $L_f$ preserves invariance and anti-invariance by $J_1^+$. This is proven in Lemma \ref{lem:ibp-for-H1} in Appendix \ref{app:3}. 

\begin{corollary}\label{cor:invariance-splitting}
Suppose $h \in C^{\infty}(M) \cap H^1_f(M)$. Then 
\begin{align*}
        \int_M (2 R(h, h) - |\nabla h|^2) \, e^{-f} d\mu_g& =  \int_M (2 R(h_I, h_I) - |\nabla h_I|^2) \, e^{-f} d\mu_g \\
        &\qquad + \int_M (2 R(h_A, h_A) - |\nabla h_A|^2) \, e^{-f} d\mu_g.
\end{align*}
\end{corollary}

\newpage

\part{Linear Stability of the Blowdown Soliton: Radially Symmetric Case}\label{part:radial stability}

In this following two sections, we will prove Theorem \ref{thm:main} in the special case that $h$ is radially symmetric. We view this as the most subtle and important setting of the theorem.  

\begin{theorem}\label{thm:main-radial}
Suppose that $h \in C^{\infty}(M) \cap H^1_f(M)$ is a radially symmetric $2$-tensor. Then 
\[
\delta^2 \nu_g(h) \leq 0. 
\]
\end{theorem}

Note it suffices to prove the theorem assuming $\mathrm{div}_f(h) = 0$ and $\Xi(h) = 0$. Broadly, the proof will proceed in two steps. Firstly, we will show that in the radial setting the second variation of Perelman's entropy in the direction $h = \sum_p h_p \mathbf{b}_p$ is controlled by the second variation in the direction $h_U =  h_0 \mathbf{b}_0 + h_1 \mathbf{b}_1 + h_6 \mathbf{b}_6$. That is, $\delta^2\nu_g(h) \leq \delta^2\nu_g(h_U)$. In particular, we show it suffices to analyze the subspace of deformations of the form $h_U$ that are divergence free and orthogonal to the Ricci tensor in $L^2_f$. In the second step, we will carefully parametrize such deformations $h_U$ and show that $\delta^2 \nu_g(h_U) \leq 0$. 

The principal subtlety in the radial setting is demonstrating stability for deformations of the form $\tilde{h} = h_0 \mathbf{b}_0 + h_1 \mathbf{b}_1$. This is because the Ricci tensor takes this form, $L_f \Ric = \Ric$, and $\mathrm{div}_f (\Ric) = 0$. We also have $L_f \nabla^2 f = 0$, but $\mathrm{div}_f \nabla^2 f  \neq 0$. In particular, one cannot hope to show that $\delta^2 \nu_g (\tilde{h})\leq 0$ by simply showing $\int_M g(L_f \tilde h , \tilde h) e^{-f} d\mu_g \leq 0$, since the latter is not true. To proceed, we need to make use of the assumption $\Xi(h) = 0$ and, to make this orthogonality condition effective, we find it very helpful to incorporate the condition $\mathrm{div}_f(h) = 0$. However, it turns out that $\mathrm{div}_f(h) = 0$ does not imply $\mathrm{div}_f(\tilde{h}) = 0$, but rather $\mathrm{div}_f(h_U) = 0$. This is the reason we also need to consider the component $h_6 \mathbf{b}_6$ in the second broad step of our proof.

Throughout this part, we assume that $h$ is a radially symmetric $2$-tensor with component functions defined by the identity \eqref{eq:h-in-basis}.

\section{Setup and reduction of the linear stability problem in the radial setting}

\begin{definition}[Radially symmetric $2$-tensors]
    We say that $h$ is \textit{radially symmetric} if the following two conditions hold: 
\begin{enumerate}
    \item[(i)] Each component function $h_p$ depends only upon the radial coordinate $r$. 
    \item[(ii)] Each component function $h_p$ extends smoothly to $r =1$ so that $h \in C^{\infty}(M)$. 
\end{enumerate}
\end{definition}

Condition (ii) implies, in particular, that 
\begin{equation}\label{eq:rto1-vanishing}
h_2(r), h_3(r), h_4(r), h_5(r), h_6(r), h_7(r), h_8(r), h_9(r)\in O(r-1)
\end{equation}
as $r \to 1$. Note that with these assumptions, $h \in H^1_f(M)$ if and only if each of the component functions is contained in $H^1_f(M)$. Indeed, while the basis of $2$-tensors $(\mathbf{b}_p)_p$ is orthonormal, hence bounded in $L^2_f$, this is not the case of their derivatives. These are computed in Proposition \ref{prop:levi-civita-on-2-tensors}. Indeed, it can be observed that, for $p\in \{2,3,4,5,6,7,8,9\}$, $\nabla_{e_1} \mathbf{b}_p$ are all pointwise orthogonal to each other  $|\nabla_{e_1} \mathbf{b}_p|$ blows up like $s^{-1}\approx (r-1)^{-\frac{1}{2}}$ by Proposition \ref{prop:Gamma-function-comps} as $r \to 1$. Now since in the radial case we have \[\sum_p\nabla (h_p\mathbf{b}_p) = d(h_p)\otimes \mathbf{b}_p + h_p \nabla \mathbf{b}_p = e_0(h_p)e^0\otimes\mathbf{b}_p +\, h_p \nabla \mathbf{b}_p,\] 
having $h=\sum_ph_p\mathbf{b}_p\in L^2_f$ and $\nabla h=\sum_p\nabla (h_p\mathbf{b}_p)\in L^2_f$ imposes $e_0(h_p)=\frac{s}{2r}h_p'\in L^2_f$ and $(r-1)^{-\frac{1}{2}}h_p\in L^2_f$ for $p\in \{2,3,4,5,6,7,8,9\}$ which in particular implies that $h_p(r)\to 0$ as $r\to 1$ for $p\notin\{0,1\}$. Furthermore, imposing that $\nabla h$ is bounded implies $h_p(r) = O(\sqrt{r-1})$ for $p\notin\{0,1\}$ while imposing $\nabla^2 h$ is bounded implies $h_p(r) = O(r-1)$ for $p\notin\{0,1\}$. For instance, $\nabla_{e_1}\nabla_{e_1}(h_6(r)\mathbf{b}_6) = h_6(r)\nabla_{e_1}\nabla_{e_1}\mathbf{b}_6$ behaves like $\frac{h_6(r)}{r-1}$ as $r\to 1$. The exact order of vanishing at $r =1$ beyond what is given in \eqref{eq:rto1-vanishing} will not be important in what follows.

In this section, using computations from Appendices \ref{app:1} and \ref{app:2}, we derive the equations obtained by imposing the conditions $\mathrm{div}_f(h) = 0$ or $\Xi(h) = 0$ on a radially symmetric $2$-tensor $h$. Then, in Lemma \ref{lem:ibp-for-H1}, we prove that  in the radial setting, the stability problem for $h$ reduces to the stability problem for the component $h_U = h_0 \mathbf{b}_0 + h_1 \mathbf{b}_1 + h_6 \mathbf{b}_6$.

\begin{lemma}\label{lem:gauge-equations}
The weighted divergence of the radially symmetric $2$-tensor $h$ is given by the formula
\begin{align*}
    \mathrm{div}_f(h) &= \frac{s}{8r^2}\left(r\big(h_0 + h_1)'  -2\sqrt{2}r^2\big (h_0 +h_1\big)+4 h_1+ r(\sqrt{2}h_6\big)' +\left(\frac{4-2r^2}{F} \right)\big(\sqrt{2}h_6\big)\right) 
    e^0 \\
   &   -\frac{s}{8r^2}\left(r(\sqrt{2}h_7)' +\left(\frac{4-2r^2}{F}\right) (\sqrt{2}h_7)\right)e^1 \\
   &  + \frac{s}{8r^2}\left(r(h_2-h_4)'+ \left(\frac{2 - r^2}{F}+2- \sqrt{2} r^2\right)(h_2 -h_4) - \left( 2  -\frac{2}{F}\right)(h_2+h_4) \right)e^- \\
    &  - \frac{s}{8r^2}\left(r(h_3-h_5)' +\left(\frac{2 - r^2}{F}+2- \sqrt{2} r^2\right) (h_3-h_5)-\left(2 -\frac{2}{F}\right) (h_3+h_5) \right)e^+.
\end{align*}

\end{lemma}

\begin{proof}
Using that the component functions of $h$ are radially symmetric and the identity $\mathrm{div}_f(uh) = h(\nabla u,\cdot) + u \,\mathrm{div}_f(h)$, from \eqref{eq:h-in-basis} we have
\begin{align*}
    \mathrm{div}_f(h) & = \sum_{p=0}^9 e_0(h_p) \mathbf{b}_p(e_0, \cdot) + h_p \mathrm{div}_f (\textbf{b}_p).
\end{align*}
By inspecting the expressions for the basis elements in terms of the $e^{ij}$ (see Section \ref{app:2}), we readily find that
\begin{align*}
\mathbf{b}_0(e_0, \cdot) & = \frac{1}{4} e^0, &\mathbf{b}_1(e_0, \cdot) & = \frac{1}{4} e^0,
\end{align*}
and
\begin{align*}
\mathbf{b}_2(e_0, \cdot) & = \frac{1}{4\sqrt{2}} (e^2-e^3) = \frac{1}{4} e^-, &\mathbf{b}_3(e_0, \cdot) & = -\frac{1}{4\sqrt{2}}(e^2+e^3) = -\frac{1}{4}e^+, \\
\mathbf{b}_4(e_0, \cdot) & = -\frac{1}{4\sqrt{2}} (e^2-e^3) = -\frac{1}{4} e^-, & \mathbf{b}_5(e_0, \cdot) & =\frac{1}{4\sqrt{2}}(e^2 + e^3) = \frac{1}{4} e^+,
\end{align*}
and
\begin{align*}
 \mathbf{b}_6(e_0, \cdot) & = \frac{1}{2\sqrt{2}} e^0, &\mathbf{b}_7(e_0, \cdot) & = -\frac{1}{2\sqrt{2}} e^1, &\mathbf{b}_8(e_0, \cdot) & = 0, &\mathbf{b}_9(e_0, \cdot) & =0.
\end{align*}
Using this and the structure of identities obtained in Proposition \ref{prop:divf-on-basis}, we observe
\begin{align*}
    \mathrm{div}_f(h) &= \left(\frac{1}{4}e_0(h_0 + h_1 + \sqrt{2}h_6) + h_0\langle \mathrm{div}_f(\mathbf{b}_0), e^0\rangle + h_1\langle \mathrm{div}_f(\mathbf{b}_1), e^0\rangle  + h_6\langle \mathrm{div}_f(\mathbf{b}_6), e^0\rangle  \right) e^0 \\
   & \quad  +\left(-\frac{1}{2\sqrt{2}} e_0(h_7) + h_7 \langle \mathrm{div}_f(\mathbf{b}_7), e^1 \rangle\right)e^1 \\
   & \quad + \left(\frac{1}{4}e_0(h_2-h_4) + h_2 \langle \mathrm{div}_f(\mathbf{b}_2), e^-  \rangle + h_4 \langle \mathrm{div}_f(\mathbf{b}_4), e^-  \rangle \right)e^- \\
    & \quad + \left(\frac{1}{4}e_0(h_5-h_3) + h_3 \langle \mathrm{div}_f(\mathbf{b}_3),e^+  \rangle + h_5 \langle \mathrm{div}_f(\mathbf{b}_5), e^+ \rangle \right) e^+.
\end{align*}
Using $e_0 = \frac{s}{2r} \partial_r$ and substituting in the identities from Proposition \ref{prop:divf-on-basis} proves the lemma. 
\end{proof}

We note by consequence of \eqref{eq:volume-density}, the following useful lemma for integrating radial functions on the FIK shrinking soliton.

\begin{lemma}\label{lem:radial-integration}
    If $a = a(r)$ is a radially symmetric function on $M$, then 
    \[
    \int_M a \, e^{-f} d\mu_g = 32\pi^2 \int_1^\infty a \, r^3 \, e^{-f} \, dr.
    \]
\end{lemma}

For radial functions, we have the following special $L^2_f$-decay estimate (which is a special case of a more general estimate Lemma 1.51 in \cite{CM} that holds for gradient Ricci shrinkers).

\begin{proposition}\label{prop:L2f-decay}
    Suppose $u= u(r) \in H^1_f(M)$ is a radially symmetric function on $M$. Then, $r u \in L^2_f(M)$ and there exists a universal constant $C$ such that $\|ru\|_{L^2_f}^2 \leq C \|u\|_{H^1_f}^2$ and 
    \[
    \int_r^\infty u^2 \,t^3 e^{-f} \, dt \leq \frac{C}{r^2} \|u\|_{H^1_f}^2. 
    \]
\end{proposition}
\begin{proof}
   Recalling \eqref{eq:lap-radial-coordinate} and that $\nabla r = e_0(r) e_0  = \frac{s^2}{4r^2}\partial_r = \frac{F}{4} \partial_r $, we compute 
    \begin{align*}
        \mathrm{div}_f(u^2 \nabla r^2 ) & = 4 ur \langle \nabla u, \nabla r \rangle+ u^2 \Delta_f r^2  = F u r \, du(\partial_r)  + u^2 (2 - r^2).
    \end{align*}
    Therefore, integrating over $M$, making use of Lemma \ref{lem:integration-by-parts} and Lemma \ref{lem:radial-integration} above, we have 
    \[
    \int_1^\infty u^2 \, r^5 e^{-f} dr \leq \int_1^\infty  (F u ru' ) \, r^3 e^{-f} dr +2 \int_1^\infty u^2 \, r^3 e^{-f} \, dr. 
    \]
    Using $F r^2 = s^2$, we estimate the second term 
    \[
     \int_1^\infty u^2 r^5 e^{-f} dr  \leq 4\int_1^\infty F\cdot(u')^2 \, r^3 e^{-f} \, dr + \int_1^\infty s^2u^2 \, r^3 e^{-f} \, dr  + 2 \int_1^\infty u^2 \, r^3 e^{-f} dr.
    \]
    Hence, using $r^2 -s^2 \geq (1 - \frac{1}{\sqrt{2}}) r^2 = \frac{c_0}{\sqrt{2}}r^2$, we have 
    \begin{align*}
      \frac{c_0}{\sqrt{2}} \int_1^\infty u^2 \, r^5 e^{-f} dr & \leq \int_1^\infty u^2 (r^2 - s^2) r^3 e^{-f} dr \\
      & \leq 4\int_1^\infty F\cdot (u')^2 \, r^3 e^{-f} \, dr +   2 \int_1^\infty u^2 \, r^3 e^{-f} dr, 
    \end{align*}
    Note that since $|\nabla u|^2  = \frac{F}{4} (u')^2$ up to a constant factor, the terms on the right are bounded by $\|u\|_{H^1_f}^2$. 
    We conclude 
    \[
    \int_1^\infty u^2 r^5 e^{-f}\, dr \leq C \|u\|_{H^1_f}^2,
    \]
    for some universal constant $C$. This shows $r u \in L^2_f(M)$. 
    Moreover, 
    \[
    r_0^2 \int_{r_0}^\infty u^2 \, r^3 e^{-f} \, dr \leq \int_{r_0}^\infty u^2 r^5 e^{-f}\, dr   \leq C \|u\|_{H^1_f}^2,
    \]
    which implies the second claim.
\end{proof}

\begin{remark}
 It is well known that in dimension one, the Sobolev space $H^1$ embeds into $L^{\infty}$. If $u : M \to \mathbb{R}$ is a smooth radially symmetric function such that $u \in H^1_f(M)$, then using the fundamental theorem of calculus and the assumption that $u \in H^1_f(M)$, we have 
    \begin{align*}
        F(r)u(r)  &= \int_1^r Fu' \, dt   + \int_1^r F' u \, dt\\
        & \leq \left(\int_1^r F (u')^2 r^3\,  e^{-f} \, dt \right)^{\frac{1}{2}}\left(\int_1^r F r^{-3} e^f \, dt \right)^{\frac{1}{2}}  +  \left(\int_1^r u^2 r^3 \, e^{-f}\, dt \right)^{\frac{1}{2}}\left(\int_1^r (F')^2 r^{-3} e^f \, dt \right)^{\frac{1}{2}}  \\
        & \leq C \|\nabla u \|_{L^2_f} r^{-2} e^{\frac{f}{2}} + C \|u\|_{L^2_f} r^{-4} e^{\frac{f}{2}}
    \end{align*}
    We have used that $\int_1^r t^k e^{\sqrt{2} t^2} \, dt \leq C_k\, r^{k-1} e^{\sqrt{2} r^2}$. We have also used that $(F')^2 \leq r^{-6}$. Therefore, assuming that $u(1)$ is bounded and using that for large enough $r$, we have $F \approx \frac{1}{\sqrt{2}}$, we conclude that 
    \[
    u r^{\frac{3}{2}} e^{-\frac{f}{2}} \leq C \|\nabla u\|_{H^1_f}r^{-\frac{1}{2}}.
    \]
\end{remark}

\begin{lemma}\label{lem:ricci-orthogonality}
   The $L^2_f$ inner product of the radially symmetric $2$-tensor $h$ with the Ricci tensor is given by 
   \begin{align*}
        \int_M g(\mathrm{Ric}, h) \, e^{-f} d\mu_g &= 8\pi^2\sqrt{2} c_0 \int_1^\infty\left(h_0 - \left(1+\frac{\sqrt{2}}{r^2} \right)h_1 \right) r e^{-f} dr, 
    \end{align*}
\end{lemma}

\begin{proof}
    By Corollary \ref{cor:ricci-scalar}, the Ricci tensor has the form 
    \begin{align}
        \nonumber \Ric &= \left(-\frac{c_0}{2r^4}\right)(e^{00} + e^{11}) + \left(\frac{c_0}{2r^4} + \frac{c_0\sqrt{2}}{2r^2}\right)(e^{22} + e^{33}) \\
       \nonumber  & = \left( \frac{c_0\sqrt{2}}{4r^2} \right)(e^{00} + e^{11} + e^{22} + e^{33})+ \left(-\frac{c_0}{2r^4} - \frac{c_0\sqrt{2}}{4r^2}\right)(e^{00} + e^{11} - e^{22} - e^{33}) \\
       & =  \left( \frac{c_0\sqrt{2}}{r^2} \right)\mathbf{b}_0+ \left(-\frac{2c_0}{r^4} - \frac{c_0\sqrt{2}}{r^2}\right)\mathbf{b}_1\label{eq:ricci-in-basis}
    \end{align}
    Taking the inner product with radial $h$, using Lemma \eqref{lem:radial-integration}, orthogonality of the basis elements, and $|\mathbf{b}_p|^2 = \frac{1}{4}$  yields the equation asserted. 
\end{proof}

We end this section with a reduction of the stability problem to certain components of $h$.

\begin{corollary} \label{cor:stability-reduction}
Suppose that $h \in C^{\infty}(M) \cap H^1_f(M)$ is a radially symmetric $2$-tensor with component functions defined by the identity \eqref{eq:h-in-basis}. Additionally, assume that $\mathrm{div}_f(h) = 0$ and $\Xi(h) = 0$. Let $h_U = h_0 \mathbf{b}_0 + h_1 \mathbf{b}_1 + h_6 \mathbf{b}_6$. Then $\mathrm{div}_f(h_U) = \Xi(h_U) = 0$, and
\begin{equation}\label{eq:reduction-hU}
\delta^2 \nu_g(h) \leq \delta^2\nu_g(h_U), 
\end{equation}
where
\begin{align}\label{eq:d2nu-hU}
    \delta^2\nu_g(h_U) &= \frac{1}{4} \int_1^\infty \big( \Lambda_{11}^+ h_0^2  + \Lambda_{11}^{--}h_1^2 + 2\Lambda_{11}^{\pm}h_0h_1 -|\nabla h_0|^2 -|\nabla h_1|^2 \big) \,r^3 e^{-f} dr\\
  \nonumber  & \qquad + \frac{1}{4} \int_1^\infty (\Lambda_{01} h_6^2-|\nabla h_6|^2) \, r^3 e^{-f} dr.
\end{align}
The divergence free and Ricci orthogonality equations for $h_U$ take the forms
\begin{equation}\label{eq:divf-hU-zero}
    r\big(h_0 + h_1)'  -2\sqrt{2}r^2\big (h_0 +h_1\big)+4 h_1+ r(\sqrt{2}h_6\big)' +\left(\frac{4-2r^2}{F} \right)\big(\sqrt{2}h_6\big)= 0,
\end{equation}
and
\begin{equation}\label{eq:hU-ricci-orthogonality}
        \int_1^\infty\Big( h_0 - \Big( 1 + \frac{\sqrt{2}}{r^2}\Big)h_1 \Big) r e^{-f} dr = 0.
\end{equation}
\end{corollary}

\begin{proof}
By Lemma \ref{lem:stability-in-gauge} and Lemma \ref{lem:ibp-for-H1}, using that our component function are radial and recalling the definition of $\mathcal{M}_{\Lambda}, \mathcal{N}_{\Lambda}$ from Section \ref{sec:matrices}, we have

\begin{align*}
   32\pi^2 \delta^2\nu_g(h)& = \frac{1}{4} \int_M\left( \vec{h}_I \cdot \mathcal{M}_{\Lambda} \vec{h}_I - |\nabla \vec{h}_I|^2 \right) e^{-f} d\mu_g +  \frac{1}{4} \int_M \left( \vec{h}_A \cdot \mathcal{N}_{\Lambda} \vec{h}_A  - |\nabla \vec{h}_A|^2 \right) e^{-f} d\mu_g.\\
    & = \frac{1}{4} \int_M (\Lambda_{11}^{++} h_0^2 + \Lambda_{11}^{--} h_1^2 +2 \Lambda_{11}^{\pm} h_0h_1-|\nabla h_0|^2 -|\nabla h_1|^2\big)\, e^{-f} d\mu_g \\
    & \qquad + \frac{1}{4} \int_M (\Lambda_{1+}(h_2^2 + h_3^2)-|\nabla h_2|^2 -|\nabla h_3|^2\big)\, e^{-f} d\mu_g \\
    & \qquad +  \frac{1}{4}\int_M (\Lambda_{1-}(h_4^2 + h_5^2)-|\nabla h_4|^2 -|\nabla h_5|^2\big)\, e^{-f} d\mu_g \\
    & \qquad + \frac{1}{4} \int_M (\Lambda_{01}(h_6^2 + h_7^2)-|\nabla h_6|^2 -|\nabla h_7|^2\big)\, e^{-f} d\mu_g \\
    & \qquad + \frac{1}{4} \int_M (\Lambda_{23}(h_8^2 + h_9^2)-|\nabla h_8|^2 -|\nabla h_9|^2\big)\, e^{-f} d\mu_g. 
\end{align*}
Equation \eqref{eq:d2nu-hU} is an immediate consequence of this formula. By Proposition \ref{prop:Lambda-function-comps}, the $\Lambda$-functions satisfy $\Lambda_{1+}, \Lambda_{1-},  \Lambda_{01}, \Lambda_{23}, <0$. Thus the last 8 terms are each nonpositive. Throwing away all the terms save those involving $h_0, h_1$, or $h_6$ and using Lemma \eqref{lem:radial-integration} establishes \eqref{eq:reduction-hU}. By Lemmas \ref{lem:gauge-equations} and \ref{lem:ricci-orthogonality}, if $\mathrm{div}_f(h) = \Xi(h) = 0$, then $\mathrm{div}_f(h_U) = \Xi(h_U) = 0$. Equations \eqref{eq:divf-hU-zero} and \eqref{eq:hU-ricci-orthogonality} then follow by applying these lemmas to $h_U$.
\end{proof}

\section{Linear stability for radial 2-tensors in the span of $\mathbf{b}_0, \mathbf{b}_1, \mathbf{b}_6$}

By Corollary \ref{cor:stability-reduction}, to prove that the FIK shrinking soliton is linearly stable among radially symmetric deformations, it suffices to show that $\delta^2 \nu_g(h_U)  \leq 0$ for $h_U = h_0 \mathbf{b}_0 + h_1 \mathbf{b}_1 + h_6 \mathbf{b}_6$ a radially symmetric 2-tensor in $H^1_f(M)$ which satisfies $\mathrm{div}_f(h_U) = 0$ and $\Xi(h_U) = 0$. We emphasize that we assume that $h_0, h_1, h_6 \in C^{\infty}(M) \cap H^1_f(M)$ in this section, and therefore that
\[
h_6(r) \in O(r-1)
\]
as $r \to 1$. In this section, we establish linear stability of FIK among such deformations $h_U$ and thereby complete the proof of Theorem \ref{thm:main}.

\subsection{The $2$-tensors $h_a$ and $k_b$}

Our first goal is to introduce divergence free $2$-tensors that we will use to parametrize the space $2$-tensors $h_U \in C^{\infty}(M) \cap H^1_f(M)$ satisfying the gauge equation \eqref{eq:divf-hU-zero} and Ricci orthogonality \eqref{eq:hU-ricci-orthogonality}.

\begin{definition}\label{def:ha-kb}
    Given smooth radial functions $a = a(r)$ and $b = b(r)$, define radially symmetric $2$-tensors 
    \begin{equation}
    h_a := (a + p(a)) \mathbf{b}_0 + (a - p(a))\mathbf{b}_1
    \end{equation}
    and 
    \begin{equation}
    k_b := (p^\ast(b)- b) \mathbf{b}_0 + (p^\ast(b) + b) \mathbf{b }_1 + q(b) \mathbf{b}_6
    \end{equation}
    where 
    \begin{align}
    p(a) &:= \frac{1}{2}\big(ra' + 2(1-\sqrt{2} r^2)a\big) = \frac{e^f}{2r} \left(\frac{r^2}{e^{f}} a\right)', \\
    p^\ast(b) &:= -\frac{1}{2}\big(rb' + 2b\big) = -\frac{1}{2r} \left(r^2b\right)',
    \end{align}
    and $q(b)$ is the maximal solution, for $r \in [1, R_b)$ some $R_b > 1$ of the ODE  
   \begin{equation}\label{eq:q-ODE}
    \left(\frac{r^4 F}{e^f}  q\right)' =- \left(\frac{r^4F}{e^f}\right) \frac{2\sqrt{2}}{r}\big(p(p^\ast(b)) + b\big), \qquad q(1) = 0. 
\end{equation}
\end{definition}

Note that for general $b$, the $2$-tensor $k_b$ may not be defined on all of $M$ if $R_b < \infty$. However, our goal is to parametrize $2$-tensors $h_U$ which do exist on all of $M$, which will ensure $q$ does as well. The initial condition for $q$ arises in Lemma \ref{lem:q-formulas} below (all trajectories of the ODE either blow up as $r \to 1$ or else satisfy $\lim_{r \to 1^+} q(r) = 0$).  

We begin by deriving some formulas that will be important later. 

\begin{lemma}\label{lem:p-pstar-formulas}
 If $a = a(r)$ and $b = b(r)$ are smooth radial functions, then
\begin{equation}
    p^\ast( p(a)) +a =  -\frac{1}{4}\big(r^2 a'' + (5 - 2\sqrt{2} r^2) ra' - 8\sqrt{2} r^2 a \big),
\end{equation}
and
\begin{equation}
    p( p^\ast(b)) +b =  -\frac{1}{4}\big(r^2 b'' + (5 - 2\sqrt{2} r^2) rb' - 4\sqrt{2} r^2 b \big).
\end{equation}
Consequently, we have the integration by parts formulas
\begin{align*}
     \left(\frac{r^4}{2e^f}ab\right)'  & = \Big(  p(a)\cdot b- a\cdot p^\ast(b) \Big) r^3 e^{-f}
\end{align*}
and 
\begin{align*}
    \left(\frac{r^5}{4e^f}(ba' - ab')\right)' &= \Big(a\cdot p(p^\ast(b)) -b\cdot p(p^\ast(a))\Big) r^3 e^{-f}= \Big(a\cdot p^\ast(p(b)) -b\cdot p^\ast(p(a))\Big) r^3 e^{-f}.
\end{align*}
\end{lemma}
\begin{proof}
Using the definition of $p$ and $p^\ast$, we have 
    \begin{align*}
        p^\ast(p(a)) & = -\frac{1}{2r}\left(r^2 p(a)\right)' \\
        & = -\frac{1}{4r}\left(r e^f\left(\frac{r^2}{e^f} a\right)'\right)' \\
        & = - \frac{1}{4r}\left(r^3 a' + 2r^2 a-2\sqrt{2}r^4 a \right)'.
    \end{align*}
Taking the derivative and simplifying gives
\[
p^\ast(p(a)) =   -\frac{1}{4}\Big(r^2 a'' + (5 -2\sqrt{2}r^2)ra' -8\sqrt{2}r^2 a \Big)-a,
\]
which is the first asserted equation. The derivation of the second equation is similar. 

To see the first integration by parts formula,  compute that 
\begin{align*}
    b\cdot p(a) - a \cdot p^\ast(b) &= \frac{1}{2} ra b' + \frac{1}{2} r ba' + 2ab-\sqrt{2} r^2 ab\\
    & = \frac{1}{2} ra b' + \frac{1}{2} r ba' + 2ab-\frac{1}{2}f' r ab \\
    & = \frac{e^f}{2r^3}(r^4ab e^{-f})',
\end{align*}
which after some rearrangement gives the asserted identity.  The second integration by parts formula follows from the first by noting that 
\begin{align*}
    a\cdot p(p^\ast(b)) -b\cdot p(p^\ast(a))& = a\cdot p(p^\ast(b))-p^\ast(b) p^\ast(a) -\Big(b p(p^\ast(a)) -  p^\ast(a) p^\ast(b) \Big)\\
    & = \frac{e^f}{2r^3}\Big(r^4 \Big(p^\ast(b) a -b p^\ast(a)\Big)e^{-f}\Big)'\\
    & = \frac{e^f}{4r^3}\Big(r^5 \Big(ba'- ab'\Big)e^{-f}\Big)'.
    \end{align*}
The last formula follows similarly. 
\end{proof}

\begin{corollary}\label{cor:p-pstar-F}
    \begin{equation}
    p(p^\ast(F)) + F = r^2.
\end{equation}
\end{corollary}
\begin{proof}
Using the expressions from \eqref{eq:Fder1}, \eqref{eq:Fder2}, \eqref{eq:Fder3}, it is straightforward to compute that 
\begin{align*}
    & r^2 F''+ (5 -2\sqrt{2} r^2)r F'- 4\sqrt{2}r^2 F  \\
    & = r^2 \left(-\frac{10c_0\sqrt{2}}{r^6} - \frac{6c_0}{r^4}\right) + (5 -2\sqrt{2} r^2)r \left(\frac{2c_0\sqrt{2}}{r^5} + \frac{2c_0}{r^3}\right) - 4\sqrt{2} r^2 \left(-\frac{c_0\sqrt{2}}{2 r^4} - \frac{c_0}{r^2} + \frac{1}{\sqrt{2}}\right) \\
     & =  \left(-\frac{10c_0\sqrt{2}}{r^4} - \frac{6c_0}{r^2}\right) + \left(\frac{10c_0\sqrt{2}}{r^4} + \frac{10c_0}{r^2} - \frac{8c_0}{r^2} - 4\sqrt{2}c_0\right) +  \left(\frac{4c_0}{ r^2}+ 4\sqrt{2}c_0 - 4 r^2\right) \\
     &= - 4 r^2.
\end{align*}
Now the claim follows from the previous lemma. 
\end{proof}

Next, we introduce a condition that ensures $R_b = \infty$ and give a formula for $q$ in this case. 
\begin{lemma}\label{lem:q-formulas}
If $b = b(r)$ is a smooth radial function satisfying
\[
\int_1^\infty b \, r^5 e^{-f} \, dr < \infty,
\]
then $R_b = \infty$, and the function $q(b)$ is given by the formula
\begin{equation}
        q(b)(r) := -2 \sqrt{2}\frac{e^f}{r^4F}\int_1^r F\,\big(p(p^\ast(b)) + b\big) \, t^3 e^{-f} \, dt,
\end{equation}
or, equivalently,
\begin{equation}\label{eq:q-simpler-def}
    q(b) = \frac{rF}{\sqrt{2}}\left(\frac{b}{F}\right)' +  2\sqrt{2} \frac{e^f}{r^4F}\left( c_0 b(1) -   \int_1^r b\, t^5 e^{-f} \,dt\right).
\end{equation}
This is the unique solution to the ODE \eqref{eq:q-ODE} (without initial condition) up to a multiple of the solution of the homogeneous equation, namely
$\frac{\sqrt{2}}{rF} \frac{e^f}{r^3}$, which blows up as $r\to 1$. In particular, \eqref{eq:q-simpler-def} is the only solution which stays bounded as $r\to 1$.
\end{lemma}

\begin{proof}
Putting Lemma \ref{lem:p-pstar-formulas} and Corollary \ref{cor:p-pstar-F} together, we obtain
\begin{align*}
    \left(\frac{r^5}{4e^f}\big(bF' - Fb'\big)\right)' &= \Big(F\big(p(p^\ast(b)) + b\big)-\big(p(p^\ast(F))+F\big)b \Big) r^3 e^{-f},\\
    & = F\big(p(p^\ast(b))+b\big)r^3 e^{-f} -b r^5 e^{-f}.
\end{align*}
Consequently, recalling that $F(1) = 0$, $F'(1) = 2$, and $e^{-f(1)} = 2c_0$, integrating both sides gives 
\[
\frac{r^5}{4e^f}\big(bF' - Fb'\big) -  c_0 b(1) = \int_1^r F\big(p(p^\ast(b)) + b\big) t^3 e^{-f} \, dt - \int_1^r b \,t^5 e^{-f} \,dt.
\]
Rearranging and multiplying through by $2\sqrt{2} \frac{e^f}{r^4F}$ gives  
\begin{equation*}
-2\sqrt{2} \frac{e^f}{r^4F} \int_1^r F\big(p(p^\ast(b)) + b\big) t^3 e^{-f} \, dt =  \frac{rF}{\sqrt{2}}\left(\frac{b}{F}\right)' +  \frac{\sqrt{2}}{rF} \frac{e^f}{r^3} 2c_0 b(1) - 2 \frac{\sqrt{2}}{rF} \frac{e^f}{r^3} \int_1^r b\, t^5 e^{-f} \,dt,
\end{equation*}
where we note that $\frac{\sqrt{2}}{rF} \frac{e^f}{r^3} 2c_0 b(1)$ is a solution to the homogeneous equation associated to the inhomogeneous equation satisfied by $q$. 

On the other hand, by integrating \eqref{eq:q-ODE} (and using that $F(1) = 0$), we find 
\begin{align*}
    q =- 2\sqrt{2} \frac{e^f}{r^4 F}\int_1^r F \big(p(p^\ast(b)) + b\big)\, t^3 e^{-f} \, dt. 
\end{align*}
These last two equations are the asserted formulas for $q$ and the assumption that $\int_1^\infty b \, r^5 \, dr < \infty$ ensures $q$ is defined for $r \in (1, \infty)$. To see that these formulas for $q$ give the initial condition
\[
q(1) = 0
\]
we investigate the limit of the second formula as $r \to 1^+$. Observe that we can write 
\[
\frac{rF}{\sqrt{2}}\left(\frac{b}{F}\right)' +  \frac{\sqrt{2}}{rF} \frac{e^f}{r^3} 2c_0 b(1)  = \frac{r}{\sqrt{2}}\left(b' - \frac{r^5F'b-2c_0F'(1)b(1) e^f}{r^5F} \right).
\]
Hence, by L'H\^opital's rule, we have 
\begin{align*}
  \lim_{r \to 1^+}\left( \frac{rF}{\sqrt{2}}\left(\frac{b}{F}\right)' +  \frac{\sqrt{2}}{rF} \frac{e^f}{r^3} 2c_0 b(1) \right)  & = \lim_{r\to 1^+}  \frac{r}{\sqrt{2}}\left(b' - \frac{r^5F'b-2c_0F'(1)b(1) e^f}{r^5F} \right)\\
  & = \frac{1}{\sqrt{2}}\left(b'(1) - \lim_{r\to1^+}\frac{5r^4F'b + r^5F''b + r^5F' b'-2c_0 F'(1) b(1) f' e^f}{5r^4 F + r^4 F'}\right)\\
  & = \frac{1}{\sqrt{2}}\left(b'(1) - \frac{10b(1) + (-14+4\sqrt{2})b(1) + 2 b'(1)-4\sqrt{2} b(1)}{2}\right)\\
  & = \sqrt{2}b(1), 
\end{align*}
where we have used that $F''(1) = -14 + 4 \sqrt{2}$ and $f'(1) = 2\sqrt{2}$.

On the other hand, 
\begin{align*}
  \lim_{r\to1^+} \left( - 2 \frac{\sqrt{2}}{rF} \frac{e^f}{r^3} \int_1^r b\, t^5 e^{-f} \,dt\right) &= -\frac{\sqrt{2}}{c_0} \lim_{r \to 1+} \frac{\int_1^r b t^5\, e^{-f} \, dt}{F} \\
  & = -\frac{\sqrt{2}}{c_0} \lim_{r \to 1+} \frac{ b r^5\, e^{-f} }{F'} \\
  & = -\sqrt{2}\, b(1).
\end{align*}
Putting these together, we conclude that 
\[
\lim_{r \to 1^+} \left(-2\sqrt{2} \frac{e^f}{r^4F} \int_1^r F\big(p(p^\ast(b)) + b\big) t^3 e^{-f} \, dt\right) = 0.
\]
This completes the proof. 
\end{proof}

\begin{remark}\label{rem:ricci-ha-kb}
    The Ricci tensor can be realized both as a $2$-tensor of the form $h_a$ and $2$-tensor of the form $k_b$. Indeed,
    \begin{align*}
        \Ric &= h_a \qquad \text{for} \qquad a = - c_0 (r^{-4}), \\
        \Ric &= k_b \qquad\, \text{for} \qquad b = - c_0(\sqrt{2}  r^{-2} + r^{-4}).
    \end{align*}
    To see this, take $a$ and $b$ as such and verify that $p(a) = -b$ and $p^\ast(b) = a$. It follows that $p(p^\ast(b)) + b = 0$. In particular $q(b)= 0$. From these, and recalling \eqref{eq:ricci-in-basis}, the asserted formulas for Ricci follow.  
\end{remark}
\begin{remark}
     Our formula for $q(b)$ was constructed under the requirement that $q(b)(1) = 0$, which ensures that $k_b$ remains bounded as $r \to 1$. For $q(b)$ to lie in $L^2_f$ (or $H^1_f$), the presence of the factor $\frac{e^f}{r^4F}$ imposes an additional constraint though: we must have $\int_1^\infty b r^5 e^{-f} \, dr =c_0b(1)$. In particular, later we will see that for $k_b$ to be $L^2_f$ orthogonal to $\Ric$ and lie in $H^1_f$, then one needs $b(1) = 0$ and $\int_1^\infty b r^5 e^{-f} \, dr = 0$. We will see this in a discussion of $k_F$ in Remark \ref{rem:kF} below, where we take $b = F$.
\end{remark}

Having established a formula for $q$, we now observe that $h_a$ and $k_b$ are divergence free and determine their $L^2_f$ inner products. 

\begin{lemma}
If $a = a(r)$ and $b = b(r)$ are smooth radial functions, then $\mathrm{div}_f(h_a) = \mathrm{div}_f(k_b)= 0$. 
\end{lemma}

\begin{proof}
Let $h_0 = a + p(a)$, $h_1 = a - p(a)$, and $h_6 = 0$. Substituting these choices into the left hand side of \eqref{eq:divf-hU-zero}, we find 
\begin{align*}
    &r\big(h_0 + h_1)'  -2\sqrt{2}r^2\big (h_0 +h_1\big)+4 h_1+ r(\sqrt{2}h_6\big)' +\left(\frac{4-2r^2}{F} \right)\big(\sqrt{2}h_6\big) \\
    & = 2ra'  -4\sqrt{2}r^2a+4a - 4p(a) \\
    & = 0,
\end{align*}
in view of the definition of $p(a)$. This shows $\mathrm{div}_f(h_a) = 0$. 

Next, let $h_0 = p^\ast(b)-b$, $h_1 = p^\ast(b) +b$, and $h_6 = q(b)$. We begin by noting that 
\begin{equation}\label{eq:soliton-eq}
\frac{e^{f}}{r^4F}\left(\frac{r^4F}{e^f} \right)' = \frac{4 - 2r^2}{rF}.
\end{equation}
This is in fact the soliton equation for the FIK shrinker. Substituting our new choices for $h_0,h_1,h_6$ into the left hand side of \eqref{eq:divf-hU-zero}, and noting the definition of $p$ as above, we find 
\begin{align*}
     &r\big(h_0 + h_1)'  -2\sqrt{2}r^2\big (h_0 +h_1\big)+4 h_1+ r(\sqrt{2}h_6\big)' +\left(\frac{4-2r^2}{F} \right)\big(\sqrt{2}h_6\big) \\
    & = 2r(p^\ast (b))' - 4 \sqrt{2} r^2 p^\ast (b) + 4 p^\ast (b) +4b + \sqrt{2}rq' +\sqrt{2} \left(\frac{4-2r^2}{F} \right) q\\
    & = 2\left( r(p^\ast(b))' + 2(1-\sqrt{2} r^2) p^\ast(b) + 2b\right) +  \sqrt{2}r\left(q' +\left(\log\left(\frac{r^4F}{e^f}\right) \right)'q\right) \\
    & = 4\left(p(p^\ast(b)) + b\right) +  \sqrt{2}r\left(q' +\left(\log\left(\frac{r^4F}{e^f}\right) \right)'q\right)\\
    & = 0,
\end{align*}
by the definition of $q(b)$ as a solution of the ODE \eqref{eq:q-ODE}.
\end{proof}

\begin{lemma}
For radial functions $a = a(r), \tilde{a} = \tilde{a}(r), b = b(r), \tilde{b} = \tilde{b}(r)$,
    \begin{align*}
        g(h_a, h_{\tilde{a}}) &= \frac{1}{2}\Big(a\tilde{a} + p(a) p(\tilde{a})\Big), \\
        g(k_b, k_{\tilde{b}}) &= \frac{1}{2}\Big(b\tilde{b} + p^\ast(b) p^\ast(\tilde{b}) + \frac{1}{2}q(b)q(\tilde{b})\Big), \\
        g(h_a, k_b) & = \frac{1}{2}\Big(a p^\ast(b)-p(a)b\Big) = -\frac{1}{4}r^{-3}e^f  \Big( r^4 e^{-f} ab\Big)'.
    \end{align*}
    In particular, 
    \begin{align*}
        g(h_a, \Ric) & = \frac{c_0}{4}r^{-3}e^f  \Big(\big(1+\sqrt{2}  r^2 \big)ae^{-f} \Big)', \\
        g(k_b, \Ric) & = \frac{c_0}{4}r^{-3}e^f  \Big(b  e^{-f} \Big)'.
    \end{align*}
\end{lemma}

\begin{proof}
    The first two identities are immediate consequences of the general fact that $(x+y)^2 + (x-y)^2 = 2x^2 + 2y^2$ along with the definitions of $h_a, k_b$, the orthogonality of basis $2$-tensors, and the normalization $|\mathbf{b}_p|^2 = \frac{1}{4}$. The third identity follows similarly  from $(x +y)(\tilde{x} - \tilde{y}) + (x-y)(\tilde{x} + \tilde{y}) = 2x\tilde{x}-2y\tilde{y}$ along with the integration by parts identity established in Lemma \ref{lem:p-pstar-formulas}. The remaining two formulas involving Ricci follow from the third and the formulas in Remark \ref{rem:ricci-ha-kb}.
\end{proof}

\begin{corollary} \label{cor:ha-kb-inner-products}
If $a = a(r)$ and $b = b(r)$ are smooth radial functions such that $h_a, k_b \in  L^2_f(M)$, then
\begin{align*}
    \int_M g(h_a, k_b) e^{-f} d\mu_g & =- 8\pi^2 \left(\lim_{r\to\infty}\Big[r^4 ab e^{-f}\Big] - 2c_0a(1)b(1)\right),   \\
    \int_M g(h_a, \Ric) e^{-f} d\mu_g &= 8\pi^2c_0\left( \lim_{r \to \infty}\Big[(1 + \sqrt{2} r^2) a e^{-f}\Big] - 2a(1) \right), \\
     \int_M g(k_b, \Ric) e^{-f} d\mu_g &= 8\pi^2c_0 \left(\lim_{r \to \infty}\Big[ b e^{-f}\Big] - 2c_0b(1) \right).
\end{align*}
\end{corollary}

\subsection{Parametrizing $h_U$ by $h_a$ and $k_b$}
Before we show that $h_U$ can be expressed as a sum of $h_a$ and $k_b$,  let us revisit our Ricci orthogonality condition from Lemma \ref{lem:ricci-orthogonality}. 

\begin{proposition}
  Suppose that $h_U = h_0 \mathbf{b}_0 + h_1 \mathbf{b}_1 + h_6 \mathbf{b}_6$ is a radially symmetric $2$-tensor. Then  
  \begin{align}
\frac{8}{c_0}g(\Ric, h_U) r^3 e^{-f} & =  2 \sqrt{2} \Big(h_0 - \Big(1+ \frac{\sqrt{2}}{r^2}\Big) h_1\Big)r e^{-f} \\
\nonumber & = \Big( (1 + \sqrt{2} r^2) (h_0 + h_1 + \sqrt{2}h_6)e^{-f}\Big)'+ \frac{4(r^2+\sqrt{2})}{(r^2-1)(r^2+c_0) } h_6\, r^3e^{-f}\\
 \nonumber   & \qquad - (1+\sqrt{2} r^2)\frac{8r}{s}\langle \mathrm{div}_f(h_U), e^0 \rangle e^{-f}.
  \end{align}
\end{proposition}

\begin{proof}
    Taking $h = h_U$ and recalling Lemma \ref{lem:ricci-orthogonality} from above, we start from 
    \[
    \frac{8}{c_0}g(\Ric, h) r^3 e^{-f} = \left( 2\sqrt{2} r\,(h_0 -h_1) - \frac{4}{r} h_1 \right)  e^{-f}.
    \]
    By adding and subtracting the same term, we write this as
    \begin{align*}
    \frac{8}{c_0}g(\Ric, h) r^3 e^{-f} &=  2\sqrt{2} r\,\Big((h_0 -h_1)-\frac{1}{2}\Big(r(h_0 + h_1)'+2(1-\sqrt{2} r^2)(h_0 + h_1)\Big)\Big)  e^{-f} \\
    & \qquad + \left(\sqrt{2}r^2(h_0 + h_1)'+2\sqrt{2} r(1-\sqrt{2} r^2)(h_0 + h_1) - \frac{4}{r} h_1 \right) e^{-f}.
    \end{align*}
Then by recognizing the operator $p$ in the first line and rearranging the second line, we obtain: 
    \begin{align*}
    \frac{8}{c_0}g(\Ric, h) r^3 e^{-f} & =  2\sqrt{2} r\,\Big((h_0 -h_1)-p(h_0 + h_1)\Big)  e^{-f} \\
    & \qquad -\frac{1}{r}\Big(r (h_0 + h_1)'-2\sqrt{2} r^2(h_0 + h_1)+ 4 h_1\Big) e^{-f}\\
    & \qquad + \left((1+\sqrt{2}r^2)(h_0 + h_1)' - 4r^3(h_0 + h_1) \right) e^{-f}.
    \end{align*}
Now observe that the last line is a derivative, since
\begin{align*}
    \Big( (1 + \sqrt{2} r^2) (h_0 + h_1)e^{-f}\Big)' = \left((1+\sqrt{2}r^2) (h_0 + h_1)'-4 r^3 (h_0 + h_1) \right)e^{-f}.
\end{align*}
On the other hand, recalling our expression for the divergence in Lemma \ref{lem:gauge-equations}, we observe 
\begin{align*}
    2\sqrt{2} r\,\Big((h_0 -h_1)-p(h_0 + h_1)\Big) 
    & =\sqrt{2} r\,\left(\frac{e^f}{r^3 F} \left(\frac{r^4F}{e^f} \sqrt{2} h_6\right)'- \frac{8r^2}{s}\langle \mathrm{div}_f(h_U), e^0 \rangle \right),\\
   -\frac{1}{r}\Big(r (h_0 + h_1)'-2\sqrt{2} r^2(h_0 + h_1)+ 4 h_1\Big)  & =  \frac{1}{r} \left(\frac{e^f}{r^3 F} \left(\frac{r^4F}{e^f} \sqrt{2} h_6\right)'-\frac{8r^2}{s}\langle\mathrm{div}_f(h_U), e^0 \rangle \right). 
\end{align*}
Putting these formulas together, we obtain 
  \begin{align*}
\frac{8}{c_0}g(\Ric, h_U) r^3 e^{-f} 
 & = (1+\sqrt{2} r^2)\,\left(\frac{e^f}{r^4 F} \left(\frac{r^4F}{e^f} \sqrt{2} h_6\right)'- \frac{8r}{s}\langle \mathrm{div}_f(h_U), e^0 \rangle \right) e^{-f} \\
    & \qquad +  \Big( (1 + \sqrt{2} r^2) (h_0 + h_1)e^{-f}\Big)'.
\end{align*}
Finally, to simplify our formula a bit further, we move the derivative on $h_6$ so that 
  \begin{align*}
\frac{8}{c_0}g(\Ric, h_U) r^3 e^{-f} 
 & =  \Big( (1 + \sqrt{2} r^2) (h_0 + h_1 + \sqrt{2}h_6)e^{-f}\Big)'-\sqrt{2}rF\, \left(\frac{1+\sqrt{2} r^2}{r^4 F}\right)'  h_6 \,r^3e^{-f} \\
    & \qquad - (1+\sqrt{2} r^2)\frac{8r}{s}\langle \mathrm{div}_f(h_U), e^0 \rangle e^{-f}.
\end{align*}
After verifying that 
\[
\sqrt{2} r F \, \left(\frac{1+\sqrt{2} r^2}{r^4 F}\right)' = -\frac{2\sqrt{2}(r^2+\sqrt{2})}{Fr^4} =  \frac{4(r^2 + \sqrt{2})}{(r^2-1)(r^2+c_0)},
\]
this completes the derivation. 
\end{proof}

\begin{corollary}\label{cor:integral-input-one}
    Suppose that $h_U = h_0 \mathbf{b}_0 + h_1 \mathbf{b}_1 + h_6 \mathbf{b}_6$ is a smooth radially symmetric $2$-tensor on $M$ with $\mathrm{div}_f(h_U) = 0$. Then 
\begin{align*}
    \frac{1}{\sqrt{2}}\int_1^r \Big(h_0 - \Big(1 + \frac{\sqrt{2}}{t^2}\Big)h_1\Big) t \, e^{-f} \, dt  &= \frac{1}{4}(1 + \sqrt{2} r^2) (h_0 + h_1 + \sqrt{2} h_6) e^{-f} -  \frac{1}{2}\big(h_0+h_1+ \sqrt{2} h_6\big)(1)\\
    & \qquad + \int_1^r\frac{(t^2+\sqrt{2})}{(t^2-1)(t^2+c_0) } h_6\, t^3e^{-f}\, dt.  
\end{align*}
When $\lim_{r \to \infty} r^2(h_0 + h_1 + \sqrt{2} h_6) e^{-f} = 0$, then 
\begin{align*}
- \frac{1}{\sqrt{2}}\int_r^\infty \Big(h_0 - \Big(1 + \frac{\sqrt{2}}{t^2}\Big)h_1\Big) t \, e^{-f} \, dt & = \frac{1}{4}(1 + \sqrt{2} r^2) (h_0 + h_1 + \sqrt{2} h_6) e^{-f} \\
& \qquad - \int_r^\infty \frac{(t^2+\sqrt{2})}{(t^2-1)(t^2+c_0) } h_6\, t^3e^{-f}\, dt.  
\end{align*}
\end{corollary}

\begin{proof}
We integrate the expression from the previous proposition (as well as recalling $(\sqrt{2}+1)e^{-f(0)} = 2$). 

\end{proof}

By a completely analogous and somewhat simpler derivation, one can show that: 

\begin{proposition}\label{prop:integral-input-two}
Suppose $h_U = h_0 \mathbf{b}_0 + h_1 \mathbf{b}_1 + h_6 \mathbf{b}_6$ is a smooth radially symmetric $2$-tensor on $M$. Then
\begin{align*}
     \Big(h_0 - \Big(-1+ \frac{\sqrt{2}}{r^2}\Big) h_1\Big)r  = -\frac{2\sqrt{2}r}{s}\langle\mathrm{div}_f(h_U),e^0\rangle +\frac{1}{2\sqrt{2}}(h_0 + h_1 + \sqrt{2} h_6)' -\frac{r^2-2}{rF} h_6.
\end{align*}
Moreover, if $\mathrm{div}_f(h_U) = 0$, then 
\begin{align*}
 \frac{1}{\sqrt{2}} \int_1^r\Big(h_0 - \Big(-1+ \frac{\sqrt{2}}{t^2}\Big) h_1\Big)t \, dt  &= \frac{1}{4}(h_0 + h_1 + \sqrt{2} h_6)- \frac{1}{4}(h_0 + h_1 + \sqrt{2} h_6)(1) \\
  & \qquad -\int_1^r \frac{(t^2-2)}{(t^2-1)(t^2+c_0)} h_6 \, t^3 dt.
\end{align*}
\end{proposition}

\begin{proof}
We compute
\begin{align*}
        \Big(h_0 - \Big(-1+ \frac{\sqrt{2}}{r^2}\Big) h_1\Big)r  &= -\frac{1}{2\sqrt{2}r}\left(-2\sqrt{2}r^2(h_0 + h_1) +4 h_1\right) \\
        & = -\frac{2\sqrt{2}r}{s}\mathrm{div}_f(h_U) +\frac{1}{2\sqrt{2}}(h_0 + h_1)' +\frac{e^f}{2r^4F}\left(\frac{r^4F}{e^f}  h^6\right)' \\
        & = -\frac{2\sqrt{2}r}{s}\mathrm{div}_f(h_U) +\frac{1}{2\sqrt{2}}(h_0 + h_1 + \sqrt{2} h_6)' +\frac{e^f}{2r^4F}\left(\frac{r^4F}{e^f}  \right)'h_6.
\end{align*}
Recalling \eqref{eq:soliton-eq} gives the first asserted identity. Assuming that $\mathrm{div}_f(h_U) = 0$ and integrating gives the second identity. 
\end{proof}

The following is the main result of this section. 

\begin{proposition}\label{prop:parametrizing-hU}
    Suppose that $h_U = h_0\mathbf{b}_0 + h_1 \mathbf{b}_1 + h_6 \mathbf{b}_6$ is a smooth radially symmetric $2$-tensor satisfying $\mathrm{div}_f(h_U) = 0$ and $ \Xi(h_U) = 0$, then there exists smooth radial functions $a = a(r), b = b(r)$ satisfying $a, b, p(a), p^\ast(b), q(b) \in H^1_f(M)$,
    \[
    \lim_{r \to \infty} r^3 \, a \,e^{-\frac{f}{2}} = \lim_{r \to \infty} r^3 \, b \, e^{-\frac{f}{2}} = 0,  
    \]
    and  $a(1) = b(1) = 0$,  such that 
    \[
    h_U = h_a + k_b. 
    \]
    Moreover, we have the orthogonality relations
    \[
    \int_M g(h_a, \Ric) e^{-f} d\mu_g  = \int_M g(k_b, \Ric) e^{-f} d\mu_g = \int_M g(h_a, k_b) e^{-f} d\mu_g = 0.   
    \]
\end{proposition}
\begin{proof}
Let us begin by considering functions $a = a(r)$ and $b = b(r)$ for $r \in [1, \infty)$ that solve the linear system 
\begin{equation}\label{eq:ode-ab}
\begin{bmatrix}
    a'\\ b'
\end{bmatrix}+
\frac{2}{r}\begin{bmatrix} 1-
     \sqrt{2} r^2  & -1\\ -1  & 1 
\end{bmatrix}
\begin{bmatrix}
   a\\b 
\end{bmatrix}
= \frac{1}{r} 
\begin{bmatrix}
  ( h_0 - h_1) \\-(h_0 + h_1)
\end{bmatrix}.
\end{equation}
The equations are exactly the relations
\begin{equation}\label{eq:ode-ab-2}
 p(a) - b  = \frac{1}{2} (h_0 -h_1), \qquad p^\ast(b) +a = \frac{1}{2}(h_0 + h_1),
\end{equation}
or, equivalently, 
\[
h_0 = a+p(a) +p^\ast(b) - b, \qquad h_1 = a - p(a) +p^\ast(b) +b. 
\]
Introducing labels for the integrals in Corollary \ref{cor:integral-input-one} and Proposition \ref{prop:integral-input-two}, 
\begin{align*}
    I_1(r)& := \frac{1}{\sqrt{2}} \int_1^r  \Big(h_0 - \Big(-1+ \frac{\sqrt{2}}{t^2}\Big) h_1\Big) t \, dt, \\
    I_2(r)&:= \frac{1}{\sqrt{2}}\int_1^r  \Big( h_0 -\Big(1+ \frac{\sqrt{2}}{t^2}\Big) h_1 \Big) t e^{-f}dt, 
\end{align*}
we can verify that a general solution of such a system is given by 
\begin{align}
a &= \frac{1}{r^4}\Big[C_1c_0 +I_1(r)\Big]  + \frac{\sqrt{2} r^2 - 1}{r^4}e^f \Big[C_2 + I_2(r) \Big], \label{eq:def-a}\\
b & = -\frac{\sqrt{2} r^2 + 1}{r^4}\Big[C_1 c_0 + I_1(r) \Big]  +\frac{1}{r^4} e^f \Big[C_2 + I_2(r) \Big] \label{eq:def-b},
\end{align}
where $C_1, C_2$ are undetermined constants. 

To confirm this, first observe that the homogeneous solution of the system is given by  
\[
\begin{bmatrix} 
a \\ b 
\end{bmatrix} 
= C_1\frac{c_0}{r^4}\begin{bmatrix} 1 \\ -(\sqrt{2} r^2+1) \end{bmatrix} +C_2  \frac{e^{f}}{r^4} \begin{bmatrix}  (\sqrt{2}r^2 -1)\\ 1 \end{bmatrix}.
\]
Then, after labeling a fundamental matrix and the inhomogeneous term
\[
\mathbf{X} :=  \frac{1}{r^4}\begin{bmatrix} c_0 & e^f( \sqrt{2} r^2-1)  \\ -c_0(\sqrt{2} r^2 +1)  & e^f\end{bmatrix}, \qquad \mathbf{y}:= \frac{1}{r} 
\begin{bmatrix}
  ( h_0 - h_1) \\-(h_0 + h_1)
\end{bmatrix}, 
\]
it can be computed that
\[
\mathbf{X}^{-1} \mathbf{y} = \begin{bmatrix}
\frac{1}{c_0\sqrt{2}}  \left(h_0 -\left(-1 +  \frac{\sqrt{2}}{r^2}\right) h_1\right)r \\ \frac{1}{\sqrt{2}} \left(h_0 -\left(1 + \frac{\sqrt{2}}{r^2}\right)h_1\right) r e^{-f} 
\end{bmatrix} = \begin{bmatrix}\frac{1}{c_0} I_1'(r) \\ I_2'(r) \end{bmatrix}.
\]
It follows that a particular solution of the ODE is given by 
\[
\mathbf{X} \begin{bmatrix} \frac{1}{c_0}I_1(r) \\ I_2(r) \end{bmatrix} =\frac{1}{r^4} \begin{bmatrix} I_1(r)  + e^f (\sqrt{2} r^2 -1) I_2(r) \\ -(\sqrt{2} r^2 + 1) I_1(r) + e^f I_2(r) \end{bmatrix}.
\]
Putting this together with the homogeneous solution completes the proof of the claim. 

From \eqref{eq:def-a} and \eqref{eq:def-b}, we observe that 
\begin{align*}
a(1) &=c_0 C_1 +\frac{1}{2}C_2, \\
b(1) &=- C_1 + \frac{1}{2c_0}C_2. 
\end{align*}
So to ensure $a(1) = b(1) = 0$, we must take $C_1 = C_2 = 0$. Then
\begin{align}
a &= \frac{1}{r^4}I_1(r) + \frac{\sqrt{2} r^2 - 1}{r^4}e^f I_2(r), \label{eq:def-a-2}\\
b & = -\frac{\sqrt{2} r^2 + 1}{r^4} I_1(r)   +\frac{1}{r^4} e^f  I_2(r) \label{eq:def-b-2}. 
\end{align}
Next, we observe that 
\begin{align*}
(1 + \sqrt{2} r^2) ae^{-f} &= \frac{ \sqrt{2} r^2 + 1}{r^4}I_1(r) e^{-f} + \frac{2 r^4 - 1}{r^4} I_2(r),\\
be^{-f}  & = -\frac{\sqrt{2} r^2 + 1}{r^4} I_1(r)e^{-f}  +\frac{1}{r^4}  I_2(r), \\
r^4 ab e^{-f} 
& = - \frac{\sqrt{2} r^2 + 1}{r^4} I_1(r)^2 e^{-f}  - \frac{2 r^4 -2}{r^4} I_1(r) I_2(r)  + \frac{\sqrt{2} r^2 -1}{r^4} I_2(r)^2 e^{f}. 
\end{align*}
To evaluate the limits of these as $r \to \infty$, we verify the following claims:
\begin{enumerate}
\item[(i)]  $|I_1(r)| \leq C \|h\|_{L^2_f} \,r^{-2} e^{\frac{f(r)}{2}}$.
\item[(ii)] $|I_2(r)|\leq C \|h\|_{H^1_f}\,  r^{-2} e^{-\frac{f(r)}{2}}$. 
\end{enumerate}
To prove these, we will make use of the following integral exponential integral estimates 
\[
\int_1^r t^k e^{\sqrt{2} t^2} \, dt \leq C_k r^{k-1} e^{\sqrt{2} r^2}, \qquad \int_r^\infty t^k e^{-\sqrt{2} t^2} \, dt \leq C_k r^{k-1} e^{-\sqrt{2}r^2}. 
\]
Now for (i), note that for $t \in [1, r]$, $|-1 +\sqrt{2}/t^2| \leq1$. Additionally, since $h_0, h_1 \in H^1_f(M)$, by Proposition \ref{prop:L2f-decay} we have $\|r h_0\|_{L^2_f},\|r h_1\|_{L^2_f} \leq C \|h\|_{H^1_f}$. Using these facts and H\"older's inequality, we have 
\begin{align*}
| I_1(r)|  &\leq  \frac{1}{\sqrt{2}} \int_1^r  \big(|h_0| + |h_1|\big) t  dt  \\
& \leq  \left(\int_1^r (|h_0|^2 + |h_1|^2) t^5 e^{-f(t)} \, dt  \right)^{\frac{1}{2}} \left(\int_1^r t^{-3} \,e^{f(t)} dt \right)^{\frac{1}{2}} \\
& \leq C\|h\|_{H^1_f}\,  r^{-2} e^{\frac{f(r)}{2}}.
\end{align*}

For (ii), we note that by Lemma \ref{lem:ricci-orthogonality} and the fact that $h_0, h_1 \in L^2_f(M)$, we have 
\begin{align*}
|I_2(r)|&=\frac{1}{\sqrt{2}} \bigg|\int_r^\infty \Big( h_0 -\Big(1+ \frac{\sqrt{2}}{r^2}\Big) h_1 \Big) t e^{-f(t)}dt\bigg| \\
& \leq \frac{\sqrt{2} + 1}{\sqrt{2}} \int_r^\infty \big(|h_0| + |h_1|\big) t e^{-f(t)} \, dt \\
& \leq (\sqrt{2} +1) \left(\int_r^\infty\big(|h_0|^2 + |h_1|^2\big) t^3 e^{-f(t)} \, dt \right)^{\frac{1}{2}} \left(\int_r^\infty t^{-1} e^{-f(t)}\, dt\right)^{\frac{1}{2}}. 
\end{align*}
Using Proposition \ref{prop:L2f-decay}, we can bound the first integral by $Cr^{-1} \|h\|_{H^1_f}$. Using the integral estimates from above, we conclude $|I_2(r)| \leq C \|h\|_{H^1_f} r^{-2} e^{-\frac{f(r)}{2}}$,  which is (ii). 

In particular, we note that $\lim_{r \to \infty} I_2(r) = 0$ (which is the Ricci orthogonality condition) and that $\lim_{r \to \infty} I_1(r) I_2(r) = 0$ (since $|I_1(r) I_2(r)| \leq C\|h\|_{H^1_f}^2  r^{-4}$). 

Putting the estimates above together, we readily conclude 
\begin{align*}
\lim_{r \to \infty} (1 + \sqrt{2} r^2) ae^{-f} &=0 ,\\
\lim_{r \to \infty} be^{-f}  & = 0, \\
\lim_{r \to \infty} r^4 ab e^{-f} &= 0.
\end{align*}
Together with $a(1) = b(1) = 0$, the asserted orthogonality of $h_a$ and $k_b$ with $\Ric$ and each other now follows from Corollary \ref{cor:ha-kb-inner-products}. 
In fact, our estimates above show
\begin{align*}
    |r^3 \, a e^{-\frac{f}{2}}| &\leq r^{-1} |I_1(r)|e^{-\frac{f(r)}{2}} + C r |I_2(r)| e^{\frac{f(r)}{2}} \leq C\|h\|_{H^1_f} \big(r^{-3}+ r^{-1}\big) \leq C \|h\|_{H^1_f}r^{-1}, \\
    |r^3 b e^{-\frac{f}{2}}| & \leq Cr|I_1(r)| e^{-\frac{f(r)}{2}} +  r^{-1} |I_2(r)| e^{\frac{f(r)}{2}} \leq C\|h\|_{H^1_f} (r^{-1}+ r^{-3})\leq C \|h\|_{H^1_f}r^{-1}, 
\end{align*}
and hence we obtain the slightly stronger asymptotics
\begin{align}
    \lim_{r \to \infty} r^3 \,a\, e^{-\frac{f}{2}} &= 0, \\
    \lim_{r \to \infty} r^3 \,b \,e^{-\frac{f}{2}} & = 0. 
\end{align}

The fact that $a, b \in H^1_f(M)$ can be seen in a few ways. The most straightforward of these is to note that by the estimates above, we evidently have $r^\frac{3}{2} a, r^{\frac{3}{2}} b \in L^2_f$ (since $r^3 (a^2 + b^2)\,r^3 e^{-f} \leq C \|h\|_{H^1_f}^2 r^{-2}$). So from the ODE \eqref{eq:ode-ab} and our assumption that $r h_0, rh_1 \in L^2_f$, we conclude $r^{\frac{1}{2}}a', r^{\frac{1}{2}}b' \in L^2_f$ (in fact $r^2 b' \in L^2_f$). Now \eqref{eq:ode-ab-2} gives $rp(a), rp^\ast(b) \in L^2_f$. Differentiation shows $p(a)', p^\ast(b)' \in L^2_f$. We conclude $a, b, p(a), p^\ast(b) \in H^1_f(M)$. 

Finally, observe that since $\mathrm{div}_f(h_U) = 0$, by construction we have  
\begin{align*}
r(\sqrt{2}h_6)' +\frac{4-2r^2}{F} (\sqrt{2} h_6) &= -r(h_0 + h_1)' + 2\sqrt{2} r^2 (h_0 + h_1) - 4h_1 \\
& = -2r(p^\ast(b) + a)' + 4\sqrt{2} r^2(p^\ast(b) + a) - 4(a - p(a) +p^\ast(b) +b) \\
& = -\Big(2r(p^\ast(b))' +4(1-\sqrt{2} r^2 )p^\ast (b) + 4b  +2ra' + 4(1-\sqrt{2} r^2)a - 4 p(a)\Big) \\
& = -4 \big(p(p^\ast(b)) + b\big).
\end{align*}
Hence
\[
   \left(\frac{r^4F}{e^f}h_6\right)'   =  -\left(\frac{r^4F}{e^f}\right)\frac{2\sqrt{2}}{r} \big(p(p^\ast(b)) + b\big).
\]
Since by assumption $h_6 = O(r-1))$ as $r \to 1$ and $h_6$ satisfies the same ODE as $q(b)$,  we conclude that $q(b) - h_6$ is a solution to the homogeneous ODE associated with \eqref{eq:q-ODE}, hence $q(b)-h_6 = \frac{c}{rF} \frac{e^f}{r^3}$ for some $c\in \mathbb{R}$. Since $\frac{1}{rF} \frac{e^f}{r^3}$ is unbounded at $r=1$ while $q(b) - h_6$ stays bounded, we deduce that $q(b) = h_6$ so that $h = h_a + k_b$ as desired. In particular, $q(b) = h_6 \in H^1_f(M)$. This completes the proof of our desired parametrization. 
\end{proof}

\begin{remark}\label{rem:simple-form-a}
In the case where $h_6 \equiv 0$, it follows from Corollary \ref{cor:integral-input-one} and Proposition \ref{prop:integral-input-two} that $a = \frac{1}{2}(h_0 + h_1)$ and $b  \equiv 0$. Indeed, after canceling terms, we note that $b$ can be expressed as
\begin{align*}
 b  &=  -\frac{\sqrt{2} r^2 + 1}{r^4} I_1(r)   +\frac{1}{r^4} e^f  I_2(r)  \\
& =  -\frac{\sqrt{2} r^2 + 1}{r^4} \left(   -\int_1^r \frac{(t^2-2)}{(t^2-1)(t^2+c_0)} h_6 \, t^3 dt\right) + \frac{1}{r^4} e^f \left( \int_1^r\frac{(t^2+\sqrt{2})}{(t^2-1)(t^2+c_0) } h_6\, t^3e^{-f}\, dt\right),
\end{align*}
from which the assertion that $h_6 \equiv 0 \implies b \equiv 0$ follows. 
\end{remark}

A consequence of our parametrization above and approximation is the following corollary. 

\begin{corollary}\label{cor:parametrizing-hU-cs}
     Suppose that $h_U = h_0\mathbf{b}_0 + h_1 \mathbf{b}_1 + h_6 \mathbf{b}_6$ is a smooth radially symmetric $2$-tensor satisfying $\mathrm{div}_f(h_U) = 0$ and $ \Xi(h_U) = 0$. Then there exists a sequence of compactly supported, radially symmetric functions $a_l, b_l\in C^{\infty}_0(M)$ satisfying $a_l(1) = b_l(1) = 0$ such that 
     \[
     h_{a_l} +k_{b_l} \to h_U
     \]
     in the $H^1_f(M)$ norm. 
     
     In particular, if $ \delta^2\nu_g(h_a + k_b) \leq 0$ for all compactly supported, radially symmetric functions $a, b$ with $a(1) = b(1) = 0$, then $\delta^2 \nu_g(h_U) \leq 0$. 
\end{corollary}
\begin{proof}[Sketch of proof]
     From Proposition \ref{prop:parametrizing-hU}, we know that $h_U = h_a+k_b$ with $a$ and $b$ smooth (including at $r=1$) functions such that $a(1) = b(1) = 0$, $\lim_{r \to \infty} r^3 \, a \,e^{-\frac{f}{2}} = \lim_{r \to \infty} r^3 \, b \, e^{-\frac{f}{2}} = 0$, and $a, b, p(a), p^\ast(b), q(b) \in H^1_f(M)$ The only subtle issue in the approximation is that cutting off $a$ and $b$ is not directly equivalent to cutting off $h_0$, $h_1$ and $h_6$ (under the condition $\mathrm{div}_f(h_U) = 0$). 

    Consider $\chi = \chi(r)$ a cut-off function supported on $[0,2]$ and equal to $1$ on $[0,1]$, and define $\chi_n(r) = \chi(r/n)$ for $r\geq 1$ which is supported on $[0,2n]$ and equal to $1$ on $[0,n]$ and whose $k^{th}$ derivative is $O(n^{-k})$. 
    
    The main difficulty in approximation comes from $h_6$ given as a solution of the ODE 
    $\left(\frac{r^4F}{e^f}h_6\right)'   =  -\left(\frac{r^4F}{e^f}\right)\frac{2\sqrt{2}}{r} \big(p(p^\ast(b)) + b\big)$ depending on $b.$
    If for some function $b_n$, a cut-off of $b$, one has $b_n(r)=0$ for $r\in [2n,+\infty)$, then the associated $h_6$ becomes a multiple of $\frac{r^4F}{e^f}$ on $[2n,+\infty)$, in which case the approximation $h_6$ blows up at infinity. Since we have $$q(b) = \frac{rF}{\sqrt{2}}\left(\frac{b}{F}\right)' -  2\frac{\sqrt{2}}{rF} \frac{e^f}{r^3}\left( \int_1^r b\, t^5 e^{-f} \,dt\right),$$
    we need to cut off in a way that preserves the condition
    \begin{equation}\label{eq: need 0 integral for b}
        \int_1^{\infty} b\, t^5 e^{-f} \,dt = 0,
    \end{equation}
    in order to ensure that $h_6=q(b)$ vanishes in a neighborhood of infinity. Towards this goal, consider a nonnegative bump function $\phi\geq0$ supported on $[2,3]$ such that $\int_1^{+\infty} \phi\, t^5 e^{-f} \,dt = 1$. We then define 
    \[
    a_n:= \chi_na \qquad \text{ and } \qquad b_n: = \chi_nb+c_n\phi,
    \]
    where 
    \[
    c_n :=\int_1^{\infty} (1-\chi_n)b\, t^5 e^{-f} \,dt = \int_n^{2n} (1 - \chi_n) b \, t^5 e^{-f} \, dt. 
    \]
    This definition of $c_n$ ensures that  
    \[
    \int_1^{\infty} b_n\, t^5 e^{-f} \,dt = \int_1^\infty \chi_n b\, t^5 e^{-f} \, dt + c_n\int_1^\infty \phi \, t^5 e^{-f} \, dt  = \int_1^\infty b\, t^5 e^{-f} \, dt  = 0. 
    \]
    From $\lim_{r \to \infty} r^3 \, b \, e^{-\frac{f}{2}} = 0$, we additionally get that 
    \[
    c_n \leq \sup_{r\in [n, 2n]} (r^6 \,b \,e^{-f}) \int_n^{2n} t^{-1} \,dt \leq \log (2) \sup_{r\in [n, 2n]} (r^6 \,b \,e^{-f}),
    \]
    so $c_n \to 0$. Thus,
    \[
    h_{6,n} = \frac{rF}{\sqrt{2}}\left(\frac{b_n}{F}\right)' -  2\frac{\sqrt{2}}{rF} \frac{e^f}{r^3}\left( \int_1^r b_n\, t^5 e^{-f} \,dt\right)
    \]
    is smooth, equal to the initial $h_{6}$ on $[1,2]$ (which ensures the regularity of the tensor), vanishes for $r>2n$, and $h_{6,n}\to h_6$ smoothly and in $H^1_f$.

    Similarly, the functions $h_{0,n}$ and $h_{1,n}$ explicitly obtained from $a_n$ and $b_n$ using \eqref{eq:ode-ab} are smooth, equal to the initial $h_{0}$ and $h_1$ on $[1,2]$, vanish for $r>2n$ and $(h_{0,n},h_{1,n})\to (h_{0},h_1)$ smoothly and in $H^1_f$. This completes the sketch proof. 
\end{proof}

\subsection{The formulas for $L_fh_a$ and $L_fk_b$ and stability}

We next turn our attention to deriving formulas for action of $L_f$ on $h_a$ and $k_b$ and linear stability. Throughout this section we consider the spaces 
\[
\mathcal{A} := \{ u \in C^{\infty}_0([1, \infty)) : u(1) = 0\}. 
\]
By our work in the previous section, for any $a, b \in \mathcal{A}$, we have $h_a, k_b \in C^{\infty}_0(M)$ are $L^2_f$-orthogonal to each other and to the Ricci tensor. 

In this section, we will show that $\delta^2 \nu_g(h_a + k_b) \leq 0$ for all $a, b \in \mathcal{A}$. 

\begin{lemma}\label{lem:lap-radial}
If $u = u(r)$ is a smooth radial function, then 
\begin{equation}
    \Delta_f u = \frac{F}{4} \left(u'' + \left(\log(r^3F)- f\right)' u' \right). 
\end{equation}
Equivalently,
\begin{equation}
 \Delta_f u = \frac{F}{4} \left(\frac{e^f}{Fr^3}\right)\left( \frac{r^3 F}{e^f} u' \right)'.
\end{equation}
\end{lemma}

\begin{proof}
By definition of the Laplacian, our computation of the connection coefficient in Lemma \ref{lem:levi-civita}, and the assumption that $a$ is radial, we have
\begin{align*}
    \Delta_f a &= e_0(e_0(a)) - \sum_{i =1}^3 (\nabla_{e_i} e_i)(a) - e_0(f) e_0(a) \\
    & =  \frac{s}{2r} \left(\frac{s}{2r} a'\right)' +\left(\frac{s'}{2r} + \frac{s}{r^2}\right) \frac{s}{2r} a'- \frac{s^2}{4r^2} f'a' \\
    & = \frac{F}{4} (a''-f'a') + \frac{\sqrt{F}}{4}\left(\sqrt{F}\right)' a' + \frac{F}{4}\left(\frac{ r(\log F)' + 6}{2r}\right) a' \\
    & = \frac{F}{4} (a''-f'a')  + \frac{F}{4}\big(\log (r^3F)\big)' a' \\
    & = \frac{F}{4} \left(a'' + \left(\log(r^3F)- f\right)' a' \right).
\end{align*}
\end{proof}

\begin{definition}\label{def:P-and-Q}
    Given smooth radial functions $a = a(r)$ and $b = b(r)$, we define second order linear operators 
    \begin{equation}
        P(a):= \frac{F}{4r^2}(r^2 a'' + (5 - 2\sqrt{2} r^2) ra' - 8\sqrt{2} r^2 a) + a.
    \end{equation}
    and
    \begin{equation}
        Q(b):=\frac{F}{4r^2}(r^2 b'' + (5 - 2\sqrt{2} r^2) rb' - 4\sqrt{2} r^2 b) + b.
    \end{equation}
    In particular, in view of Lemma \ref{lem:p-pstar-formulas}, 
    \[
    P(a) -a = -\frac{F}{r^2}(p^\ast(p(a)) + a), \qquad Q(b) -b = -\frac{F}{r^2}(p(p^\ast(b)) +b).
    \]
    We also note that we may express $P, Q$ by the identities 
    \[
    P(a) = \frac{F}{4} \left(\frac{e^f}{r^5}\left(\frac{r^5}{e^f} a\right)' -8\sqrt{2} a\right) + a, \qquad  Q(b) = \frac{F}{4} \left(\frac{e^f}{r^5}\left(\frac{r^5}{e^f} b\right)' -4\sqrt{2} b\right) + b.
    \]
\end{definition}
These operators are closely related to the action of $L_f$
\begin{proposition}
    Suppose $a, b \in \mathcal{A}$. Then 
    \begin{equation}
        L_f h_a = h_{P(a)}
    \end{equation}
    and
    \begin{equation}
        L_f k_b = k_{Q(b)}
    \end{equation}
\end{proposition}

\begin{proof}
We begin by noting that, in general, on a Ricci gradient shrinker one has the identities 
\[
\mathrm{div}_f L_f h = \Big(\Delta_f + \frac{1}{2}\Big) \mathrm{div}_f h, \qquad L_f \Ric = \Ric. 
\]
First, let us consider $h_a$ for some $a \in \mathcal{A}$. It follows that $\mathrm{div}_fL_fh_a = 0$. Additionally, because $L_f$ is self-adjoint, we have 
\[
\int_M \langle L_f h_a, \Ric \rangle \, e^{-f} d\mu_g = \int_M \langle h_a, L_f \Ric \rangle \, e^{-f} d\mu_g = \int_M \langle h_a, \Ric\rangle \, e^{-f} d\mu_g = 0. 
\]
Lastly, it follows from our computations in Corollary \ref{cor:Lf-action-on-basis}, $L_f$ preserves the span of $\mathbf{b}_0$ and $\mathbf{b}_1$ in the radial setting. Indeed, we directly have that
\begin{align*}
L_fh_a &= \Big(\Delta_f\big(a + p(a)\big) + \Lambda_{11}^{++}(a + p(a)) + \Lambda_{11}^{\pm} (a - p(a)) \Big) \mathbf{b}_0 \\
& \qquad + \Big(\Delta_f\big(a - p(a)\big) + \Lambda_{11}^{--}(a - p(a)) + \Lambda_{11}^{\pm} (a+ p(a)) \Big) \mathbf{b}_1. 
\end{align*}
Thus, by our work on parametrizing divergence free and Ricci-orthogonal deformations in the previous section, we conclude that $L_f h_a = h_{P(a)}$ for some $P(a)$. 

To find the form of $P(a)$, we observe by the definition of $h_a$
\begin{align*}
P(a) &=2 \langle L_f h_a, \mathbf{b}_0 \rangle + 2 \langle L_f h_a, \mathbf{b}_1 \rangle \\
& = \Delta_f a +\frac{1}{2}\Big(\Lambda_{11}^{++} + \Lambda_{11}^{\pm} \Big)\big(a+p(a)\big)+ \frac{1}{2}\Big(\Lambda_{11}^{--} + \Lambda_{11}^{\pm}\Big)\big(a-p(a)\big). 
\end{align*}
Now computations from Proposition \ref{prop:Lambda-function-comps} give that 
\begin{align*}
    \frac{1}{2}\Big(\Lambda_{11}^{++} + \Lambda_{11}^{\pm} \Big)&=  -\frac{c_0}{2r^4}, \\
    \frac{1}{2}\Big(\Lambda_{11}^{--} + \Lambda_{11}^{\pm} \Big)&=  \frac{3c_0\sqrt{2}}{2r^6} + \frac{3c_0}{2r^4} - \frac{\sqrt{2}}{2r^2}. 
\end{align*}
Recalling the definition of $p(a)$,  we therefore obtain that 
\begin{align*}
    P(a) & = \Delta_f a +\Big(-\frac{c_0}{2r^4}\Big)\big(a+p(a)\big)+\Big(\frac{3c_0\sqrt{2}}{2r^6} + \frac{3c_0}{2r^4} - \frac{\sqrt{2}}{2r^2}\Big)\big(a-p(a)\big) \\
    & = \Delta_f a +\Big(-\frac{c_0}{2r^4}\Big)\Big(\frac{1}{2} r a' + (2 - \sqrt{2} r^2) a\Big)+\Big(\frac{3c_0\sqrt{2}}{2r^6} + \frac{3c_0}{2r^4} - \frac{\sqrt{2}}{2r^2}\Big)\Big(-\frac{1}{2} ra' + \sqrt{2} r^2 a\Big) \\
    & = \Delta_f a +\frac{1}{2r} \Big(-\frac{3c_0\sqrt{2}}{2r^4} - \frac{2c_0}{r^2} + \frac{1}{\sqrt{2}} \Big)a' + \Big(\frac{2c_0}{r^4} + \frac{2c_0\sqrt{2}}{r^2} - 1\Big)a . 
\end{align*}
Recalling \eqref{eq:Fder1} and \eqref{eq:Fder2}, we observe this last identity can be written 
\begin{align*}
P(a) &= \Delta_f a - \frac{1}{2r}\Big(\frac{r}{2} F' - F\Big)a' + (1 - 2\sqrt{2} F) a \\
& = \Delta_f a - \frac{F}{4} \Big(\log(r^{-2} F)\Big)'a' - \frac{F}{4}(8\sqrt{2}a) + a.
\end{align*}
Incorporating the formula for $\Delta_f a$ from Lemma \ref{lem:lap-radial} above, we conclude 
\begin{align*}
    P(a) &= \frac{F}{4} \Big(a'' + \left(\log(r^3F) - \log(r^{-2} F)- f\right)' a' - 8 \sqrt{2} a \Big) + a \\
    & = \frac{F}{4} \Big(a'' + r^{-1}\left( 5 - 2\sqrt{2} r^2\right)' a' - 8 \sqrt{2} a \Big) + a. 
\end{align*}
This establishes the asserted formula for $P(a)$. 

Next, if given $k_b$ for some $b \in \mathcal{B}$, we note that for any $\tilde{a} \in \mathcal{A}$, we have 
\[
\int_M \langle L_f k_b, h_{\tilde{a}} \rangle \, e^{-f} d\mu_g =\int_M \langle k_b, L_f h_{\tilde{a}} \rangle \, e^{-f} d\mu_g = \int_M \langle k_b, h_{P(\tilde{a})} \rangle \, e^{-f} d\mu_g = 0. 
\]
Therefore, by the same reasoning as above, it follows that $L_f k_b = k_{Q(b)}$ for some $Q(b)$. Recalling the definition of $k_b$ and taking the approach taken above gives 
\begin{align*}
Q(b) &= - 2\langle L_f k_b, \mathbf{b}_0\rangle+ 2\langle L_f k_b, \mathbf{b}_1\rangle   \\
& = \Delta_f b +\frac{1}{2}\Big(-\Lambda_{11}^{++}+\Lambda_{11}^{\pm} \Big)\big(p^\ast(b) - b\big)+ \frac{1}{2}\Big(\Lambda_{11}^{--} - \Lambda_{11}^{\pm}\Big)\big(p^\ast(b)+b\big).
\end{align*}
Proceeding as before, from Proposition \ref{prop:Lambda-function-comps}, we obtain 
\begin{align*}
    \frac{1}{2}\Big(-\Lambda_{11}^{++} + \Lambda_{11}^{\pm} \Big)&=  -\frac{c_0}{2r^4} - \frac{\sqrt{2}c_0}{2r^2}, \\
    \frac{1}{2}\Big(\Lambda_{11}^{--} - \Lambda_{11}^{\pm} \Big)&=  \frac{3c_0\sqrt{2}}{2r^6} + \frac{5c_0}{2r^4} - \frac{c_0}{r^2}. 
\end{align*}
Recalling the definition of $p^\ast(b)$,  we therefore obtain that 
\begin{align*}
Q(b) &= \Delta_f b + \Big( -\frac{c_0}{2r^4} - \frac{\sqrt{2}c_0}{2r^2}\Big) \big( p^\ast(b)-b\big) +\Big(\frac{3c_0\sqrt{2}}{2r^6} + \frac{5c_0}{2r^4} - \frac{c_0}{r^2}\Big)\big(p^\ast(b) + b\big)  \\
& = \Delta_f b +\Big( -\frac{c_0}{2r^4} - \frac{\sqrt{2}c_0}{2r^2}\Big)\Big(-\frac{1}{2} r b' -2b\Big)+\Big(\frac{3c_0\sqrt{2}}{2r^6} + \frac{5c_0}{2r^4} - \frac{c_0}{r^2}\Big)\Big(-\frac{1}{2} rb' \Big) \\
    & = \Delta_f b +\frac{1}{2r} \left(-\frac{3c_0\sqrt{2}}{2r^4} - \frac{2c_0}{r^2} + \frac{1}{\sqrt{2}} \right)b' +\Big( \frac{c_0}{r^4} + \frac{\sqrt{2}c_0}{r^2}\Big) b .
\end{align*}
Only the last term differs by a factor. Thus
\begin{align*}
Q(b) &= \Delta_f b - \frac{1}{2r}\Big(\frac{r}{2} F' - F\Big)b'  +(1-\sqrt{2}F)b \\
& = P(b) +\sqrt{2} F b. 
\end{align*}
This completes the proof. 
\end{proof}

Notice that 
\[
P(r^{-4}) = r^{-4} ,\qquad Q(-\sqrt{2} r^{-2} - r^{-4}) = - \sqrt{2} r^{-2} - r^{-4}. 
\]
Both of these can be checked by hand, but are consequences of Remark \ref{rem:ricci-ha-kb} and the fact that $L_f \Ric = \Ric$. Stability of the FIK shrinker is essentially related to the eigenvalues of $L_f$ on the  space $L^2_f$-orthogonal to $\Ric$. In the radial setting and current framework, the space of deformations $L^2_f$-orthogonal to $\Ric$ are parametrized by those functions $a, b$ with $a(1) = b(1) = 0$. The following lemma is the key to stability. Since $F(1) = 0$ and $F(r) > 0$ for $r > 1$, the lemma essentially shows that the largest eigenvalue (with convention $L_f h = \lambda h$) of $L_f$ on these spaces is nonnegative.

\begin{lemma}\label{lem:barta2}
The operators $P$ and $Q$ satisfy
\begin{align*}
    P(F) = -\sqrt{2} F^2 \leq 0, \qquad\text{and} \qquad   Q(F) = 0.
\end{align*}
\end{lemma}
\begin{proof}
Using Definition \eqref{def:P-and-Q}, these formulas follows almost immediately from Corollary \ref{cor:p-pstar-F}. Indeed, 
\[
Q(F) = F - \frac{F}{r^2}(p(p^\ast(F)) + F) = F - \frac{F}{r^2} r^2 = 0.
\]
On the other hand, the definition of $P$ gives us
\[
P(F) = Q(F) -\frac{F}{4r^2} (4\sqrt{2} r^2 F ) = 0- \sqrt{2} F^2 = -\sqrt{2}F^2. 
\]
\end{proof}

\begin{lemma}\label{lem:hahPkbkQ}
Suppose that $a,b \in \mathcal{A}$. Let $P = P(a)$ and $Q = Q(b)$. Then 
\begin{equation}
 \frac{1}{32\pi^2} \int_M g(h_a, h_{P}) \,e^{-f} d\mu_g = -\int_1^\infty \frac{r^2}{2F}(P-a) P \, r^3e^{-f} \, dr, 
\end{equation}
and 
\begin{equation}\label{eq:kbkQ}
 \frac{1}{32\pi^2} \int_M   g(k_b, k_{Q}) \,e^{-f}d\mu_g =- \int_1^\infty \frac{r^2}{2F}(Q-b) Q\, r^3 e^{-f} dr -\int_1^\infty \frac{r^2}{2F}  \big(Q - b\big)^2\, r^3 e^{-f} dr. 
\end{equation}
\end{lemma}

\begin{remark}
Note that in the formulas above $\frac{r^2}{F} (P - a) = -(p^\ast(p(a))+a)$ and $\frac{r^2}{F}(Q-b) = - (p(p^\ast(b)) + b)$ are bounded as $r \to 1$, so the expressions are indeed integrable. In fact, since we also assume that $a(1) = b(1) = 0$ (which implies $P(1) = Q(1)= 0$), there is no trouble at $r = 1$ when dividing by $F$ in the integrals that follow. 
\end{remark}

\begin{proof}[Proof of Lemma \ref{lem:hahPkbkQ}]
    We begin by recalling that 
    \begin{align*}
        g(h_a, h_{P(a)}) r^3 e^{-f} = \frac{1}{2}(a P  + p(a) p(P)), 
    \end{align*}
    and
    \begin{align*}
        g(k_b, k_{Q(b)}) r^3 e^{-f} = \frac{1}{2}\Big(b \,Q  + p^\ast(b) p^\ast(Q) + \frac{1}{2} q(b) q(Q)\Big).
    \end{align*}
    Recalling Definitions \ref{def:ha-kb} and \ref{def:P-and-Q} and using the integration by parts formulas from Lemma \ref{lem:p-pstar-formulas}, we find 
    \begin{align*}
        \big(a P  + p(a) p(P)\big) r^3 e^{-f}& =  aP r^3 e^{-f}+ p^\ast(p(a)) Pr^3 e^{-f} + \frac{1}{2}\Big(r^4 p(a) P e^{-f} \Big)' \\
        & =  -\frac{r^2}{F}(P-a) P\,r^3 e^{-f}+ \frac{1}{2}\Big(r^4 p(a) Pe^{-f} \Big)'.
    \end{align*}
    Similarly, 
    \begin{align*}
        \big(b Q + p^\ast(b) p^\ast(Q)\big)r^3 e^{-f} &= b Qr^3 e^{-f}+ p(p^\ast(b)) Qr^3 e^{-f}- \frac{1}{2}\Big(r^4 p^\ast(b) Qe^{-f} \Big)' \\
        & =  -\frac{r^2}{F}(Q-b) Q\, r^3 e^{-f}- \frac{1}{2}\Big(r^4 p^\ast(b) Q e^{-f} \Big)'. 
    \end{align*}
    Note that because $P(1) = Q(1) = 0$ and $a, b$ have compact support, 
    \[
    \int_1^\infty \Big(r^4 p(a) Pe^{-f} \Big)' \, dr = \int_1^\infty \Big(r^4 p^\ast(b) Q e^{-f} \Big)' dr = 0. 
    \]

    Next, since $b$ is compactly supported and $b(1) = 0$, we have that $Q$ is compactly supported and $Q(1) = 0$. From \eqref{lem:q-formulas} and again using $Q = b -\frac{F}{r^2}(p(p^\ast(b)) - b))$ from Definition \ref{def:P-and-Q}, we observe that 
    \begin{align*}
    - 2\frac{\sqrt{2}}{rF} \frac{e^f}{r^3} \int_1^r Q \, t^5 e^{-f} \, dt &= - 2\frac{\sqrt{2}}{rF} \frac{e^f}{r^3} \int_1^r b \, t^5 e^{-f} \, dt  + 2\frac{\sqrt{2}}{rF} \frac{e^f}{r^3} \int_1^r F\big(p(p^\ast(b)) + b\big) \, t^3 e^{-f} \, dt \\
    & = - 2\frac{\sqrt{2}}{rF} \frac{e^f}{r^3} \int_1^r b \, t^5 e^{-f} \, dt  -q(b) \\
    & = - \frac{rF}{\sqrt{2}}\left(\frac{b}{F}\right)'. 
    \end{align*}
    Therefore, 
    \begin{align*}
    q(Q) &= \frac{rF}{\sqrt{2}}\left(\frac{Q}{F}\right)'- 2\frac{\sqrt{2}}{rF} \frac{e^f}{r^3} \int_1^r Q \, t^5 e^{-f} \, dt \\
    & = \frac{rF}{\sqrt{2}}\left(\frac{Q-b}{F}\right)'. 
    \end{align*}
    In summary, we have
    \[
   q(b) =  2\frac{\sqrt{2}}{rF} \frac{e^f}{r^3} \int_1^r\big(Q - b\big)\, t^5 e^{-f} \, dt, \qquad   q(Q) = \frac{rF}{\sqrt{2}}\left(\frac{Q-b}{F}\right)'. 
    \]
    It then follows that
    \begin{align*}
        q(b) q(Q)r^3 e^{-f} &= 2\left(\frac{Q-b}{F}\right)'  \int_1^r\big(Q - b\big)\, t^5 e^{-f} \, dt  \\
        & = 2\left(\frac{Q-b}{F}  \int_1^r\big(Q - b\big)\, t^5 e^{-f} \, dt \right)' -2\frac{r^2}{F}  \big(Q - b\big)^2\, r^3 e^{-f}.   
    \end{align*}
    As above, using that $Q$ and $b$ have compact support, 
    \begin{align*}
     \int_1^\infty \left(\frac{Q-b}{F}  \int_1^r\big(Q - b\big)\, t^5 e^{-f} \, dt \right)' \, ds & = \left(\lim_{r \to \infty} \frac{Q-b}{F}\right) \int_1^\infty (Q- b) t^5 \, e^{-f} \, dt = 0. 
    \end{align*}
    Putting everything above together, using Lemma \ref{lem:radial-integration}, and integrating completes the proof. 
\end{proof}

\begin{remark}\label{rem:kF}
Earlier we observed that $Q(F) = 0$. But note that the identity \eqref{eq:kbkQ} above does not hold for $b = F$, since the left hand side would seem to vanish, but the right hand side does not. Of course $F$ is not compactly supported, so its not a counterexample to the lemma, but $F$ does lie in $H^1_f$. Given that $Q(F) = 0$, it is natural to wonder whether the deformation $k_F$ is an $H^1_f$ deformation of the FIK metric. Let us show that it is not. We begin by observing that 
\begin{align*}
p^\ast(F) & = - \frac{1}{2} r F' - F = -\frac{c_0}{\sqrt{2}r^4} - \frac{1}{\sqrt{2}}. 
\end{align*}
Consequently, 
\begin{align*}
    p^\ast(F) - F & = \frac{c_0}{r^2}  - \sqrt{2}\\
    p^\ast(F) + F & = -\frac{c_0}{r^2} -\frac{\sqrt{2}c_0}{r^4}. 
\end{align*}
Next, using the first formula from Lemma \ref{lem:q-formulas} and recalling Corollary \ref{cor:p-pstar-F}, we have
\begin{align*}
    q(F) &= -2\sqrt{2} \frac{e^f}{r^4F} \int_1^r F t^5 e^{-f}\, dt. 
\end{align*}
Using that 
\begin{align*}
    \int t e^{-\sqrt{2} t^2} \, dt &= -\frac{1}{4} \sqrt{2} e^{-\sqrt{2} t^2},   \\
    \int t^3 e^{-\sqrt{2} t^2} \, dt &= -\frac{1}{4} (1 + \sqrt{2}t^2) e^{-\sqrt{2} t^2}, \\ 
    \int t^5 e^{-\sqrt{2} t^2} \, dt &= -\frac{1}{4} (\sqrt{2}+2t^2 + \sqrt{2}t^4) e^{-\sqrt{2} t^2}, 
\end{align*}
we find
\begin{align*}
\int_1^r F t^5 e^{-f} \, dt = \int_1^r \left(\frac{1}{\sqrt{2}} t^5 -c_0 t^3 - \frac{c_0}{\sqrt{2}} t\right) e^{-f} \, dt =- \frac{1}{4} (r^2 + c_0)^2 e^{-f} + c_0.  
\end{align*}
Thus, recalling the definitions of $F$ and $f$, we have 
\begin{align*}
q(F) = \frac{(r^2 + c_0)^2}{\sqrt{2}r^4F} -2\sqrt{2} c_0 \frac{e^f}{r^4 F} = \frac{r^2 + c_0}{r^2-1} -2\sqrt{2} c_0 \frac{e^f}{r^4 F} = \frac{(r^2-1+\sqrt{2})^2-2e^{\sqrt{2} (r^2 -1)}}{(r^2-1)(r^2+c_0)}. 
\end{align*}
As in Lemma \ref{lem:q-formulas}, it can be checked that $q(F) \to 0$ as $r \to 1$, so that $q(F)$ is a smooth function on $[1, \infty)$.  However, we now observe the problem: $q(F)$ grows too quickly as $r \to \infty$ for it to lie in the space $L^2_f$. Indeed, solutions of the ODE satisfied by $q$ must either satisfy $q(1) = 0$ or else they must blow up as $r \to 1$ by consequence of the fact that $\frac{e^f}{r^4F}$ is the homogeneous solution of the ODE (and $F(1) = 0$). Since we want $k_F$ to be bounded at $r = 1$, $q$ must take the form above. Putting our expressions for $p^\ast(b) -b$, $p^\ast(b) +b$, and $q(b)$ together, we obtain 
\[
k_F = \left( \frac{c_0}{r^2}  - \sqrt{2} \right)\mathbf{b}_0 + \left( -\frac{c_0}{r^2} -\frac{\sqrt{2}c_0}{r^4}\right)\mathbf{b}_1 + \left(\frac{r^2 + c_0}{r^2-1} -2\sqrt{2} c_0 \frac{e^f}{r^4 F}\right)\mathbf{b}_6.
\]
The deformation $k_F$ has the properties $L_f k_F = k_{Q(F)} = 0$ and $\mathrm{div}_f k_F = 0$. However, $k_F$ does not lie in $L^2_f$ since $|k_F|^2r^3 e^{-f}$ grows like $r^{-5} e^{f}$. 
\end{remark}

\begin{remark}
In fact, if we let $\hat{S}_F := q(F) \mathbf{b}_6$, it can be checked that 
\[
k_F= - \frac{1}{\sqrt{2}}\nabla^2 f + \hat{S}_F.
\]
Moreover, we separately have
\[
L_f \nabla^2 f = 0, \qquad L_f \hat{S}_F = 0. 
\]
Now $\nabla^2 f \in \mathrm{span}\{\mathbf{b}_0, \mathbf{b}_1\}$ is not itself in gauge since $\mathrm{div}_f(\nabla^2 f ) = -\frac{1}{2}\nabla f$. So $\hat{S}_F$ can be seen as a gauge correcting tensor, but correcting for the gauge forces the tensor outside of $L^2_f$. 
\end{remark}

\begin{corollary}\label{cor:hahPkbkQ-int-by-parts}
    Suppose that $a,b \in \mathcal{A}$. Let $P = P(a)$ and $Q = Q(b)$. Then 
\[
 \frac{1}{32\pi^2} \int_M g(h_a, h_{P}) \,e^{-f} d\mu_g = \frac{1}{2}\int_1^\infty \left( \Big(\frac{1}{F}- 2\sqrt{2}\Big) a^2  -\frac{1}{4}(a')^2\right) r^5 e^{-f} \, dr -\int_1^\infty \frac{r^2}{2F}P^2 r^3e^{-f} \, dr, 
\]
and 
\[
 \frac{1}{32\pi^2} \int_M   g(k_b, k_{Q}) \,e^{-f}d\mu_g =\frac{1}{2}\int_1^\infty \left( \Big(\frac{1}{F}- \sqrt{2}\Big) b^2  -\frac{1}{4}(b')^2\right) r^5 e^{-f} \, dr -\int_1^\infty \frac{r^2}{2F}  \big(Q - b\big)^2\, r^3 e^{-f} dr. 
\]
\end{corollary}
\begin{proof}
Recall from Definition \ref{def:P-and-Q} that
\[
P= \frac{F}{4}\left(\frac{e^f}{r^5}\left(\frac{r^5}{e^f} a'\right)'- 8\sqrt{2}a\right)  +a.
\]
Now 
\begin{align*}
 \frac{r^5}{2F} aP e^{-f} &=  \frac{1}{8}a\left(\frac{r^5}{e^f} a'\right)'- \sqrt{2} a^2 r^5 e^{-f}  +\frac{r^5}{2F} a^2e^{-f} \\
 & = \frac{1}{2}\left(\frac{1}{F} - 2\sqrt{2}\right)a^2 r^5 e^{-f} + \frac{1}{8}a\left(\frac{r^5}{e^f} a'\right)'.
\end{align*}
Integrating by parts and using that $a \in \mathcal{A}$ we find 
\begin{align*}
\int_1^\infty  \frac{r^5}{2F} aP\,  e^{-f} \, dr  & = \frac{1}{2}\int_1^\infty \left(\frac{1}{F} - 2\sqrt{2}\right)a^2 r^5 e^{-f} + \frac{1}{4}a\left(\frac{r^5}{e^f} a'\right)' \, dr \\
& = \frac{1}{2}\int_1^\infty\left(\frac{1}{F} - 2\sqrt{2}\right)a^2 r^5 e^{-f} + \frac{1}{4}\left(a\frac{r^5}{e^f} a'\right)'  - \frac{1}{4} (a')^2 r^5 e^{-f}\, dr \\
& = \frac{1}{2}\int_1^\infty \left( \Big(\frac{1}{F}- 2\sqrt{2}\Big) a^2  -\frac{1}{4}(a')^2\right) r^5 e^{-f} \, dr.
\end{align*}
This shows the first formula. Using 
\[
Q= \frac{F}{4}\left(\frac{e^f}{r^5}\left(\frac{r^5}{e^f} b'\right)'- 4\sqrt{2}b\right)  +b, 
\]
so that 
\[
  \frac{r^5}{2F} bQ e^{-f}  = \frac{1}{2}\left(\frac{1}{F} - \sqrt{2}\right)b^2 r^5 e^{-f} + \frac{1}{8}b\left(\frac{r^5}{e^f} b'\right)',
\]
a completely analogous computation gives the second.  
\end{proof}

\begin{lemma}\label{lem:barta1}
   Suppose that $a,b \in \mathcal{A}$. Let $\phi = \phi(r)$ be a bounded $C^2$ radial function such that $\phi(1) = 0$, $\phi'(1) > 0$, and $\phi(r) > 0$ for $r > 1$. Set $w(r) := \phi(r)^{-1} a(r)$ for $r > 1$.  Then
    \begin{equation}\label{eq:bartas-trick} 
        \int_1^\infty \left(\left(\frac{1}{F} - 2\sqrt{2}\right) a^2 - \frac{1}{4}(a')^2\right) r^5 e^{-f}\, dr = \int_1^\infty \left(\frac{1}{F} \phi w^2 P(\phi)  - \frac{1}{4} \phi^2 (w')^2 \right) r^5 e^{-f} \, dr. 
    \end{equation}
    Similarly, if we $w(r) := \phi^{-1}(r) b(r)$ for $r > 1$. Then 
    \begin{equation}\label{eq:bartas-trick-2} 
        \int_1^\infty \left(\left(\frac{1}{F} - \sqrt{2}\right) b^2 - \frac{1}{4}(b')^2\right) r^5 e^{-f}\, dr = \int_1^\infty \left(\frac{1}{F} \phi w^2 Q(\phi)  - \frac{1}{4} \phi^2 (w')^2 \right) r^5 e^{-f} \, dr. 
    \end{equation}
\end{lemma}

\begin{proof}
With $a = \phi w$, we compute 
\begin{align*}
(a')^2 &= (\phi' w + \phi w')^2\\
& = (\phi')^2 w^2 + \phi^2 (w')^2 + 2  w \phi w' \phi'. \\
& = \phi'(\phi w^2)' + \phi^2 (w')^2.
\end{align*}
Now
\begin{equation}\label{eq:barta-temp-1}
\left(\phi' \phi w^2 r^5 e^{-f}\, \right)' - \phi' (\phi w^2)' r^5 e^{-f}  = \left(\phi' r^5 e^{-f}\right)' \phi w^2,
\end{equation}
while recalling the last formula for $P(\phi)$ in Definition \ref{def:P-and-Q} gives
\begin{equation}\label{eq:barta-temp-2}
\frac{1}{4} \left(\phi' r^5 e^{-f}\right)' \phi w^2= \left(\frac{1}{F}\phi w^2 P(\phi) -\left(\frac{1}{F} - 2\sqrt{2} \right)\phi^2w^2 \right) r^5 e^{-f}.
\end{equation}
Note that $\phi' \phi w^2 = \phi' \phi^{-1} a^2$. Using that $\phi'(1) > 0$ and $a(1) = 0$, we have
\[
\lim_{r\to1^+}\frac{a(r)^2}{\phi(r)} = \lim_{r\to1^+}\frac{2a(r)a'(r)}{\phi'(r)}  = 0. 
\]
This, together with the assumption that $a$ is compactly supported, gives
\begin{align*}
\int_1^\infty \left(\phi' \phi w^2 r^5 e^{-f}\, \right)' \, dr = \lim_{r \to \infty} \phi'(r) a(r) w(r) r^5 e^{-f(r)} -  \phi'(1) e^{-f(1)}\lim_{r \to 1^+} \frac{a(r)^2}{\phi(r)} = 0. 
\end{align*}
Therefore, putting the identities \eqref{eq:barta-temp-1} and \eqref{eq:barta-temp-2} together and integrating gives
\begin{align*}
-\frac{1}{4}\int_1^\infty \phi' (\phi w^2)' r^5 e^{-f}  \, dr =\int_1^\infty \left(\frac{1}{F}\phi w^2 P(\phi) -\left(\frac{1}{F} - 2\sqrt{2} \right)\phi^2w^2 \right) r^5 e^{-f} \, dr. 
\end{align*}
Using this in the integral expression for $a$, we obtain 
\begin{align*}
 &  \int_1^\infty \left(\left(\frac{1}{F} - 2\sqrt{2}\right) a^2 - \frac{1}{4}(a')^2\right) r^5 e^{-f}\, dr \\
 & =  \int_1^\infty \left(\left(\frac{1}{F} - 2\sqrt{2}\right) \phi^2 w^2  - \frac{1}{4}\phi^2 (w')^2 \right)r^5 e^{-f}\, dr - \frac{1}{4} \int_1^\infty  \phi' (\phi w^2)' r^5 e^{-f}\, dr\\
   & = \int_1^\infty \left(\frac{1}{F} \phi w^2 P(\phi)  - \frac{1}{4} \phi^2 (w')^2 \right) r^5 e^{-f} \, dr. 
\end{align*}
A completely analogous computation gives the second formula involving $b$. This completes the proof.  
\end{proof}

\begin{remark}
The proof above is known as Barta's trick \cite{Bar} for showing the integral on the left hand side of \eqref{eq:bartas-trick} (which recall is essentially $\int_1^\infty \frac{1}{F} aP(a) r^5 e^{-f} \, dr$) has a sign. Indeed, when one can find a positive function $\phi$ as in the lemma above such that $P(\phi) \leq 0$, the desired result will follow. 
\end{remark}

\begin{corollary}\label{cor:bartas-applied}
Suppose that $a,b \in \mathcal{A}$. Then
\[
\int_1^\infty \left(\left(\frac{1}{F} - 2\sqrt{2}\right) a^2 - \frac{1}{4}(a')^2\right) r^5 e^{-f}\, dr \leq 0,
\]
and 
\[
\int_1^\infty \left(\left(\frac{1}{F} - \sqrt{2}\right) b^2 - \frac{1}{4}(b')^2\right) r^5 e^{-f}\, dr \leq 0.
\]
\end{corollary}

Of course, the second estimate is stronger than the first, but we continue with the pattern of addressing tensors $h_a$ and $k_b$ simultaneously. 

\begin{proof}
Since $F(1) = 0$, $F'(1) > 0$, $F(r) > 0$ and $F$ is bounded, we may take $\phi = F$ in Lemma \ref{lem:barta1} above. In view of Lemma \ref{lem:barta2}, and setting $w = F^{-1} a$, we obtain
\begin{align*}
\int_1^\infty \left(\left(\frac{1}{F} - 2\sqrt{2}\right) a^2 - \frac{1}{4}(a')^2\right) r^5 e^{-f}\, dr &= \int_1^\infty \left(w^2 P(F)  - \frac{1}{4} F^2 (w')^2 \right) r^5 e^{-f} \, dr \\
& = -\int_1^\infty \left(\sqrt{2}  w^2  + \frac{1}{4} (w')^2 \right) F^2r^5 e^{-f} \, dr  \leq 0. 
\end{align*}
Completely analogously, we have 
\begin{align*}
\int_1^\infty \left(\left(\frac{1}{F} - \sqrt{2}\right) b^2 - \frac{1}{4}(b')^2\right) r^5 e^{-f}\, dr &= \int_1^\infty \left(w^2 Q(F)  - \frac{1}{4} F^2 (w')^2 \right) r^5 e^{-f} \, dr \\
& = -\int_1^\infty \frac{1}{4} (w')^2 F^2r^5 e^{-f} \, dr  \leq 0.
\end{align*}
\end{proof}

We now prove the main theorem of this section, thereby completing the proof that the FIK shrinking soliton is stable under radially symmetric deformations. 

\begin{theorem}
    Suppose $a,b  \in \mathcal{A}$. Then 
    \[
    \delta^2\nu_g(h_a + k_b) \leq 0. 
    \]
    Consequently, if $h_U = h_0 \mathbf{b}_0+ h_1 \mathbf{b}_1 + h_6 \mathbf{b}_6$ is any radially symmetric $2$-tensor in $H^1_f(M)$ which satisfies $\mathrm{div}_f(h_U) = 0$ and $\Xi(h_U) = 0$, then 
    \[
    \delta^2 \nu_g(h_U) \leq 0. 
    \]
\end{theorem}

\begin{proof}
    Set $P = P(a)$ and $Q = Q(b)$. Using that $h_a$ and $k_b$ are compactly supported and smooth; 2-tensors are orthogonal (to each other and to the Ricci tensor); and integration by parts gives that 
    \begin{align*}
    \delta^2 \nu_g (h_a + k_b) &= \frac{1}{32\pi^2} \int_M g(L_f h_a + L_f k_b, h_a+  k_b) \, e^{-f} d\mu_g\\
    & = \frac{1}{32\pi^2} \int_M g(h_a, h_{P}) \, e^{-f} d\mu_g +  \frac{1}{32\pi^2} \int_M g(k_b, k_Q) \, e^{-f} d\mu_g. 
    \end{align*}
    Then Corollary \ref{cor:hahPkbkQ-int-by-parts} and Corollary \ref{cor:bartas-applied}
    show
    \[
    \delta^2 \nu_g(h_a + k_b) \leq 0. 
    \]
    The last assertion now follows from Corollary \ref{cor:parametrizing-hU-cs}.
\end{proof}

\newpage

\part{Linear Stability of the Blowdown Soliton: General Case}\label{part:higher stability}

\section{Wigner functions, higher modes, and setup to nonradial stability}

\subsection{Introduction} 
Having completed a proof of the linear stability of the FIK metric among radially symmetric deformations, we turn our discussion towards the general case. Our principal tool is to introduce a separation of variables via the basis of Wigner functions, see Section \ref{sec:Wigner fct}. Wigner functions are a well-known special basis of eigenfunctions of $\Delta_{\mathbb{S}^3}$, that have the additional property of interacting well with derivatives by the left-invariant frame $X_1, X_2, X_3$ on $\mathbb{S}^3$. (They also behave well with respect to derivatives by the right-invariant frame.) In particular, they can be used to provide a useful Fourier decomposition for functions on Berger spheres, which, in our setting, are precisely the radial cross-sections of the FIK metric. By indexing the eigenvalues of $\Delta_{\mathbb{S}^3}$ by $J \in \mathbb{N}$ (our convention: $0 \in \mathbb{N}$), we will be able to deal with general deformations $h$ separately for each mode $J$.

To that end, let us now describe the procedure in more detail. In this section, we consider a general deformation 
$$h = \sum_{p = 0}^9 h_p \mathbf{b}_p = h_I + h_A$$
where we no longer assume the coefficients $h_p$ are radial. That is, $h_p = h_p(r,\Theta)$ for $r>1$, with $\Theta \in \mathbb{S}^3$ the spherical variable. Recall that $h_I, h_A$ denote the $J_1^+$-invariant and $J_1^+$-anti-invariant parts of $h$. 

As before, our goal is to show that $\delta^2\nu_g(h) \leq 0$. As in the radial setting, we may assume without loss of generality that $\mathrm{div}_f(h) = \Xi(h) = 0$, so that 
\begin{align*}
\delta^2 \nu_g(h) &=\frac{1}{32\pi^2} \int_M (2R(h, h) - |\nabla h|^2) e^{-f} d\mu_g \\
& = \frac{1}{32\pi^2}\int_M (2R(h_I, h_I) - |\nabla h_I|^2)\, e^{-f} d\mu_g  + \frac{1}{32\pi^2} \int_M (2R(h_A, h_A) - |\nabla h_A|^2)\,  e^{-f} d\mu_g,
\end{align*} 
where the second equality follows from Lemma \ref{lem:ibp-for-H1}. Note that the assumption $\Xi(h) = 0$ implies $\Xi(h_I) = \Xi(h_A) = 0$ (the anti-invariant part is pointwise orthogonal to the Ricci tensor), but as we saw in the radial case, the condition $\mathrm{div}_f(h) = 0$ only implies $\mathrm{div}_f(h_I) =- \mathrm{div}_f(h_A)$. The gauge equations couple the invariant and anti-invariant parts of a deformation, which complicates the study of $\delta^2\nu_g(h)$. In the radial setting, we partially accepted this added complexity and used one of the gauge equations to carefully parametrize the subtle deformation direction given by $h_U= h_0 \mathbf{b}_0 + h_1 \mathbf{b}_1 + h_6 \mathbf{b}_6$. To show $\delta^2 \nu_g(h) \leq 0$, it would suffice to show that
\begin{equation}\label{eq:nonradial-desired-ineq}
\int_M (2R(h_I, h_I) - |\nabla h_I|^2)\, e^{-f} d\mu_g \leq 0, \qquad \text{and} \qquad \int_M (2R(h_A, h_A) - |\nabla h_A|^2)\,  e^{-f} d\mu_g \leq 0,
\end{equation}
for \textit{any} smooth $J_1^+$-invariant and $J_1^+$-anti-invariant tensors $h_I, h_A$ satisfying $\Xi(h_I) = \Xi(h_A) = 0$ (without assuming that $\mathrm{div}_f(h_I)$ or $\mathrm{div}_f(h_A)$ are zero). We will see that this is possible for $J \geq 3$, but not possible for $J \in \{ 1, 2\}$. We have already seen this is not possible for $J = 0$.

We combine this $J_1^+$-invariance decomposition above with a decomposition of the component functions $h_p$ into sums of Wigner functions 
\begin{equation}
    \{ D^J_{M,M'} \}\,\text{ for }\, J\in\mathbb{N} \,\text{ and }\,  M,M'\in\mathcal{J}:=\{-J,-J+2,\dots,J-2,J\}. 
\end{equation}
Key properties of these functions are summarized in Appendix \ref{app:3}, Section \ref{sec:Wigner fct}. In the discussion that follows, we will use the properties of Wigner functions given in Proposition \ref{prop:wigner}. For each $J$, the set of $(J+1)^2$ Wigner functions $D^J_{M, M'}$ for $M, M' \in \mathcal{J}$ form a (complex) basis of the $-J(J+2)$ eigenspace of $\Delta_{\mathbb{S}^3}$.\\

\textbf{Convention note:}
Wigner functions and their derivatives are complex-valued, so in what follows we now let $g = \langle \cdot \,, \cdot \rangle$ denote the complex linear extension of our real 4D FIK metric and adopt the convention that $|T|^2 = \langle T, \overline{T}\rangle$ for complex-valued tensors $T$. Similarly $|u|^2 = u \overline{u}$ for complex-valued functions. We analogously extend other tensors and operators (such as $\nabla, R, \Ric$) to be complex linear. \\

For $J \geq 0$ and $M' \in \mathcal{J}$, we let $\mathcal{T}^J_{M'}$ denote the vector space of (complex-valued) $2$-tensors in $C^{\infty}(M) \cap H^1_f(M)$ with component functions $h_{p}(r,\Theta)$ such that for each $r$, the functions $\Theta\mapsto h_{p}(r,\Theta)$ belong to the subspace of linear combinations of Wigner functions 
\begin{equation}
\mathcal{D}^J_{M'}:=\operatorname{span}\left\{D^J_{M,M'}:  M\in \mathcal{J}\right\}.
\end{equation}
More concretely, we say $h \in \mathcal{T}^J_{M'}$ if for  each $p \in \{0, \dots, 9\}$, we have a decomposition of $h_p := 4\langle h, \mathbf{b}_p\rangle$ into the Wigner basis
\begin{equation}
    h_p(r,\Theta) = \sum_{M\in\mathcal{J}} h_{p,M}(r)\,D^J_{M,M'}(\Theta) \in \mathcal{D}^J_{M'}, 
\end{equation}
for complex-valued functions $h_{p, M}(r)$. Note that for $J = 0$, we must have $M = M' = 0$ and $D^0_{0, 0} = 1$, and so $\mathcal{T}^0_0$ is precisely the set of a radially symmetric deformations. Our next goal is to decompose $h$ into an orthogonal sum of deformations in the spaces $\mathcal{T}^J_{M'}$.

Before proceeding, let us take a moment to observe that the spaces $\mathcal{D}^J_{M'}$ and $\mathcal{D}^J_{-M'}$ are complex conjugates of one another and consequently that $\mathcal{T}^J_{M'}$ and $\mathcal{T}^J_{-M'}$ are conjugate tensor bundles. When $M' = 0$, this means $\mathcal{T}^J_{M'}$ contains its conjugates. Indeed, since $\overline{D^J_{M, M'}} = (-1)^{\frac{M-M'}{2}}D^J_{-M,-M'}$, if we have $h = \sum_p h_p \mathbf{b}_p \in \mathcal{T}^J_{M'}$ (so $h_p \in \mathcal{D}^J_{M'}$), then $\overline{h} = \sum_{p} \overline{h}_p \mathbf{b}_p$ has component functions 
\[
\overline{h_p} = \sum_{M \in \mathcal{J}} \overline{h_{p,M}} \overline{D^J_{M, M'}} = \sum_{M \in \mathcal{J}}(-1)^{\frac{M-M'}{2}} \overline{h_{p,M}} D^J_{-M, -M'}   = \sum_{M \in \mathcal{J}} (-1)^{\frac{M+M'}{2}} \overline{h_{p, -M}}\, D^J_{M, -M'} \in \mathcal{D}^J_{-M'}. 
\]
It follows that $\overline{h} \in \mathcal{T}^J_{-M'}$. 

Continuing now with our goal, consider a general real-valued deformation $h$ as above with $h_p = 4 \langle h, \mathbf{b}_p \rangle$. For every $J \in \mathbb{N}$ and $M', M \in \mathcal{J}$, define 
\begin{align*}
(h^J_{M'})_{p, M}(r) &:= \frac{J+1}{2\pi^2} \int_{\mathbb{S}^3} h_p(r, \Theta) \, \overline{D^J_{M, M'}}(\Theta) \, d\mu_{\mathbb{S}^3}(\Theta),
\end{align*}
and then set 
\begin{align*}
(h^J_{M'})_p(r, \Theta) &:= \sum_{M \in \mathcal{J}} (h^J_{M'})_{p, M}(r) \, D^J_{M,M'}(\Theta), \\
h^J_{M'} &:= \sum_{p = 0}^9 (h^J_{M'})_{p} \, \mathbf{b}_p.
\end{align*}
Because the Wigner functions form a full $L^2$ orthonormal basis on $\mathbb{S}^3$, we have 
\begin{equation}\label{eq:hp-full-wigner}
h_p(r , \Theta) = \sum_{J =0}^\infty \sum_{M, M' \in \mathcal{J}}  (h^J_{M'})_{p, M} (r) \, D^J_{M, M'} (\Theta).
\end{equation}
We conclude that we have a decomposition
\begin{align}\label{eq:h-hJMp-decomposition}
    h &= \sum_{p =0}^9 h_p \mathbf{b}_p 
     = \sum_{J = 0}^\infty \sum_{M' \in \mathcal{J}} \sum_{p =0}^9\left( \sum_{M \in \mathcal{J}} (h^J_{M'})_{p, M} D^J_{M, M'} \right)\mathbf{b}_p 
     = \sum_{J =0}^\infty \sum_{M' \in\mathcal{J}}  h^J_{M'}. 
\end{align}
Since we assume that $h$ is a real-valued deformation (i.e. its component functions $h_p$ are real-valued), we note that $\overline{(h^J_{M'})_{p, M}} = (-1)^{\frac{M-M'}{2}} (h^J_{-M'})_{p, -M}$ and thus that 
\begin{equation}\label{eq:hJMp-conjugation-rule}
\overline{h^J_{M'}} = h^J_{-M'}. 
\end{equation}
Let us introduce the notation $h^J_{|M'|} := h^J_{M'} + h^J_{-M'}$ if $M' \neq 0$ and $h^J_{|M'|} = h^J_{M'}$ if $M' = 0$. Then, whenever $h$ is a real-valued deformation, we have that $h^J_{|M'|}$ are real-valued deformations as well.

Importantly, the spaces $\mathcal{T}^J_{M'}$ are $L^2_f$ orthogonal to one another because the Wigner functions are $L^2$ orthogonal on $\mathbb{S}^3$. In particular, if $h = \sum_{J, M'} h^J_{M'}$ and $k =\sum_{J, M'} k^J_{M'}$ are real-valued deformations decomposed as above, then 
\begin{equation}
\int_M \langle h, k \rangle \, e^{-f} d\mu_g = \sum_{J = 0}^\infty \sum_{M' \in \mathcal{J}} \int_M \langle h^J_{M'}, \overline{k^J_{M'}} \rangle \, e^{-f} d\mu_g.  
\end{equation}
Indeed, note that from $d\mu_g = \frac{2r}{s} d\sigma_g(r) dr = 16 r^3 d\mu_{\mathbb{S}^3} dr$,  we have $d\sigma_g(r) = 8 s r^2 d\mu_{\mathbb{S}^3}$. Using that $|\nabla r| = \frac{s}{2r}$ and the co-area formula,  we obtain that for any function $u = u(r, \Theta)$,
\begin{equation}
\int_M u \, e^{-f} d\mu_g = \int_1^\infty  \left(\int_{\{r = t\}}  u \,  d\sigma_g(t) \right)  \frac{2t}{s} e^{-f}dt = 16 \int_1^\infty  \left(\int_{\mathbb{S}^3} u(r, \Theta) \, d\mu_{\mathbb{S}^3}(\Theta)\right)r^3 e^{-f} dr. 
\end{equation}
It follows that, 
\begin{align*}
&\int_M \langle h^J_{M'}, \overline{k^{J_0}_{M_0'}} \rangle e^{-f} d\mu_g =  4 \sum_p \sum_{M, M_0} \int_1^\infty (h^J_{M'})_{p, M}(r) \overline{(k^{J_0}_{M_0'})_{p, M_0}}(r)\, \left(\int_{\mathbb{S}^3} D^J_{M, M'} \overline{D^{J_0}_{M_0, M_0'}}\right) \, r^3 e^{-f} dr,
\end{align*}
which is only nonzero if $J_0=J$ and $M_0' = M'$. 

Here is the main result of this introduction. We intend to apply this proposition separately to both the invariant and anti-invariant parts of a deformation $h$.

\begin{proposition}\label{prop:pres of TJM'}
    Suppose a deformation $h \in C^{\infty}(M) \cap H^1_f(M)$ satisfying $\mathrm{div}_f(h) = \Xi(h) = 0$ is decomposed into an $L^2_f$ orthogonal sum 
    \[
    h = \sum_{J = 0}^\infty \sum_{M' \in \mathcal{J}} h^J_{M'}, 
    \]
    with $h^J_{M'} \in \mathcal{T}^J_{M'}$ and $h^J_{|M'|} \in\mathcal{T}^J_{M'} \oplus \mathcal{T}^J_{-M'} $ defined as above. 
    Then 
     \[
    \int_M (2R(h, h) - |\nabla h|^2) \, e^{-f} d\mu_g = \sum_{J = 0}^\infty \sum_{M' \in \mathcal{J}} \int_M (2R(h^J_{M'},\overline{h^J_{M'}}) - |\nabla h^J_{M'}|^2) \, e^{-f} d\mu_g.
    \]
    Moreover, for every $J$ and $M'$, one has $\mathrm{div}_f(h^J_{M'}) = \Xi(h^J_{M'}) = 0$.  
    As a consequence, if it holds that for any $J, M'$
    \[
    \delta^2 \nu_g(h^J_{|M'|}) =\frac{1}{16\pi^2} \int_M (2R(h^J_{M'},\overline{h^J_{M'}}) - |\nabla h^J_{M'}|^2) \, e^{-f} d\mu_g\leq 0, 
    \]
    then $\delta^2 \nu_g(h) \leq 0$. 
\end{proposition}

\begin{proof}
    The metric and its curvatures are radially symmetric. Additionally, covariant derivatives preserve the space $\mathcal{D}^J_{M'}$ by Proposition \ref{prop:wigner} (see also Lemma \ref{lem:fik-derivatives-on-wigner}). Indeed,  $X_1(D^J_{M, M'})$ is a (complex) multiple of $D^J_{M, M'}$, while $(X_2 \pm iX_3)(D^J_{M, M'})$ is a (complex) multiple of $D^J_{M\pm 2, M'}$, both of which leave $J$ and $M'$ fixed. Thus the actions of $\partial_r, X_1, X_2, X_3$ preserve $\mathcal{D}^J_{M'}$, from which it follows that $e_0, e_1, e_2, e_3$ preserve $\mathcal{D}^J_{M'}$. We conclude that if $h \in \mathcal{T}^J_{M'}$, then $R(h) \in \mathcal{T}^J_{M'}$ and $\nabla_{e_i} h \in \mathcal{T}^J_{M'}$ for any $i \in\{0,1,2,3\}$.

    As we noted above, the spaces $\mathcal{T}^J_{M'}$ and $\mathcal{T}^{J_0}_{M_0'}$ are orthogonal if  $(J,M') \neq (J_0, M_0')$. Noting that $h = \overline{h}$, it follows that  
    \begin{align*}
        \int_M (2R(h, h) - |\nabla h|^2) \, e^{-f} d\mu_g & = \int_M (2R(h, \overline{h}) - \langle \nabla h, \nabla \overline{h}\,\rangle) \, e^{-f} d\mu_g\\
        & = \sum_{J, J_0}\sum_{M', M_0'} \int_M (2R(h^J_{M'}, \overline{h^{J_0}_{M_0'}}) - \sum_i\langle \nabla_{e_i} h^J_{M'}, \nabla_{e_i}\overline{h^{J_0}_{M_0'}}\rangle) \, e^{-f} d\mu_g \\
        & = \sum_{J, M'} \int_{M} (2R(h^J_{M}, \overline{h^J_{M'}} ) - |\nabla h^J_{M'} |^2) \, e^{-f} d\mu_g,
    \end{align*}
    recalling that by convention $|\nabla h^J_{M'} |^2 = \langle \nabla h^J_{M'},\nabla \overline{h^J_{M'}} \rangle $.
    
    Reasoning as we did with the spaces $\mathcal{T}^J_{M'}$, if we define 
    \[
    X^J_{M'} := \mathrm{div}_f(h^J_{M'}), \qquad (X^J_{M'})_j := \langle X^J_{M'}, e^j\rangle
    \]
    then the component functions satisfy $(X^J_{M'})_j \in \mathcal{D}^J_{M'}$. If $\mathrm{div}_f(h) =\sum_{J, M'} \mathrm{div}_f(h^J_{M'}) = 0$, then the functions $X_j := \sum_{J, M'} (X^J_{M'})_j = 0$ for each $j$, and consequently (for any $J, M, M'$)
    \[
    0 = \int_{\mathbb{S}^3} X_j(r, \Theta) \overline{D^J_{M, M'}}(\Theta)  = \int_{\mathbb{S}^3} (X^J_{M'})_j(r, \Theta) \overline{D^J_{M, M'}}(\Theta)
    \]
    which is only possible if $(X^J_{M'})_j = 0$ for each $J, M'$. We conclude $\mathrm{div}_f(h^J_{M'}) = 0$ for each $J, M'$. It is even more straightforward to see that $\Xi(h^J_{M'}) = 0$ since $\Ric \in \mathcal{T}^0_0$. 

    Finally, the expression for $\delta^2 \nu_g(h^J_{|M'|})$ follows from the definition of $h^J_{|M'|}$ and the orthogonality discussed above. In particular, if $ \delta^2 \nu_g(h^J_{|M'|}) \leq 0$ for all $h^J_{M'}\in\mathcal{T}^J_{M'}$, then $\delta^2\nu_g(h) \leq 0$.
\end{proof}

\subsection{Overview of linear stability proof in the general case}

Given Proposition \ref{prop:pres of TJM'}, we will prove the linear stability of the FIK metric among general deformations as follows.
\begin{enumerate}
    \item Decompose any real-valued deformation $h$ as an $L_f^2$–orthogonal sum
    \[
    h=h_R+h_N=h_R+(h_I+h_A),
    \]
    where $h_R\in\mathcal{T}^0_0$ is the radial part and $h_N$ is the nonradial part split into its $J^+_1$–invariant and $J^+_1$–anti-invariant pieces, $h_I$ and $h_A$
    \[
      h_I,\,h_A \in \bigoplus_{J\ge1}\ \bigoplus_{M'\in\mathcal{J}}\ \mathcal{T}^J_{M'}.
    \]
    By Proposition \ref{prop:pres of TJM'}, this decomposition preserves the gauge constraints:
    \(
      \operatorname{div}_f(h_R)=\Xi(h_R)=0
      \) and \(
      \operatorname{div}_f(h_N)=\Xi(h_N)=0
    \),
    and the second variation splits as
    \[
      \delta^2\nu_g(h)=\frac{1}{32\pi^2}\sum_{S\in\{R,I,A\}}\int_M\bigl(2R(h_S,h_S)-|\nabla h_S|^2\bigr)\,e^{-f}\,d\mu_g.
    \]
    
    \item We have already shown in Theorem \ref{thm:main-radial} that for any $h_R\in\mathcal{T}^0_0$ with $\operatorname{div}_f(h_R)=\Xi(h_R)=0$,
    \[
      \int_M \bigl(2R(h_R,h_R)-|\nabla h_R|^2\bigr)\,e^{-f}\,d\mu_g \leq 0.
    \]

    \item The tensors $h_I$ and $h_A$ are decomposed in a natural complex basis of $2$-tensors based on Wigner functions in Section \ref{sec:expanding Wigner}, introduced below.

    \item We will show in Section \ref{sec: Wigner function and stab for Jgeq 3} that instability can only come from lower frequency modes $J\leq 2$ through a rough estimate of the operator on high frequency modes. 
    
    \item In Section \ref{sec:further block decomposition}, we further refine our block decomposition for the action of $L_f$ on $J_1^+$-invariant and $J_1^+$-anti-invariant deformations reducing the problem to blocks of size at most $4\times 4$. This is a major improvement over the initial size $10(J+1)\times 10(J+1)$ for each $\mathcal{T}^J_{M'}$.
    
    \item Studying the few remaining blocks, we show in Section \ref{sec:stability for J leq 2} that instability can only come from four $3\times 3$ blocks for $h_A$ in $J=1$ or three $4\times 4$ blocks for $h_I$ in $J=2$.
    
    \item These are indeed problematic blocks: it is shown in Section \ref{sec:nonnegative eigenvalues} that $L_f$ does have nonnegative (and even positive) eigenvalues on these blocks. All eigendeformations are identified as Hessians of explicit functions.
    \item Finally, in Section \ref{sec:P1Q2}, we bound from above the action of $L_f$ on these remaining blocks by a diagonal operator, thereby reducing the problem in each case to a $1$-dimensional Sturm-Liouville operator.
    
    We show, using a variant of Barta's trick, that this Sturm-Liouville operator can at most have one nonnegative eigenvalue. It follows by comparison that $L_f$ can at most have four nonnegative eigenvalues if $J = 1$ and three if $J = 2$, one in each problematic block. Such \textit{purely gauge} positive eigendirections are explicitly identified in Section \ref{sec:nonnegative eigenvalues}. Consequently, on $\ker\operatorname{div}_f\cap\,\mathcal{T}^J_{M'}$, which is orthogonal to gauge transformations, all of the other eigenvalues must be negative.
\end{enumerate}

\subsection{A new complex basis of $2$-tensors}

In the nonradial setting, expanding $|\nabla h|^2$ introduces many first order terms that arise as cross terms in the action of $\nabla$ on $h_p \mathbf{b}_p$. (In the radial setting, recall these did not appear because $\nabla_{e_0} \mathbf{b}_p = 0$). These new first order terms are the principal difficulty in the study of stability of the FIK metric among deformations drawn from the space $\mathcal{T}^J_{M'}$ for $J \geq 1$. In this section, we introduce new bases of complex-valued $2$-tensors and complex-valued vector fields. Our reason for introducing these new bases is firstly to diagonalize the action of the Levi-Civita connection $\nabla$ on basis elements as much as possible, and secondly to make use of the action of derivatives on the basis of Wigner functions given in Proposition \ref{prop:wigner}. 

Define a new basis of complex $2$-tensors $\mathbf{C}$, made out of basis elements in $\mathbf{B}$, consisting of  
\begin{align*}
&\mathbf b_0, && \mathbf b_1,\\ 
&\mathbf b_+ := \frac{1}{\sqrt{2}}\big(\mathbf b_2 + i\,\mathbf b_3 \big),
&&\mathbf b_- := \frac{1}{\sqrt{2}}\big(\mathbf b_2  - i\,\mathbf b_3 \big),\\ 
&\mathbf k_1 := \frac{1}{\sqrt{2}}\big(\mathbf b_{4}+  i\,\mathbf b_{5}\big),
&&\mathbf k_{\bar{1}}  := \frac{1}{\sqrt{2}}\big(\mathbf b_{4}- i\,\mathbf b_{5}\big),\\ 
&\mathbf k_2  := \frac{1}{\sqrt{2}}\big(\mathbf b_{6}+i\,\mathbf b_{7}\big),
&&\mathbf k_{\bar{2}} := \frac{1}{\sqrt{2}}\big(\mathbf b_{6}- i\,\mathbf b_{7}\big),\\ 
&\mathbf k_3  := \frac{1}{\sqrt{2}}\big(\mathbf b_{8}+ i\,\mathbf b_{9}\big),
&&\mathbf k_{\bar{3}} := \frac{1}{\sqrt{2}}\big(\mathbf b_{8}- i\,\mathbf b_{9}\big).
\end{align*}
By construction $\langle \mathbf{c}_p, \overline{\mathbf{c}_q} \rangle = \frac{1}{4} \delta_{pq}$ for any $\mathbf{c}_p, \mathbf{c}_q \in \mathbf{C}$. 
Introduce a new complex basis of vector fields given by 
\[
e_0 = \frac{s}{2r} \partial_r , \qquad e_1 = \frac{1}{2s} X_1, \qquad D_{\pm} := \frac{1}{\sqrt{2}}(e_+ \pm i \, e_-) = \frac{1\mp i}{4r}(X_3 \pm iX_2), 
\]
with duals $e^0, e^1, \sigma^+, \sigma^-$, where 
\[
\sigma^+ := \frac{1}{\sqrt{2}}(e^+ - i \,e^-), \qquad \sigma^- := \frac{1}{\sqrt{2}}(e^+ + i\, e_-). 
\]
These definitions ensure that $| e^{0}|=| e^{1}|=| D_{\pm}|=1$ and
\[
\sigma^{\pm}(D_{\pm}) = \frac{1}{2}(e^+ \mp i e^-) (e^+ \pm i e^-) = 1, \qquad \sigma^{\mp}(D_\pm) = \frac{1}{2}(e^+ \pm i e^-) (e^+ \pm i e^-) = 0. 
\]

Given $h = \sum_{p =0}^9 h_p \mathbf{b}_p$, we can express $h$ in the new basis 
\begin{equation}\label{eq:h-in-complex-basis}
h=
 h_0\,\mathbf b_0
  +
 h_1\,\mathbf b_1
  +
 h_-\,\mathbf b_+
  +
 h_+\,\mathbf b_{-}
  +
 \sum_{j=1}^{3} 
 \big(
     k_j\,\mathbf k_j
      +
     k_{\bar{j}}\,\mathbf k_{\bar{j}}
 \big),
\end{equation}
with new coordinate functions 
\begin{align*}
h_0 &= h_0, &
h_1 &= h_1, & h_{\pm}      &= \frac{h_2 \pm i\,h_3}{\sqrt{2}},\\
k_1  &= \frac{h_4 - i\,h_5}{\sqrt{2}}, &k_2  &= \frac{h_6 - i\,h_7}{\sqrt{2}}, & k_3   &= \frac{h_8 - i\,h_9}{\sqrt{2}} \\
k_{\bar{1}}  &= \frac{h_4 + i\,h_5}{\sqrt{2}}, & k_{\bar{2}} &= \frac{h_6 + i\,h_7}{\sqrt{2}}, & k_{\bar{3}}  &= \frac{h_8 + i\,h_9}{\sqrt{2}}.
\end{align*}

Note that functions with `$\pm$' are paired with basis elements of the opposite sign `$\mp$'. 
The $J^1_+$-invariant part is given by $h_I = h_0 \mathbf{b}_0 + h_1 \mathbf{b}_1 + h_- \mathbf{b}_+ + h_+ \mathbf{b}_-$. The $J^1_+$-anti-invariant is given by $h_A = \sum_{j=1}^{3}(k_j\,\mathbf k_j  +  k_{\bar{j}}  \mathbf k_{\bar{j}} )$.  On the other hand, we can recover the component functions of $h$ in the basis $\mathbf{B}$ by the formulas
\begin{align*}
h_2 &= \frac{h_++h_-}{\sqrt{2}}, & h_4 &= \frac{k_{\bar{1}}  + k_1}{\sqrt{2}}, & h_6 &= \frac{k_{\bar{2}} +k_2 }{\sqrt{2}}, & h_8 &= \frac{k_{\bar{3}} +k_3}{\sqrt{2}}, \\
h_3 &= \frac{h_+-h_-}{i\sqrt{2}}, & h_5 &= \frac{k_{\bar{1}}  - k_1}{i\sqrt{2}},& h_7 &= \frac{ k_{\bar{2}}-k_2 }{i\sqrt{2}}, & h_9 &= \frac{k_{\bar{3}}  -k_3}{i\sqrt{2}} .
\end{align*}
As above, we define 
\[
(h^J_{M'})_{p, M} := \frac{J+1}{2\pi^2} \int_{\mathbb{S}^3} h_p \, \overline{D^J_{M, M'}}, \qquad (k^J_{M'})_{q, M} := \frac{J+1}{2\pi^2} \int_{\mathbb{S}^3} k_q \, \overline{D^J_{M, M'}}, 
\]
for $p \in \{0, 1, +, -\}$ and $q \in \{1, 2, 3, \bar{1}, \bar{2}, \bar{3}\}$. 
Note that 
\[
\overline{(h^J_{M'})_{p, M}} = (-1)^{\frac{M-M'}{2}} (h^J_{-M'})_{p, -M}, \qquad \overline{(h^J_{M'})_{\pm, M}} = (-1)^{\frac{M-M'}{2}} (h^J_{M'})_{\mp, -M},
\]
for $p \in \{0, 1\}$ and 
\[
\overline{(k^J_{M'})_{q, M}} = (-1)^{\frac{M-M'}{2}} (k^J_{-M'})_{\bar{q}, -M}, \qquad \overline{(k^J_{M'})_{\bar{q}, M}} = (-1)^{\frac{M-M'}{2}} (k^J_{-M'})_{q, -M},
\]
for $q \in\{1, 2, 3\}$. 


As a corollary of Proposition \ref{prop:levi-civita-on-2-tensors}, we have the following. 

\begin{proposition}\label{prop:levi-civita-complex-basis}
    The Levi-Civita connection acts on the complex basis of $2$-tensors by the formulas
    \begin{align*}
        \nabla \mathbf{b}_0 & = 0, \\
        \nabla \mathbf{b}_1
        & = i \Gamma_1^- \sigma^+ \otimes \mathbf{b}_+ + i\Gamma_1^- \sigma^- \otimes (-\mathbf{b}_-),\\
        \nabla \mathbf{b}_+ 
        & =i\Gamma_{23} \,e^1 \otimes (-\mathbf{b}_+) +i\Gamma_1^-\, \sigma^-\otimes \mathbf{b}_1, \\
        \nabla \mathbf{b}_- 
        & =i\Gamma_{23} \,e^1 \otimes\mathbf{b}_- +i\Gamma_1^-\,\sigma^+ \otimes (-\mathbf{b}_1),  \end{align*}
        and
        \begin{align*}
        \nabla \mathbf{k}_1 
        & = i \Gamma_1^+ \, e^1 \otimes \mathbf{k}_1  + i \Gamma_1^-\,\sigma^- \otimes (- \mathbf{k}_2) + i \Gamma_1^-  \sigma^+ \otimes \mathbf{k}_3, \\
        \nabla \mathbf{k}_2 & = i (\Gamma_1^+ + \Gamma_{23}^-)\, e^1 \otimes \mathbf{k}_2+ i \Gamma_1^- \, \sigma^+\otimes(- \mathbf{k}_1), \\
        \nabla \mathbf{k}_3 & = i (\Gamma_1^+ - \Gamma_{23}^-) \,e^1 \otimes \mathbf{k}_3 + i \Gamma_1^- \, \sigma^- \otimes \mathbf{k}_1, 
    \end{align*}
    and
    \begin{align*}
        \nabla \mathbf{k}_{\bar{1}} & = i \Gamma_1^+ \, e^1 \otimes (-\mathbf{k}_{\bar{1}})  + i \Gamma_1^-\,\sigma^+ \otimes  \mathbf{k}_{\bar{2}} + i \Gamma_1^-  \sigma^- \otimes (-\mathbf{k}_{\bar{3}} ), \\
        \nabla \mathbf{k}_{\bar{2}} & = i (\Gamma_1^+ + \Gamma_{23}^-)\, e^1 \otimes(- \mathbf{k}_{\bar{2}}) + i \Gamma_1^- \, \sigma^-\otimes \mathbf{k}_{\bar{1}} , \\
        \nabla \mathbf{k}_{\bar{3}} & = i (\Gamma_1^+ - \Gamma_{23}^-) \, e^1 \otimes (-\mathbf{k}_{\bar{3}}) + i \Gamma_1^- \, \sigma^+\otimes (-\mathbf{k}_{\bar{1}}).
    \end{align*}
\end{proposition}

\begin{proof}
    This is a computational corollary of Proposition \ref{prop:levi-civita-on-2-tensors}. Note that the formula involving $\mathbf{b}_-$ is the conjugation of the formula involving $\mathbf{b}_+$. Similarly, the formulas for $\mathbf{k}_{\bar{j}}$ are the conjugations of those for $\mathbf{k}_j$.  To check it by hand, it is  useful to observe the general identities 
    \[
    \frac{1}{2}(X+  i Y) (A - i B) +\frac{1}{2}(X -  i Y) (A +  i B)  = A X + BY
    \]
    and 
    \[
    \frac{1}{2}(X+  i Y) (A + i B) +\frac{1}{2}(X -  i Y) (A -  i B)  = A X - BY.
    \]
    Additionally, note that
    \[
    i \sigma^+ = \frac{1}{\sqrt{2}}( e^-+ie^+ ), \qquad i \sigma^- = -\frac{1}{\sqrt{2}} (e^- - i e^+). 
    \]
\end{proof}
From Proposition \ref{prop:divf-on-basis}, we have: 
\begin{corollary}\label{cor:divf-complex-basis}
    The divergences of the complex basis of $2$-tensors are given by the formulas 
    \begin{align*}
        \mathrm{div}_f(\mathbf{b}_0) & = - \frac{s}{8r^2}\Big( 2 \sqrt{2} r^2\Big) e^0, & \mathrm{div}_f(\mathbf{b}_1) & = \frac{s}{8r^2}\Big(4 - 2 \sqrt{2} r^2\Big) e^0, \\
        \mathrm{div}_f(\mathbf{b}_+) &= \frac{s}{8r^2} \Big(\frac{4-r^2}{F} - \sqrt{2} r^2\Big) (-i \sigma^-), & \mathrm{div}_f(\mathbf{b}_-) &= \frac{s}{8r^2} \Big(\frac{4-r^2}{F} - \sqrt{2} r^2\Big) (i \sigma^+), \\
        \mathrm{div}_f(\mathbf{k}_1) & = \frac{s}{8r^2} \Big( \frac{4F-r^2}{F} - \sqrt{2} r^2\Big) (i\sigma^-), &\mathrm{div}_f(\mathbf{k}_{\bar{1}}) & = \frac{s}{8r^2} \Big( \frac{4F-r^2}{F} - \sqrt{2} r^2\Big) (-i\sigma^+), \\
        \mathrm{div}_f(\mathbf{k}_2) & = \frac{s}{8r^2} \Big( \frac{4-2r^2}{F} \Big) (e^0-ie^1), &  \mathrm{div}_f(\mathbf{k}_{\bar{2}}) & = \frac{s}{8r^2} \Big( \frac{4-2r^2}{F} \Big) (e^0 + ie^1 ), \\
        \mathrm{div}_f(\mathbf{k}_3) & = 0, & \mathrm{div}_f(\mathbf{k}_{\bar{3}}) & = 0. 
    \end{align*}
\end{corollary}

As a corollary of Corollary \ref{prop:riemann-02basis}, we have:

\begin{proposition}\label{prop:riemann-complex-basis}
    The Riemann curvature acts on the complex basis of $2$-tensors by the formulas 
    \begin{align*}
R(\mathbf{b}_0)&
=\frac{c_0\sqrt2}{4\,r^2}\mathbf{b}_0 - \frac{rF'}{4\sqrt{2}} \mathbf{b}_1 & R(\mathbf{b}_1)&=
- \frac{rF'}{4\sqrt{2}} \mathbf{b}_0+ \Big(\frac{c_0\sqrt2}{4\,r^2}+ \frac{rF'}{2r^2}\Big)\mathbf{b}_1, 
\end{align*}
and 
\begin{align*}
R(\mathbf{b}_+)  &
=-\frac{rF'}{4r^2}\mathbf{b}_+ 
& R(\mathbf{k}_1) &
= \frac{rF'}{2r^2} \mathbf{k}_1, & 
R(\mathbf{k}_2) &
=-\frac{rF'+2F-\sqrt{2}}{2r^2}\mathbf{k}_2,& R(\mathbf{k}_3)&=-\frac{1-F}{r^2}\mathbf{k}_3, \\
R(\mathbf{b}_-)  &
=-\frac{rF'}{4r^2}\mathbf{b}_-, & R(\mathbf{k}_{\bar{1}}) &
= \frac{rF'}{2r^2} \mathbf{k}_{\bar{1}}, 
&R(\mathbf{k}_{\bar{2}}) &
=-\frac{rF'+2F-\sqrt{2}}{2r^2}\mathbf{k}_{\bar{2}},&
R(\mathbf{k}_{\bar{3}}) &
=-\frac{1-F}{r^2}\mathbf{k}_{\bar{3}}.
\end{align*}
\end{proposition}

\begin{proof}
    We can verify that 
    \begin{align*}
        -\frac{F''}{16} - \frac{7F'}{16r} + \frac{1-F}{2r^2} & = \frac{c_0\sqrt2}{4\,r^2}, \\
       -\frac{F''}{16} + \frac{F'}{16r} + \frac{1-F}{2r^2} & = \frac{c_0\sqrt2}{4\,r^2}+ \frac{rF'}{2r^2}, \\
         -\frac{F''}{16} - \frac{3F'}{16r} - \frac{1-F}{2r^2}  & = - \frac{rF'}{4\sqrt{2}}.
    \end{align*}
    Now the first formulas follow from Proposition \ref{prop:riemann-02basis}. The remaining formulas follow similarly after recalling the definition of $\mathbf{C}$ above (and in the case of $\mathbf{k}_2, \mathbf{k}_{\bar{2}}$ noting that $\frac{F''}{8} + \frac{3F'}{8r} =-\frac{rF'+2F-\sqrt{2}}{2r^2}$). 
\end{proof}

\subsection{Integrals and derivatives in the Wigner basis}

We summarize several useful integral and derivative identities for functions expressed with respect to the Wigner basis on the FIK shrinking soliton. Define 
\begin{equation}\label{eq:CJM}
C^J_{M\pm} := \sqrt{J (J+2) - M (M \pm 2)}= \sqrt{(J+1)^2-(M\pm1)^2}.
\end{equation}

We observe the useful identities 
\[
C^J_{(M+2)-} = C^J_{M+}, \qquad C^J_{(M-2)+} = C^J_{M-}, \qquad C^J_{-M-} = C^J_{M+} , \qquad C^J_{-M+} = C^J_{M-},
\]
\[
C^J_{J+} = C^J_{-J-} = 0,
\]
and
\[
\frac{(C^J_{M-})^2 + (C^J_{M+})^2}{2} = J(J+2) - M^2.
\]
We begin with a lemma on the actions of derivatives. By convention, we take $D^J_{M, M'} = 0$ if $M \not \in \mathcal{J}$.

\begin{lemma}\label{lem:fik-derivatives-on-wigner}
Suppose $u(r)$ is a complex-valued radial function on FIK. The basis of complex vector fields act by the formulas
\begin{align*}
e_0\Big( u(r)D^J_{M,M'}\Big) &= \frac{s}{2r} u'(r)D^J_{M,M'}, & e_0\Big( u(r)\overline{D^J_{M,M'}}\Big) &= \frac{s}{2r} u'(r) \overline{D^J_{M,M'}}, \\
e_1\Big(u(r)D^J_{M,M'}\Big) &= -i\frac{M}{2s}\, u(r)\,D^J_{M,M'},& e_1\Big(u(r)\overline{D^J_{M,M'}}\Big) &= i\frac{M}{2s}\, u(r)\,\overline{D^J_{M,M'}},\\
D_+\Big(u(r)D^J_{M,M'}\Big) &=-\frac{1+i}{4r}u(r)C^J_{M+}\,D^J_{M+2,M'}, & D_+\Big(u(r)\overline{D^J_{M,M'}}\Big) &=\frac{1+i}{4r}u(r)C^J_{M-}\,\overline{D^J_{M-2,M'}},\\
D_-\Big(u(r)D^J_{M,M'}\Big) &= \frac{1-i}{4r}u(r)C^J_{M-}\,D^J_{M-2,M'}, & D_-\Big(u(r)\overline{D^J_{M,M'}}\Big) &=-\frac{1-i}{4r}u(r)C^J_{M+}\,\overline{D^J_{M+2,M'}}. 
\end{align*}
\end{lemma}
\begin{proof}
    These are straightforward consequences of the definitions above and Proposition \ref{prop:wigner}. For instance, 
    \begin{align*}
    D_{\pm}(D^J_{M,M'}) & = \frac{1\mp i}{4r}(X_3 \pm iX_2)(D^J_{M,M'})  = -\frac{1\mp i}{4r} \,i \,C^J_{M\pm} D^J_{M \pm 2, M'} = -\frac{i\pm 1}{4r}  C^J_{M\pm} D^J_{M \pm 2, M'}.
    \end{align*}
    The latter formulas involving the conjugates of Wigner functions follow by simply conjugating the first. Alternatively, one can verify them directly using that $\overline{D^J_{M, M'}} = (-1)^{\frac{M-M'}{2}} D^J_{-M, -M'}$ and the useful identities involving $C^J_{M\pm}$ noted above. 
\end{proof}

\begin{lemma}\label{lem:fik-wigner-integrations}
    Suppose that $u, v$ are a complex-valued functions on $\mathbb{S}^3 \times (1, \infty)$ such that
    \[
    u(r, \Theta) = \sum_{M \in \mathcal{J}} u_M(r) \, D^J_{M, M'}(\Theta), \qquad v(r, \Theta) = \sum_{M \in \mathcal{J}} v_M(r) \, D^{J}_{M, M'}(\Theta).
    \]
    Then
    \[
    \int_{\mathbb{S}^3} u\,\overline{v}  = \frac{2\pi^2}{J+1} \sum_{M \in \mathcal{J}} u_M \overline{v_M}. 
    \]
    and
    \begin{align*}
    \int_{\mathbb{S}^3} |v|^2 & = \frac{2\pi^2}{J+1} \sum_{M \in \mathcal{J}} |v_M|^2, \\
    \int_{\mathbb{S}^3} |e_0(v)|^2 & =  \frac{2\pi^2}{J + 1} \sum_{M \in \mathcal{J}} \frac{F}{4} |v_M'|^2, \\
    \int_{\mathbb{S}^3} |e_1(v)|^2  &=  \frac{2\pi^2}{J+1} \sum_{M \in \mathcal{J}}\frac{M^2}{4s^2}|v_M|^2, \\
    \int_{\mathbb{S}^3} |D_{\pm}(v)|^2  & =   \frac{2\pi^2}{J+1} \sum_{M \in \mathcal{J}} \frac{(C^J_{M\pm})^2}{8r^2} |v_M|^2.
    \end{align*}
    Consequently, 
    \begin{align*}
     \int_{\mathbb{S}^3} |D_-(v)|^2  + |D_+(v)|^2  &=\frac{2\pi^2}{J+1} \sum_{M \in \mathcal{J}} \frac{J (J+2) - M^2}{4r^2}  |v_M|^2, \\
    \int_{\mathbb{S}^3} |\nabla v|^2& = \frac{2\pi^2}{J+1} \sum_{M \in \mathcal{J} } \frac{F}{4} |v_M'|^2  -\tilde{\Delta}^J_M|v_M|^2.
    \end{align*}
    where 
\begin{equation}\label{eq:berger-eigenvalues}
\tilde{\Delta}^J_M(r):=-\frac{1}{4r^2}\left(\Big(\frac{1}{F}-1\Big) M^2+J (J+2) \right) = \frac{M^2 - J(J+2)}{4r^2} - \frac{M^2}{4s^2}.
\end{equation}
\end{lemma}

\begin{proof}
Recalling again Proposition \ref{prop:wigner}, the first and second expressions are directly a consequence of the orthogonality relations for Wigner functions. The third expression is similarly proved using that $|e_0(v)|^2 = \frac{F}{4} |v'|t^2$. One way to obtain the fourth expression is to integrate by parts on $\mathbb{S}^3$ and use that 
\[
e_1(e_1(D^J_{M,M'})) = \frac{1}{4s^2} X_1 X_1(D^J_{M,M'}) = - \frac{M^2}{4s^2} D^J_{M,M'}. 
\]
For the fifth expression, we observe that
\begin{align*}
|D_{\pm}(v)|^2  = D_{\pm} (v)D_{\mp}(\overline{v}) &= \sum_{M,M_0 \in \mathcal{J}} v_M \overline{v_{M_0}} \, \frac{1}{8r^2} C^J_{M\pm}C^J_{M_0\pm} \, D^J_{M\pm2,M'}\overline{D^J_{M_0\pm2,M'}}.
\end{align*} 
Now integrating and using orthogonality of the Wigner functions gives the claim. 
Finally, for the penultimate identity, we can use that 
\[
(D_-D_-+D_+D_+)(D^J_{M,M'}) = \frac{1}{4r^2} (X_2X_2 + X_3 X_3)(D^J_{M,M'}) = - \frac{J(J+2)-M^2}{4r^2} D^J_{M,M'}.
\]

The last claim is a consequence of the others. 
\end{proof}

We will also need the following corollary in the computations that follow. 

\begin{corollary}\label{cor:fik-wigner-integrations}
Suppose $u, v$ are complex-valued functions on $\mathbb{S}^3 \times (1, \infty)$ such that
    \[
    u(r, \Theta) = \sum_{M \in \mathcal{J}} u_M(r) \, D^J_{M, M'}(\Theta), \qquad v(r, \Theta) = \sum_{M \in \mathcal{J}} v_M(r) \, D^{J}_{M, M'}(\Theta).
    \]
Let $\Gamma$ be real-valued radially symmetric function. 
Then
\begin{align*}
\frac{J+1}{2\pi^2} \int_{\mathbb{S}^3} |D_+(u) \pm i\Gamma v|^2 & = \sum_{M \in \mathcal{J}} \frac{(C^J_{M+})^2}{8r^2} |u_M|^2 + \Gamma^2  |v_M|^2  \\
& \qquad  \mp\sum_{M \in \mathcal{J}} C^J_{M-} \Gamma \left(  \frac{1 - i}{4r} u_{M-2} \overline{v_M} + \frac{1+i}{4r}  \overline{u_{M-2}} v_M \right),
\end{align*}
and
\begin{align*}
\frac{J+1}{2\pi^2} \int_{\mathbb{S}^3} |D_-(u) \pm i\Gamma v|^2 & = \sum_{M \in \mathcal{J}} \frac{(C^J_{M-})^2}{8r^2} |u_M|^2 + \Gamma^2  |v_M|^2  \\
& \qquad \mp\sum_{M \in \mathcal{J}}  C^J_{M+}  \Gamma\left( \frac{1 + i}{4r}  u_{M+2} \overline{v_M} + \frac{1-i}{4r} \overline{u_{M+2}} v_M \right). 
\end{align*}
\end{corollary} 

\begin{proof}
We compute that 
\begin{align*}
 |D_+(u) + i\Gamma v|^2  &=  |D_+(u)|^2 + \Gamma^2 |v|^2 - i \Gamma D_+(u)\overline{v} +i\Gamma D_-(\overline{u}) v \\
 & = |D_+(u)|^2 + \Gamma^2 |v|^2 +  i \frac{1+i}{4r}  \Gamma \sum_{M, M_0 \in \mathcal{J}}  C^J_{M+} u_M \overline{v_{M_0}} D^J_{M+2, M'} \overline{D^J_{M_0, M'}}  \\
 & \qquad - i \frac{1-i}{4r} \Gamma \sum_{M, M_0 \in \mathcal{J}} C^J_{M+} \overline{u_M} v_{M_0} \overline{D^J_{M+2, M'} }D^J_{M_0, M'} \\
 & = |D_+(u)|^2 + \Gamma^2 |v|^2 -  \frac{1-i}{4r}  \Gamma \sum_{M, M_0 \in \mathcal{J}}  C^J_{M-} u_{M-2} \overline{v_{M_0}} D^J_{M, M'} \overline{D^J_{M_0, M'}}  \\
 & \qquad - \frac{1+i }{4r} \Gamma \sum_{M, M_0 \in \mathcal{J}} C^J_{M-} \overline{u_{M-2}} v_{M_0} \overline{D^J_{M, M'} }D^J_{M_0, M'}. 
\end{align*}
In the third equality, we re-indexed ($M \to M -2$) used that $C^J_{J+} = 0$ and $C^J_{(M-2)+} = C^J_{M-}$. Integrating then gives
\begin{align*}
\int_{\mathbb{S}^3}  |D_+(u) + i\Gamma v|^2 & = \int_{\mathbb{S}^3}  |D_+(u)|^2 + \Gamma^2  \int_{\mathbb{S}^3} |v|^2  \\
& \qquad - \frac{2\pi^2}{J+1} \left(\sum_{M \in \mathcal{J}} \frac{1 - i}{4r} \Gamma  C^J_{M-} u_{M-2} \overline{v_M} + \frac{1+i}{4r} \Gamma  C^J_{M-} \overline{u_{M-2}} v_M \right). 
\end{align*} 
This gives the first identity with `$+$'. Swapping the sign of $\Gamma$ in this formula gives the first identity with `$-$'. The derivation of the other identity is similar. 
\end{proof}

\section{The operator $L_f$ in the Wigner basis and stability for $J \geq 3$}

In this section, we consider a $J_1^+$-invariant $2$-tensor $h_I$ given as
\begin{equation}\label{eq:hI-complex-components}
h_I=
 h_0\,\mathbf b_0
  +
 h_1\,\mathbf b_1
  +
 h_-\,\mathbf b_+
  +
 h_+\,\mathbf b_-, 
\end{equation}
and a $J^1_+$-anti-invaraint $2$-tensor $h_A$ given as 
\begin{equation}\label{eq:hA-complex-components}
    h_A=  k_1 \mathbf{k}_1 +  k_2 \mathbf{k}_2 + k_3 \mathbf{k}_3 +k_{\bar{1}} \mathbf{k}_{\bar{1}}+ k_{\bar{2}} \mathbf{k}_{\bar{2}} + k_{\bar{3}} \mathbf{k}_{\bar{3}}. 
\end{equation}
In general, we allow each of the component functions to be complex-valued. When $h_I$ is real-valued, then $h_0, h_1$ must be real valued and $h_-$ and $h_+$ must be complex conjugates of one another. When $h_A$ is real-valued, $k_p$ and $k_{\bar{p}}$ must be complex conjugates of one another for $p \in \{1, 2,3 \}$.

\subsection{Expanding $|\nabla h|^2$ and $R(h, \overline{h})$ in the complex basis}

Our first goal is to expand the gradient term 
\begin{equation}
  |\nabla h|^2
    =|\nabla_{e_{0}}h|^2
    +|\nabla_{e_1}h|^2
    +|\nabla_{D_+}h|^2
    +|\nabla_{D_-}h|^2.
\end{equation}
as well as the curvature term $R(h, \overline{h})$ in the complex basis $\mathbf{C}$ defined above. As discussed before, in view of Propositions \ref{prop:levi-civita-complex-basis} and \ref{prop:riemann-complex-basis}, we can treat the invariant and anti-invariant cases separately. 

\begin{lemma}\label{lem:expanding-hI-complex}
    Suppose $h_I$ is a complex-valued $J_1^+$-invariant $2$-tensor expressed by component components as in \eqref{eq:hI-complex-components}. Then, 
    \begin{align*}
|\nabla_{e_{0}}h_I|^2
  &=\frac14\sum_{p\in\{0,1,-,+\}}| e_{0}(h_{p})|^2,\\
|\nabla_{e_1}h_I|^2
  &=\frac14\Big(
      | e_1(h_{0})|^2
     +| e_1(h_1)|^2
     +| e_1(h_+)+i\Gamma_{23}^- h_+|^2
     +| e_1(h_-)-i\Gamma_{23}^- h_-|^2
     \Big),\\
|\nabla_{D_+}h_I|^2
  &=\frac14\Big(
      | D_+(h_{0})|^2
     +| D_+(h_1)-i\Gamma_1^- h_+|^2
     +| D_+(h_-)+i\Gamma_1^- h_1|^2
     +| D_+(h_+)|^2
     \Big),\\
|\nabla_{D_-}h_I|^2
  &=\frac14\Big(
      | D_-(h_{0})|^2
     +| D_-(h_1)+i\Gamma_1^- h_-|^2
     +| D_-(h_-)|^2
     +| D_-(h_+)-i\Gamma_1^- h_1|^2
     \Big).
\end{align*}
Additionally, 
\begin{align*}
R(h_I, \overline{h_I}) &= \frac{c_0\sqrt{2}}{16r^2}|h_0|^2 +\Big(\frac{c_0\sqrt{2}}{16r^2} + \frac{rF'}{8r^2} \Big) 
|h_1|^2 - \frac{rF'}{16\sqrt{2}} (h_0 \overline{h_1} + \overline{h_0} h_1) -\frac{rF'}{16r^2}(|h_+|^2+ |h_-|^2).
\end{align*}
\end{lemma}

\begin{proof}
Using Proposition \ref{prop:levi-civita-complex-basis}, we compute
\begin{align*}
\nabla_{e_0} h_I &= e_0(h_0)\,\mathbf b_0 + e_0(h_1)\,\mathbf b_1 + e_0(h_-)\,\mathbf b_+ + e_0(h_+)\,\mathbf b_-, \\
 \nabla_{e_1} h_I &= e_1(h_0)\,\mathbf b_0 + e_1(h_1)\,\mathbf b_1 + \big(e_1(h_-)- i\Gamma_{23} h_-\big)\,\mathbf b_+ + \big(e_1(h_+)+ i \Gamma_{23} h_+)\,\mathbf b_-,\\
 \nabla_{D_+} h_I &= D_+(h_0)\,\mathbf b_0 + \big(D_+(h_1) - i\Gamma_1^- h_+\big)\,\mathbf b_1 + \big(D_+(h_-)+ i\Gamma_1^- h_1\big)\,\mathbf b_+ + D_+(h_+)\,\mathbf b_-, \\
\nabla_{D_-} h_I &= D_-(h_0)\,\mathbf b_0 + \big(D_-(h_1) +  i\Gamma_1^- h_-\big)\,\mathbf b_1 + D_-(h_-)\,\mathbf b_+ + \big(D_-(h_+) - i \Gamma_1^- h_1\big)\,\mathbf b_-. 
\end{align*}
Now the first set of formulas follow by taking norms. Note the third and fourth formulas are complex conjugates whenever $h_I$ is real-valued (and in this case, one has $|\nabla_{D_+}h_I|^2 = |\nabla_{D_-}h_I|^2$). 

Using Proposition \ref{prop:riemann-complex-basis}, we have 
\[
R(h_I) = \Big(\frac{c_0\sqrt{2}}{4r^2} h_0- \frac{rF'}{4\sqrt{2}} h_1\Big) \mathbf{b}_0 +\Big( \Big(\frac{c_0\sqrt{2}}{4r^2} + \frac{rF'}{2r^2} \Big)h_1 - \frac{rF'}{4\sqrt{2}} h_0 \Big)\mathbf{b}_1 -\frac{rF'}{4r^2}(h_- \mathbf{b}_+ + h_+ \mathbf{b}_-).
\]
Now the last formula follows by taking the inner product with 
\[
\overline{h_I} = \overline{h_0}\,\mathbf b_0
  +
\overline{h_1}\,\mathbf b_1
  +
 \overline{h_-}\,\mathbf b_-
  +
 \overline{h_+}\,\mathbf b_+.
\]
\end{proof}

\begin{lemma}\label{lem:expanding-hA-complex}
    Suppose $h = h_A$ is a complex-valued $J^1_+$-anti-invariant $2$-tensor expressed by component functions as in \eqref{eq:hA-complex-components}. Then 
    \begin{align*}
        |\nabla_{e_{0}}h_A|^2
        &=\frac14\sum_{j=1}^{3}
          \big(| e_{0}(k_j)|^2
               +| e_{0}(k_{\bar{j}})|^2\big)\\
        |\nabla_{e_1}h_A|^2
      &=\frac14 
        \Big(
           | e_1(k_1)+i\Gamma_1^+k_1|^2 +| e_1(k_2)+i(\Gamma_1^+ + \Gamma_{23}^-)k_2|^2 +| e_1(k_3)+i(\Gamma_1^+ - \Gamma_{23}^-)k_3|^2 \\
          & \quad \; +| e_1(k_{\bar{1}})-i\Gamma_1^+k_{\bar{1}}|^2+  | e_1(k_{\bar{2}})-i(\Gamma_1^+ + \Gamma_{23}^-)k_{\bar{2}}|^2 
          +| e_1(k_{\bar{3}})-i(\Gamma_1^+ - \Gamma_{23}^-)k_{\bar{3}}|^2
        \Big)\\
        |\nabla_{D_+}h_A|^2
      &=\frac14\Big(
         | D_+(k_1)-i\Gamma_1^- k_2|^2+| D_+(k_2)|^2 +| D_+(k_3)+i\Gamma_1^- k_1|^2
        \\
        &\quad \; +| D_+(k_{\bar{1}})-i\Gamma_1^- k_{\bar{3}}|^2   +| D_+(k_{\bar{2}})+i\Gamma_1^- k_{\bar{1}}|^2
        +| D_+(k_{\bar{3}})|^2\Big)\\
        |\nabla_{D_-}h_A|^2
      &=\frac14\Big(
         | D_-(k_1)+i\Gamma_1^- k_3|^2+| D_-(k_2)-i\Gamma_1^- k_1|^2 +| D_-(k_3)|^2 \\
        &\quad \; +| D_-(k_{\bar{1}})+i\Gamma_1^- k_{\bar{2}}|^2 +| D_-(k_2)|^2
        +| D_-(k_{\bar{3}})-i\Gamma_1^- k_{\bar{1}}|^2\Big).
    \end{align*}
Additionally, 
\begin{align*}
    R(h_A, \overline{h_A}) = \frac{rF'}{8r^2}\big(|k_1|^2+|k_{\bar{1}}|^2\big) -\frac{rF'+2F-\sqrt{2}}{8r^2} \big(|k_2|^2+|k_{\bar{2}}|^2\big) - \frac{1-F}{4r^2}\big(|k_3|^2+|k_{\bar{3}}|^2\big).
\end{align*}
\end{lemma}

\begin{proof}
     There is an evident conjugation symmetry in our formulas. 
    Using Proposition \ref{prop:levi-civita-complex-basis}, we compute
    \begin{align*}
        \nabla_{e_0} h_A &= \sum_{j =1}^3 e_0(k_j) \mathbf{k}_j + e_0(k_{\bar{j}}) \mathbf{k}_{\bar{j}}, \\
        \nabla_{e_1} h_A & = \big(e_1(k_1) + i\Gamma_1^+ k_1 \big) \mathbf{k}_1 + \big(e_1(k_2) + i(\Gamma_1^++\Gamma_{23}^-) k_2\big) \mathbf{k}_2 + \big(e_1(k_3) + i(\Gamma_1^+ -\Gamma_{23}^-)k_3\big) \mathbf{k}_3\\
        &   + \big(e_1(k_{\bar{1}}) - i\Gamma_1^+ k_{\bar{1}}\big) \mathbf{k}_{\bar{1}} + \big(e_1(k_{\bar{2}}) - i(\Gamma_1^+ +\Gamma_{23}^-)k_{\bar{2}}\big) \mathbf{k}_{\bar{2}} + \big(e_1(k_{\bar{3}}) - i(\Gamma_1^+-\Gamma_{23}^-) k_{\bar{3}}\big) \mathbf{k}_{\bar{3}}, \\
        \nabla_{D_+} h_A & = \big(D_+(k_1)-i\Gamma_1^- k_2\big) \mathbf{k}_1 + D_+(k_2)\mathbf{k}_2 + \big(D_+(k_3) + i \Gamma_1^- k_1\big) \mathbf{k}_3 \\
        &  + \big(D_+ (k_{\bar{1}})- i \Gamma_1^- k_{\bar{3}}\big) \mathbf{k}_{\bar{1}} + \big(D_+(k_{\bar{2}}) + i \Gamma_1^- k_{\bar{1}}\big) \mathbf{k}_{\bar{2}}  + D_+(k_{\bar{3}}) \mathbf{k}_{\bar{3}}, \\
        \nabla_{D_-} h_A & = \big(D_-(k_1)+i\Gamma_1^- k_3\big) \mathbf{k}_1  + \big(D_-(k_2)- i \Gamma_1^- k_1\big)\mathbf{k}_2 + D_-(k_3) \mathbf{k}_3\\
        &  + \big(D_- (k_{\bar{1}})+ i \Gamma_1^- k_{\bar{2}}\big) \mathbf{k}_{\bar{1}} + D_-(k_{\bar{2}}) \mathbf{k}_{\bar{2}}  + \big(D_-(k_{\bar{3}})- i \Gamma_1^- k_{\bar{1}}\big) \mathbf{k}_{\bar{3}}. 
    \end{align*}
    The first set of formulas follows by taking norms.  Note the third and fourth formulas are complex conjugates whenever $h_A$ is real-valued (and in this case, one has $|\nabla_{D_+}h_A|^2 = |\nabla_{D_-}h_A|^2$).

    Using Proposition \ref{prop:riemann-complex-basis}, we have 
    \begin{align*}
        R(h_A) & = \frac{rF'}{2r^2}(k_1 \mathbf{k}_1 + k_{\bar{1}} \mathbf{k}_{\bar{1}}) -\frac{rF'+2F-\sqrt{2}}{2r^2}(k_2 \mathbf{k}_2 + k_{\bar{2}} \mathbf{k}_{\bar{2}}) - \frac{1-F}{r^2} (k_3 \mathbf{k}_3 + k_{\bar{3}} \mathbf{k}_{\bar{3}}). 
    \end{align*}
    Now the last formula follows by taking the inner product with 
    \[
    \overline{h_A} = \overline{k_1} \,\mathbf{k}_{\bar{1}} + \overline{k_{\bar{1}}}\, \mathbf{k}_{1} +\overline{k_2} \,\mathbf{k}_{\bar{2}} + \overline{k_{\bar{2}}}\, \mathbf{k}_2  + \overline{k_3} \,\mathbf{k}_{\bar{3}} + \overline{k_{\bar{3}}}\, \mathbf{k}_{3}. 
    \] 
\end{proof}

\subsection{Expressing $2R(h, \overline{h})- |\nabla h|^2$  in each mode $J \in \mathbb{N}$}\label{sec:expanding Wigner}

Having expanded the curvature and gradient terms in the complex basis of $2$-tensors, our next goal is to express these quantities in the Wigner basis. As discussed above, since $\int 2R(h, h) -|\nabla h|^2$ decomposes over the spaces $\mathcal{T}^J_{M'}$, in this section we fix $J  \in \mathbb{N}$ and $M' \in \mathcal{J}$. We consider a complex-valued deformation 
\[
h = h_I + h_A \in \mathcal{T}^J_{M'}
\]
with $h_I$ as in \eqref{eq:hI-complex-components} and $h_A$ as in \eqref{eq:hA-complex-components}. In particular, we assume $h_p, k_q \in \mathcal{D}^J_{M'}$ for each $p \in \{0, 1, -, +\}$ and $q \in \{1, 2, 3, \bar{1},\bar{2}, \bar{3}\}$. Then, we have
 \[
    h_p  = \sum_{M \in \mathcal{J} }h_{p,M} D^J_{M, M'}, \qquad h_{p, M}(r) := \frac{J+1}{2\pi^2} \int_{\mathbb{S}^3} h_p(r, \cdot) \overline{D^J_{M, M'}},
\]
for $p \in \{0, 1, + , -\}$ and 
\[
    k_q = \sum_{M \in \mathcal{J} } k_{q,M} \, D^J_{M, M'}, \qquad k_{q, M}(r) := \frac{J+1}{2\pi^2} \int_{\mathbb{S}^3} k_q(r, \cdot) \overline{D^J_{M, M'}},
\]
for $q \in \{1, 2, 3, \bar{1}, \bar{2}, \bar{3}\}$. 

We begin by integrating the identities obtained in the previous section over $\mathbb{S}^3$. Given a function $u = u(r, \Theta)$, let us introduce the shorthand notation  
\begin{equation}
\int_{\mathbb{S}^3}^J u := \frac{J + 1}{2\pi^2} \int_{\mathbb{S}^3} u(r, \Theta) \, d\mu_{\mathbb{S}^3}(\Theta).
\end{equation}

\begin{lemma}\label{lem:hI-mode-ints}
    Suppose $h_I$ is a complex-valued $J_1^+$-invariant $2$-tensor in $\mathcal{T}^J_{M'}$. For any $r \in (1,\infty)$, 
    \begin{align}
        \frac{J+1}{2\pi^2} \int_{\mathbb{S}^3} 8 R(h_I, \overline{h_I}) & =  \sum_{M \in \mathcal{J}} \frac{c_0\sqrt{2}}{2r^2} |h_{0, M}|^2 - \frac{rF'}{2\sqrt{2}}\big(h_{0, M} \overline{h_{1, M}} + h_{1, M} \overline{h_{0, M}}\Big) \label{eq:hI-curve-int} \\
        &  +\sum_{M \in \mathcal{J}} \Big(\frac{c_0\sqrt{2}}{2r^2} + \frac{rF'}{r^2}\Big)|h_{1, M}|^2 - \frac{rF'}{2r^2}\big(|h_{+, M}|^2 + |h_{-, M}|^2\big), \nonumber 
    \end{align}
    \begin{align}
        \frac{J+1}{2\pi^2} \int_{\mathbb{S}^3} 4|\nabla_{e_0} h_I|^2 & = \sum_{M \in \mathcal{J}} \frac{F}{4} \Big(|h_{0, M}'|^2 + |h_{1, M}'|^2 + |h_{+, M}'|^2 + |h_{-, M}'|^2\Big), \label{eq:hI-e0-int}\\
        \frac{J+1}{2\pi^2} \int_{\mathbb{S}^3} 4|\nabla_{e_1} h_I|^2 & =\sum_{M \in \mathcal{J}} \frac{M^2}{4s^2} \Big(|h_{0,M}|^2 + |h_{1, M}|^2\Big) \label{eq:hI-e1-int}\\
        & + \sum_{M \in \mathcal{J}} \frac{(M - 2 -\frac{1}{2} rF')^2}{4s^2} |h_{+, M}|^2 +  \frac{(M + 2 +\frac{1}{2} rF')^2}{4s^2} |h_{-, M}|^2,   \nonumber
    \end{align}
    and 
    \begin{align}
        \frac{J+1}{2\pi^2} \int_{\mathbb{S}^3} 4|\nabla_{D_+} h_I|^2 &+4|\nabla_{D_-} h_I|^2\nonumber \\
        & = \sum_{M \in \mathcal{J}} \frac{J(J+2) - M^2}{4r^2} \Big(|h_{0, M}|^2 + |h_{1, M}|^2 + |h_{+, M}|^2 + |h_{-, M}|^2\Big) \label{eq:hI-Dpm-int} \\
        & + \sum_{M \in \mathcal{J}} \frac{F}{r^2} \Big(2|h_{1, M}|^2 + |h_{+, M}|^2 + |h_{-, M}|^2 \Big), \nonumber \\
        & +  \sum_{M \in \mathcal{J}} \frac{1-i}{2}  \frac{\sqrt{F}}{r^2} C^J_{M+} \Big(h_{-, M}  \overline{h_{1, M+2}}  - h_{1, M} \overline{h_{+,M+2}} \Big)\nonumber \\
        & +\sum_{M \in \mathcal{J}} \frac{1+i}{2}  \frac{\sqrt{F}}{r^2}  C^J_{M-} \Big(h_{1, M}  \overline{h_{-, M-2}} - h_{+, M}\overline{h_{1,M -2}}  \Big).\nonumber
    \end{align}
\end{lemma}

\begin{proof}
The first two identities following directly from Lemmas \ref{lem:fik-wigner-integrations} and  \ref{lem:expanding-hI-complex} after integration. By Lemma \ref{lem:fik-derivatives-on-wigner}, we have
\[
e_1(h_\pm) \pm i \Gamma_{23}^{-} h_\pm = \sum_{M \in \mathcal{J}} -i \Big(\frac{M}{2s} \mp \Gamma_{23}^-\Big)  h_{\pm, M} D^J_{M, M'}. 
\]
Recalling from Definition \ref{def:levi-civita-omega-coef-functions} that  $4s\Gamma_{23}^- = 4 +rF'$, we obtain 
\[
e_1(h_\pm) \pm i \Gamma_{23}^{-} h_\pm = \sum_{M \in \mathcal{J}} -i\frac{1}{2s} \Big(M \mp 2\mp \frac{1}{2}rF' \Big)  h_{\pm, M} D^J_{M, M'}. 
\]
Consequently, 
\begin{align*}
    \int_{\mathbb{S}^3}^J|e_1(h_+) + i \Gamma_{23}^-h_+|^2 &= \sum_{M \in \mathcal{J}} \frac{(M - 2 -\frac{1}{2}rF')^2}{4s^2}  |h_{+, M}|^2, \\
    \int_{\mathbb{S}^3}^J |e_1(h_-) - i \Gamma_{23}^-h_-|^2 & = \sum_{M \in \mathcal{J}} \frac{(M + 2+ \frac{1}{2}rF')^2}{4s^2}|h_{-, M}|^2. 
\end{align*}
Putting this together with the integrations of $|e_1(h_0)|^2 + |e_1(h_1)|^2$ gives the third identity. Finally, we turn to the last identity. First, we note 
\begin{align*}
     &\int_{\mathbb{S}^3}^J |D_+(h_+)|^2 + |D_-(h_-)|^2 + |D_+(h_0)|^2 + |D_-(h_0)|^2\\
     &= \sum_{M \in \mathcal{J}} \frac{J(J+2) -M^2}{4r^2} |h_{0, M}|^2 + \frac{(C^J_{M+})^2}{8r^2}|h_{+, M}|^2 +\frac{(C^J_{M-})^2}{8r^2}|h_{-, M}|^2. 
\end{align*}
Next, making use of Corollary \ref{cor:fik-wigner-integrations} with $\Gamma = \Gamma_1^- = \frac{\sqrt{F}}{r}$, we observe 
\begin{align*}
    &\int_{\mathbb{S}^3}^J |D_+(h_1) - i \Gamma_{1}^-h_+|^2 +|D_+(h_-) + i \Gamma_{1}^-h_1|^2\\ 
    & = \sum_{M \in \mathcal{J}} \frac{(C^J_{M+})^2}{8r^2} \big(|h_{1, M}|^2 + |h_{-, M}|^2 \big) + \frac{F}{r^2} \big(|h_{1, M}|^2 + |h_{+, M}|^2\big) \\
    &  + \sum_{M \in \mathcal{J}} C^J_{M-} \frac{\sqrt{F}}{r}\Big(\frac{1-i}{4r} \big( h_{1, M-2} \overline{h_{+, M}} -h_{-, M-2}\overline{h_{1, M}}\big)  + \frac{1+i}{4r}\big(h_{+, M} \overline{h_{1, M-2}} - h_{1, M} \overline{h_{-, M-2}}\big)\Big).
\end{align*}
Similarly 
\begin{align*}
    &\int_{\mathbb{S}^3}^J  |D_-(h_1) + i \Gamma_{1}^-h_-|^2 + |D_-(h_+) - i \Gamma_{1}^-h_1|^2\\ 
    & = \sum_{M \in \mathcal{J}} \frac{(C^J_{M-})^2}{8r^2} \big(|h_{1, M}|^2 + |h_{+, M}|^2 \big) + \frac{F}{r^2} \big(|h_{1, M}|^2 + |h_{-, M}|^2\big) \\
    &  + \sum_{M \in \mathcal{J}} C^J_{M+} \frac{\sqrt{F}}{r}\Big(\frac{1+i}{4r} \big( h_{+, M+2} \overline{h_{1, M}} -h_{1, M+2}\overline{h_{-, M}}\big)  + \frac{1-i}{4r}\big(h_{1, M} \overline{h_{+, M+2}} - h_{-, M} \overline{h_{1, M+2}}\big)\Big).
\end{align*}
Observe that by after re-indexing (using that $C^J_{(M-2)+} = C^J_{M-}$ and $C^J_{(M+2)-} = C^J_{M+}$), 
\begin{align*}
 \sum_{M \in \mathcal{J}} C^J_{M+}  \big( h_{+, M+2} \overline{h_{1, M}} -h_{1, M+2}\overline{h_{-, M}}\big)  & =  \sum_{M \in \mathcal{J}} C^J_{M-} \big(h_{+, M} \overline{h_{1, M-2}} - h_{1, M} \overline{h_{-, M-2}}\big) , \\
 \sum_{M \in \mathcal{J}} C^J_{M-} \big( h_{1, M-2} \overline{h_{+, M}} -h_{-, M-2}\overline{h_{1, M}}\big)& =  \sum_{M \in \mathcal{J}}C^J_{M+} \big(h_{1, M} \overline{h_{+, M+2}} - h_{-, M} \overline{h_{1, M+2}}\big) .
\end{align*}
Using these and putting the previous three identities together completes the derivation of the last identity asserted in the lemma. This completes the proof. 
\end{proof}

\begin{lemma}\label{lem:hA-mode-ints}
      Suppose $h_A$ is a complex-valued $J_1^+$-anti-invariant $2$-tensor in $\mathcal{T}^J_{M'}$. For any $r \in (1,\infty)$, 
    \begin{align}
        \frac{J+1}{2\pi^2} \int_{\mathbb{S}^3} 8 R(h_A, \overline{h_A}) & = \sum_{M \in \mathcal{J}} \frac{rF'}{r^2} \big(|k_{1,M}|^2+|k_{\bar{1},M}|^2\big)-\frac{rF'+2F-\sqrt{2}}{r^2} \big(|k_{2,M}|^2+|k_{\bar{2},M}|^2\big)  \label{eq:hA-curve-int} \\
        &  +\sum_{M \in \mathcal{J}} - \frac{2-2F}{r^2}\big(|k_{3,M}|^2+|k_{\bar{3},M}|^2\big), \nonumber 
    \end{align}
    \begin{align}
        \frac{J+1}{2\pi^2} \int_{\mathbb{S}^3} 4|\nabla_{e_0} h_A|^2 & = \sum_{M \in \mathcal{J}} \frac{F}{4} \Big(|k_{1, M}'|^2 + |k_{2, M}'|^2 +|k_{3, M}'|^2 +|k_{\bar 1, M}'|^2 +|k_{\bar 2, M}'|^2 +|k_{\bar 3, M}'|^2\Big), \label{eq:hA-e0-int}\\
        \frac{J+1}{2\pi^2} \int_{\mathbb{S}^3} 4|\nabla_{e_1} h_A|^2 & =\sum_{M \in \mathcal{J}} \frac{(M +2-  \frac{1}{2}rF'-2F)^2}{4s^2} |k_{1, M}|^2 + \frac{(M -2 +  \frac{1}{2}rF'+2F)^2}{4s^2} |k_{\bar{1}, M}|^2 \label{eq:hA-e1-int}\\
         & +\sum_{M \in \mathcal{J}} \frac{(M - rF'-2F)^2}{4s^2} |k_{2, M}|^2 + \frac{(M + rF'+ 2F)^2}{4s^2} |k_{\bar{2}, M}|^2 \nonumber \\
       & + \sum_{M \in \mathcal{J}} \frac{(M+4 -2F)^2}{4s^2} |k_{3, M}|^2 + \frac{(M -4+ 2F)^2}{4s^2} |k_{\bar{3}, M}|^2  \nonumber, 
    \end{align}
    and 
    \begin{align}
        &\frac{J+1}{2\pi^2} \int_{\mathbb{S}^3} 4|\nabla_{D_+} h_A|^2 +4|\nabla_{D_-} h_A|^2 \nonumber \\
        & = \sum_{M \in \mathcal{J}} \frac{J(J+2) - M^2}{4r^2} \Big(|k_{1, M}|^2 + |k_{2, M}|^2 +|k_{3, M}|^2 +|k_{\bar 1, M}|^2 +|k_{\bar 2, M}|^2 +|k_{\bar 3, M}|^2\Big) \label{eq:hA-Dpm-int}\\
        & + \sum_{M \in \mathcal{J}} \frac{F}{r^2} \Big(2|k_{1, M}|^2 + |k_{2, M}|^2 +|k_{3, M}|^2 +2|k_{\bar 1, M}|^2 +|k_{\bar 2, M}|^2 +|k_{\bar 3, M}|^2\Big),  \nonumber \\
        & + \sum_{M \in \mathcal{J}} \frac{1-i}{2} \frac{\sqrt{F}}{r^2} C^J_{M+}\Big( k_{3, M} \overline{k_{1, M+2}} - k_{1,M}\overline{k_{2,M+2}}  + k_{\bar{2}, M}  \overline{k_{\bar{1}, M+2}}-    k_{\bar{1},M} \overline{k_{\bar{3},M+2}}   \Big) \nonumber \\
        & + \sum_{M \in \mathcal{J}} \frac{1+i}{2} \frac{\sqrt{F}}{r^2} C^J_{M-}\Big(  k_{1, M} \overline{k_{3, M-2}} - k_{2,M} \overline{k_{1,M-2}} + k_{\bar{1}, M}\overline{k_{\bar{2}, M-2}} -  k_{\bar{3},M} \overline{k_{\bar{1},M-2}}  \Big) \nonumber.
    \end{align}
\end{lemma}

\begin{proof}
The proof proceeds as in Lemma \ref{lem:hI-mode-ints}. The first two identities again following directly from Lemmas \ref{lem:fik-wigner-integrations} and  \ref{lem:expanding-hA-complex} after integration. Recall from Definition \ref{def:levi-civita-omega-coef-functions} that 
\[
2s \Gamma_1^+ = \frac{1}{2}rF'+2F-2 \qquad 2s(\Gamma_1^+ + \Gamma_{23}^-)  = rF'+2F, \qquad 2s(\Gamma_1^+-  \Gamma_{23}^-) = 2F-4.
\]
On the other hand, we have
\[
e_1(k) +i\,  \Gamma\, k  = \sum_{M \in \mathcal{J}} -i \frac{1}{2s}(M - 2s\,\Gamma)\, k_{M}\, D^J_{M, M'}  
\]
for any $k \in \{k_1, k_2, k_3, k_{\bar 1}, k_{\bar 2}, k_{\bar 3}\}$. By the fourth identity from Lemma \ref{lem:fik-wigner-integrations}, 
\begin{align*}
 \int_{\mathbb{S}^3}^J |e_1(k) + i \Gamma k|^2 & = \sum_{M \in \mathcal{J}} \frac{(M - 2s \Gamma)^2}{4s^2} |k_{M}|^2. 
\end{align*}
Taking $k \in \{k_1, k_2, k_3, k_{\bar 1}, k_{\bar 2} , k_{\bar 3}\}$ and $\Gamma \in \{\Gamma_1^+, \Gamma_1^+ + \Gamma_{23}^-, \Gamma_1^+ - \Gamma_{23}^-\}$, and using the second identity in Lemma \ref{lem:expanding-hA-complex} gives the third identity. 

Finally, we once more make use of Corollary \ref{cor:fik-wigner-integrations} to obtain the final identity asserted in the lemma. Exactly as in the proof of Lemma \ref{lem:hI-mode-ints} (substituting $h_1 \to k_1, h_+ \to k_2, h_- \to k_3$), we have 
\begin{align*}
    &\int_{\mathbb{S}^3}^J |D_+(k_1) - i \Gamma_{1}^-k_2|^2 +|D_+(k_3) + i \Gamma_{1}^-k_1|^2\\ 
    & = \sum_{M \in \mathcal{J}} \frac{(C^J_{M+})^2}{8r^2} \big(|k_{1, M}|^2 + |k_{3, M}|^2 \big) + \frac{F}{r^2} \big(|k_{1, M}|^2 + |k_{2, M}|^2\big) \\
    &  + \sum_{M \in \mathcal{J}} C^J_{M-} \frac{\sqrt{F}}{r}\Big(\frac{1-i}{4r} \big( k_{1, M-2} \overline{k_{2, M}} -k_{3, M-2}\overline{k_{1, M}}\big)  + \frac{1+i}{4r}\big(k_{2, M} \overline{k_{1, M-2}} - k_{1, M} \overline{k_{3, M-2}}\big)\Big),
\end{align*}
and
\begin{align*}
    &\int_{\mathbb{S}^3}^J  |D_-(k_1) + i \Gamma_{1}^-k_3|^2 + |D_-(k_2) - i \Gamma_{1}^-k_1|^2\\ 
    & = \sum_{M \in \mathcal{J}} \frac{(C^J_{M-})^2}{8r^2} \big(|k_{1, M}|^2 + |k_{2, M}|^2 \big) + \frac{F}{r^2} \big(|k_{1, M}|^2 + |k_{3, M}|^2\big) \\
    &  + \sum_{M \in \mathcal{J}} C^J_{M+} \frac{\sqrt{F}}{r}\Big(\frac{1+i}{4r} \big( k_{2, M+2} \overline{k_{1, M}} -k_{1, M+2}\overline{k_{3, M}}\big)  + \frac{1-i}{4r}\big(k_{1, M} \overline{k_{2, M+2}} - k_{3, M} \overline{k_{1, M+2}}\big)\Big).
\end{align*}
Including the conjugates (an identical formula with $k_1 \to k_{\bar{1}}, k_2 \to k_{\bar{3}}, k_3 \to k_{\bar{2}}$) and combining the cross terms as we did before completes the derivation of the last identity. This completes the proof. 
\end{proof}

Having computed integrations over Berger spheres in Lemmas \ref{lem:hI-mode-ints} and \ref{lem:hA-mode-ints}, we now define several Hermitian matrices, which will express the invariant and anti-invariant components of $\int_{\mathbb{S}^3} 2R(h, \overline{h}) - |\nabla h|^2$.

\begin{definition}\label{def:prelim-matrices}
Define a $J \times J$ diagonal matrix 
\[
\mathcal{C}^J := \mathrm{diag}(C^J_{M+})_{M \in \mathcal{J} \setminus \{ J\}} =  \mathrm{diag}(C^J_{M-})_{M \in \mathcal{J} \setminus \{ -J\}}, 
\]
where $C^J_{M\pm}$ is given by \eqref{eq:CJM}. Note the definition is well-defined recalling that $C^J_{(M+2)-} = C^J_{M+}$, and $C^J_{(M-2)+} = C^J_{M-}$. 
Define $(J+1) \times (J+1)$ matrices
\[
\mathcal{U} = \Lambda_{11}^\pm \, I_{J +1}, \qquad \mathcal{K} = \frac{1+i}{2} \frac{\sqrt{F}}{r^2} \begin{bmatrix} 0_{1 \times J} & \mathcal{C}^J \\ 0_{1 \times 1} & 0_{J \times 1} \end{bmatrix}, \qquad  \mathcal{K}^\dagger = \frac{1-i}{2} \frac{\sqrt{F}}{r^2} \begin{bmatrix} 0_{J \times1} & 0_{1 \times 1} \\ \mathcal{C}^J  & 0_{1 \times J} \end{bmatrix}.
 \]
 
Equivalently, the entries of $\mathcal{K}$ are given by
\begin{align*}
\mathcal{K}_{MM_0} &= \begin{cases} \frac{1 + i}{2r^2} \sqrt{F} C^J_{M+} & \text{if } \;M_0 = M+ 2, \;\; M \in \mathcal{J}  \setminus\{J\} \\
0 & \text{otherwise}
\end{cases}, \\
\mathcal{K}^\dagger_{MM_0} &= \begin{cases} \frac{1 - i}{2r^2} \sqrt{F} C^J_{M-} & \text{if } \;M_0 = M - 2, \;\; M \in \mathcal{J}  \setminus\{-J\} \\
0 & \text{otherwise}
\end{cases},
\end{align*}
for $M, M_0 \in \mathcal{J}$.
 \end{definition}
For the next definition, we recall that $\tilde{\Delta}^J_M = \frac{M^2 - J(J+2)}{4r^2} - \frac{M^2}{4s^2} \leq 0$. 

\begin{definition}\label{def:QmatJinv}
For each $r \in (1, \infty)$, define a $4(J+1)\times 4(J+1)$ Hermitian matrix $\mathcal{Q}^J = \mathcal{Q}^J(r)$ expressed as a $4 \times 4$ block matrix $\mathcal{Q}^J = [\mathcal{Q}^J_{pq}]_{p,q \in \{0, 1, +, -\}}$ of $(J+1) \times (J+1)$ matrices $\mathcal{Q}^J_{pq}$ as follows. The diagonal blocks are diagonal matrices given by 
\begin{align*}
    \mathcal{Q}^J_{00} &= \mathrm{diag}\Big(  \Lambda_{11}^{++}+\tilde{\Delta}^J_M \Big)_{M \in \mathcal{J}}, \\
    \mathcal{Q}^J_{11} &= \mathrm{diag} \Big(\Lambda_{11}^{--}+\tilde{\Delta}^J_M \Big)_{M \in \mathcal{J}} \\
    \mathcal{Q}^J_{++} &= \mathrm{diag}\Big(\Lambda_{1+}+\tilde{\Delta}^J_M  +M\frac{4+rF'}{4s^2}\Big)_{M \in \mathcal{J}}, \\
     \mathcal{Q}^J_{--} &= \mathrm{diag}\Big(\Lambda_{1+}+\tilde{\Delta}^J_M  -M \frac{4 + rF'}{4s^2} \Big)_{M \in \mathcal{J}}.
\end{align*}
The nonzero off-diagonal blocks are given by 
\begin{align*}
\mathcal{Q}^J_{10} &= \mathcal{U}, & \mathcal{Q}^J_{1+}& = -\mathcal{K},  &  \mathcal{Q}^J_{1-} &= \mathcal{K}^\dagger\\
\mathcal{Q}^J_{01} &= \mathcal{U},  & \mathcal{Q}^J_{+1} &= -\mathcal{K}^\dagger,&  \mathcal{Q}^J_{-1} &= \mathcal{K}.
\end{align*}
The remaining 6 off-diagonal blocks are all zero: 
\begin{align*}
\mathcal{Q}^J_{0+} = \mathcal{Q}^J_{+0} =  \mathcal{Q}^J_{0-} = \mathcal{Q}^J_{-0}  = \mathcal{Q}^J_{+-} = \mathcal{Q}^J_{-+} & = 0.
\end{align*}
\end{definition}

\begin{definition}\label{def:QmatJanti}
    For each $r \in (1, \infty)$, define a $6(J+1) \times 6(J+1)$ Hermitian matrix $\hat{\mathcal{Q}}^J = \hat{\mathcal{Q}}^J(r)$ expressed as a $6 \times 6$ block matrix $\hat{\mathcal{Q}} = [\hat{\mathcal{Q}}^J_{pq}]_{p,q \in \{1, 2, 3,\bar{1}, \bar{2},  \bar{3}\}}$ of $(J+1) \times (J+1)$ matrices  $\hat{\mathcal{Q}}^J_{pq}$ as follows. The diagonal blocks are diagonal matrices given by 
    \begin{align*}
        \hat{\mathcal{Q}}^J_{11} &= \mathrm{diag}\Big(\Lambda_{1-} + \tilde{\Delta}^J_{M} -M\frac{4-4F-rF'}{4s^2} \Big)_{M \in \mathcal{J}}, &  \hat{\mathcal{Q}}^J_{\bar 1 \bar 1} &= \mathrm{diag}\Big(\Lambda_{1-} + \tilde{\Delta}^J_{M} +M\frac{4-4F-rF'}{4s^2}\Big)_{M \in \mathcal{J}}, \\
        \hat{\mathcal{Q}}^J_{22}  &= \mathrm{diag}\Big( \Lambda_{01} + \tilde{\Delta}^J_M + M \frac{4F+ 2rF'}{4s^2} \Big)_{M \in \mathcal{J}}, &  \hat{\mathcal{Q}}^J_{\bar 2 \bar 2}  &= \mathrm{diag}\Big(\Lambda_{01} + \tilde{\Delta}^J_M - M \frac{4F+ 2rF'}{4s^2} \Big)_{M \in \mathcal{J}}, \\
        \hat{\mathcal{Q}}^J_{33}  &= \mathrm{diag}\Big(\Lambda_{23} + \tilde{\Delta}^J_M-M \frac{8-4F}{4s^2} \Big)_{M \in \mathcal{J}}, & \hat{\mathcal{Q}}^J_{\bar 3 \bar 3}  &= \mathrm{diag}\Big(\Lambda_{23} + \tilde{\Delta}^J_M + M \frac{8-4F}{4s^2} \Big)_{M \in \mathcal{J}}.
    \end{align*}  
    The nonzero off-diagonal blocks are given by 
    \begin{align*}
        \hat{\mathcal{Q}}^J_{12} &=- \mathcal{K}, & \hat{\mathcal{Q}}^J_{13}  &= \mathcal{K}^\dagger, &  \hat{\mathcal{Q}}^J_{\bar{1}\bar{2}}  &= \mathcal{K}^\dagger, & 
        \hat{\mathcal{Q}}^J_{\bar{1}\bar{3}} &= - \mathcal{K}, \\
        \hat{\mathcal{Q}}^J_{21} &=-\mathcal{K}^\dagger,& \hat{\mathcal{Q}}^J_{31}  &=  \mathcal{K}, & \hat{\mathcal{Q}}^J_{\bar{2}\bar{1}}  &= \mathcal{K},  & \hat{\mathcal{Q}}^J_{ \bar{3} \bar{1}} &= -\mathcal{K}^\dagger.
    \end{align*}
    The remaining 22 off-diagonal blocks are all zero:
    \begin{align*}
    \hat{\mathcal{Q}}^J_{1\bar{1}} &= \hat{\mathcal{Q}}^J_{\bar{3}2} = \hat{\mathcal{Q}}^J_{12}  = \hat{\mathcal{Q}}^J_{\bar{1}{3}} = \hat{\mathcal{Q}}^J_{1\bar{3}} =\hat{\mathcal{Q}}^J_{\bar{2}{3}}
     = \hat{\mathcal{Q}}^J_{2\bar{2}} = \hat{\mathcal{Q}}^J_{\bar{3}{3}} = \hat{\mathcal{Q}}^J_{23} =\hat{\mathcal{Q}}^J_{\bar{2}\bar{1}}= \hat{\mathcal{Q}}^J_{\bar{2}\bar{3}}= 0, \\
      \hat{\mathcal{Q}}^J_{\bar{1}1}&= \hat{\mathcal{Q}}^J_{2\bar{3}} = \hat{\mathcal{Q}}^J_{21} =\hat{\mathcal{Q}}^J_{3\bar{1}} = \hat{\mathcal{Q}}^J_{\bar{3}1} = \hat{\mathcal{Q}}^J_{3\bar{2}}
    = \hat{\mathcal{Q}}^J_{\bar{2} 2} = \hat{\mathcal{Q}}^J_{3\bar{3}} = \hat{\mathcal{Q}}^J_{32} = \hat{\mathcal{Q}}^J_{\bar{1}\bar{2}}  =  \hat{\mathcal{Q}}^J_{\bar{3}\bar{2}}=0.
    \end{align*}
\end{definition}
In summary, the matrices defined above take the following forms:
\[
\mathcal{Q}^J = \begin{bmatrix} 
\mathcal{Q}^J_{00} & \mathcal{U} & 0 & 0 \\ 
\mathcal{U} & \mathcal{Q}^J_{11} & -\mathcal{K} & \mathcal{K}^\dagger \\
0 & -\mathcal{K}^\dagger & \mathcal{Q}^J_{++} & 0 \\
0 & \mathcal{K} & 0 & \mathcal{Q}^J_{--}\\
\end{bmatrix} , \qquad \hat{\mathcal{Q}}^J = \begin{bmatrix} 
\hat{\mathcal{Q}}^J_{11} & -\mathcal{K}   &  \mathcal{K}^\dagger  & 0 & 0 & 0 \\ 
 -\mathcal{K}^\dagger& \hat{\mathcal{Q}}^J_{22}   & 0 &  0 & 0 & 0\\
\mathcal{K}&  0& \hat{\mathcal{Q}}^J_{33}  & 0 & 0 & 0\\
0 & 0 & 0 & \hat{\mathcal{Q}}^J_{\bar{1}\bar{1}}   & \mathcal{K}^\dagger  & -\mathcal{K} \\
0 & 0 & 0 &\mathcal{K}  &  \hat{\mathcal{Q}}^J_{\bar{2}\bar{2}}  & 0\\
0 & 0 & 0 & -\mathcal{K}^\dagger& 0 & \hat{\mathcal{Q}}^J_{\bar{3}\bar{3}}
\end{bmatrix}.
\]

The main result of this section is:

\begin{proposition}\label{prop:matrix-stability-identities}
Suppose $h = h_I + h_A \in \mathcal{T}^J_{M'}$ is a complex-valued deformation. For any $r \in (1, \infty)$, 
    \[
    \frac{J+1}{2\pi^2}\int_{\mathbb{S}^3} 8R(h_I, \overline{h_I}) - 4|\nabla h_I|^2  = \mathcal{Q}^J \eta \cdot \overline{\eta} -  \frac{F}{4}|\eta'|^2
    \]
    and 
    \[
    \frac{J+1}{2\pi^2}\int_{\mathbb{S}^3} 8R(h_A, \overline{h_A}) - 4|\nabla h_A|^2  = \hat{\mathcal{Q}}^J \, \kappa \cdot \overline{\kappa} -  \frac{F}{4}|\kappa'|^2
    \]
    where `\,$\cdot$' is the usual inner product and 
    \[
    \eta_p = \eta_p (r):= \big( h_{p, M}(r) \big)_{M \in \mathcal{J}} \in \mathbb{C}^{J+1}, \qquad \kappa_j = \kappa_j (r):= \big( k_{j, M}(r) \big)_{M \in \mathcal{J}} \in \mathbb{C}^{J+1}, 
    \]
    for $p \in \{0, 1, + , -\}$ and $j \in \{1, 2, 3, \bar{1}, \bar{2}, \bar{3}\}$, and
    \[
    \eta := (\eta_0, \eta_1, \eta_+, \eta_-) \in \mathbb{C}^{4(J+1)}\qquad \kappa:= (\kappa_1, \kappa_2, \kappa_3, \kappa_{\bar{1}}, \kappa_{\bar{2}}, \kappa_{\bar{3}}) \in \mathbb{C}^{6(J+1)}. 
    \]
\end{proposition}

Before diving into the derivation, we offer several remarks.  

\begin{remark}
Expanded, the above identities are
\[
\frac{J+1}{2\pi^2} \int_{\mathbb{S}^3} 8R(h_I, \overline{h_I}) - 4|\nabla h_I|^2  =  \begin{bmatrix} 
\mathcal{Q}^J_{00} & \mathcal{U} & 0 & 0 \\ 
\mathcal{U} & \mathcal{Q}^J_{11} & -\mathcal{K} & \mathcal{K}^\dagger \\
0 & -\mathcal{K}^\dagger & \mathcal{Q}^J_{++} & 0 \\
0 & \mathcal{K} & 0 & \mathcal{Q}^J_{--}\\
\end{bmatrix} \begin{bmatrix} \eta_0 \\ \eta_1 \\ \eta_+ \\ \eta_- \end{bmatrix} \cdot \begin{bmatrix} \overline{\eta_0} \\ \overline{\eta_1} \\  \overline{\eta_+} \\ \overline{\eta_-}\end{bmatrix} 
 - \frac{F}{4} \begin{bmatrix} \eta_0' \\ \eta_1' \\ \eta_+' \\ \eta_-' \end{bmatrix}\cdot \begin{bmatrix} \overline{\eta_0'} \\ \overline{\eta_1'} \\  \overline{\eta_+'} \\ \overline{\eta_-'}\end{bmatrix},
 \]
 and 
 \[
\frac{J+1}{2\pi^2} \int_{\mathbb{S}^3} 8R(h_A, \overline{h_A}) - 4|\nabla h_A|^2  =  \begin{bmatrix} 
\hat{\mathcal{Q}}^J_{11} & -\mathcal{K}   &  \mathcal{K}^\dagger  & 0 & 0 & 0 \\ 
 -\mathcal{K}^\dagger& \hat{\mathcal{Q}}^J_{22}   & 0 &  0 & 0 & 0\\
\mathcal{K}&  0& \hat{\mathcal{Q}}^J_{33}  & 0 & 0 & 0\\
0 & 0 & 0 & \hat{\mathcal{Q}}^J_{\bar{1}\bar{1}}   & \mathcal{K}^\dagger  & -\mathcal{K} \\
0 & 0 & 0 &\mathcal{K}  &  \hat{\mathcal{Q}}^J_{\bar{2}\bar{2}}  & 0\\
0 & 0 & 0 & -\mathcal{K}^\dagger& 0 & \hat{\mathcal{Q}}^J_{\bar{3}\bar{3}}
\end{bmatrix} \begin{bmatrix} \kappa_1 \\ \kappa_2 \\ \kappa_3 \\\kappa_{\bar{1}} \\ \kappa_{\bar{2}} \\ \kappa_{\bar{3}} \end{bmatrix} \cdot \begin{bmatrix} \overline{\kappa_1} \\ \overline{\kappa_2} \\ \overline{\kappa_3} \\\overline{\kappa_{\bar{1}}} \\\overline{\kappa_{\bar{2}}} \\ \overline{\kappa_{\bar{3}}} \end{bmatrix}
 - \frac{F}{4} \begin{bmatrix} \kappa_1' \\ \kappa_2' \\ \kappa_3' \\\kappa_{\bar{1}}' \\ \kappa_{\bar{2}}' \\ \kappa_{\bar{3}}'\end{bmatrix} \cdot \begin{bmatrix} \overline{\kappa_1'} \\ \overline{\kappa_2'} \\ \overline{\kappa_3'} \\\overline{\kappa_{\bar{1}}'} \\\overline{\kappa_{\bar{2}}'} \\ \overline{\kappa_{\bar{3}}'} \end{bmatrix}. 
 \]
\end{remark}

\begin{remark}
    Evidently $\hat{\mathcal{Q}}^J$ is made up of two smaller blocks. Let us define matrices $\mathcal{P}^J(r)$ and $\tilde{\mathcal{P}}^J(r)$ by
    \begin{equation}\label{eq:def-of-P-matrices}
    \mathcal{P}^J := \begin{bmatrix} \hat{\mathcal{Q}}^J_{11} & -\mathcal{K} & \mathcal{K}^\dagger \\ -\mathcal{K}^\dagger & \hat{\mathcal{Q}}^J_{22}  & 0 \\ \mathcal{K} & 0 & \hat{\mathcal{Q}}^J_{33}  \end{bmatrix}, \qquad \tilde{\mathcal{P}}^J:= \begin{bmatrix} \hat{\mathcal{Q}}^J_{\bar{1}\bar{1}} & \mathcal{K}^\dagger & -\mathcal{K} \\ \mathcal{K} & \hat{\mathcal{Q}}^J_{\bar{2}\bar{2}}  & 0 \\ -\mathcal{K}^\dagger & 0 & \hat{\mathcal{Q}}^J_{\bar{3}\bar{3}}  \end{bmatrix}, 
    \end{equation}
    so that 
    \begin{equation}
    \hat{\mathcal{Q}}^J(r) = \begin{bmatrix} \mathcal{P}^J(r) & 0 \\ 0 & \tilde{\mathcal{P}}^J(r) \end{bmatrix}.
    \end{equation}
    Note these blocks are similar:  $\mathcal{P}^J = (\mathcal{S}^J)^{-1} \tilde{\mathcal{P}}^J \mathcal{S}^J$. If $v = (v_1, v_2, v_3) \in (\mathbb{C}^{J+1})^3$ and $v_p = (v_{p, M})_{M \in \mathcal{J}}$, then the change of basis matrix acts by $\mathcal{S}^J v = (-\tilde{v}_1, \tilde{v}_2,\tilde{v}_3)$, where $\tilde{v}_p = (v_{p, -M})_{M \in \mathcal{J}}$ (which reverses the order of the entries of $v_p$). As a consequence, the eigenvalues of $\mathcal{P}^J$ and $\tilde{\mathcal{P}}^J$ are the same and the eigenvectors are related by the transformation $\mathcal{S}^J$. 
\end{remark}

\begin{remark}\label{rem:full-matrix}
    Suppose $h \in \mathcal{T}^J_{M'}$ with $h_{p, M}, k_{j, M}$ and $\eta, \kappa$ defined as above. Let us define 
    \begin{equation}
    \mathcal{R}^J(r) := \begin{bmatrix} \mathcal{Q}^J(r) & 0 \\
    0 & \hat{\mathcal{Q}}^J(r)\end{bmatrix} \in  M_{10(J+1)}(\mathbb{C}), \qquad \zeta(r) := \big(\eta(r), \kappa(r)\big) \in \mathbb{C}^{10(J+1)}.
    \end{equation}
    Then formally, after integration by parts, the proposition gives 
    \begin{align*}
    \int_M \langle L_f h, \overline{h} \rangle \, e^{-f} d\mu_g & =16 \int_1^\infty   \left(\int_{\mathbb{S}^3} 2R(h, \overline{h}) - |\nabla h|^2\right) \, r^3 e^{-f} dr \\
    & =  \frac{8\pi^2}{J+1} \int_1^\infty \big(\mathcal{R}^J \zeta \cdot \overline{\zeta} - \frac{F}{4} |\zeta'|^2\big) \, r^3 e^{-f} dr  \\
    & = \frac{8\pi^2}{J+1}  \int_1^\infty \big(L_f^J \zeta \cdot \overline{\zeta} \big)\; r^3 e^{-f} \, dr,
    \end{align*}
    where 
    \begin{equation}
    L_f^J := \Delta_f + \mathcal{R}^J
    \end{equation}
    acts on functions $\zeta : [1, \infty) \to \mathbb{C}^{10(J+1)}$. Note that we have used that on radial functions, $(\Delta_f a)  r^3 e^{-f} = (\frac{F}{4}a'r^3 e^{-f})'  $ (see Lemma \ref{lem:lap-radial}). In other words, the operator $L_f$ acting on $\mathcal{T}^J_{M'}$ is given by $L_f^J$ acting on vector-valued radial functions $\zeta$. Since 
    \[
    \int_M |h|^2 \, e^{-f}  d\mu_g = \frac{8\pi^2}{J+1} \int_1^\infty |\zeta|^2 \, r^3 e^{-f} ,   
    \]
    we have 
    \begin{equation}\label{eq:eigenvalues-LfJ}
    \frac{\int_M \langle L_f h, \overline{h} \rangle \, e^{-f} d\mu_g }{\int_M |h|^2 \,e^{-f} d\mu_g} = \frac{\int_1^\infty \big(L_f^J \zeta \cdot \overline{\zeta} \big)\; r^3 e^{-f} \, dr}{\int_1^\infty |
    \zeta|^2  \, r^3 e^{-f} \, dr}. 
    \end{equation}
    Eigenvalues of $L_f$ restricted to the space $\mathcal{T}^J_{M'}$ must be in correspondence with eigenvalues of $L^J_f$ acting on functions $\zeta$, since both are given by the respective Rayleigh quotients. In particular, eigenvalues of $L_f$ (and corresponding eigentensors in $L^2(M, e^{-f})$) decompose into eigenvalues (and corresponding eigenvectors in $L^2([1, \infty), r^3 e^{-f})$) of each of the operators $L_f^J$ for $J \in \mathbb{N}$. A similar property holds for $\Delta_f$ acting on functions via the same separation of variables. Heuristically, as can be seen from figures in the following sections, the largest eigenvalue of the operators $L_f^J$ (with convention $L_f^J \zeta = \lambda \zeta$) is decreasing as $J \to \infty$ and negative for $J \geq 3$. Note that whenever we have $L_f h = \lambda h $ for some $\lambda \geq 0$, \eqref{eq:eigenvalues-LfJ} implies 
    \begin{equation}\label{eq:eigenvalues-LfJ-2}
    \frac{\int_1^\infty (\mathcal{R}^J \zeta \cdot \overline{\zeta}) \; r^3 e^{-f} \, dr}{\int_1^\infty |\zeta|^2 \; r^3 e^{-f} \,dr} = \lambda +  \frac{\int_1^\infty \frac{F}{4} |\zeta'|^2 \; r^3 e^{-f} \, dr}{\int_1^\infty |\zeta|^2 \; r^3 e^{-f} \,dr} > 0.
    \end{equation}
    In particular, the existence of nonnegative eigenvalues of $L_f$ must cause the matrix $\mathcal{R}^J(r)$ to have positive eigenvalues for at least some values of $r$ (so that the integral is positive). We will see this in the sections that follow. 
\end{remark}

\begin{proof}[Proof of Proposition \ref{prop:matrix-stability-identities}]
For the invariant case, we collect the coefficients of each $|h_{p, M}|^2$ from each identity in Lemma \ref{lem:hI-mode-ints} (including a minus sign for the derivative identities) and perform some simplification. Let $(\mathcal{Q}^J_{pp})_{M}$ denote the coefficient of $|h_{p, M}|^2$. Then we readily find
\begin{align*}
    (\mathcal{Q}^J_{00})_{M} &= \frac{M^2 - J(J+2) + 2\sqrt{2}c_0}{4r^2}- \frac{M^2}{4s^2}, \\
    (\mathcal{Q}^J_{11})_{M} &= \frac{M^2-J(J+2)+2\sqrt{2} c_0- 8F+4rF' }{4r^2} - \frac{M^2}{4s^2}\\
    (\mathcal{Q}^J_{++})_{M} &= \frac{M^2-J(J+2)-4F -2rF'}{4r^2}  -\frac{(M-2-\frac{1}{2}rF')^2}{4s^2}, \\
    (\mathcal{Q}^J_{--})_{M} &= \frac{M^2-J(J+2)-4F-2rF'}{4r^2} - \frac{(M+2+\frac{1}{2}rF')^2}{4s^2}.
\end{align*}
As we saw in the proofs of Lemma \ref{lem:ibp-for-H1} and Corollary \ref{cor:stability-reduction}, when $J = 0$, the terms which do not involve $J, M$ must become the $\Lambda$-functions previously found. Moreover, recall from \eqref{eq:berger-eigenvalues} that 
\[
\tilde{\Delta}^J_{M} = \frac{M^2 - J(J+2)}{4r^2} -\frac{M^2}{4s^2}.
\]
These observations and some simplification now yield the formulas for $\mathcal{Q}^J_{pp}$ in Definition \ref{def:QmatJinv}. For the off-diagonal terms, we observe that 
\begin{align*}
     \sum_{M \in \mathcal{J}} \frac{1-i}{2}  \frac{\sqrt{F}}{r^2} C^J_{M+} \big(h_{-, M}  \overline{h_{1, M+2}} - h_{1, M}  \overline{h_{+, M+2}}\big)& = (\mathcal{K}^\dagger \eta_{-}) \cdot \overline{\eta_1}-(\mathcal{K}^\dagger \eta_1) \cdot \overline{\eta_+}, \\
    \sum_{M \in \mathcal{J}}  \frac{1+i}{2}  \frac{\sqrt{F}}{r^2}  C^J_{M-} \big(h_{1, M}  \overline{h_{-, M-2}}- h_{+, M}\overline{h_{1,M -2}}\big)& = (\mathcal{K}\eta_1) \cdot \overline{\eta_-}-(\mathcal{K} \eta_+) \cdot \overline{\eta_1}. 
\end{align*}
Similarly, noting that $- \frac{rF'}{2\sqrt{2}}= \Lambda_{01}^\pm$, we have 
\[
 -\sum_{M \in \mathcal{J}} \frac{rF'}{2\sqrt{2}} \big( h_{0, M} \overline{h_{1,M}} + \overline{h_{0,M}}h_{1,M}\big) = (\mathcal{U} \eta_0)\cdot \overline{\eta_1} + (\mathcal{U}\eta_1)\cdot \overline{\eta_0}. 
\]
For the anti-invariant case, we proceed similarly. As before, we read the diagonal terms by collecting the coefficients of each $|k_{p, M}|^2$ in each identity in Lemma \ref{lem:hA-mode-ints}. Denoting this coefficient by $(\hat{Q}^J_{jj})_M$, we obtain 
\begin{align*}
    (\hat{\mathcal{Q}}^J_{11})_{M} &= \frac{M^2 - J(J+2) -8F+4rF'}{4r^2} -\frac{(M+2-2F-\frac{1}{2}rF')^2}{4s^2} , \\
    (\hat{\mathcal{Q}}^J_{22})_{M}   &= \frac{M^2 - J(J+2)+4\sqrt{2}-12F-4rF'}{4r^2}- \frac{(M  - rF'-2F)^2}{4s^2}, \\
    (\hat{\mathcal{Q}}^J_{33})_{M}   &= \frac{M^2 - J(J+2)-8+4F}{4r^2} -\frac{( M+4 - 2F)^2}{4s^2}, \\
    (\hat{\mathcal{Q}}^J_{\bar 1 \bar 1})_{M}  &= \frac{M^2 - J(J+2) -8F+4rF'}{4r^2} -\frac{(M-2+2F+\frac{1}{2}rF')^2}{4s^2} , \\
    (\hat{\mathcal{Q}}^J_{\bar 2 \bar 2})_{M}   &=\frac{M^2 - J(J+2)+4\sqrt{2}-12F-4rF'}{4r^2} - \frac{(M + rF'+2F)^2}{4s^2}, \\
    (\hat{\mathcal{Q}}^J_{\bar 3 \bar 3})_{M}   &= \frac{M^2 - J(J+2)-8+4F}{4r^2} -\frac{(M-4 +2F)^2}{4s^2} .
\end{align*}
Also as before, the collection of terms which do not involve $J, M$ in the first three lines must be $\Lambda_{1-}, \Lambda_{01}$ and $\Lambda_{23}$, respectively. Similarly for the latter three.  Expanding the square and collecting terms gives the asserted formulas along the diagonals. As in the invariant case, the off diagonal terms can be expressed in matrix multiplication as
\begin{align*}
& (\mathcal{K}^\dagger \kappa_{3}) \cdot \overline{\kappa_1}-(\mathcal{K}^\dagger \kappa_1) \cdot \overline{\kappa_2} + (\mathcal{K}^\dagger \kappa_{\bar{2}}) \cdot \overline{\kappa_{\bar{1}}}-(\mathcal{K}^\dagger \kappa_{\bar{1}}) \cdot \overline{\kappa_{\bar{3}}}\\
 & \qquad + (\mathcal{K}\kappa_1) \cdot \overline{\kappa_3}-(\mathcal{K} \kappa_2) \cdot \overline{\kappa_1} +  (\mathcal{K}\kappa_{\bar{1}}) \cdot \overline{\kappa_{\bar{2}}}-(\mathcal{K} \kappa_{\bar{3}}) \cdot \overline{\kappa_{\bar{1}}}.
\end{align*}
This completes the proof. 
\end{proof}

\subsection{Auxiliary estimates for stability for $J \geq 3$}

With the formulas of Proposition \ref{prop:matrix-stability-identities} in hand, we turn to proving linear stability of the FIK soliton under deformations $h \in \mathcal{T}^J_{M'}$ for $J \geq 3$. To that end, in this section, we establish several simple estimates to bound the stability operator. 

Recall that $\int^J_{\mathbb{S}^3} = \frac{J+1}{2\pi^2} \int_{\mathbb{S}^3}$. 

\begin{lemma}\label{lem:curv-int-est}
Suppose $h = h_I + h_A \in \mathcal{T}^J_{M'}$ for some $J \geq 3$ and $M' \in \mathcal{J}$. Then 
\begin{equation}
    4r^2 \int_{\mathbb{S}^3}^J 2 R(h_I, \overline{h}_I) \leq \sum_{M \in \mathcal{J}} 3.5|h_{0, M}|^2 + (0.45+ rF')|h_{1, M}|^2,
\end{equation}
and 
\begin{equation}
    4r^2 \int_{\mathbb{S}^3}^J 2 R(h_A, \overline{h}_A) \leq \sum_{M \in \mathcal{J}} rF'\big(|k_{1, M}|^2 + |k_{\bar{1}, M}|^2\big)  - 0.5\big(|k_{3, M}|^2 + |k_{\bar{3}, M}|^2\big).
\end{equation}
\end{lemma}
\begin{proof}
The first follows from \eqref{eq:hI-curve-int} in Lemma \ref{lem:hI-mode-ints}. We first estimate 
\[
 - \frac{r^3F'}{2\sqrt{2}}\big(h_{0, M} \overline{h_{1, M}} + h_{1, M} \overline{h_{0, M}}\Big) \leq \sqrt{2}|h_{0, M}||h_{1, M}|  \leq   3.5|h_{0, M}|^2 + \frac{1}{7} |h_{1, M}|^2, 
\]
where we have used \eqref{eq:Fder2} to deduce that $r^3F' = 2c_0(1 + \sqrt{2} r^{-2}) \leq 2$ and Young's inequality $ab \leq \frac{t}{2} a^2 + \frac{1}{2t} b^2$ with $t =\frac{7}{\sqrt{2}}$. Then, using $\frac{c_0}{\sqrt{2}} < 0.3$, $\frac{1}{7} < 0.15$, we have
\begin{align*}
&\frac{c_0\sqrt{2}}{2} |h_{0, M}|^2 - \frac{r^3F'}{2\sqrt{2}}\big(h_{0, M} \overline{h_{1, M}} + h_{1, M} \overline{h_{0, M}}\Big)+\Big(\frac{c_0\sqrt{2}}{2} + rF'\Big)|h_{1, M}|^2 \\
& \leq 3.5 |h_{0, M}|^2 + (0.45 + rF')|h_{1, M}|^2.
\end{align*}
The second follows similarly from \eqref{eq:hA-curve-int} from Lemma \ref{lem:hA-mode-ints}, using that $rF' + 2F - \sqrt{2} \geq 0$ and $2 -2F \geq 2 -\sqrt{2} > 0.5$. 
\end{proof}

\begin{lemma}\label{lem:grad-terms-ests}
    Suppose $L \in \mathbb{Z}$ and $r \in (1, \infty)$. Then we have the estimates: 
    \begin{enumerate}
        \item
        \[
        \frac{1}{4F} \geq \max \big\{0.35, 0.5 rF'\}
        \]
        \item 
        \[
        (L - 2-\frac{1}{2}rF')^2 \geq\begin{cases} (L-3)^2 & L \geq 3 \\
       (L-2)^2 & L \leq 2
        \end{cases}.
        \]
        \item
        \[
        (L + 2 - \frac{1}{2} rF' - 2F)^2 \geq \begin{cases} (L+0.5)^2 & L \geq 0 \\
    (L+1)^2 & L  \leq -1
        \end{cases}.
        \]
        \item
        \[
        (L -rF' - 2F)^2 \geq\begin{cases} (L-2)^2 & L \geq 2 \\
       (L-1.4)^2& L \leq 1
        \end{cases}.
        \]
        \item
        \[
        (L +  4 - 2F)^2  \geq\begin{cases} 
        (L+2.5)^2 & L \geq- 2\\
         (L+4)^2& L \leq- 4
        \end{cases}.
        \]
    \end{enumerate}
\end{lemma}
\begin{proof}
The first part of the first estimate follows simply because $\frac{1}{4F} \geq \frac{1}{2\sqrt{2}} > 0.35$. The latter part of the first estimate follows from the fact that
\[
\frac{2\sqrt{2}c_0(r^2+c_0)(r^2+ \sqrt{2})}{r^4} \leq 4\leq \frac{r^4}{r^2-1},
\] 
so that 
\[
r(F^2)' = \frac{2\sqrt{2}c_0(r^2 -1)(r^2 +c_0)(r^2+\sqrt{2})}{r^8} \leq 1.
\]
The second inequality follows from the fact that $\frac{1}{2} rF' \in (0, 1]$, so that $L - 2 -\frac{1}{2} rF' \in [L-3, L-2)$. For the third, we have $\frac{1}{2}rF' + 2F \in (1, \sqrt{2}] \subset (1, 1.5]$ so that $L +2 -\frac{1}{2} rF' -2F \in [L+0.5, L+1)$. For the fourth inequality, we use that $rF' + 2F \in (\sqrt{2}, 2] \subset (1.4, 2]$, so that $L - rF' -2F \in [L-2, L-1.4)$. Finally, for the fifth inequality, we have $4 - 2F \in (4-\sqrt{2}, 4] \subset (2.5, 4]$ so that $L +4 - 2F \in (L+2.5, L+4]$.
\end{proof}

\begin{lemma}\label{lem:cross-section-ests}
    Suppose $h = h_I + h_A \in \mathcal{T}^J_{M'}$ for $J \geq 3$ and $M' \in \mathcal{J}$. Then
\begin{equation}
    4r^2 \int^J_{\mathbb{S}^3} 2R(h_I, \overline{h_I}) -  |\nabla_{e_1} h_I|^2 - |\nabla_{D_+} h_I|^2 - |\nabla_{D_-} h_I|^2  \leq  \mathcal{E}_I^J - \mathcal{G}_I^J,
\end{equation}
and 
\begin{equation}
    4r^2 \int^J_{\mathbb{S}^3} 2R(h_A, \overline{h_A}) -  |\nabla_{e_1} h_A|^2 - |\nabla_{D_+} h_A|^2 - |\nabla_{D_-} h_A|^2  \leq  \mathcal{E}_{A}^J - \mathcal{G}_{A}^J + \mathcal{E}_{\bar A}^J - \mathcal{G}_{\bar{A}}^J,
\end{equation}
where, setting $N^J_M := \frac{J(J+2)-M^2}{4}$,
 \begin{align*}
\mathcal{G}_I^J& :=  \sum_{M \in \mathcal{J}} \Big( N^J_M + \frac{M^2}{4F} -3.5\Big) |h_{0, M}|^2 + \Big( N^J_M +2F + \frac{M^2}{4F}-0.45-rF'\Big) |h_{1, M}|^2 \\
&  + \sum_{M \in \mathcal{J}} \Big( N^J_M +F  +\frac{(M - 2-\frac{1}{2} rF')^2}{4F} \Big) |h_{+, M}|^2+  \Big( N^J_M+F +\frac{(-M -2 -\frac{1}{2} rF')^2}{4F}  \Big) |h_{-, M}|^2,\\
\mathcal{E}_I^J &:=  \sqrt{F}\sum_{M \in \mathcal{J}} \sqrt{2}C^J_{M+} \Big(|h_{-, M}|  |h_{1, M+2}|  + |h_{1, M}||h_{+,M+2}| \Big),
\end{align*}
and
\begin{align*}
\mathcal{G}_A^J & := \sum_{M \in \mathcal{J}} \Big(N^J_M + 2F + \frac{(M+2 - \frac{1}{2} rF'-2F)^2}{4F} - rF'\Big)|k_{1, M}|^2 \\
& + \sum_{M \in \mathcal{J}}  \Big(N^J_M + F + \frac{(M - rF' - 2F)^2}{4F} \Big)|k_{2, M}|^2+ \Big(N^J_M + F + \frac{(M +4 -2F)^2}{4F} \Big)|k_{3, M}|^2,\\ 
\mathcal{E}_A & := \sqrt{F}\sum_{M \in \mathcal{J}} \sqrt{2}C^J_{M+} \Big(|k_{3, M}||k_{1, M+2}| + |k_{1, M}||k_{2, M+2}|\Big),
\end{align*}
and 
\begin{align*}
\mathcal{G}_{\bar A}^J & := \sum_{M \in \mathcal{J}} \Big(N^J_M + 2F + \frac{(-M+2 - \frac{1}{2} rF'-2F)^2}{4F} - rF'\Big)|k_{\bar 1, M}|^2 \\
& + \sum_{M \in \mathcal{J}}  \Big(N^J_M  + F + \frac{(-M - rF' - 2F)^2}{4F}\Big)|k_{\bar 2, M}|^2+ \Big(N^J_M+ F + \frac{(-M +4 -2F)^2}{4F} \Big)|k_{\bar 3, M}|^2,\\ 
\mathcal{E}_{\bar A}^J & := \sqrt{F}\sum_{M \in \mathcal{J}} \sqrt{2}C^J_{M+} \Big(|k_{\bar 2, M}||k_{\bar 1, M+2}| + |k_{\bar 1, M}||k_{\bar 3, M+2}|\Big).
\end{align*}
\end{lemma}

\begin{proof}
These estimates follow directly from the identities in Lemmas \ref{lem:hI-mode-ints}, \ref{lem:hA-mode-ints}, from the curvature estimates obtained in Lemma \ref{lem:curv-int-est}, and from estimates for the cross terms that appear in equations \eqref{eq:hI-Dpm-int} and \eqref{eq:hA-Dpm-int}. The latter estimates are obtained, for instance, as follows:
\begin{align*}
& \sum_{M \in \mathcal{J}}  \frac{1-i}{2} \sqrt{F} C^J_{M+} \Big(h_{-, M}  \overline{h_{1, M+2}}  - h_{1, M}\overline{h_{+,M+2}} \Big) + \frac{1+i}{2}  \sqrt{F} C^J_{M-} \Big(h_{1, M}  \overline{h_{-, M-2}} - h_{+, M}\overline{h_{1,M -2}}  \Big) \\
& \leq \frac{\sqrt{F}}{\sqrt{2}} \sum_{M \in \mathcal{J}} C^J_{M+} \Big(|h_{-, M}|  |h_{1, M+2}|  + |h_{1, M}||h_{+,M+2}| \Big)  + C^J_{M-}\Big(|h_{1, M}|  |h_{-, M-2}| + |h_{+, M}||h_{1,M -2}|  \Big)\\
& =  \sqrt{F}\sum_{M \in \mathcal{J}} \sqrt{2}C^J_{M+} \Big(|h_{-, M}|  |h_{1, M+2}|  + |h_{1, M}||h_{+,M+2}| \Big),
\end{align*}
where the last equality follows by reindexing $M \to M+2$ in the second term and using $C^J_{(M+2)-} = C^J_{M+}$. 
\end{proof}

\subsection{Stability for $J \geq 3$} \label{sec: Wigner function and stab for Jgeq 3}

We now show, through coarse estimates, that all deformations $h\in \mathcal{T}^J_{M'}$ are stable for $J\geq 3$, confirming the intuition that higher-frequency deformations correspond to more stable perturbations. This reduces the analysis of linear stability to a small space of deformations with $J \in \{1, 2\}$. In Section \ref{sec:further block decomposition}, we will see that linear stability (for any $J$) can be deduced from eigenvalue estimates for matrices whose sizes never exceeds $4\times 4$. The largest eigenvalues among these matrix blocks corresponding to $J \leq 5$ are plotted in Figures \ref{fig:max-eigenvalues-Q} and \ref{fig:max-eigenvalues-P}.

\begin{theorem}\label{thm:stability for J = 3}
    Let $h  \in \mathcal{T}^3_{M'}$ and $M' \in \mathcal{J} = \{-3, -1, 1, 3\}$. Then 
    \[
    \int_M\big( 2R(h, \overline{h}) - |\nabla h|^2 \big) e^{-f} d\mu_g \leq 0. 
    \]
    Consequently, $\delta^2 \nu_g(h + \overline{h}) \leq 0$. 
\end{theorem}

\begin{proof}
By Corollary \ref{cor:invariance-splitting}, it suffices to write $h = h_I + h_A$ and estimate the invariant and anti-invariant parts separately. For $J = 3$, we use Lemma \ref{lem:grad-terms-ests} and the definitions of $N^J_M, C^J_{M+}$ to obtain:   
\begin{align*}
\underline{\text{Expression}}   && \underline{M = -3} && \underline{M =  -1} && \underline{M = 1} && \underline{M = 3}&& \\
    \sqrt{2}C^3_{M+} =&&    2\sqrt{6}  &&    4\sqrt{2}   && 2\sqrt{6}   &&   0 &&\\  
    N^3_M =&&  1.5 && 3.5 && 3.5 && 1.5&&\\
    \frac{M^2}{4F} -3.5\geq && -0.35 && -3.15 && -3.15 && -0.35&&\\
    \frac{M^2}{4F} - 0.45 -rF'\geq && 0.16 && -1.45 && -1.45 && 0.16&&\\
    \frac{(M-2 - \frac{1}{2}rF')^2}{4F} \geq && 8.75 && 3.15 && 0.35 && 0 && \\
    \frac{(-M-2 - \frac{1}{2}rF')^2}{4F} \geq  && 0 && 0.35 && 3.15 && 8.75 &&
\end{align*}
In the third, fifth, and sixth rows above we have used $(4F)^{-1} \geq 0.35$.  We have used (1) from Lemma \ref{lem:grad-terms-ests} to obtain the fourth row. For $M = \pm 3$, we have used:
\[
\frac{M^2}{4F}- 0.45- rF' = \frac{7}{4F}- 0.45 + \frac{2}{4F}- rF' > 0.16.
\]
For $M = \pm 1$, we similarly used 
\[
\frac{M^2}{4F}- 0.45- rF' = - 0.45- 0.5rF' \geq -1.45,
\]
noting that $rF' \leq 2$. The data from our table implies that
\begin{align*}
    \mathcal{G}_I^3 
     & \geq 1.15 \big(|h_{0, -3} |^2 + |h_{0, 3}|^2\big) + 0.35 \big(|h_{0, -1}|^2 + |h_{0, 1}|^2\big) \\
    & \quad + (1.66 + 2F)\big(|h_{1, -3} |^2 + |h_{1, 3}|^2\big) + (2.05 + 2F) \big(|h_{1, -1}|^2 + |h_{1, 1}|^2\big)  \\
    & \quad +(10.25 + F)\big(|h_{+, -3} |^2 + |h_{-, 3}|^2\big) + (6.65 + F) \big(|h_{+, -1}|^2 + |h_{-, 1}|^2\big) \\
    & \quad + (3.85 + F)\big(|h_{+, 1} |^2 + |h_{-, -1}|^2\big) + (1.5 + F ) \big(|h_{+, 3}|^2 + |h_{-, -3}|^2\big), 
\end{align*}
and, estimating $2\sqrt{6} < 4.9$, $4\sqrt{2} < 5.7$, 
\begin{align*}
    \mathcal{E}^3_I &\leq 4.9\sqrt{F} \Big(|h_{-, -3}|  |h_{1, -1}|  + |h_{1, -3}||h_{+,-1}| + |h_{-, 1}|  |h_{1, 3}|  + |h_{1, 1}||h_{+,3}|\Big)\\
    & \quad + 5.7 \sqrt{F}\Big(|h_{-, -1}|  |h_{1, 1}|  + |h_{1, -1}||h_{+,1}|\Big). 
\end{align*}
Using Young's inequality, we have
\begin{align*}
   &  4.9 \sqrt{F} \Big(|h_{-,-3}| |h_{1, -1}| + |h_{+, 3}| |h_{1, 1}| \Big)  + 5.7 \sqrt{F} \Big(|h_{+,1}| |h_{1, -1}|  + |h_{-, -1}||h_{1, 1}|\Big)\\
    & \leq 2.45\sqrt{F} \big(|h_{+, 3}|^2+|h_{-,-3}|^2\big)+ \frac{2.85^2}{1.6} \sqrt{F} \big(|h_{+,1}|^2 + |h_{-,-1}|^2 \big) \\
    & \qquad + (1.6 + 2.45)\sqrt{F}\big(|h_{1, -1}|^2 + |h_{1, 1}|^2 \big) \\
    & < 2.45\sqrt{F} \big(|h_{+, 3}|^2+|h_{-,-3}|^2\big)+  5.08\sqrt{F} \big(|h_{+,1}|^2 + |h_{-,-1}|^2 \big)+  4.05\sqrt{F}\big(|h_{1, -1}|^2 + |h_{1, 1}|^2 \big).
\end{align*}
Similarly, 
\[
4.9\sqrt{F}\Big( |h_{+, -1}||h_{1, -3}|  + |h_{-, 1}| |h_{1, 3}|\Big) \leq 2.45\sqrt{F} \Big(|h_{+, -1}|^2 + |h_{-, 1}|^2 \Big) + 2.45\sqrt{F} \Big( |h_{1, -3}|^2 + |h_{1, 3}|^2 \Big).
\]
Using that $\sqrt{F} \in [0, 2^{-1/4}) \subset [0, 0.85)$ for $r \in [1, \infty)$ and comparing the lower bound for $\mathcal{G}^3_I$ to the upper bound for $\mathcal{E}^3_I$, one can readily verify -- by checking the sign of the quadratic on $[0, 0.85)$ -- that 
\begin{align*}
    2.45 \sqrt{F} &< 1.5 + F, & 5.08 \sqrt{F} &< 3.85 + F,  & 4.05\sqrt{F} &< 2.05 + 2F,\\
    2.45 \sqrt{F} &< 6.65 + F, & 2.45\sqrt{F} &< 1.66 + 2F. 
\end{align*}
We deduce that $\mathcal{E}^3_I - \mathcal{G}^3_I \leq 0$ for $r \in (1, \infty)$.

We tackle the anti-invariant setting in a similar manner. Using Lemma \ref{lem:grad-terms-ests}, we start from: 
\begin{align*}
\underline{\text{Expression}}   &&\underline{M = -3} && \underline{M =  -1} && \underline{M = 1} && \underline{M = 3}&& \\
    \sqrt{2}C^3_{M+} =&&    2\sqrt{6}  &&    4\sqrt{2}   && 2\sqrt{6}   &&   0 &&\\  
    N^3_M =&&  1.5 && 3.5 && 3.5 && 1.5 &&\\
    \frac{(M+ 2- \frac{1}{2}rF'-2F)^2}{4F} -rF'\geq && 0.7 && -2r^{-2} && 0.08 && 3.58 && \\
    \frac{(M - rF'-2F)^2}{4F} \geq  && 6.77 && 2.01 && 0.05 && 0.35 &&\\
     \frac{(M +4 -2F)^2}{4F} \geq  && 0 && 0.78  && 4.28 && 10.58 &&
\end{align*}
As above, we used $(4F)^{-1} \geq 0.35$ in rows four and five along with estimates (4) and (5) from Lemma \ref{lem:grad-terms-ests}. To obtain the third row, estimate (3) from the lemma asserts $(M+ 2- \frac{1}{2}rF'-2F)^2$ is bounded below by $X_M := 4, 0, 2.25, 12.35$ for $M = -3, -1, 1, 3$, respectively. For $M \neq -1$, we have $X_M/4F -rF'\geq 0.35(X_M-2)$. The third equation in Lemma \ref{lem:grad-terms-ests} gives no good term when $M = -1$, leaving $-rF'$. Unfortunately, here the estimate $rF' \leq 2$ would not be sufficient in what follows, so we have used a slightly sharper estimate
\[
rF' - 2r^{-2} = -\frac{2\sqrt{2}c_0(r^2-1)}{r^4} \leq 0.
\]
From the table above, we obtain 
\begin{align*}
    \mathcal{G}^3_A &\geq (2.2+ 2F)  |k_{1, -3}|^2 +(3.5-2r^{-2}+2F) |k_{1 ,-1}|^2 + (3.58+2F) |k_{1, 1}|^2 + (5.08+2F) |k_{1, 3}|^2 \\
    & \qquad + (8.27 +F) |k_{2, -3}|^2 +(5.51+F) |k_{2 ,-1}|^2 + (3.55+F) |k_{2, 1}|^2 + (1.85+F) |k_{2, 3}|^2\\
    & \qquad + (1.5+F) |k_{3, -3}|^2 +(4.28+F) |k_{3 ,-1}|^2 + (7.78+F) |k_{3, 1}|^2 + (12.08+F) |k_{3, 3}|^2,
\end{align*}
and, estimating $2\sqrt{6} < 4.9$, $4\sqrt{2} < 5.7$,  
\begin{align*}
    \mathcal{E}^3_A &\leq 4.9\sqrt{F} \Big(|k_{3, -3}|  |k_{1, -1}|  + |k_{1, -3}||k_{2,-1}| + |k_{3, 1}|  |k_{1, 3}|  + |k_{1, 1}||k_{2,3}|\Big)\\
    & \quad + 5.7 \sqrt{F}\Big(|k_{3, -1}|  |k_{1, 1}|  + |k_{1, -1}||k_{2,1}|\Big). 
\end{align*}
Using Young's inequality, we have 
\begin{align*}
    4.9\sqrt{F} |k_{3, -3}|  |k_{1, -1}|  +  5.7 \sqrt{F}|k_{2,1}| |k_{1,-1}| & \leq 2.45\sqrt{F}|k_{3, -3}|^2 +2.45\sqrt{F} |k_{1, -1}|^2\\ 
    & \qquad +  \frac{2.85}{0.6} \sqrt{F}|k_{2,1}|^2 + (0.6)2.85\sqrt{F} |k_{1, -1}|^2\\
    & = 2.45 \sqrt{F} |k_{3, -3}|^2 + 4.75 \sqrt{F} |k_{2, 1}|^2 + 4.16 \sqrt{F} |k_{1, -1}|^2. 
\end{align*}
Similarly, 
\[
 4.9\sqrt{F} |k_{2, 3}|  |k_{1, 1}|  + 5.7 \sqrt{F}|k_{3,-1}| |k_{1,1}|  \leq  2.45 \sqrt{F} |k_{2, 3}|^2 + 4.75 \sqrt{F} |k_{3, -1}|^2 + 4.16 \sqrt{F} |k_{1, 1}|^2.
\]
\[
4.9 \sqrt{F} \big(|k_{1, -3}||k_{2,-1}| + |k_{3, 1}|  |k_{1, 3}|\big) \leq 2.45\sqrt{F} \big(|k_{1, -3}|^2 + |k_{2,-1}|^2  + |k_{3, 1}|^2 +  |k_{1, 3}|^2\big).
\]
Using $\sqrt{F} \in [0, 0.85)$ for $r \in [0, \infty)$ and comparing the lower bound for $\mathcal{G}_A^3$ to the upper bound for $\mathcal{E}^3_A$, we find that (comparing coefficients in the order $|k_{3, -3}|, |k_{2, 1}|, |k_{2, 3}|$, $|k_{3, -1}|, |k_{1, 1}|, |k_{1, -3}|$, $|k_{2, -1}|, |k_{3, 1}|, |k_{1, 3}|$):
\begin{align*}
    2.45 \sqrt{F} &< 1.5 + F, & 4.75 \sqrt{F} &< 3.55 + F, &  2.45 \sqrt{F} &< 1.85 + F, \\
    4.75\sqrt{F} &< 4.28 + F, & 4.16 \sqrt{F} &< 3.58+ 2F, & 2.45 \sqrt{F} &< 2.2 + 2F, \\
    2.45 \sqrt{F} &< 5.51 + F, & 2.45 \sqrt{F} &< 7.78 + F, & 2.45 \sqrt{F} &< 5.08 + 2F. 
\end{align*}
It remains only to compare the coefficients of $|k_{1,-1}|$, for which we need
\[
 4.16 \sqrt{F} < 3.5-2r^{-2} + 2F.
\]
We use the quadratic formula and the definition of $F$ to see that $2r^{-2} = \sqrt{2+4(1+\sqrt{2})(1-\sqrt{2}F)}-\sqrt{2}$. Now the inequality follows from
\[
2 + 4(1+\sqrt{2})(1-\sqrt{2} x) < \big(2x^2 - 4.16x + 3.5+\sqrt{2}\big)^2
\]
for $x \in [0, 0.85)$. We conclude that $\mathcal{E}^3_A  - \mathcal{G}^3_A \leq 0$ for $r \in (1, \infty)$. An identical argument gives that $\mathcal{E}^3_{\bar A}  - \mathcal{G}^3_{\bar A} \leq 0$. 

Finally, by Lemma \ref{lem:cross-section-ests}, we conclude that 
\[
\int_M\big( 2R(h_I, \overline{h_I}) - |\nabla h_I|^2 \big) e^{-f} d\mu_g + \int_M\big( 2R(h_A, \overline{h_A}) - |\nabla h_A|^2 \big) e^{-f} d\mu_g \leq 0.
\]
This completes the proof. 
\end{proof}

\begin{theorem}\label{thm:stability for J = 4}
    Let $h  \in \mathcal{T}^4_{M'}$ and $M' \in \mathcal{J} = \{-4, -2, 0, 2, 4\}$. Then 
    \[
    \int_M\big( 2R(h, \overline{h}) - |\nabla h|^2 \big) e^{-f} d\mu_g \leq 0. 
    \]
    Consequently, $\delta^2 \nu_g(h + \overline{h}) \leq 0$. 
\end{theorem}

\begin{proof}
    We proceed as we did for $J = 3$, except we can afford to be coarser in our estimates than before. Write $h = h_I+ h_A$. We show that $\mathcal{E}^4_I - \mathcal{G}^4_I \leq 0$ and $\mathcal{E}^4_A - \mathcal{G}^4_A \leq 0$ for $r \in (1, \infty)$ (the estimate $\mathcal{E}^4_{\bar A} - \mathcal{G}^4_{\bar A} \leq 0$  is identical to the latter one). Using Lemma \ref{lem:grad-terms-ests}, recalling the definitions on $N^J_M$ and $C^J_{M+}$, we obtain a table:
    \begin{align*}
\underline{\text{Expression}}   && \underline{M = -4} && \underline{M =  -2} && \underline{M = 0} && \underline{M = 2}&& \underline{M = 4} &&\\
    \sqrt{2}C^4_{M+} =&&    4\sqrt{2} &&    4\sqrt{3}   && 4\sqrt{3}   &&   4\sqrt{2} && 0 &&\\  
    N^4_M =&&  2 && 5 && 6 && 5 && 2 && \\
    \frac{M^2}{4F} -3.5\geq && 2.1 && -2.1 && -3.5 && -2.1&& 2.1 &&\\
    \frac{M^2}{4F} - 0.45 -rF'\geq && 3.15 && -1.05 && -2.45 && -1.05 && 3.15 &&\\
    \frac{(M-2 - \frac{1}{2}rF')^2}{4F} \geq && 12.6 && 5.6 && 1.4 && 0 && 0.35 && \\
    \frac{(-M-2 - \frac{1}{2}rF')^2}{4F} \geq  && 0.35 && 0 && 1.4 && 5.6 && 12.6 &&
\end{align*}
We have only used $rF' \leq 2$ and $(4F)^{-1} \geq 0.35$ above. Using the table, we obtain 
\begin{align*}
    \mathcal{G}^4_I &\geq  4.1\big(|h_{0, -4}|^2 + |h_{0, 4}|^2\big)+    2.9\big(|h_{0, -2}|^2 + |h_{0, 2}|\big)^2 + 2.5  |h_{0, 0}|^2  , \\
& + (5.15 + 2F)\big(|h_{1, -4}|^2 +|h_{1, 4}|^2 \big) +  (3.95+2F)\big(|h_{1, -2}|^2+|h_{1, 2}|^2\big)  + (3.55 + 2F) |h_{1, 0}|^2 \\
& + (14.6+F)\big(|h_{+, -4}|^2 +|h_{-, 4}|^2 \big) + (10.6+F)\big(|h_{+, -2}|^2+|h_{-, 2}|^2  \big)  \\
& +(7.4+F)  \big(|h_{+, 0}|^2 +|h_{-, 0}|^2 \big) + (5+F)\big(|h_{+, 2}|^2 + |h_{-, -2}|^2 \big) + (2.35+F)\big(|h_{+, 4}|^2 + |h_{-, -4}|^2 \big), 
\end{align*}
and, estimating $4\sqrt{2} < 5.66$, $4\sqrt{3} < 6.94$, 
\begin{align*}
    \mathcal{E}^4_I &\leq 5.66\sqrt{F} \Big(|h_{-, -4}|  |h_{1, -2}|  + |h_{1, -4}||h_{+,-2}| +|h_{-, 2}|  |h_{1, 4}|  + |h_{1, 2}||h_{+,4}| \Big)\\
    & \quad + 6.94 \sqrt{F}\Big(|h_{-, -2}|  |h_{1, 0}|  + |h_{1, -2}||h_{+,0}|+ |h_{-, 0}|  |h_{1, 2}|  + |h_{1, 0}||h_{+,2}|\Big). 
\end{align*}
Using Young's inequality, we have
\begin{align*}
& 6.94 \sqrt{F}\big(|h_{1, -2}||h_{+,0}| +|h_{1, 2}| |h_{-, 0}| \big) + 5.66\sqrt{F}\big(|h_{1, -2}||h_{-, -4}|+  |h_{1, 2}||h_{+,4}|\big) \\
&\leq \frac{3.47^2}{3} \sqrt{F} \big(|h_{+, 0}|^2 + |h_{-, 0}|^2 \big) +  3 \sqrt{F} \big(|h_{1, -2}|^2 + |h_{1, 2}|^2 \big)\\
& \qquad + 2.83 \sqrt{F} \big(|h_{+, 4}|^2 + |h_{-, -4}|^2 \big) +  2.83 \sqrt{F} \big(|h_{1, -2}|^2 + |h_{1, 2}|^2 \big) \\
& \leq 4.02 \sqrt{F} \big(|h_{+, 0}|^2 + |h_{-, 0}|^2 \big) + 2.83 \sqrt{F} \big(|h_{+, 4}|^2 + |h_{-, -4}|^2 \big)+  5.83 \sqrt{F} \big(|h_{1, -2}|^2 + |h_{1, 2}|^2 \big),
\end{align*}
and 
\begin{align*}
   & 6.94 \sqrt{F} \big(|h_{-, -2}|  |h_{1, 0}| + |h_{1, 0}||h_{+,2}|\big) +  5.66\sqrt{F} \big(  |h_{1, -4}||h_{+,-2}| +|h_{-, 2}|  |h_{1, 4}|  \big) \\
    & \leq (3.47\cdot 1.5) \sqrt{F}\big(|h_{+,2}|^2 + |h_{-, -2}| ^2\big) + \frac{6.94}{1.5}\sqrt{F} |h_{1, 0}|^2 \\
    & \qquad + 2.83\sqrt{F} \big(|h_{+, -2}|^2 + |h_{-, 2}|^2 + |h_{1, -4}|^2 + |h_{1, 4}|^2 \big)\\
    & \leq 5.21\sqrt{F}\big(|h_{+,2}|^2 + |h_{-, -2}| ^2\big) +4.63 \sqrt{F} |h_{1, 0}|^2 + 2.83\sqrt{F} \big(|h_{+, -2}|^2 + |h_{-, 2}|^2 +|h_{1, -4}|^2 + |h_{1, 4}|^2 \big).
\end{align*}
Using that $\sqrt{F} \subset [0, 0.85)$ for $r \in [1, \infty)$ and comparing the lower bound for $\mathcal{G}^4_I$ to the upper bound for $\mathcal{E}^4_I$, one can readily verify -- by checking the sign of the quadratic -- that for the coefficients (in the order $|h_{\pm, 0}|, |h_{\pm,\pm4}|, |h_{1,\mp2}|, |h_{\pm, \pm2}|, |h_{1,0}|$, $ |h_{\pm, \mp2}|, |h_{1, \pm4}|$) satisfy
\begin{align*}
    4.02 \sqrt{F} &< 7.4+F, & 2.83 \sqrt{F} &< 2.35 + F, &  5.83 \sqrt{F} &< 3.95 + 2F, \\
    5.21\sqrt{F} &< 5 + F, & 4.63 \sqrt{F} &< 3.55+ 2F, & 2.83 \sqrt{F} &< 10.6 + F, \\
    2.83 \sqrt{F} &< 5.15 + 2F.
\end{align*}
This shows that $\mathcal{E}^4_I - \mathcal{G}^4_I \leq 0$ for $r \in (1, \infty)$. 

For the anti-invariant case, we start from: 

\begin{align*}
\underline{\text{Expression}}    && \underline{M = -4} && \underline{M =  -2} && \underline{M = 0} && \underline{M = 2}&& \underline{M = 4} && \\
    \sqrt{2}C^4_{M+} =&&    4\sqrt{2} &&    4\sqrt{3}   && 4\sqrt{3}   &&   4\sqrt{2} && 0 &&\\  
    N^4_M =&&  2 && 5 && 6 && 5 && 2 && \\
    \frac{(M+ 2- \frac{1}{2}rF'-2F)^2}{4F} -rF'\geq && 1.15 && -1.65 && -1.92 && 0.18 && 5.08  \\
    \frac{(M - rF'-2F)^2}{4F} \geq  && 10.2 && 4 && 0.68 && 0  && 1.4\\
     \frac{(M +4 -2F)^2}{4F} \geq  && 0 && 0.08  && 2.18 && 7.08 && 14.78
\end{align*}
Once again we have only used the coarse estimates $rF' \leq 2$ and $(4F)^{-1} \geq 0.35$ above. Using the table, we obtain 
\begin{align*}
    \mathcal{G}^4_A& \geq  (3.15 + 2F)|k_{1, -4}|^2   +(3.35+2F) |k_{1, -2}|^2\\
    & \qquad +(4.08+ 2F) |k_{1, 0}|^2+   (5.18+2F)|k_{1, 2}|^2 + (7.08 +2F)|k_{1, 4}|^2\\
&  + (12.2+ F)|k_{2, -4}|^2 + (9+F)|k_{2, -2}|^2+ (6.68 + F) |k_{2, 0}|^2 \\
& \qquad +(5+F)|k_{2, 2}|^2+ (3.4+ F)|k_{2, 4}|^2   \\
& + (2+ F)|k_{3, -4}|^2 + (5.08+F)|k_{3, -2}|^2+ (8.18+ F) |k_{3, 0}|^2\\
& \qquad + (12.08+F)|k_{3, 2}|^2  +(16.78+ F)|k_{3, 4}|^2,
\end{align*}
and,  estimating $4\sqrt{2} < 5.66$, $4\sqrt{3} < 6.94$, 
\begin{align*}
\mathcal{E}^4_A &\leq 5.66\sqrt{F}\Big(|k_{3, -4}||k_{1, -2}| + |k_{1, -4}||k_{2, -2}|  + |k_{3,2}||k_{1, 4}|+ |k_{1,2}||k_{2, 4}|  \Big)\\
& + 6.94\sqrt{F} \Big(|k_{3, -2}| |k_{1, 0}| + |k_{1, -2}| |k_{2, 0}|+ |k_{3, 0}||k_{1, 2}|  + |k_{1, 0}||k_{2, 2}|\Big). 
\end{align*}
Using Young's inequality, we estimate these 
\begin{align*}
    5.66\sqrt{F} |k_{3, -4}||k_{1, -2}|+  6.94\sqrt{F} |k_{2, 0}||k_{1, -2}|  & \leq 2.83 \sqrt{F} |k_{3, -4}|^2 + \frac{3.47^2}{2.5} \sqrt{F} |k_{2, 0}|^2  \\
    & \qquad + (2.83 + 2.5)\sqrt{F} |k_{1, -2}|^2, \\
    5.66\sqrt{F} |k_{2, 4}||k_{1,2}| +  6.94\sqrt{F}|k_{3, 0}||k_{1, 2}|  & \leq 2.83 \sqrt{F} |k_{2, 4}|^2 + \frac{3.47^2}{2.5} \sqrt{F} |k_{3, 0}|^2  \\
    & \qquad + (2.83 + 2.5)\sqrt{F} |k_{1, 2}|^2,  \\
    5.66\sqrt{F}\big(|k_{3, -2}| |k_{1, 0}| +|k_{2, 2}||k_{1, 0}|\big)  & \leq 2.83 \sqrt{F} |k_{3, -2}|^2 + 2.83 \sqrt{F} |k_{2, 2}|^2  \\
    & \qquad + (2.83 + 2.83)\sqrt{F} |k_{1, 0}|^2, \\
    6.94\sqrt{F} \big(|k_{2, -2}| |k_{1, -4}|+  |k_{3,2}||k_{1, 4}|\big)& \leq  3.47 \sqrt{F} |k_{2, -2}|^2 + 3.47 \sqrt{F} |k_{3, 2}|^2  \\
    & \qquad + 3.47\sqrt{F} |k_{1, -4}|^2 + 3.47\sqrt{F} |k_{1, 4}|^2 . 
\end{align*}
As we have done above, the coefficients can be compared demonstrating the $\mathcal{E}^4_A - \mathcal{G}^4_A \leq 0$. By Lemma \ref{lem:cross-section-ests}, this completes the proof of the theorem. 
\end{proof}

\begin{remark}
    The estimates used in the proof of these two theorems above are not as delicate as numerical constants may make them appear. See the maximum eigenvalues of $r^2\mathcal{Q}^J, r^2\mathcal{P}^J,   r^2\tilde{\mathcal{P}}^J$ for $J\leq 5$ in Figures \ref{fig:max-eigenvalues-Q}, \ref{fig:max-eigenvalues-P}. 
\end{remark}

As we noted above, the proof of stability for $J = 4$ required only coarser estimates than the proof for $J =3$. This pattern persists, and arguing as we have above, one can straightforwardly give a unified proof of the stability of the FIK shrinker with respect to deformations in $h \in \mathcal{T}^J_{M'}$ for all $J \geq 5$. For brevity, we assert the result, but omit the straightforward proof. 

\begin{theorem}\label{thm:stability for J geq 5}
    Let $h \in \mathcal{T}^J_{M'}$ for $J \geq 5$ and $M' \in \mathcal{J}$. Then 
    \[
    \int_M \big(2R(h, \overline{h}) - |\nabla h|^2 \big) e^{-f} d\mu_g \leq 0. 
    \]
    Consequently, $\delta^2 \nu_g(h + \overline{h}) \leq 0$. 
\end{theorem}

\subsection{Summary and eigenvalue plots}\label{sec:summary eigenvalue plots}
Let $h = h_I + h_A\in \mathcal{T}^J_{M'}$ with $\mathrm{div}_f(h) =0$ and as above let
\begin{equation}
\eta := (\eta_0, \eta_1, \eta_+, \eta_-),\qquad  \kappa := (\kappa_1, \kappa_2, \kappa_3), \qquad \tilde{\kappa} := (\kappa_{\bar{1}}, \kappa_{\bar{2}}, \kappa_{\bar{3}}),
\end{equation}
denote tuples of complex vector-valued functions (of $r$) in Proposition \ref{prop:matrix-stability-identities}. We have thus shown that 
\begin{align*}
&2(J+1)\; \delta^2 \nu_g(h + \overline{h}) \\ &= \int_1^\infty \left(4 \int_{\mathbb{S}^3}^J 2 R(h_I, \overline{h_I}) - |\nabla h_I|^2 \right) r^3 e^{-f} dr + \int_1^\infty \left(4 \int_{\mathbb{S}^3}^J 2 R(h_A, \overline{h_A}) - |\nabla h_A|^2 \right) r^3 e^{-f} dr  \\
& =   \int_1^\infty \left((\mathcal{Q}^J \eta) \cdot \overline{\eta} - \frac{F}{4} |\eta'|^2 \right) r^3 e^{-f} dr + \int_1^\infty \left((\mathcal{P}^J \kappa) \cdot \overline{\kappa} - \frac{F}{4} |\kappa'|^2  \right) r^3 e^{-f} dr\\
& \qquad + \int_1^\infty \left((\tilde{\mathcal{P}}^J \tilde{\kappa}) \cdot \overline{\tilde{\kappa}} - \frac{F}{4} |\tilde{\kappa}'|^2  \right) r^3 e^{-f} dr. 
\end{align*} 
When the maximum eigenvalues of the matrices $\mathcal{Q}^J(r)$ and $\mathcal{P}^J(r)$ are nonpositive for all $r \in (1,\infty)$ (recall that $\tilde{\mathcal{P}}^J(r)$ has the same eigenvalues as $\mathcal{P}^J(r)$), it follows that $\delta^2 \nu_g(h + \overline{h}) \leq 0$, since the derivative term has a favorable sign and can be discarded. In the previous section, we proved that $\mathcal{Q}^J(r), \mathcal{P}^J(r)$, and $\tilde{\mathcal{P}}^J(r)$ are nonpositive for all $J \geq 3$. Numerically, the maximum eigenvalues of $\mathcal{Q}^J(r)$ and $\mathcal{P}^J(r)$ for $J \leq 5$ are displayed below.

Note that $\mathcal{Q}^J$ has size $4(J+1)\times 4(J+1)$ and $\mathcal{P}^J$ has size $3(J+1)\times 3(J+1)$. In the next section we will see that both matrices admit a further block decomposition. As a result, estimating the maximum eigenvalue of $\mathcal{Q}^J(r)$ reduces to estimating eigenvalues of blocks of size at most $4 \times 4$, and for $\mathcal{P}^J(r)$ of blocks of size at most $3 \times 3$. Moreover, by Theorems \ref{thm:main-radial}, \ref{thm:stability for J = 3}, \ref{thm:stability for J = 4}, and \ref{thm:stability for J geq 5}, we need only carry this out for $J \in \{1,2\}$.

In the figures below, one sees that the maximum eigenvalues of $\mathcal{Q}^0$, $\mathcal{Q}^2$, and $\mathcal{P}^1$ (hence also $\tilde{\mathcal{P}}^1$) are not nonpositive for all $r \in (1,\infty)$. Additional work is therefore required to establish that $\delta^2 \nu_g$ is nonpositive in these cases. For $\mathcal{Q}^0$ (the radially symmetric setting), we did this using the gauge equations, which couple the analysis of $\mathcal{Q}^0$ and $\hat{\mathcal{Q}}^0$, together with Barta’s trick. In Section \ref{sec:P1Q2}, we will adopt a different approach and employ a 1-dimensional reduction to a Sturm-Liouville problem to address the positivity of $\mathcal{Q}^2$ and $\mathcal{P}^1$ (and $\tilde{\mathcal{P}}^1$). 

 \begin{figure}[H]
 \includegraphics[scale=0.7]{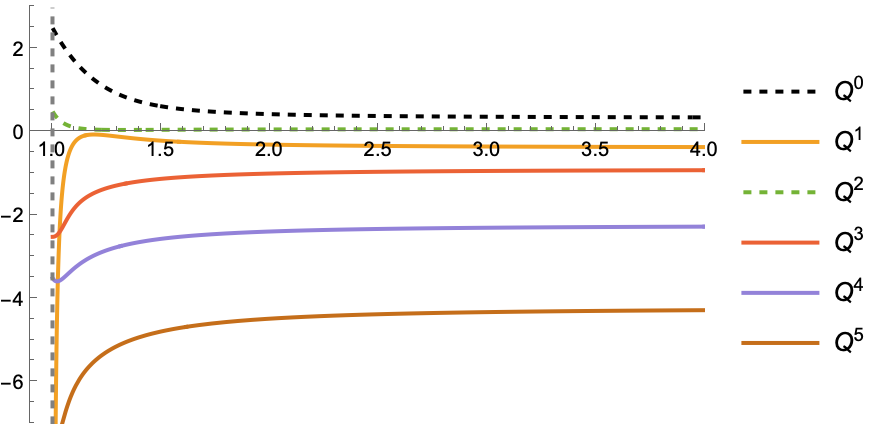}
   \caption{The maximum eigenvalue of $r^2 \mathcal{Q}^J(r)$ (as a function of $r$) for $J \in \{0,1,2,3,4,5\}$ plotted on the intervals  $r \in [1, 4]$.} \label{fig:max-eigenvalues-Q}
\end{figure}
\begin{figure}[H]
  \includegraphics[scale=0.7]{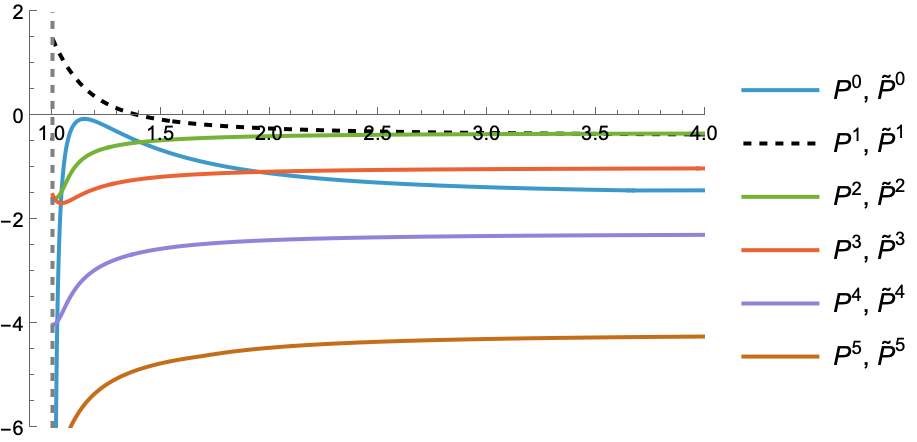}
      \caption{The maximum eigenvalues of $r^2 \mathcal{P}^J(r)$ and $r^2 \tilde{\mathcal{P}}^J(r)$ (as functions of $r$) for $J \in \{0,1,2,3,4,5\}$ plotted on the intervals $r \in [1, 4]$. }\label{fig:max-eigenvalues-P}
\end{figure}

Note that the maximum eigenvalue of $\mathcal{P}^0$ (and $\tilde{\mathcal{P}}^0$), which is close to positive near $r =1$, is just $\Lambda_{1-}$, the sign of which we addressed in Lemma \ref{lem:Lambda1m-negativity}. We will discuss $\mathcal{Q}^1$ in slightly more detail in the proof of Proposition \ref{prop:J1-except-bad-block} below.

\begin{figure}[H]
  \includegraphics[scale=0.7]{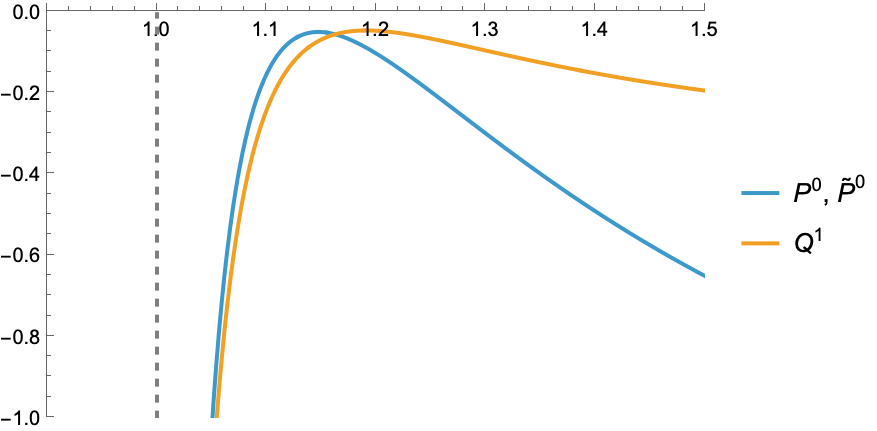}
      \caption{The maximum eigenvalue of $r^2 \mathcal{P}^0(r),r^2 \tilde{\mathcal{P}}^0(r)$, and $r^2 \mathcal{Q}^1(r)$ (as a function of $r$) plotted on the intervals $r \in [1, 1.5]$. }\label{fig:close-calls}
\end{figure}

\section{A further block decomposition of $L_f$ and positive eigentensors}\label{sec:further block decomposition}

\subsection{A further block decomposition of $L_f$} 

In this section, we observe that the matrices $\mathcal{Q}^J$, $\mathcal{P}^J$, and $\tilde{\mathcal{P}}^J$ admit further block decompositions, with blocks of size at most $4 \times 4$. This reduction makes demonstrating the nonpositivity of blocks quite tractable, in principle. We shall only need it when $J \in \{1, 2\}$ in any event. We begin with the smaller blocks in the $J_1^+$–anti-invariant case.

\subsubsection{The block decomposition in the $J_1^+$-anti-invariant setting}
Recall $\mathcal{P}^J, \tilde{\mathcal{P}}^J$ defined in \eqref{eq:def-of-P-matrices}. Because $\hat{\mathcal{Q}}^J_{pp}$ is diagonal and $\mathcal{K}$ only has entries along the superdiagonal, it can be readily checked by examining nonzero elements of rows and columns that each of $\mathcal{P}^J$ and $\tilde{\mathcal{P}}^J$ further block decompose into two $1 \times 1$ blocks, two $2\times 2$ blocks, and $(J-1)$ $3 \times 3$ blocks for a total of $J + 3$ blocks. 

We will specify blocks by indicating which entries of $\kappa = (k_{p, M})$ and $\tilde{\kappa} = (k_{\bar{p}, M})$ the blocks act upon for $p \in \{1, 2 ,3\}$ and $M \in \mathcal{J}$ (equipped with the usual ordering $-J, -J+2, \dots, J-2, J$). The block decomposition of $\mathcal{P}^J$ and $\tilde{\mathcal{P}}^J$ are as follows. \\

\noindent \underline{\textbf{$1 \times 1$ blocks of $\mathcal{P}^J, \tilde{\mathcal{P}}^J$:}} For all $J \geq 0$, $\mathcal{P}^J$ and $\tilde{\mathcal{P}}^J$ preserve the blocks of functions:
\[
\begin{bmatrix}
  k_{2, -J}
\end{bmatrix},
\begin{bmatrix}
  k_{3, J}
\end{bmatrix},
\begin{bmatrix}
  k_{\bar{2}, J}
\end{bmatrix},
\begin{bmatrix}
  k_{\bar{3}, -J}
\end{bmatrix}.
\]
They interact by the $1 \times 1$ blocks
\begin{align*}
    \mathcal{P}^J_{2, -J} &= \begin{bmatrix}\Lambda_{01} + \tilde{\Delta}^J_J -J\frac{4F+ 2rF'}{4s^2} \end{bmatrix}, \\[4pt]
    \mathcal{P}^J_{3,J} &= \begin{bmatrix} \Lambda_{23} + \tilde{\Delta}^J_J-J \frac{8-4F}{4s^2}\end{bmatrix}, \\[4pt]
    \tilde{\mathcal{P}}^J_{\bar{2}, J} &= \begin{bmatrix}\Lambda_{01} + \tilde{\Delta}^J_J -J\frac{4F+ 2rF'}{4s^2} \end{bmatrix}, \\[4pt]
    \tilde{\mathcal{P}}^J_{\bar{3},-J} &= \begin{bmatrix} \Lambda_{23} + \tilde{\Delta}^J_J-J \frac{8-4F}{4s^2}\end{bmatrix}. 
\end{align*}
Note that $\tilde{\Delta}^J_{-M} = \tilde{\Delta}^J_{M}$ for any $M$. Note also that $\mathcal{P}^J_{2, -J} = \tilde{\mathcal{P}}^J_{\bar{2}, J}$ and $\mathcal{P}^J_{3, J} = \tilde{\mathcal{P}}^J_{\bar{3}, -J}$. When $J = 0$, there are additional $1 \times 1$ blocks that interact with $k_{1, 0}$ and $k_{\bar{1},0}$ given by $\mathcal{P}^J_{1, 0}  = \tilde{\mathcal{P}}^J_{\bar{1}, 0} = \begin{bmatrix}\Lambda_{1-}\end{bmatrix}$. For $J = 0$, these are the only blocks that appear. \\
    
\noindent \underline{\textbf{$2 \times 2$ blocks of $\mathcal{P}^J, \tilde{\mathcal{P}}^J$:}} For all $J \geq 1$, $\mathcal{P}^J$ and $\tilde{\mathcal{P}}^J$ preserve the blocks of functions
\[
\begin{bmatrix}
  k_{1, -J} \\
  k_{2, -J + 2}
\end{bmatrix},
\begin{bmatrix}
  k_{1, J} \\
  k_{3, J - 2}
\end{bmatrix},
\begin{bmatrix}
  k_{\bar{1}, J} \\
  k_{\bar{2}, J - 2}
\end{bmatrix},
\begin{bmatrix}
  k_{\bar{1}, -J} \\
  k_{\bar{3}, -J + 2}
\end{bmatrix}.
\]
They interact by the $2 \times 2$ blocks:  
\begin{align*}
    \mathcal{P}^J_{12,-J} &= \frac{1}{r^2}\begin{bmatrix}  r^2(\hat{\mathcal{Q}}^J_{11})_{-J} & -(1+i) \sqrt{JF} \\[4pt]
     -(1-i) \sqrt{JF} &  r^2(\hat{\mathcal{Q}}^J_{22})_{-J+2} \end{bmatrix},\\[3pt]
    \mathcal{P}^J_{13, J} &= \frac{1}{r^2}\begin{bmatrix} r^2(\hat{\mathcal{Q}}^J_{11})_{J} & (1-i) \sqrt{JF}\\[4pt]
   (1+i) \sqrt{JF}  & r^2(\hat{\mathcal{Q}}^J_{33})_{J-2} \end{bmatrix}, \\[3pt]
       \tilde{\mathcal{P}}^J_{\bar{1}\bar{2}, J} &= \frac{1}{r^2}\begin{bmatrix} r^2(\hat{\mathcal{Q}}^J_{\bar{1}\bar{1}})_{J} & (1-i) \sqrt{JF}\\[4pt]
   (1+i) \sqrt{JF}  & r^2(\hat{\mathcal{Q}}^J_{\bar{2}\bar{2}})_{J-2} \end{bmatrix},\\[3pt]
   \tilde{\mathcal{P}}^J_{\bar{1}\bar{3},-J} &= \frac{1}{r^2}\begin{bmatrix}  r^2(\hat{\mathcal{Q}}^J_{\bar{1}\bar{1}})_{-J} & -(1+i) \sqrt{JF}\\[4pt]
     -(1-i) \sqrt{JF} &  r^2(\hat{\mathcal{Q}}^J_{\bar{3}\bar{3}})_{-J+2} \end{bmatrix}.
\end{align*}
We have used that $C^J_{-J+}=C^J_{J-} = \sqrt{4J}$. The definition of $\hat{\mathcal{Q}}^J$ gives  
\begin{align*}
    (\hat{\mathcal{Q}}^J_{11})_{-J} =  (\hat{\mathcal{Q}}^J_{\bar 1\bar 1})_{J} &= \Lambda_{1-} + \tilde{\Delta}^J_{J} +J\frac{4-4F-rF'}{4s^2},\\ (\hat{\mathcal{Q}}^J_{22})_{-J+2} = (\hat{\mathcal{Q}}^J_{\bar{2} \bar{2}})_{J-2}&=  \Lambda_{01} + \tilde{\Delta}^J_{J-2} - (J-2) \frac{4F+ 2rF'}{4s^2},\\
    (\hat{\mathcal{Q}}^J_{11})_{J} = (\hat{\mathcal{Q}}^J_{\bar 1\bar 1})_{-J}  &=\Lambda_{1-} + \tilde{\Delta}^J_{J} -J\frac{4-4F-rF'}{4s^2}, \\
    (\hat{\mathcal{Q}}^J_{33})_{J-2} =(\hat{\mathcal{Q}}^J_{\bar 3 \bar 3})_{-J+2} &= \Lambda_{23} + \tilde{\Delta}^J_{J-2} -(J-2) \frac{8-4F}{4s^2}.
\end{align*}
Note that the eigenvalues of $\mathcal{P}^J_{12, -J}$ and $ \tilde{\mathcal{P}}^J_{\bar{1}\bar{2}, J}$ are the same (with eigenvectors related by the correspondence $(v_1,v_2) \to(-\overline{v_1}, v_2)$). Similarly for $\mathcal{P}^J_{13, J}$ and $ \tilde{\mathcal{P}}^J_{\bar{1}\bar{3}, -J}$. When $J = 1$, these and the blocks above are the only ones that appear. 
\\
    
\noindent\underline{\textbf{$3 \times 3$ blocks of $\mathcal{P}^J$ and $\tilde{\mathcal{P}}^J$:}} For all $J \geq 2$, and $M \in \mathcal{J} \setminus \{\pm J\}$, $\mathcal{P}^J$ and $\tilde{\mathcal{P}}^J$ preserve the blocks of functions
\[
\begin{bmatrix}
  k_{1, M} \\
  k_{2, M + 2} \\
  k_{3, M - 2}
\end{bmatrix},
\begin{bmatrix}
  k_{\bar{1}, M} \\
  k_{\bar{2}, M - 2} \\
  k_{\bar{3}, M + 2}
\end{bmatrix}.
\]
They interact by the $3 \times 3$ blocks: 
\begin{align*}
    \mathcal{P}^J_{123,M} &= \frac{1}{2r^2} \begin{bmatrix} 2r^2(\hat{\mathcal{Q}}^J_{11})_{M}   & -(1+i) \sqrt{F} C^J_{M+} & (1-i)\sqrt{F} C^J_{M-}\\[4pt]
     -(1-i) \sqrt{F} C^J_{M+}& 2r^2(\hat{\mathcal{Q}}^J_{22})_{M+2}  & 0 \\[4pt]
     (1+i) \sqrt{F} C^J_{M-} & 0 & 2r^2(\hat{\mathcal{Q}}^J_{33})_{M-2}  \end{bmatrix},\\[4pt]
    \tilde{\mathcal{P}}^J_{\bar{1}\bar{2}\bar{3},M} &= \frac{1}{2r^2} \begin{bmatrix} 2r^2(\hat{\mathcal{Q}}^J_{\bar{1}\bar{1}})_{M}   & (1-i)\sqrt{F} C^J_{M-} & -(1+i) \sqrt{F} C^J_{M+} \\[4pt]
    (1+i) \sqrt{F} C^J_{M-}  & 2r^2(\hat{\mathcal{Q}}^J_{\bar{2}\bar{2}})_{M-2}  & 0 \\[4pt]
     -(1-i) \sqrt{F} C^J_{M+} & 0 & 2r^2(\hat{\mathcal{Q}}^J_{\bar{3}\bar{3}})_{M+2}  \end{bmatrix}.
\end{align*}
We have used $C^J_{(M-2)+} = C^J_{M-}$. The definition of $\hat{\mathcal{Q}}^J$ gives 
\begin{align*}
    (\hat{\mathcal{Q}}^J_{11})_{M} &=\Lambda_{1-} + \tilde{\Delta}^J_{M} -M\frac{4-4F-rF'}{4s^2}, &  (\hat{\mathcal{Q}}^J_{\bar{1}\bar{1}})_{M} &=\Lambda_{1-} + \tilde{\Delta}^J_{M} +M\frac{4-4F-rF'}{4s^2}, \\
    (\hat{\mathcal{Q}}^J_{22})_{M+2}   &= \Lambda_{01} + \tilde{\Delta}^J_{M+2} + (M+2) \frac{4F+ 2rF'}{4s^2}, & (\hat{\mathcal{Q}}^J_{\bar{2}\bar{2}})_{M-2}   &= \Lambda_{01} + \tilde{\Delta}^J_{M-2} - (M-2) \frac{4F+ 2rF'}{4s^2},  \\
    (\hat{\mathcal{Q}}^J_{33})_{M-2}  &= \Lambda_{23} + \tilde{\Delta}^J_{M-2}-(M-2) \frac{8-4F}{4s^2}, &  (\hat{\mathcal{Q}}^J_{\bar{3}\bar{3}})_{M+2}  &= \Lambda_{23} + \tilde{\Delta}^J_{M+2}+(M+2) \frac{8-4F}{4s^2}.
\end{align*}
Note that $\mathcal{P}^J_{123,M}$ and $\tilde{\mathcal{P}}^J_{\bar 1\bar 2 \bar 3, -M}$ have the same eigenvalues (and eigenvectors related by the correspondence $(v_1, v_2, v_3) \to (-\overline{v_1}, v_2, v_3)$). \\

\subsubsection{The block decomposition in the $J_1^+$-invariant setting}
A similar picture holds in the invariant setting. Recall $\mathcal{Q}^J$ from Definition \ref{def:QmatJinv}.
As above, because each $\mathcal{Q}^J_{pp}$ is diagonal, $\mathcal{U}$ is diagonal, and $\mathcal{K}$ is superdiagonal, the matrix $\mathcal{Q}^J$ decomposes into two $1 \times 1$ blocks, two $3 \times 3$ blocks, and $(J-1)$ $4 \times 4$ blocks for a total of $J + 3$ blocks. As before we specify blocks by indicating which entries of $\eta = (h_{p, M})$ they act upon. The block decomposition of $\mathcal{Q}^J$ is as follows. \\

\noindent \underline{\textbf{$1 \times 1$ blocks of $\mathcal{Q}^J$:}} For all $J \geq 0$,  $\mathcal{Q}^J$ preserves the blocks of functions: 
\[
\begin{bmatrix}
  h_{+, -J}
\end{bmatrix},
\begin{bmatrix}
  h_{-, J}
\end{bmatrix}.
\]
It interacts by the (identical) $1 \times 1$ blocks: 
\begin{align*}
    \mathcal{Q}^J_{+, -J} &= \begin{bmatrix}\Lambda_{1+}+\tilde{\Delta}^J_J  -J\frac{4+rF'}{4s^2} \end{bmatrix}, \\[4pt]
    \mathcal{Q}^J_{-,J} &= \begin{bmatrix}\Lambda_{1+}+\tilde{\Delta}^J_J  -J \frac{4 + rF'}{4s^2}\end{bmatrix}. 
\end{align*}
When $J = 0$, there is an additional $2 \times 2$ block for the entries interacting with $(h_{0,0}, h_{1, 0})$ given by $Q^J_{01, 0} = \begin{bmatrix}\Lambda_{11}^{++} & \Lambda_{11}^{\pm} \\\Lambda_{11}^{\pm} & \Lambda_{11}^{--} \end{bmatrix}$. For $J = 0$, these are the only blocks that appear. \\

\noindent \underline{\textbf{$3 \times 3$ blocks of $\mathcal{Q}^J$:}}  For all $J \geq 1$, $\mathcal{Q}^J$ preserves the blocks of functions:
\[
\begin{bmatrix}
  h_{0, -J} \\
  h_{1, -J} \\
  h_{+, -J + 2}
\end{bmatrix}, 
\begin{bmatrix}
  h_{0, J} \\
  h_{1, J} \\
  h_{-, J - 2}
\end{bmatrix}, 
\]
It interacts by the $3 \times 3$ blocks:
\begin{align*}
    \mathcal{Q}^J_{01+,-J} &= \frac{1}{r^2}\begin{bmatrix}  r^2(\mathcal{Q}^J_{00})_{-J} & r^2\Lambda_{11}^\pm & 0 \\[4pt]
   r^2  \Lambda_{11}^\pm & r^2(\mathcal{Q}^J_{11})_{-J} &  -(1+i) \sqrt{JF}\\[4pt]
     0 & -(1-i) \sqrt{JF}  &  r^2(\mathcal{Q}^J_{++})_{-J+2} \end{bmatrix},\\[4pt]
    \mathcal{Q}^J_{01-, J} &= \frac{1}{r^2}\begin{bmatrix}  r^2(\mathcal{Q}^J_{00})_{J} & r^2\Lambda_{11}^\pm & 0 \\[4pt]
   r^2  \Lambda_{11}^\pm & r^2(\mathcal{Q}^J_{11})_{J} &  (1-i) \sqrt{JF}\\[4pt]
     0 & (1+i) \sqrt{JF}  &  r^2(\mathcal{Q}^J_{--})_{J-2} \end{bmatrix}. 
\end{align*}
We have used that $C^J_{-J+}=C^J_{J-} = \sqrt{4J}$. The definition of $\mathcal{Q}^J$ gives 
\begin{align*}
    (\mathcal{Q}^J_{00})_{\pm J} &= \Lambda_{11}^{++} + \tilde{\Delta}^J_{J} ,\\
    (\mathcal{Q}^J_{11})_{\pm J} &= \Lambda_{11}^{--} + \tilde{\Delta}^J_{J} , \\
    (\mathcal{Q}^J_{++})_{-J+2} &=  \Lambda_{1+}+\tilde{\Delta}^J_{J-2} -(J-2)\frac{4+rF'}{4s^2},\\
    (\mathcal{Q}^J_{--})_{J-2} &= \Lambda_{1+}+\tilde{\Delta}^J_{J-2}  -(J-2) \frac{4 + rF'}{4s^2} .
\end{align*}
Note that the eigenvalues of $\mathcal{Q}^J_{01+,-J}$ and $\mathcal{Q}^J_{01-,J}$ are the same. When $J = 1$, these and the blocks above are the only ones that appear. \\

\noindent \underline{\textbf{$4 \times 4$ blocks of $\mathcal{Q}^J$:}} For all $J \geq 2$, for $M \in \mathcal{J} \setminus \{\pm J\}$, $\mathcal{Q}^J$ preserves the blocks of functions:
\[
\begin{bmatrix}
  h_{0, M} \\
  h_{1, M} \\
  h_{+, M + 2} \\
  h_{-, M - 2}
\end{bmatrix}.
\]
It interacts by the $4 \times 4$ blocks: 
\begin{align*}
    \mathcal{Q}^J_{01+-,M} &=\frac{1}{2r^2}\begin{bmatrix}  2r^2(\mathcal{Q}^J_{00})_{M} & r^2\Lambda_{11}^\pm & 0 & 0\\[4pt]
  2 r^2  \Lambda_{11}^\pm & 2r^2(\mathcal{Q}^J_{11})_{M} &  -(1+i) \sqrt{F} C^J_{M+} & (1-i)\sqrt{F} C^J_{M-}\\[4pt]
     0 & -(1-i) \sqrt{F} C^J_{M+}  &  2r^2(\mathcal{Q}^J_{++})_{M+2} & 0\\[4pt]
      0 & (1+i) \sqrt{F} C^J_{M-}  &  0 & 2r^2(\mathcal{Q}^J_{--})_{M-2}\end{bmatrix}.
\end{align*}
We have used that $C^J_{(M-2)+} = C^J_{M-}$. The definition of $\mathcal{Q}^J$ gives
\begin{align*}
    (\mathcal{Q}^J_{00})_{M} &=   \Lambda_{11}^{++}+\tilde{\Delta}^J_M, \\
    (\mathcal{Q}^J_{11})_{M}  &= \Lambda_{11}^{--}+\tilde{\Delta}^J_M, \\
    (\mathcal{Q}^J_{++})_{M+2}  &= \Lambda_{1+}+\tilde{\Delta}^J_{M+2}  +(M+2)\frac{4+rF'}{4s^2}, \\
     (\mathcal{Q}^J_{--})_{M-2}  &= \Lambda_{1+}+\tilde{\Delta}^J_{M-2}  -(M-2) \frac{4 + rF'}{4s^2} .
\end{align*}
We note there is an additional decoupling of the $4 \times 4$ block for $M = 0$ whenever $0 \in \mathcal{J}$. In this case, $\Lambda_{1+} + \Delta^J_2 + \frac{4+rF'}{2s^2}$ is an eigenvalue of $\mathcal{Q}^J_{01+-,0}$ with eigenvector given by $(0, 0, 1-i, 1+i)$.

\subsection{Eigenvalue plots, and stability for $J \in\{1, 2\}$ outside of $3$ blocks}\label{sec:stability for J leq 2}

We have seen the FIK shrinking soliton is linearly stable among deformations contained in any $\mathcal{T}^J_{M'}$ for $J \geq 3$ (Theorems \ref{thm:stability for J = 3}, \ref{thm:stability for J = 4}, and \ref{thm:stability for J geq 5}) as well as $J = 0$ (Theorem \ref{thm:main-radial}). Above, we have shown that the matrices $\mathcal{Q}^J, \mathcal{P}^J, \tilde{\mathcal{P}}^J$ further block decompose in at most $4\times4$ blocks. In this section, we will argue that the FIK shrinking soliton is linearly stable for $J \in \{1, 2\}$ except for deformations contained within 3 blocks (for fixed $M'$). 

\begin{remark}\label{rem:sign-justification}
Our main goal in this section is to identify which of the remaining blocks for $J \in \{1, 2\}$ exhibit positivity, as such positivity could be the source of instability of the FIK shrinking soliton. Showing all but 3 of the remaining blocks are nonpositive is straightforward, but somewhat computational. There are three reasonable approaches: 
\begin{enumerate}
\item[(1)] We can compute the maximum eigenvalue (as a function of $r$) of each of the blocks and show it has a sign.
\item[(2)] We can apply Sylvester's criterion (computing leading principal minors) to each of the blocks and show the determinants have a sign. 
\item[(3)] We can seek refined, block specific versions of the proof used to show stability for $J \in \{3, 4\}$ in Section \ref{sec: Wigner function and stab for Jgeq 3}. Essentially this involves careful applications of Young's inequality for blocks that interact with components of $h_0, h_1, k_1, k_{\bar{1}}$ (where the curvature is positive). Any block that does not interact with these must have a sign. 
\end{enumerate}

The first approach is more direct, but it relies more heavily on computer-assisted computation. The second reduces nonpositivity to showing that certain polynomials in $r^2$ have a sign for $r \in [1, \infty)$. Because of the bounded block size, these polynomials always have degree $\leq 8$. Since including all of these computations would be lengthy and not particularly illuminating, we have chosen only to outline an example of the procedure (2) in Proposition \ref{prop:J1-except-bad-block} below. For the remaining cases we either note (3) or pursue approach (1) and plot the maximum eigenvalues over intervals sufficiently large that the asymptotics and negativity are clear. 

Explicit expressions for all of the blocks for $J \in \{1, 2\}$, obtained with computer assisted computation, can be found in \cite{NO2}.
\end{remark}

\subsubsection{$J=0$} Let us begin with a review of the radially symmetric case. We proved in Theorem \ref{thm:main-radial} the linear stability of the FIK shrinking soliton in this setting. 
For $J = 0$, there are blocks 9 matrix blocks
\[
\mathcal{Q}^0_{01, 0}, \; \mathcal{Q}^0_{+,0},\; \mathcal{Q}^0_{-,0},\; \mathcal{P}^0_{1, 0}, \;\mathcal{P}^0_{2, 0}, \;\mathcal{P}^0_{3, 0}, \;\tilde{\mathcal{P}}^0_{\bar 1, 0}, \;\tilde{\mathcal{P}}^0_{\bar 2, 0}, \;\tilde{\mathcal{P}}^0_{\bar 3, 0},
\]
which act on the blocks of functions:
\[
\begin{bmatrix}
  h_{0, 0} \\
  h_{1, 0}
\end{bmatrix},
\begin{bmatrix}
  h_{+, 0}
\end{bmatrix},
\begin{bmatrix}
  h_{-, 0}
\end{bmatrix},
\begin{bmatrix}
  k_{1, 0}
\end{bmatrix},
\begin{bmatrix}
  k_{2, 0}
\end{bmatrix},
\begin{bmatrix}
  k_{3, 0}
\end{bmatrix},
\begin{bmatrix}
  k_{\bar{1}, 0}
\end{bmatrix},
\begin{bmatrix}
  k_{\bar{2}, 0}
\end{bmatrix},
\begin{bmatrix}
  k_{\bar{3}, 0}
\end{bmatrix}.
\]
The only block which is not everywhere nonpositive is $\mathcal{Q}^0_{01, 0}$. The maximum eigenvalues of all the blocks and all the eigenvalues of the problematic block are plotted in Figures \ref{fig:max-eigenvalues-blocks-j0} and \ref{fig:eigenvalues-Q0}.

\begin{figure}[H]
 \includegraphics[scale=0.7]{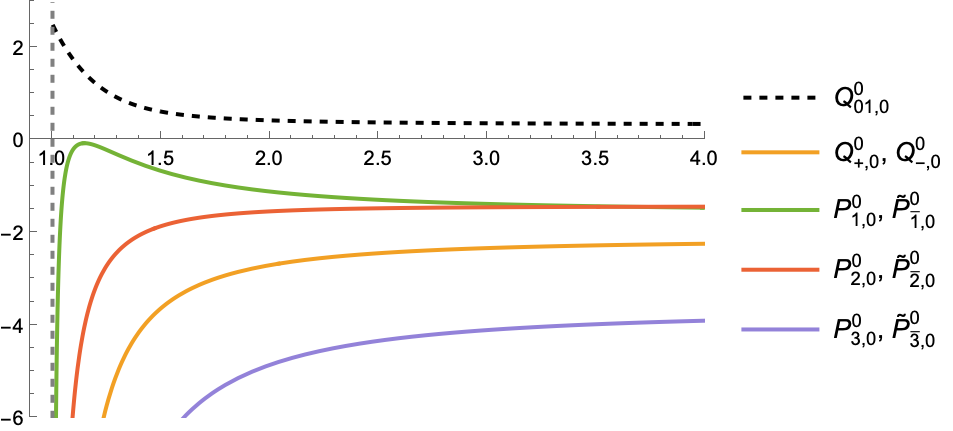}
   \caption{$J = 0$: The maximum eigenvalues (as functions of $r$) of the 9 blocks in the further block decomposition $r^2\mathcal{R}^0(r)$ plotted on the interval $r \in [1, 4]$. The block $\mathcal{Q}^0_{10, 0}$ is the block with the positive eigenvalue.} \label{fig:max-eigenvalues-blocks-j0}
\end{figure}

\begin{figure}[H]
 \includegraphics[scale=0.7]{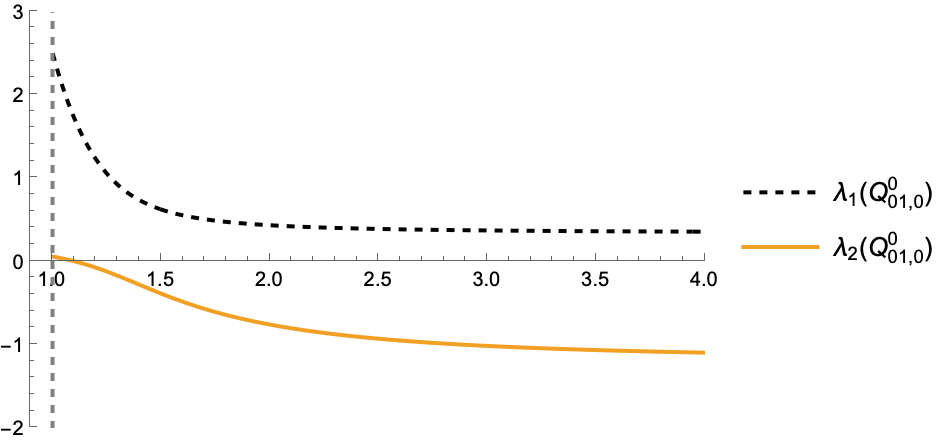}
   \caption{$J = 0$: The eigenvalues of $r^2\mathcal{Q}^0_{01,0}(r)$ (as functions of $r$) plotted on the interval $r \in [1, 4]$ with $\lambda_1 >  \lambda_2$.} \label{fig:eigenvalues-Q0}
\end{figure}

\subsubsection{$J=1$}
For $J = 1$, there are 12 matrix blocks 
\[
\mathcal{Q}^1_{01+, -1}, \; \mathcal{Q}^1_{01-, 1}, \;\mathcal{Q}^1_{+,-1},\; \mathcal{Q}^1_{-,1},\; \mathcal{P}^1_{12, -1},\;  \mathcal{P}^1_{13, 1}, \;\mathcal{P}^1_{2, -1}, \;\mathcal{P}^1_{3, 1}, \;\tilde{\mathcal{P}}^1_{\bar 1 \bar 2, 1}, \;\tilde{\mathcal{P}}^1_{\bar 1 \bar 3, -1}, \;\tilde{\mathcal{P}}^1_{\bar 2, 1}, \;\tilde{\mathcal{P}}^1_{\bar 3, -1},
\]
which act on the blocks of functions: 
\[
\begin{bmatrix}
  h_{0, -1} \\
  h_{1, -1} \\
  h_{+, 1}
\end{bmatrix},
\begin{bmatrix}
  h_{0, 1} \\
  h_{1, 1} \\
  h_{-, -1}
\end{bmatrix},
\begin{bmatrix}
  h_{+, -1}
\end{bmatrix},
\begin{bmatrix}
  h_{-, 1}
\end{bmatrix},
\begin{bmatrix}
  k_{1, -1} \\
  k_{2, 1}
\end{bmatrix},
\begin{bmatrix}
  k_{1, 1} \\
  k_{3, -1}
\end{bmatrix},
\begin{bmatrix}
  k_{2, -1}
\end{bmatrix},
\begin{bmatrix}
  k_{3, 1}
\end{bmatrix},
\begin{bmatrix}
  k_{\bar{1}, 1} \\
  k_{\bar{2}, -1}
\end{bmatrix},
\begin{bmatrix}
  k_{\bar{1}, -1} \\
  k_{\bar{3}, 1}
\end{bmatrix},
\begin{bmatrix}
  k_{\bar{2}, 1}
\end{bmatrix},
\begin{bmatrix}
  k_{\bar{3}, -1}
\end{bmatrix}.
\]
The only blocks which are not everywhere nonpositive are $\mathcal{P}^1_{12, -1}, \tilde{\mathcal{P}}^1_{\bar{1}\bar{2}, 1}$. These blocks have the same eigenvalues. The maximum eigenvalues of all the blocks and all the eigenvalues of the problematic blocks are plotted in Figures \ref{fig:max-eigenvalues-blocks-j1} and \ref{fig:eigenvalues-Q1}.

\begin{figure}[H]
 \includegraphics[scale=0.7]{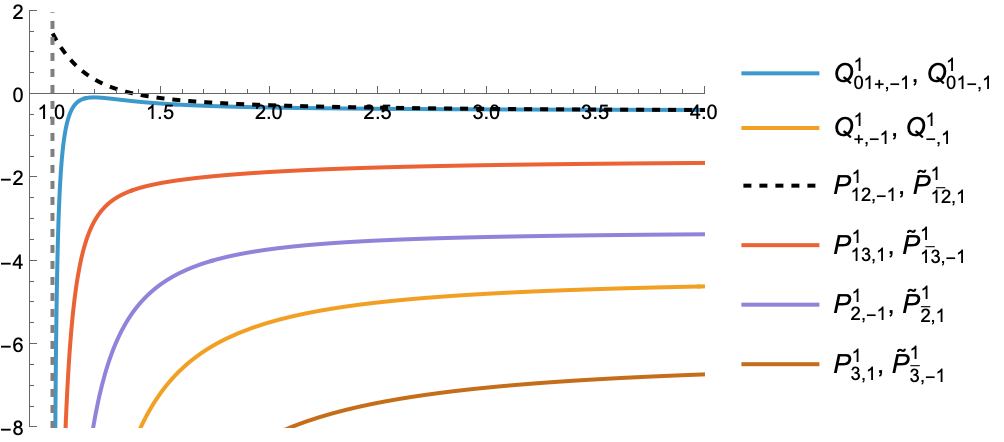}
   \caption{$J = 1$: The maximum eigenvalues (as functions of $r$) of the 12 blocks in the further block decomposition of $r^2 \mathcal{R}^1(r)$  plotted on the interval  $r \in [1, 4]$. The blocks $\mathcal{P}^1_{12, -1}, \tilde{\mathcal{P}}^1_{\bar{1}\bar{2}, 1}$ are the blocks with the positive eigenvalue.} \label{fig:max-eigenvalues-blocks-j1}
\end{figure}

\begin{figure}[H]
 \includegraphics[scale=0.7]{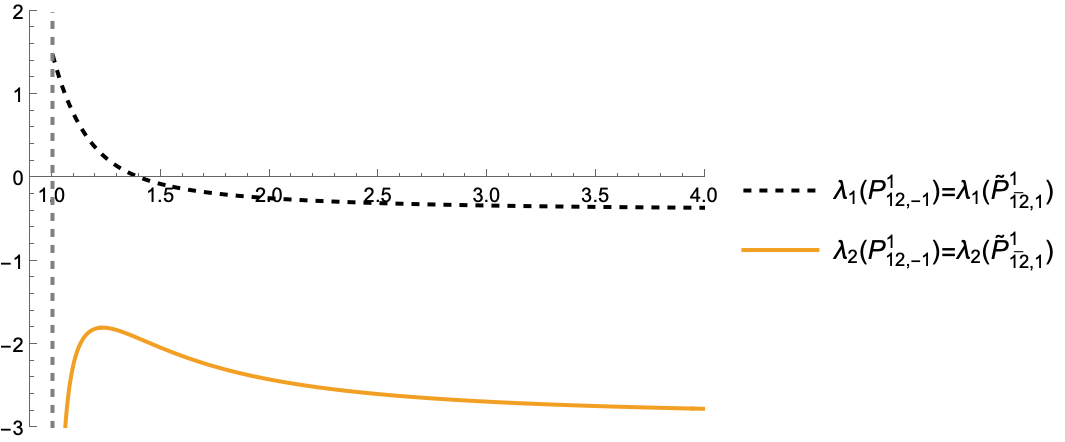}
   \caption{$J = 1$: The eigenvalues of $r^2\mathcal{P}^1_{12,-1}(r)\,,\, r^2 \tilde{\mathcal{P}}^1_{\bar{1}\bar{2}, 1}(r)$ (as functions of $r$) plotted on the interval $r \in [1, 4]$ with $\lambda_1 > \lambda_2$.} \label{fig:eigenvalues-Q1}
\end{figure}

\begin{proposition}\label{prop:J1-except-bad-block}
Let $h\in \mathcal{T}^1_{M'}$ and $M'\in\mathcal{J}$. Let $h_{p, M}$ and $k_{j, M}$ be the component functions of $h$. Let 
\begin{align*}
S &= k_{1,-1} D^1_{-1, M'} \mathbf{k}_1 + k_{2, 1} D^1_{1, M'} \mathbf{k}_2, & T &=  k_{\bar 1,1} D^1_{1, M'} \mathbf{k}_{\bar{1}} + k_{\bar{2}, -1} D^1_{-1, M'} \mathbf{k}_{\bar{2}}.
\end{align*}
Then 
\begin{align*}
\int_M\big(2R(h, \overline{h}) - |\nabla h|^2 ) e^{-f}d\mu_g &\leq \int_M (2 R(S, \overline{S}) - |\nabla S|^2) e^{-f} d\mu_g \\
& \qquad + \int_M (2 R(T, \overline{T}) - |\nabla{T}|^2) e^{-f} d\mu_g.
\end{align*}
\end{proposition}

\begin{proof}

The nonpositivity of any of the blocks outside of $\mathcal{P}^1_{12, -1}$ and $\tilde{\mathcal{P}}^1_{\bar 1 \bar 2, 1}$ ultimately comes down to the sign of a number of polynomials in $r^2$. By symmetries of the operator, one only needs to check the signs of three $1 \times 1$ blocks, one $2 \times 2$ block, and one $3 \times 3$ block. By (3) in Remark \ref{rem:sign-justification}, negativity of the curvature implies the $1 \times 1$ blocks all have a sign. We investigate the sign of the $3 \times 3$ block $\mathcal{Q}^1_{01+, - 1}$ -- hence also $\mathcal{Q}^1_{01-,1}(r)$ -- which are the most subtle blocks for $J = 1$. A similar investigation, or else a computation of eigenvalues, can be done for the $2 \times 2$ blocks.  See \cite{NO2} for explicit formulas for the blocks and also Figure \ref{fig:max-eigenvalues-blocks-j1}. 

Let us show in more detail the positivity of the $3\times 3$ matrix $B =-4s^2\mathcal{Q}^1_{01+, -1}(r)$. From the definition, one can show: 
\begin{align*}
B_{11} &= \frac{4 - 3\sqrt{2} - (6 - 4\sqrt{2})r^2 + (3 - \sqrt{2})r^4}{r^4},\\
B_{22} &= \frac{12(3 - 2\sqrt{2}) -20 (4 - 3\sqrt{2})r^2 + 17(4 - 3\sqrt{2})r^4 - 2(15 - 8\sqrt{2})r^6 + (7 - \sqrt{2})r^8}{r^8},\\
B_{33} &= \frac{5 - 3\sqrt{2} - 2(2 - \sqrt{2})r^2 + (3 + \sqrt{2})r^4}{r^4},\\
B_{12}=B_{21} &= 2(\sqrt{2}-1)\frac{(r^2-1)( r^2+ \sqrt{2} -1 )(r^2+\sqrt{2})}{r^6},\\
B_{13}=B_{31} &= 0,\\
B_{23} &= (1+i)2^{\frac{5}{4}}\frac{\big((r^2-1)( r^2+\sqrt{2} - 1 )\big)^{\frac{3}{2}}}{r^6},\\
B_{32} &= (1-i)2^{\frac{5}{4}}\frac{\big(( r^2-1)(r^2 + \sqrt{2} -1)\big)^{\frac{3}{2}}}{r^6}.
\end{align*}
Let $p_1 := B_{11}$, $p_2:= \mathrm{det}\begin{bmatrix} B_{11} & B_{12} \\ B_{21} & B_{22} \end{bmatrix}$, and $p_3:= \mathrm{det}(B)$. Then 
\begin{align*}
    r^4p_1& = 4 - 3\sqrt{2} - (6 - 4\sqrt{2})r^2 + (3 - \sqrt{2})r^4,  \\
    r^{12}p_2& = 4(38 - 27\sqrt{2} )
- 8(64 - 45\sqrt{2})r^2 
+ 2(349 - 246\sqrt{2})r^4- 8(62 -45\sqrt{2})r^6 \\
& \qquad 
+ 2(104 - 75\sqrt{2})r^8 
- 4(15 - 8\sqrt{2})r^{10} 
+ (11 - 2\sqrt{2})r^{12}, \\
 r^{16} p_3 & = 4(516 - 365\sqrt{2} )
- 8(1184 - 837\sqrt{2})r^2 
+ 2(9605 - 6791\sqrt{2})r^4 \\
& \qquad - 12(1882 - 1331\sqrt{2})r^6 
+ 2(8437 -5968\sqrt{2})r^8 
-4 (2049 - 1453\sqrt{2})r^{10} \\
&\qquad 
+ (2535 - 1789\sqrt{2})r^{12} 
-2 (236 - 153\sqrt{2})r^{14} 
+ (45 - 19\sqrt{2})r^{16}.
\end{align*}
For these identities, we have the following approximations:
\begin{align*}
    r^4 p_1& \approx -0.243 + 0.343r^2 + 1.586r^4, \\
   r^{12} p_2  & \approx -0.735 - 2.883 r^2 + 2.207 r^4 + 13.117 r^6 - 4.132 r^8 - 
 14.745 r^{10} + 8.172 r^{12}, \\
   r^{16} p_3 & \approx -0.752 - 2.426 r^2 + 2.151 r^4 + 3.819 r^6 - 6.053 r^8 +
 23.409 r^{10} + 4.972 r^{12} \\
 & \qquad - 39.251 r^{14} + 18.130 r^{16}. 
\end{align*}
It can be checked straightforwardly with computer assistance that $p_1, p_2, p_3 > 0$ on the interval $r \in [1, \infty)$. Indeed, for $r > 2$, it is clear that $p_1, p_2, p_3$ simply by positivity of the dominant constant term. For $r \in [1, 4]$, positivity can be verified to sufficient precision. It follows from Sylvester's criterion that the matrix $B$ is positive definite and hence that $\mathcal{Q}^1_{01+, -1}$ is negative definite for all $r\in (1, \infty)$.

Once nonpositivity of the blocks outside $\mathcal{P}^1_{12, -1}$ and $\tilde{\mathcal{P}}^1_{\bar 1 \bar 2, 1}$ is established, the theorem follows from Proposition \ref{prop:matrix-stability-identities} by throwing away the nonpositive blocks and noting that $\mathcal{P}^1_{12, -1}$ and $\tilde{\mathcal{P}}^1_{\bar 1 \bar 2, 1}$ act precisely upon the vector representations of $S$ and $T$. 
\end{proof}

\begin{figure}[H]
 \includegraphics[scale=0.7]{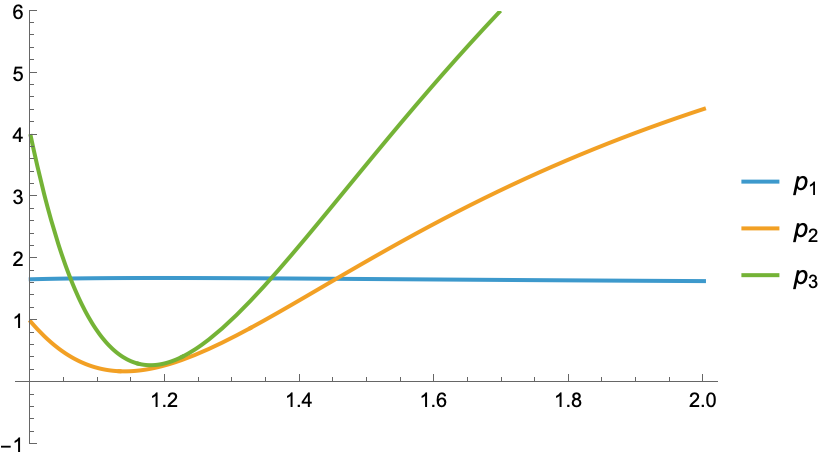}
   \caption{$J = 1$: The plots of $p_1, p_2, p_3$ from the proof of Proposition \ref{prop:J1-except-bad-block}. }
\end{figure}

\subsubsection{$J=2$}

For $J = 2$, there are 15 matrix blocks
\[
\mathcal{Q}^2_{01+-,0}, \; \mathcal{Q}^2_{01+, -2}, \; \mathcal{Q}^2_{01-, 2}, \;\mathcal{Q}^2_{+,-2},\; \mathcal{Q}^2_{-,2},
\]
\[
\mathcal{P}^2_{123, 0},\; \mathcal{P}^2_{12, -2},\;  \mathcal{P}^2_{13, 2}, \;\mathcal{P}^2_{2, -2}, \;\mathcal{P}^2_{3, 2}, \;\tilde{\mathcal{P}}^2_{\bar 1 \bar 2 \bar{3}, 0}, \;\tilde{\mathcal{P}}^2_{\bar 1 \bar 2, 2}, \;\tilde{\mathcal{P}}^2_{\bar 1 \bar 3, -2}, \;\tilde{\mathcal{P}}^2_{\bar 2, 2}, \;\tilde{\mathcal{P}}^2_{\bar 3, -2},
\]
which act upon the blocks of functions
\[
\begin{bmatrix}
  h_{0, 0} \\
  h_{1, 0} \\
  h_{+, 2} \\
  h_{-, -2}
\end{bmatrix},
\begin{bmatrix}
  h_{0, -2} \\
  h_{1, -2} \\
  h_{+, 0}
\end{bmatrix},
\begin{bmatrix}
  h_{0, 2} \\
  h_{1, 2} \\
  h_{-, 0}
\end{bmatrix},
\begin{bmatrix}
  h_{+, -2}
\end{bmatrix},
\begin{bmatrix}
  h_{-, 2}
\end{bmatrix},
\]
\[
\begin{bmatrix}
  k_{1, 0} \\
  k_{2, 2} \\
  k_{3, -2}
\end{bmatrix},
\begin{bmatrix}
  k_{1, -2} \\
  k_{2, 0}
\end{bmatrix},
\begin{bmatrix}
  k_{1, 2} \\
  k_{3, 0}
\end{bmatrix},
\begin{bmatrix}
  k_{2, -2}
\end{bmatrix},
\begin{bmatrix}
  k_{3, 2}
\end{bmatrix}, 
\begin{bmatrix}
  k_{\bar{1}, 0} \\
  k_{\bar{2}, -2} \\
  k_{\bar{3}, 2}
\end{bmatrix},
\begin{bmatrix}
  k_{\bar{1}, 2} \\
  k_{\bar{2}, 0}
\end{bmatrix},
\begin{bmatrix}
  k_{\bar{1}, -2} \\
  k_{\bar{3}, 0}
\end{bmatrix},
\begin{bmatrix}
  k_{\bar{2}, 2}
\end{bmatrix},
\begin{bmatrix}
  k_{\bar{3}, -2}
\end{bmatrix}.
\]
The only block which is not everywhere nonpositive is $\mathcal{Q}^2_{01+-, 0}$. The maximum eigenvalues of all the blocks and all the eigenvalues of the problematic block are plotted in Figures \ref{fig:max-eigenvalues-blocks-j2} and \ref{fig:eigenvalues-Q2}.

\begin{figure}[H]
 \includegraphics[scale=0.7]{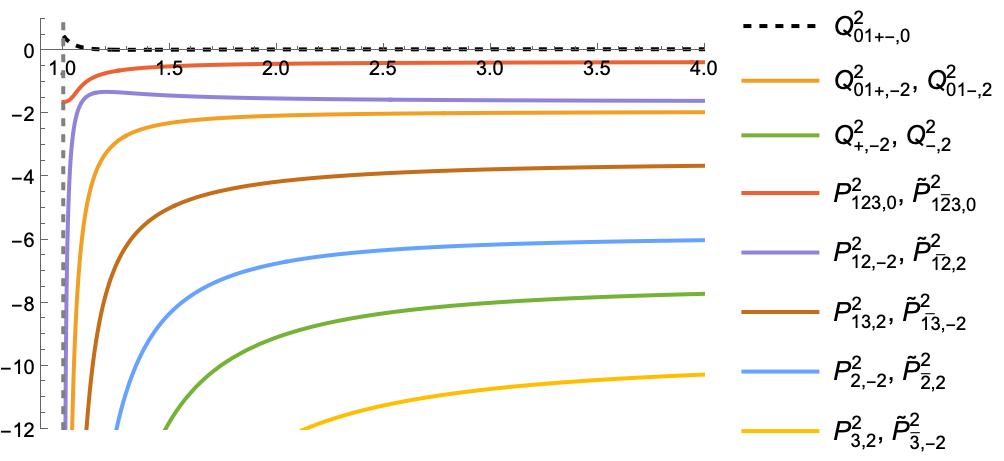}
   \caption{$J = 2$: The maximum eigenvalues (as functions of $r$) of the 15 blocks in the further block decomposition of $r^2\mathcal{R}^2(r)$ plotted on the interval  $r \in [1, 4]$. The block $\mathcal{Q}^2_{01+-, 0}$ is the block with the positive eigenvalue.} \label{fig:max-eigenvalues-blocks-j2}
\end{figure}

\begin{figure}[H]
 \includegraphics[scale=0.7]{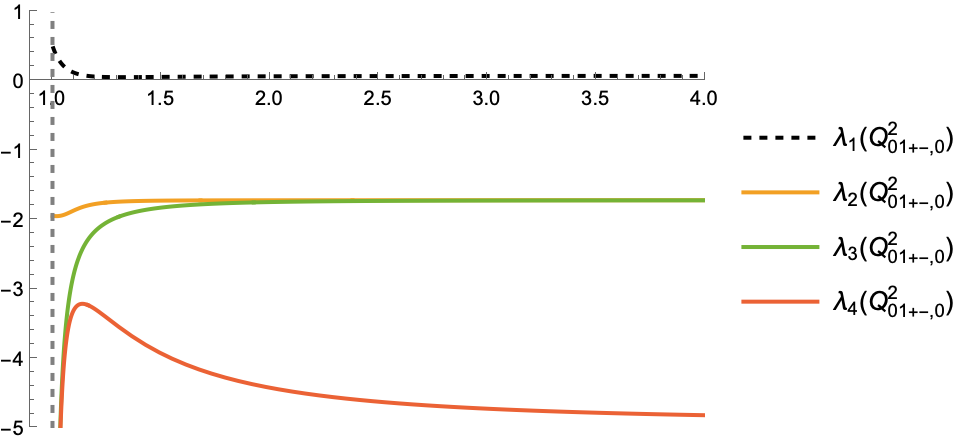}
   \caption{$J = 2$: The eigenvalues of $r^2\mathcal{Q}^2_{01+-,0}(r)$ (as functions of $r$) plotted on the interval $r \in [1, 4]$ with $\lambda_1 > \lambda_2 \geq \lambda_3 > \lambda_4$.} \label{fig:eigenvalues-Q2}
\end{figure}

\begin{proposition}\label{prop:J2-except-bad-block}
Let $h\in \mathcal{T}^2_{M'}$ and $M'\in\mathcal{J}$. Let $h_{p, M}$ and $k_{j, M}$ be the component functions of $h$. Let 
\begin{align*}
U &= h_{0, 0} D^2_{0, M'} \mathbf{b}_0 +h_{1, 0} D^2_{0, M'} \mathbf{b}_1+ h_{+, 2} D^2_{2, M'} \mathbf{b}_- + h_{-, -2} D^2_{-2, M'} \mathbf{b}_+
\end{align*}
Then
\[
\int_M\big(2R(h, \overline{h}) - |\nabla h|^2 ) e^{-f}d\mu_g \leq \int_M (2 R(U, \overline{U}) - |\nabla U|^2) e^{-f} d\mu_g.
\]
\end{proposition}

\begin{proof}
We can proceed as we did in Proposition \ref{prop:J1-except-bad-block} above for any of the blocks outside of $\mathcal{Q}^2_{01+-,0}$. By symmetries of the operator, one only needs to check the signs of three $1 \times 1$ blocks, two $2 \times 2$ blocks, and two $3 \times 3$ blocks. By (3) in Remark \ref{rem:sign-justification}, the $1 \times 1$ blocks all have a sign. See \cite{NO2} for explicit formulas for the blocks. See also Figure \ref{fig:max-eigenvalues-blocks-j2}.
\end{proof}

\subsection{Nonnegative eigenvalues of $L_f$ in the $J =0, 1, 2$ modes}\label{sec:nonnegative eigenvalues}

We have seen above that the matrices $\mathcal{R}^0, \mathcal{R}^1, \mathcal{R}^2$ have positive eigenvalues for some values of $r \in (1, \infty)$, stemming from positivity of the blocks $\mathcal{Q}^0_{01, 0}$, $\mathcal{P}^1_{12, -1}$, $\tilde{\mathcal{P}}^1_{\bar 1 \bar 2, 1}$, and $\mathcal{Q}^2_{01+-, 0}$. This observation suggests that $L_f$ has nonnegative eigenvalues on the corresponding spaces $\mathcal{T}^0 = \mathcal{T}^0_0$, $\mathcal{T}^1 = \mathcal{T}^1_{-1} \oplus \mathcal{T}^1_1$ and $\mathcal{T}^2 = \mathcal{T}^2_{-2} \oplus \mathcal{T}^2_{0} \oplus \mathcal{T}^2_{2}$. In this section, we verify the existence of positive eigenvalues of $L_f$, showing that each arises via the action of the gauge group.

\subsubsection{Nonnegative eigenvalues when $J = 0$}
For $J = 0$, we have found that one block, $\mathcal{Q}^0_{10,0}$, exhibits two eigenvalues with some positivity (the second eigenvalue in Figure \ref{fig:eigenvalues-Q0} becomes positive near $r = 1$), suggesting that $L_f$ admits two nonnegative eigentensors in $\mathcal{T}^0$.  We know one of these positive eigenvalues is due to the Ricci tensor, which satisfies 
\[
L_f \Ric = \Ric, \qquad \mathrm{div}_f (\Ric) = 0. 
\]
A deformation of the FIK shrinking soliton in the direction of $\Ric$ corresponds to simply rescaling the soliton. 
The other positive eigenvalue may be due to the fact that
\[
L_f \nabla^2 f = 0, \qquad \mathrm{div}_f(\nabla^2 f) = - \frac{1}{2} \nabla f,
\]
which corresponds to a deformation of the metric by a diffeomorphism. Both deformations satisfy $N_f \Ric = N_f \nabla^2 f = 0$, so neither is a genuine instability of the FIK metric. 

Since $\Ric$ and $\nabla^2f$ are both acted upon by the same block $\mathcal{Q}^0_{01,0}(r)$, there is not a straightforward relationship between the eigenvalues of $\mathcal{Q}^0_{01,0}(r)$ -- for any fixed values of $r$ -- and the presence of these tensors. However, by \eqref{eq:eigenvalues-LfJ-2}, whenever $L_f$ has a nonnegative eigenvalue for a tensor in a given block, acting by say $\Delta_f + \mathcal{B}$, we can be sure the corresponding matrix $\mathcal{B}(r)$ has a nonnegative eigenvalue for at least some values of $r$. In the present setting, the existence of either $\Ric$ or $\nabla^2 f$ suffices to explain why $\lambda_1(\mathrm{Q}^0_{01, 0})(r) > 0$, and the existence of both may be the reason that $\lambda_2(\mathcal{Q}^0_{01,0})(r) > 0$ for values of $r$ near $r = 1$. 

Finally, as $\mathcal{Q}^0_{01, 0}$ is a $2\times 2$ matrix, its eigenvalues and eigenvectors can be readily computed for any $r$. The resulting expressions, however, are complicated and have little to do with the directions determined by $\Ric$ or $\nabla^2 f$.

\subsubsection{Nonnegative eigenvalues when $J = 1, 2$}
For $J = 1$, we discovered 4 blocks (two for each choice of $M'$) with the same positive eigenvalue (as a function of $r$). Thus, we expect $L_f$ to have nonnegative eigenvalue on $\mathcal{T}^1$ with a 4-dimensional eigenspace. For $J = 2$, we discovered 3 blocks (one for each choice of $M'$) with the same positive eigenvalue (as a function of $r$). Thus, we expect $L_f$ to have a nonnegative eigenvalue on $\mathcal{T}^2$ with a 3-dimensional eigenspace. Here we verify both of these expectations and show that these nonnegative eigenvalues of $L_f$ arise as Hessians of functions. Consequently, they correspond to deformations of the FIK metric by diffeomorphisms. 

First, let us recall general identities (cf. \cite{CM}) which hold on any gradient Ricci shrinker $(M, g, f)$:  
\begin{align}
    L_f (\mathcal{L}_Y g) &= \mathcal{L}_{\Delta_f Y + \frac{1}{2} Y} g, \\
    \mathrm{div}_f(\mathcal{L}_Yg) & = \Delta_f Y + \frac{1}{2} Y + \nabla \mathrm{div}_f(Y).
\end{align}
Specializing to the case $Y = \nabla u$ and using the weighted Bochner formula, gives 
\begin{align}
    L_f (\nabla^2 u) &= \nabla^2 (\Delta_f u + u),\label{eq:Lf-Hessians} \\
    \mathrm{div}_f(\nabla^2 u ) & = \nabla (\Delta_f u + \frac{1}{2} u).
\end{align}
In particular, we may discover eigentensors of $L_f$ via the Hessians of non-constant eigenfunctions of $\Delta_f$. 

Next, let us recall that for any Wigner function $D^J_{M, M'}$, we have 
\[
\Delta_f D^J_{M, M'} = \tilde{\Delta}^J_{M}(r)  \, D^J_{M, M'}, \qquad \tilde{\Delta}^J_M(r) = \frac{M^2 - J(J+2)}{4r^2} - \frac{M^2}{4s^2} \leq 0. 
\]
In particular, if $u = u(r)$ is a radial function, then we have 
\[
\Delta_f (u D^J_{M, M'}) = \Big(\Delta_fu + \tilde{\Delta}^J_M(r)u\Big)D^J_{M, M'}.
\]
We conclude that whenever we can find a function $u$ such that $\Delta_f u + \tilde{\Delta}^J_{M'}(r) u = \lambda u$ with $\lambda \geq -1$, then the 2-tensor $S = \nabla^2(u D^J_{M, M'})$ will satisfy 
\[
L_f S = (1+ \lambda) S,
\]
yielding a nonnegative eigenvalue of $L_f$. For $J \in\{0, 1, 2\}$, $\tilde{\Delta}^J_M$ is given by  
\[
\tilde{\Delta}^0 = 0, \qquad \tilde{\Delta}^1_{\pm1} = -\frac{1}{4s^2} - \frac{1}{2r^2}, \qquad \tilde{\Delta}^2_{0} = -\frac{2}{r^2}, \qquad \tilde{\Delta}^2_{\pm2}= - \frac{1}{s^2} -\frac{1}{r^2}. 
\]
Now let us show when acting on radially symmetric functions the operators $\Delta_f + \tilde{\Delta}^1_{\pm1}(r) $ and $\Delta_f + \tilde{\Delta}^2_0(r)$ have eigenfunctions with eigenvalues $\lambda \geq -1$. We will also point out that the first eigenvalue of the operator $\Delta_f + \tilde{\Delta}^2_{\pm 2}(r)$ is $-\sqrt{2} < -1$. 

Define functions 
\begin{equation}
\hat{u}:= (r^2 -1)^{\frac{1}{2}}(r^2 + c_0)^{\frac{c_0}{2}},\qquad \hat{v} :=  r^2, \qquad \hat{w}:= \hat{u}^2. 
\end{equation}
These functions satisfy 
\begin{align*}
\hat{u}' = \frac{1}{rF} \hat{u}, 
\qquad \hat{v}'  =2r = \frac{1}{r} \hat{v},  \implies e_0(\hat{v}) = \frac{s}{2r^2} \hat{v}.
\qquad \hat{w}' = \frac{2}{rF}\hat{w}. 
\end{align*}
Consequently, we compute that
\[
\Delta_f \hat{u} = \frac{e^f}{4r^3}\Big(\frac{r^3F}{e^f} \hat{u}'\Big)' =  \frac{e^f}{4r^3}\Big(\frac{r^2}{e^f}\hat{u} \Big)' = \Big(\frac{1}{4s^2} +\frac{1}{2r^2}-\frac{f'}{4r}\Big)\hat{u}  =  \Big(\frac{1}{4s^2} +\frac{1}{2r^2}-\frac{1}{\sqrt{2}}\Big)\hat{u}.
\]
and -- see \eqref{eq:soliton-eq} -- that
\[
\Delta_f \hat{v} = \frac{e^f}{4r^3}\Big(\frac{r^3F}{e^f} \hat{v}'\Big)' =  \frac{rF}{2}\frac{e^f}{r^4F}\Big(\frac{Fr^4}{e^f} \Big)' = \frac{4-2r^2}{2} = \Big(\frac{2}{r^2} - 1\Big)\hat{v}^2.
\]
Similarly, 
\[
\Delta_f \hat{w} = \Big( \frac{1}{s^2} + \frac{1}{r^2} - \sqrt{2}\Big) \hat{w}. 
\]
We conclude 
\[
\big(\Delta_f + \tilde{\Delta}^1_{\pm 1}\big) \hat{u} = -\frac{1}{\sqrt{2}}\,\hat{u}, \qquad \big(\Delta_f + \tilde{\Delta}^2_0\big) \hat{v} = - \hat{v}, \qquad (\Delta_f + \tilde{\Delta}^2_{\pm2})\hat{w} = -\sqrt{2} \,\hat{w}. 
\]
Moreover, note that since the functions $\hat{u}, \hat{v}, \hat{w}$ are all nonnegative, they are the first eigenfunctions of the respective operators. Multiplying by the appropriate Wigner functions gives
\[
\Delta_f(\hat{u} D^1_{M, {M'}}) = - \frac{1}{\sqrt{2}} (\hat{u} D^1_{M, M'}), \qquad \Delta_f(\hat{v} D^2_{0, M'}) = - (\hat{v} D^2_{0, M'}), 
\]
and 
\[
\Delta_f(\hat{w} D^2_{\pm 2, M'}) = -\sqrt{2}\big(\hat{w} D^2_{\pm 2, M'}). 
\]
As eigenfunctions, $\hat{u} D^J_{M, M'}, \hat{v} D^2_{0, M'}, \hat{w} D^2_{\pm 2, M'} \in H^k_f(M)$ are smooth including at the tip of the soliton. We note however that both $\hat{u}$ and $\hat{w}$ need to vanish at $r =1$, as they do, because $D^1_{M, M'}$ and $D^2_{\pm 2, M'}$ do not converge smoothly as the Hopf fiber collapses (because $M \neq 0$). On the other hand, $D^2_{0, M'}$ is smooth as the Hopf fiber collapses, so $\hat{v}$ does not and need not vanish.

In any event, by the above it now follows that 
\begin{equation}
S^1_{M, M'} := \nabla^2 (\hat{u} D^1_{M, M'})
\end{equation}
for $M, M' \in \{-1, 1 \}$ satisfies 
\begin{equation}
L_f (S^1_{M, M'}) = \frac{c_0}{\sqrt{2}} (S^1_{M, M'}), \qquad \mathrm{div}_f (S^1_{M, M'} ) = - \frac{c_0}{2} \nabla \big(\hat{u} D^1_{M, M'} \big),
\end{equation}
which yields a 4-dimensional space of $\frac{c_0}{\sqrt{2}}$-eigentensors of $L_f$ lying in $\mathcal{T}^1$. Similarly 
\begin{equation}
S^2_{0, M'} := \nabla^2 (\hat{v} D^2_{0, M'}) 
\end{equation}
for $M' \in \{-2, 0, 2\}$ satisfies
\begin{equation}
L_f(S^2_{0, M'}) = 0, \qquad \mathrm{div}_f(S^2_{0, M'}) = -\frac{1}{2} \nabla \big(\hat{v} D^2_{0, M'}\big), 
\end{equation}
which yields a 3-dimensional space of $0$-eigentensors of $L_f$ lying in $\mathcal{T}^2$. We have also found an analogously defined 6-dimensional space of $(-c_0)$-eigentensors of $L_f$ in $\mathcal{T}^2$ spanned by $S^2_{\pm 2, M'} = \nabla^2 (\hat{w} D^J_{\pm 2, M'})$. Of course, in this presentation, these are complex subspaces, but real-valued deformations are obtained by taking the real parts. The existence of the tensors $S^1_{M, M'}$ and $S^2_{0, M'}$ explain the presence of positive eigenvalues in the blocks of $\mathcal{R}^1$ (specifically $\mathcal{P}^1_{12, -1}$ and $\tilde{\mathcal{P}}^1_{\bar 1\bar 2, 1}$) and $\mathcal{R}^2$ (specifically $\mathcal{Q}^2_{01+-, 0}$). Since these are Hessians of functions, they correspond to deformations of the FIK metric by diffeomorphisms and therefore lie in the kernel of the full stability operator $N_f$ of the Perelman entropy. 

\subsubsection{Expressions for $S^1_{M, M'}$ and $S^2_{0, M'}$}

It is straightforward, albeit a bit computational, to derive formulas for the $S^1_{M, M'}$ and $S^2_{0, M'}$ in our complex basis of $2$-tensors. We shall not need these explicit expressions, but the interested reader may verify that
\begin{align}
S^1_{1, M'} &= -(1+i) \frac{s}{r} \Big(\frac{\hat{u}}{r}\Big)' D^1_{-1, M'} \mathbf{k}_{1} + \frac{s}{r}\Big(\frac{\hat{u}}{s}\Big)' D^1_{1,M'} \mathbf{k}_{2}, \\
S^1_{-1, M'} &= (1-i) \frac{s}{r} \Big(\frac{\hat{u}}{r}\Big)' D^1_{1, M'} \mathbf{k}_{\bar{1}} + \frac{s}{r}\Big(\frac{\hat{u}}{s}\Big)' D^1_{-1,M'} \mathbf{k}_{\bar{2}},
\end{align}
and 
\begin{align}
    S^2_{0, M'} &=-c_0\frac{1 +\sqrt{2}r^2}{r^2} D^2_{0, M'} \mathbf{b}_0  + \Big(2 + c_0 \frac{r^2 + \sqrt{2}}{r^4}\Big)D^2_{0, M'} \mathbf{b}_0\\
    & \qquad -\sqrt{2}(1-i)\frac{s}{r} D^2_{2,M'} \mathbf{b}_{-}+\sqrt{2} (1+i)\frac{s}{r} D^2_{-2, M'} \mathbf{b}_+. \nonumber
\end{align}
To verify these expressions, one checks that 
\[
\hat S^1_1:= \frac{s}{r}\begin{bmatrix} -(1+i)  \big(\frac{\hat{u}}{r}\big)' \\ \big(\frac{\hat{u}}{s}\big)' \end{bmatrix}, \qquad \hat S^1_{-1}:= \frac{s}{r}\begin{bmatrix}(1-i) \big(\frac{\hat{u}}{r}\big)' \\ \big(\frac{\hat{u}}{s}\big)' \end{bmatrix}, \qquad \hat S^2_0 := \begin{bmatrix}
-c_0\frac{1 +\sqrt{2}r^2}{r^2}\\ 2 + c_0 \frac{r^2 + \sqrt{2}}{r^4} \\-\sqrt{2}(1-i)\frac{s}{r}\\\sqrt{2} (1+i)\frac{s}{r} 
\end{bmatrix},
\]
satisfy
\[
\big(\Delta_f + \mathcal{P}^1_{12,-1}\big) \hat S^1_1 = \frac{c_0}{\sqrt{2}} \hat S^1_1 \qquad  \big(\Delta_f + \mathcal{P}^1_{\bar 1 \bar 2,1}\big) \hat S^1_{-1} = \frac{c_0}{\sqrt{2}} \hat S^1_{-1},\qquad (\Delta_f + \mathcal{Q}^2_{01+-,0}) \hat{S}^2_0 = 0, 
\]
using the expressions for $\mathcal{P}^1_{12, -1}$, $\tilde{\mathcal{P}}^1_{\bar 1 \bar 2, 1}$, and $\mathcal{Q}^2_{01+-,0}$ given in the next section.

\section{Addressing positivity for $J \in \{1, 2\}$ via reduction to a Sturm-Liouville problem}\label{sec:P1Q2}

In this section, we take $J \in \{1, 2\}$. For each choice of $J$, we fix some $M' \in \mathcal{J}$. In what follows, let us introduce the shorthand 
\[
\mathcal{P} = \mathcal{P}^1_{12, -1}, \qquad \tilde{\mathcal{P}} = \tilde{\mathcal{P}}^1_{\bar 1 \bar 2, 1}, \qquad \mathcal{Q} := \mathcal{Q}^2_{01+-,0}.
\]
In this section, we address the nonnegativity we have discovered for the operators $\Delta_f + \mathcal{B}$ for $\mathcal{B} \in \{\mathcal P, \tilde{\mathcal P}, \mathcal{Q}\}$. This is done by showing that the second eigenvalue of $\Delta_f + \mathcal{B}$  is negative by comparing the operator to a 1-dimensional Sturm-Liouville operator. This will show linear stability of the FIK shrinking soliton for $J \in \{1, 2\}$ and complete our proof of linear stability. 

By the formulas obtained in Section \ref{sec:further block decomposition}, we begin by noting that 
\begin{align*}
  \mathcal{P} = \begin{bmatrix}
\frac{36 - 24\sqrt{2} + (-44 + 36\sqrt{2})r^2 + (21 - 13\sqrt{2})r^4 + (-7 + \sqrt{2})r^6}{2 \sqrt{2}r^6 (r^2+c_0)}
&
 -(1 + i) \frac{s}{r^3}
\\[1.5em]
 -(1 - i)\frac{ s}{r^3}
&
\frac{-2 + \sqrt{2} - 2(-3 + \sqrt{2})r^2 + (-5 + \sqrt{2})r^4}{2 \sqrt{2}r^2 (r^2-1)(r^2 + c_0)}
\end{bmatrix}, 
\end{align*}
$\tilde{\mathcal{P}}$ is the same, except its off-diagonal elements are conjugated and multiplied by $-1$, and 
\begin{align*}
\mathcal{Q} =\begin{bmatrix}
-\frac{\sqrt{2}+1}{\sqrt{2} r^{2}} 
& -\frac{c_0(r^2 + \sqrt{2})}{\sqrt{2} r^{4}} 
& 0 
& 0 \\
-\frac{c_0(r^2 + \sqrt{2})}{\sqrt{2} r^{4}} 
& \frac{6\sqrt{2}c_0 + 8c_0 r^{2} -(2 + 3\sqrt{2}) r^{4}}{2 r^{6}} 
& -(1+i)\frac{\sqrt{2}s}{r^3} 
& (1-i)\frac{\sqrt{2}s}{r^3} \\
0 
& -(1-i)\frac{\sqrt{2}s}{r^3} 
& \frac{1 + 2\sqrt{2} r^{2} - 2(1+\sqrt{2}) r^{4}}{2 \sqrt{2}r^{2} (r^{2}-1)(r^2 + c_0)} 
& 0 \\
0 
& (1+i)\frac{\sqrt{2}s}{r^3} 
& 0 
& \frac{1 + 2\sqrt{2} r^{2} - 2(1+\sqrt{2}) r^{4}}{2\sqrt{2} r^{2} (r^{2}-1)(r^2 + c_0)}
\end{bmatrix}.
\end{align*}
The following is the key result that enables comparison to a Sturm-Liouville operator.

\begin{proposition}
For any $r \in (1, \infty)$, we have 
\[
 \mathcal{P}(r), \tilde{\mathcal{P}}(r) \leq \frac{3}{2} r^{-2} \begin{bmatrix} 1 & 0 \\ 0 & 0\end{bmatrix},  \qquad \mathcal{Q}(r)\leq  \frac{3}{2} r^{-2} \begin{bmatrix} 0 & 0 & 0 & 0 \\ 0 & 1 & 0 & 0 \\ 0 & 0 & 0 & 0 \\ 0 & 0 & 0 & 0 \end{bmatrix}.
\]

\end{proposition}
\begin{proof}
The proof of the estimate for $\tilde{\mathcal P}$ is the same as the proof for $\mathcal P$, so let us focus on the latter. To establish the first estimate, we begin with by proving that any $r \in [1, \infty)$,
\[
\mathcal{P}_{11} + \frac{4s^2}{r^4} \leq \frac{3}{2} r^{-2}, \qquad \mathcal{P}_{22} + \frac{1}{2r^2} < 0. 
\]
A computation gives that 
\begin{align*}
r^2\Big( \mathcal{P}_{11} + \frac{4s^2}{r^4}\Big) &= \frac{4c_0(r^2+c_0) + 
3 c_0 r^4 + c_0^{-1}r^6}{2\sqrt{2} r^4(r^2 + c_0)} =\frac{\sqrt{2}c_0}{r^4} +\frac{c_0^2}{2(r^2 + c_0)} + \frac{1}{2\sqrt{2}c_0}, \\
r^2\Big( \mathcal{P}_{22} + \frac{1}{2r^2}\Big)& = \frac{2 (c_0-1) - 
4 (c_0-1) r^2 + (2c_0-3) r^4}{2\sqrt{2} (r^2-1) (r^2+c_0)}. 
\end{align*}
From the first line, it is evident that $r^2(\mathcal{P}_{11} + \frac{4s^2}{r^4})$ is monotone decreasing for $r \in [1, \infty)$. At $r = 1$, we have $r^2(\mathcal{P}_{11} + \frac{4s^2}{r^4})= \sqrt{2}c_0  + \frac{c_0^2}{2\sqrt{2}} + \frac{1}{2\sqrt{2}c_0}= \frac{3}{2}$. We conclude $r^2(\mathcal{P}_{11} + \frac{4s^2}{r^4}) \leq \frac{3}{2}$. From the second line, it is straightforward to check that $2 (c_0-1) - 
4 (c_0-1) r^2 + (2c_0-3) r^4 < 0$, by checking that the roots of the associated quadratic lie to the left of $r = 1$. 

Now let $\xi_1 = (1,0)$, $\xi_2 = (0, 1)$, and $\xi = u_1 \xi_1 + u_2 \xi_2 \in \mathbb{C}^2$. With the above estimates in hand, we show $\mathcal{P} \xi \cdot \overline{\xi} \leq \frac{3}{2} r^{-2} |u_1|^2$. First note Cauchy's inequality gives
\[
2\frac{\sqrt{2}s}{r^3}|u_1||u_2| \leq \frac{4s^2}{r^4}|u_1|^2 + \frac{1}{2r^2}|u_2|^2.
\]
Then, using the expression for $\mathcal{P}$ above, we have
\begin{align*}
\mathcal{P} \, \xi \cdot\overline{ \xi} & = \mathcal{P}_{11} |u_1|^2 +\mathcal{P}_{22} |u_2|^2 -\frac{\sqrt{2}s}{r^3}\Big(\frac{1-i}{\sqrt{2}} u_1 \overline{u_2} +\frac{1+i}{\sqrt{2}} \overline{u_1} u_2\Big),\\
& \leq \mathcal{P}_{11} |u_1|^2 + \mathcal{P}_{22} |u_2|^2 + 2\sqrt{2}\frac{s}{r^3}|u_1||u_2|,\\
& \leq \Big(\mathcal{P}_{11} + \frac{4s^2}{r^4}\Big)  |u_1|^2 + \Big(\mathcal{P}_{22} + \frac{1}{2r^2}\Big) |u_2|^2 \\
&\leq \frac{3}{2} r^{-2} |u_1|^2.
\end{align*}

Next, we prove the estimate for $\mathcal{Q}$. Let us consider the change of basis: 
\[
(\mathbf{b}_0, \mathbf{b}_1, \mathbf{b}_-, \mathbf{b}_+) \to (\mathbf{b}_1, \mathbf{b}_0, \mathbf{b}_-, \mathbf{b}_+), \qquad \mathcal{S} := \begin{bmatrix} 0 & 1 & 0 & 0\\
1 & 0 & 0 & 0 \\
0 & 0 & 1 & 0  \\
0 & 0 & 0 & 1\end{bmatrix}. 
\]
Then, from the expression for $\mathcal{Q}$ above, one finds
\[
\mathcal{S}^{-1} \mathcal{Q}\mathcal{S} := r^{-2}\begin{bmatrix} a & b \\
b^\dagger & C \end{bmatrix}
\]
with 
\begin{align*}
    a &:= \frac{6\sqrt{2}c_0+ 8 c_0 r^2 - (2 + 3 \sqrt{2}) r^4}{2 r^4}, \\
    b &:= \begin{bmatrix} \frac{2 - 2 \sqrt{2} + (-2 + \sqrt{2}) r^2}{2 r^2} & -(1+i) \frac{\sqrt{2}s}{r} &  (1-i) \frac{\sqrt{2}s}{r}\end{bmatrix},\\
    C &:= \begin{bmatrix} -\frac{2+\sqrt{2}}{2} & 0 & 0  \\
    0 & d  & 0 \\ 
    0 & 0 & d
    \end{bmatrix}, \\
  d & : =  \frac{1 + 2 \sqrt{2} r^2 - 
 2 (1 + \sqrt{2}) r^4}{2\sqrt{2} (r^2-1) (r^2 + c_0)}.
\end{align*}
Since $d < 0$, we have that $C < 0$ is invertible. We want to show that 
\[
\begin{bmatrix} a & b  \\
b^\dagger & C  \end{bmatrix} \leq \frac{3}{2} \begin{bmatrix} 1 & 0 \\ 0 & 0 \end{bmatrix}. 
\]
Using the Schur complement, this inequality holds if and only if
\[
a - b^\dagger C^{-1} b \leq \frac{3}{2}. 
\]
By direction computation 
\[
\alpha(r):= a - b^\dagger C^{-1} b = \frac{2 (10 - 6 \sqrt{2} + 5 (-4 + 3 \sqrt{2}) r^2 + (17 - 13 \sqrt{2}) r^4 + 
   4 (-2 + \sqrt{2}) r^6 + 2 r^8)}{(2 + \sqrt{2}) r^4 ( 2 (1 + \sqrt{2}) r^4 - 
   2 \sqrt{2} r^2-1)}.
\]
In fact, the maximum of $\alpha(r)$ on $[1, \infty)$ occurs at $r = 1$ and $\alpha(1) = 2 -\sqrt{2}< 0.59 < \frac{3}{2}$. So there is room to spare in our estimate (owing to the fact that $\mathcal{Q}$ is a fair bit less positive than $\mathcal{P}$ and $\tilde{\mathcal{P}}$). We do not need a sharp estimate, so to finish the proof we may instead note that $2(1+\sqrt{2})r^4 - 2\sqrt{2} r^2 - 1 \geq r^4$ for $r \in [1, \infty)$ and thus that 
\begin{align*}
\alpha(r) &\leq \frac{2 (10 - 6 \sqrt{2} + 5 (-4 + 3 \sqrt{2}) r^2 + (17 - 13 \sqrt{2}) r^4 + 
   4 (-2 + \sqrt{2}) r^6 + 2 r^8)}{(2 + \sqrt{2}) r^8 } \\
   & \leq 1.18 - 1.38r^{-2} - 0.82 r^{-4} + 0.72 r^{-6} + 0.89 r^{-8} <1.18 <  \frac{3}{2}. 
\end{align*}
This completes the proof. 
\end{proof}

\begin{corollary}\label{lem:1D-reduction}
    If $\xi(r) = (v_1(r), v_2(r)) \in \mathbb{C}^2$ and $\mathcal{B} \in\{\mathcal{P}, \tilde{\mathcal{P}}\} $, then 
    \[
   \mathcal{B} \xi \cdot \overline{\xi} - \frac{F}{4} |\xi'|^2 \leq \frac{3}{2} r^{-2} |v_1|^2 - \frac{F}{4} |v_1'|^2.
    \]
    If $\chi(r) = (v_1(r), v_2(r), v_3(r), v_4(r)) \in \mathbb{C}^4$, then 
    \[
    \mathcal{Q} \chi \cdot \overline{\chi} - \frac{F}{4}|\chi'|^2 \leq \frac{3}{2} r^{-2}|v_2|^2 - \frac{F}{4} |v_2'|^2.
    \]
\end{corollary}

The corollary shows that we can compare the operators $\Delta_f + \mathcal{B}$ for $\mathcal{B} \in \{\mathcal P, \tilde{\mathcal P}, \mathcal{Q}\}$ to the operator $\Delta_f + \frac{3}{2} r^{-2}$ on radially symmetric functions. In the following, let us denote by $H^1_f(M)_{\mathrm{rad}} \subset H^1_f(M)$, the set of functions in $H^1_f(M)$ that are radially symmetric (i.e. the closure of radially symmetric functions in $C^{\infty}_0(M)$ with respect to the Sobolev norm).

\begin{lemma}\label{lem: Rayleigh deltaf 32r}
    There exist smooth (real-valued) radially symmetric functions $a_1, a_2 \in H^1_f(M)_{\mathrm{rad}}$, satisfying $|a_1|_{L^2_f(M)} = |a_2|_{L^2_f(M)} = 1$, and real numbers $\tilde{\mu}_1> \tilde{\mu}_2$ such that 
    \[
    \big(\Delta_f + \frac{3}{2} r^{-2}\big)a_1 =\tilde \mu_1 a_1, \qquad \big(\Delta_f + \frac{3}{2} r^{-2}\big)a_2 = \tilde\mu_2 a_2,
    \]
    additionally satisfying 
    \[
    \sup_{a \in H^1_f(M)_{\mathrm{rad}}} \frac{\int_1^\infty \big(\frac{3}{2} r^{-2} |a|^2 - \frac{F}{4} |a'|^2 \big)r^3 e^{-f}\, dr }{\int_1^\infty|a|^2 r^3 e^{-f} \, dr} \leq \tilde\mu_1, 
    \]
    and 
    \[
    \sup_{\substack{a \in H^1_f(M)_{\mathrm{rad}}\\a\perp a_1} }\frac{\int_1^\infty \big(\frac{3}{2} r^{-2} |a|^2 - \frac{F}{4} |a'|^2 \big)r^3 e^{-f}\, dr }{\int_1^\infty|a|^2 r^3 e^{-f} \, dr} \leq \tilde\mu_2. 
    \]
    Here $a \perp a_1$ means orthogonality with respect to the $L^2_f(M)$ inner product.
\end{lemma}

\begin{proof}
    It is known that $H^1_f(M) \hookrightarrow L^2_f(M)$ is a compact embedding under quite general assumptions (e.g. if one has a log-Sobolev inequality \cite{cz17}). The compact embedding $H^1_f(M) \hookrightarrow L^2_f(M)$ restricts to a compact embedding $H^1_f(M)_{\mathrm{rad}} \hookrightarrow L^2_f(M)_{\mathrm{rad}}$ as $H^1_f(M)_{\mathrm{rad}}$ is a closed subspace of $H^1_f(M)$. It follows that $\Delta_f$ has a discrete spectrum on $H^1_f(M)_{\mathrm{rad}}$. As $V(r)=\frac{3}{2}r^{-2}$ is a smooth uniformly bounded potential, we conclude the operator $\Delta_f + \frac{3}{2} r^{-2}$ has a discrete spectrum on $H^1_f(M)_{\mathrm{rad}}$ as well.
\end{proof}

Following the strategy of Lemma \ref{lem:barta1}, we apply Barta's trick to our operator $\Delta_f+\frac{3}{2r^2}$ on sub-intervals of $(1,+\infty)$. 

\begin{lemma}\label{lem:barta deltaf}
    Assume that on an interval $(r_1,r_2)\subset (1,\infty)$, one has smooth functions $a,\phi:(r_1,r_2)\to \mathbb{R}$ with $a\neq0$ and $\phi>0$, 
    \begin{equation}\label{eq:condition phi barta deltaf}
        \left(\Delta_f +\frac{3}{2}r^{-2}\right)\phi<0,
    \end{equation}
    and 
    \begin{equation}\label{eq:bdry cond barta deltaf}
        \lim_{r \to r_2} F(r)\frac{\phi'(r)}{\phi(r)}  a(r)^2 r^3 e^{-f(r)} - \lim_{r \to r_1} F(r)\frac{\phi'(r)}{\phi(r)}  a(r)^2 r^3 e^{-f(r)}=0.
    \end{equation}
    Then, we have 
    $$ \int_{r_1}^{r_2} \left(-\frac{F}{4}(a')^2+\frac{3}{2r^2}  a^2\right)r^3e^{-f} dr < 0. $$  
\end{lemma}
    \begin{proof}
Define $w$ through $a = \phi w$. We start from the identity
\begin{equation}\label{eq:barta-temp-1-deltaf}
\left(\phi' \phi w^2 F r^3 e^{-f}\, \right)' - \phi' (\phi w^2)' F r^3 e^{-f}  = \left(\phi' F r^3 e^{-f}\right)' \phi w^2,
\end{equation}
where one recognizes $\left(\phi' F r^3 e^{-f}\right)' = 4\Delta_f\phi \,\, r^3 e^{-f}$. Note also that $\phi' \phi w^2 = \phi' \phi^{-1} a^2$. 

Integrating \eqref{eq:barta-temp-1-deltaf} over $(r_1,r_2)\subset(1,\infty)$ gives
\begin{align*}
&\int_{r_1}^{r_2} \left( \phi' (\phi w^2)' F r^3 e^{-f} + \left(\phi' F r^3 e^{-f}\right)' \phi w^2\right) dr\\
&=\int_{r_1}^{r_2} \left(\phi' \phi w^2 F r^3 e^{-f}\, \right)' \, dr \\
&= \lim_{r \to r_2} F(r)\frac{\phi'(r)}{\phi(r)}  a(r)^2 r^3 e^{-f(r)} - \lim_{r \to r_1} F(r)\frac{\phi'(r)}{\phi(r)}  a(r)^2 r^3 e^{-f(r)}. 
\end{align*}
Now, from $a = \phi w$, we compute $(a')^2 = \phi'(\phi w^2)' + \phi^2 (w')^2$ and rewrite
\begin{align*}
&\int_{r_1}^{r_2} \left((a')^2-\phi^2(w')^2 + 4  w^2\phi\Delta_f\phi\right) F r^3 e^{-f}dr\\
&=\int_{r_1}^{r_2} \left( \phi' (\phi w^2)' F r^3 e^{-f} + \left(\phi' F r^3 e^{-f}\right)' \phi w^2\right) dr\\
&= \lim_{r \to r_2} F(r)\frac{\phi'(r)}{\phi(r)}  a(r)^2 r^3 e^{-f(r)} - \lim_{r \to r_1} F(r)\frac{\phi'(r)}{\phi(r)}  a(r)^2 r^3 e^{-f(r)}. 
\end{align*}
In particular, if $\lim_{r \to r_2} F(r)\frac{\phi'(r)}{\phi(r)}  a(r)^2 r^3 e^{-f(r)} - \lim_{r \to r_1} F(r)\frac{\phi'(r)}{\phi(r)}  a(r)^2 r^3 e^{-f(r)}=0$, then
\begin{equation}\label{eq:towards barta deltaf}
    \int_{r_1}^{r_2} (a')^2 F r^3 e^{-f}dr = \int_{r_1}^{r_2} \left(\phi^2(w')^2 - 4  w^2\phi\Delta_f\phi\right) F r^3 e^{-f}dr,
\end{equation}
which lets us conclude by multiplying \eqref{eq:towards barta deltaf} by $-\frac14$ and adding $\frac{3}{2r^2}a^2 = \frac{3}{2r^2}\phi^2w^2$ to the integrands:
\[
\int_{r_1}^{r_2} \Big(-\frac{F}{4} (a')^2 + \frac{3}{2r^2} a^2 \Big)r^3 e^{-f}dr = \int_{r_1}^{r_2} \left(-\frac{1}{4}\phi^2(w')^2 +  w^2\phi\Big(\Delta_f\phi + \frac{3}{2r^2} \phi\Big) \right) F r^3 e^{-f}dr. 
\]
\end{proof}

Let $\bar{r}_1:=\sqrt\frac{3}{2}$. We will apply Lemma \ref{lem:barta deltaf}  to the functions:
\begin{itemize}
    \item $\phi_1(r) = \frac{3}{2}-r^2,$ which is positive on $(1,\bar{r}_1)$, vanishes at $\bar{r}_1$, and satisfies $(\Delta_f + \frac{3}{2} r^{-2})\phi_1(r)<0$ for $r \in (1,\bar{r}_1)$. 
    \item $\phi_2(r) = 1-\frac{3}{2}r^{-2},$ which is positive on $(\bar{r}_1,\infty)$, vanishes at $\bar{r}_1$, and satisfies $(\Delta_f + \frac{3}{2} r^{-2})\phi_2(r)<0$ for  $r\in (\bar{r}_1,\infty)$.
\end{itemize}
The positivity of the functions $\phi_1,\phi_2$ is clear, and the negativity of $\Delta_f+\frac{3}{2r^2}$ can be checked from the formulae:
\begin{align*}
    \Big(\Delta_f  + \frac{3}{2} r^{-2}\Big)\phi_1 &= r^2 -\frac{7}{2} + \frac{9}{4} r^{-2}, \\
    \Big(\Delta_f  + \frac{3}{2} r^{-2}\Big)\phi_2 & =\frac{- 3(c_0+ \sqrt{2}) r^4+
    12 c_0 r^2+  6 \sqrt{2}c_0 }{4 r^8}. 
\end{align*}

\begin{figure}[H]
\includegraphics[scale=0.7]{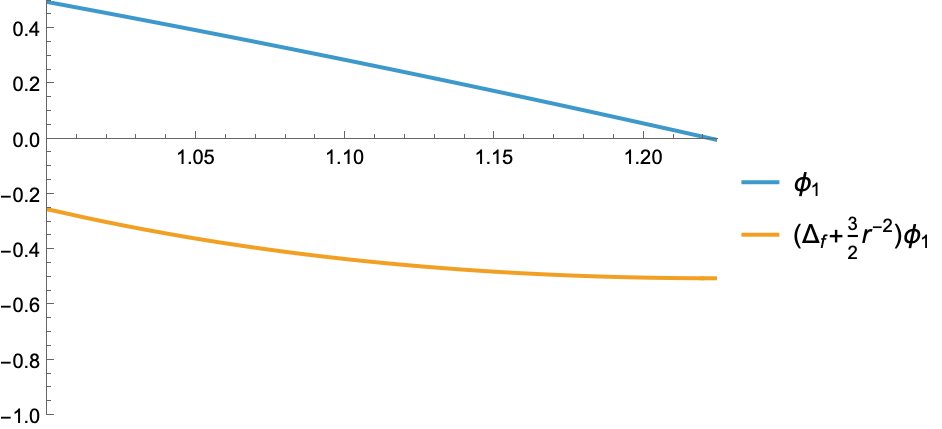}
\caption{$\phi_1\geq 0$ and $(\Delta_f + \frac{3}{2} r^{-2}) \phi_1 < 0$ plotted for $r \in (1, \bar{r}_1)$.}\label{fig:barta-j1-1}
\end{figure}
\begin{figure}[H]
\includegraphics[scale=0.7]{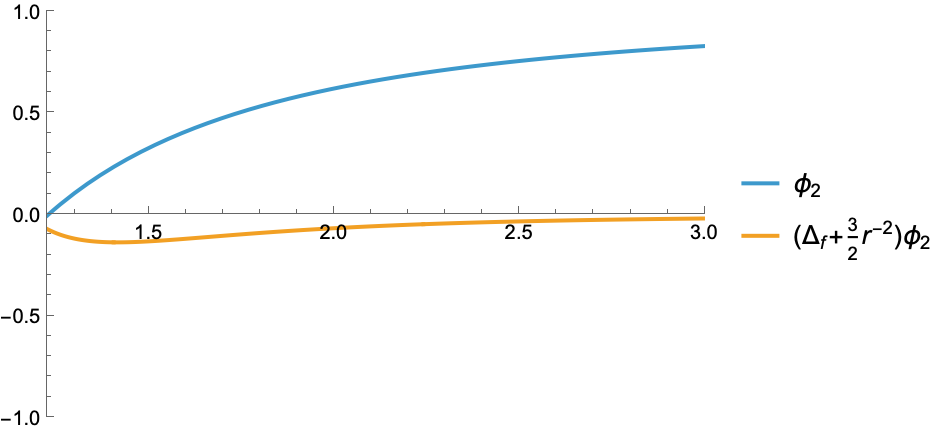}
\caption{$\phi_2\geq0$ and $(\Delta_f + \frac{3}{2} r^{-2}) \phi_2 < 0$ plotted for $r \in (\bar{r}_1, \infty)$.}\label{fig:barta-j1-2}
\end{figure}

\begin{proposition}\label{prop:negativity-of-mu2}
    The second eigenvalue of the operator $\Delta_f+\frac{3}{2}r^{-2}$ in $H^1_f(M)_{\mathrm{rad}}$ is negative. That is,  $\tilde{\mu}_2  < 0$. 
\end{proposition}
\begin{proof}
    Let $a_1, a_2:(1,\infty) \to \mathbb{R}$ be the first two eigenfunctions for the operator $\Delta_f+\frac{3}{2}r^{-2}$ defined as in Lemma \ref{lem: Rayleigh deltaf 32r}, normalized such that $\|a_1\|_{L^2_f}=\|a_2\|_{L^2_f}=1$. Then, by taking $|a_1|$ as a competitor test function, we show classically that we can choose $a_1$ positive everywhere.
    
    Now, since $\int_1^{+\infty}a_1a_2r^3e^{-f}dr = 0$ and $a_1>0$, there exists $r_0\in(1,+\infty)$ such that $a_2(r_0)=0$. Additionally, $\int_1^\infty F|a_2'|^2r^3e^{-f}dr<+\infty$ and $a_2(r_0)=0$ imply some pointwise bounds on $a_2$:
\begin{align*}
    |a_2(r)|^2 &\leq \bigg(\int_{r_0}^r |a_2'(r)|dr\bigg)^2\\
    &\leq \bigg(\int_1^\infty F|a_2'|^2r^3e^{-f}dr\bigg)\bigg(\int_{r_0}^r \frac{e^{f}}{F r^3}dr\bigg).
\end{align*}
Now, as $r\to 1$, one has
\[
\int_{r_0}^r \frac{e^{f}}{F r^3}dr\sim \int_{r_0}^r \frac{1}{r-1} dr \sim |\log(r-1)|
\]
and as $r\to +\infty$, one has
\[
\int_{r_0}^r \frac{e^{f}}{F r^3}dr\sim\int_{r_0}^r r^{-3} e^{f}dr \sim r^{-4} e^{f}.
\]
Recall $\bar{r}_1 = \sqrt{\frac32}$. Now, there are two situations: either (1) $r_0\leq\bar{r}_1$ or (2) $r_0\geq \bar{r}_1$. If the first happens, we consider $\phi_1(r) = r^2-\frac32$ directly verity that 
\begin{equation}\label{eq:bdry phi1}
\begin{aligned}
     &\lim_{r \to 1} F(r)\frac{\phi_1'(r)}{\phi_1(r)}  a(r)^2 r^3 e^{-f(r)} =\lim_{r \to r_0} F(r)\frac{\phi_1'(r)}{\phi_1(r)}  a(r)^2 r^3 e^{-f(r)}=0.
\end{aligned}
\end{equation}
Otherwise, we consider $\phi_2(r) = 1-\frac{3}{2}\frac1{r^2}$, and directly verify that
\begin{equation}\label{eq:bdry phi2}
\begin{aligned}
     &\lim_{r \to r_0} F(r)\frac{\phi_2'(r)}{\phi_2(r)}  a(r)^2 r^3 e^{-f(r)}=\lim_{r \to \infty} F(r)\frac{\phi_2'(r)}{\phi_2(r)}  a(r)^2 r^3 e^{-f(r)}=0.
\end{aligned}
\end{equation}
In either situation, we conclude that $\tilde{\mu}_2 < 0$ as follows: 
\begin{itemize}
   \item If $r_0\leq\bar{r}_1$, then we consider $\phi_1 = r^2-\frac32$, which is positive and satisfies \eqref{eq:condition phi barta deltaf} i.e. $(\Delta_f+\frac{3}{2}r^{-2})\phi_1<0$ on $(1,r_0)$, and the boundary conditions \eqref{eq:bdry cond barta deltaf} are satisfied thanks to \eqref{eq:bdry phi1}. We can therefore apply Lemma \ref{lem:barta deltaf} with $\phi = \phi_1$ on $(r_1,r_2) = (1,r_0)$ yields
    \begin{align*}
        \tilde\mu_2\int_1^{r_0}u_2^2\,\,r^3e^{-f}dr= \tilde\mu_2\int_1^{\bar{r}_1}u_2^2\,\,r^3e^{-f}dr &=\int_1^{\bar{r}_1} \Big(-\frac{F}{4}(u_2')^2 + \frac{3}{2r^2} u_2^2\Big)\, r^{3}e^{-f}dr < 0.
    \end{align*} 
    \item If $r_0\geq\bar{r}_1$, then the same argument applied to $\phi_2(r) = 1-\frac{3}{2}\frac1{r^2}$ on $(r_0,\infty)$, which satisfies $(\Delta_f+\frac{3}{2}r^{-2})\phi_2<0$ and the boundary conditions \eqref{eq:bdry cond barta deltaf} by \eqref{eq:bdry phi2}, yields the same estimate
    \begin{align*}
    \tilde\mu_2\int_{r_0}^{\infty}u_2^2\,\,r^3e^{-f}dr \leq\tilde\mu_2\int_{\bar{r}_1}^{\infty}u_2^2\,\,r^3e^{-f}dr &=\int_{\bar{r}_1}^{\infty} \Big(-\frac{F}{4}(u_2')^2 + \frac{3}{2r^2} u_2^2\Big)\, r^{3}e^{-f}dr <0.
    \end{align*}
\end{itemize}
\end{proof}

Consider tensors $S, T, U \in H^1_f(M)$ given by  
\begin{align*}
S &= u_1 D^1_{-1, M'} \mathbf{k}_1 + u_2 D^1_{1, M'} \mathbf{k}_2,  \\ 
T &=  \tilde{u}_1 D^1_{1, M'} \mathbf{k}_{\bar{1}} + \tilde{u}_2 D^1_{-1, M'} \mathbf{k}_{\bar{2}},\\
U &= v_1 D^2_{0, M'} \mathbf{b}_0 +v_2 D^2_{0, M'} \mathbf{b}_1+ v_3 D^2_{2, M'} \mathbf{b}_- + v_4 D^2_{-2, M'} \mathbf{b}_+. 
\end{align*}
Assume that 
\begin{align*}
\langle S, \overline{S^1_{1, M'}} \rangle_{L^2_f(M)} &= 0, &\langle T,\overline{ S^1_{-1, M'} }\rangle_{L^2_f(M)} &= 0, &\langle U, \overline{S^2_{0, M'}} \rangle_{L^2_f(M)} &= 0.
\end{align*}

\begin{lemma}\label{lem:negativity-of-bad-blocks}
For $S, T, U$ as above and $V \in \{S, T, U\}$, 
\begin{align*}
\int_M \big(2 R(V, \overline{V}) - |\nabla V|^2 \big) e^{-f} d\mu_g  \leq 0.
\end{align*}
\end{lemma}

\begin{proof}
    Let us focus on the proof for $V = S$. The proof of the other cases are similar. Assume for sake of contradiction that 
    \[
    \int_{M} \big(2 R(S, \overline{S}) - |\nabla S|^2) e^{-f} d\mu_g > 0. 
    \]
    Consider the tensor $S_1 := S^1_{1, M'}$ discovered in Section \ref{sec:nonnegative eigenvalues} and write $S_1 = w_1 D^1_{-1, M'} \mathbf{k}_1 + w_2 D^1_{1, M'} \mathbf{k}_2$. Recall that $L_f S_1 = \frac{c_0}{\sqrt{2}} S_1$ and, in particular, $S_1 \in H^2_f(M)$. For any $c_1, c_2 \in \mathbb{C}$, consider $S_2:= c_1 S + c_2 S_1$. Integrating by parts, using the $L_f S_1 = \frac{c_0}{\sqrt{2}} S_1$ and $\langle S, \overline{S_1} \rangle_{L^2_f(M)} = 0$, we obtain
    \begin{align*}
       & \int_M \big(2R(S_2, \overline{S_2}) - |\nabla S_2|^2 \big) e^{-f} d\mu_g \\
       &= |c_1|^2\int_M \big(2R(S, \overline{S}) - |\nabla S|^2 \big) e^{-f} d\mu_g +|c_2|^2 \int_M \big(2R(S_1, \overline{S_1}) - |\nabla S_1|^2 \big) e^{-f} d\mu_g\\
        & \qquad  + c_1 \overline{c_2}\int_M \big(2 R(S, \overline{S_1}) - \langle \nabla S, \nabla \overline{S_1} \rangle \big)e^{-f} d\mu_g + \overline{c_1}c_2 \int_M \big(2 R(S_1, \overline{S}) - \langle \nabla  S_1, \nabla \overline{S} \rangle \big) e^{-f} d\mu_g \\
        & = |c_1|^2\int_M \big(2R(S, \overline{S}) - |\nabla S|^2 \big) e^{-f} d\mu_g + |c_2|^2\int_M \langle L_f S_1, \overline{S_1\rangle} e^{-f} d\mu_g \\
        & \qquad + c_1 \overline{c_2} \int_M \langle S, L_f \overline{S_1} \rangle e^{-f} d\mu_g  +\overline{c_1} \overline{c_2} \int_M  \langle L_f S_1, \overline{S} \rangle e^{-f} d\mu_g \\
        & =  |c_1|^2 \int_M \big(2R(S, \overline{S}) - |\nabla S|^2 \big) e^{-f} d\mu_g + \frac{c_0}{\sqrt{2}}|c_2|^2 \int_M |S_1|^2 e^{-f} d\mu_g \geq 0,  
    \end{align*}
    with equality if and only if $c_1 = c_2 = 0$. On the other hand, there exists a choice of constants $(c_1, c_2) \neq (0, 0)$ such that 
    \[
    \int_M \big( c_1 u_1 + c_2 w_1) a_1 \, e^{-f} d\mu_g = 0,
    \]
    where $a_1$ is the function from Lemma \ref{lem: Rayleigh deltaf 32r}. Fixing this choice of constants, let $v_1 = c_1 u_1 + c_2 w_1$ and $v_2 = c_1 u_1 + c_2 w_2$ so that the vector representation of $S_2$ is $\xi = (v_1, v_2) \in \mathbb{C}^2$. Then have 
    \begin{align*}
        0 < \int_M \big(2R(S_2, \overline{S_2}) - |\nabla S_2|^2 \big) e^{-f} d\mu_g & = 4\pi^2 \int_1^\infty \big(\mathcal{P} \xi \cdot \overline{\xi} - \frac{F}{4} |\xi'|^2 \big)r^3 e^{-f} \, dr \\
        & \leq 4\pi^2 \int_1^\infty \Big(\frac{3}{2}r^{-2} |v_1|^2 - \frac{F}{4} |v_1'|^2 \Big) r^3 e^{-f} dr. 
    \end{align*}
    As $v_1$ is $L^2_f$ orthogonal to $a_1$, by Lemma \ref{lem: Rayleigh deltaf 32r}, we obtain 
    \begin{align*}
       0 < \int_M \big(2R(S_2, \overline{S_2}) - |\nabla S_2|^2 \big) e^{-f} d\mu_g & \leq 4\pi^2 \tilde{\mu}_2 \int_1^\infty |v_1|^2 r^3 e^{-f} dr \\
        & \leq  4\pi^2 \tilde{\mu}_2 \int_1^\infty \big(|v_1|^2 + |v_2|^2\big) r^3 e^{-f} dr \\
        & = \tilde{\mu}_2 \int_M |S_2|^2 e^{-f} d\mu_g. 
    \end{align*}
    But this gives $\tilde{\mu}_2 >  0$, contradicting Proposition \ref{prop:negativity-of-mu2}. Thus, we must have 
    \[
    \int_M \big(2 R(S, \overline{S}) - |\nabla S|^2\big) e^{-f} d\mu_g \leq 0.
    \]
\end{proof}

We now complete the proof of linear stability of the FIK soliton for deformations in the modes $J = 1$ and $J = 2$. 

\begin{theorem}\label{thm:J1}
Let $h\in \mathcal{T}^1_{M'}$ and $M'\in\mathcal{J}$. If $\mathrm{div}_f(h) = 0$, then
\[
\int_M\big(2R(h, \overline{h}) - |\nabla h|^2 ) e^{-f}d\mu_g \leq 0.
\]
Consequently, for any $h \in \mathcal{T}^1_{M'}$ and $M' \in \mathcal{J}$, $\delta^2 \nu_g(h + \overline{h}) \leq 0$. 
\end{theorem}

\begin{proof}
    This follows from Proposition \ref{prop:J1-except-bad-block} and Lemma \ref{lem:negativity-of-bad-blocks}. Let $h_{p, M}$ and $k_{j, M}$ be the component functions of $h$. Let 
    \begin{align*}
    S &= k_{1,-1} D^1_{-1, M'} \mathbf{k}_1 + k_{2, 1} D^1_{1, M'} \mathbf{k}_2, & T &=  k_{\bar 1,1} D^1_{1, M'} \mathbf{k}_{\bar{1}} + k_{\bar{2}, -1} D^1_{-1, M'} \mathbf{k}_{\bar{2}}.
    \end{align*}
    By assumption $\mathrm{div}_f(h) = 0$, while the tensors $S^1_{M, M'}$ are Hessians, so
    \begin{align*}
        0 = \int_M \langle h, \overline{S^1_{1, M'}} \rangle e^{-f} d\mu_g = \int_M \langle S, \overline{S^1_{1, M'}} \rangle e^{-f} d\mu_g, 
    \end{align*}
    and 
    \begin{align*}
        0 = \int_M \langle h, \overline{S^1_{-1, M'}} \rangle e^{-f} d\mu_g = \int_M \langle T, \overline{S^1_{-1, M'}} \rangle e^{-f} d\mu_g. 
    \end{align*}
    By Lemma \ref{lem:negativity-of-bad-blocks}, we conclude that 
    \[
     \int_M (2 R(S, \overline{S}) - |\nabla S|^2) e^{-f} d\mu_g + \int_M (2 R(T, \overline{T}) - |\nabla T|^2) e^{-f} d\mu_g \leq 0. 
    \]
    By Proposition \ref{prop:J1-except-bad-block}, it follows that 
    \[
    \int_M \big(2R(h, \overline{h}) - |\nabla h|^2 \big) e^{-f} d\mu_g \leq 0. 
    \]
    This completes the proof. 
\end{proof}

\begin{theorem}\label{thm:J2}
Let $h\in \mathcal{T}^2_{M'}$ and $M'\in\mathcal{J}$. 
If $\mathrm{div}_f(h) = 0$, then
\[
\int_M\big(2R(h, \overline{h}) - |\nabla h|^2 ) e^{-f}d\mu_g \leq 0.
\]
Consequently, for any $h \in \mathcal{T}^2_{M'}$ and $M' \in \mathcal{J}$, $\delta^2 \nu_g(h + \overline{h}) \leq 0$. 
\end{theorem}

\begin{proof} 
This similarly follows from Proposition \ref{prop:J2-except-bad-block} and Lemma \ref{lem:negativity-of-bad-blocks}. Let $h_{p, M}$ and $k_{j, M}$ be the component functions of $h$. In this case, we let 
\begin{align*}
U &= h_{0, 0} D^2_{0, M'} \mathbf{b}_0 +h_{1, 0} D^2_{0, M'} \mathbf{b}_1+ h_{+, 2} D^2_{2, M'} \mathbf{b}_- + h_{-, -2} D^2_{-2, M'} \mathbf{b}_+,
\end{align*}
and, assuming $\mathrm{div}_f(h) = 0$, we have $ \langle U, S^2_{0, M'} \rangle_{L^2_f(M)} = 0$. It follows from Lemma \ref{lem:negativity-of-bad-blocks} that 
\[
\int_M (2 R(U, \overline{U}) - |\nabla U|^2) e^{-f} d\mu_g \leq 0.
\]
Then via Proposition \ref{prop:J2-except-bad-block}, we obtain 
\[
\int_M\big(2R(h, \overline{h}) - |\nabla h|^2 ) e^{-f}d\mu_g \leq 0.
\]
This completes the proof. 
\end{proof}

\subsection{Conclusion}

The combination of Theorems \ref{thm:main-radial} ($J=0$), \ref{thm:stability for J = 3} ($J=3$),  \ref{thm:stability for J = 4} ($J=4$), \ref{thm:stability for J geq 5} ($J \geq 5$), \ref{thm:J1} ($J=1$), and \ref{thm:J2} ($J = 2$) together with the mode decomposition Proposition \ref{prop:pres of TJM'} completes the proof of linear stability of the blowdown soliton in dimension four, Theorem \ref{thm:main}.

\begin{remark}\label{rem:all-eigenvalues}
Our proof shows that, aside from the eigenvalues
\begin{itemize}
\item $1$, with 1-dimensional eigenspace spanned by $\Ric$,
\item $1-\frac{1}{\sqrt{2}}$, with 4-dimensional eigenspace spanned by $\nabla^2(\hat{u} D^1_{M,M'})$, $M,M'\in \{-1,1\}$,
\item $0$, with 4-dimensional eigenspace spanned by $\nabla^2 f$ and $\nabla^2(\hat{v} D^2_{0,M'})$, $M'\in \{-2,0,2\}$,
\end{itemize}
the operator $L_f$ has strictly negative eigenvalues in $H^1_f(M)$.
Let $\mathcal{T}^J := \bigoplus_{M'\in\mathcal{J}}\mathcal{T}^J_{M'}\cap H^1_f(M)$.
For $J\geq3$, $L_f$ is strictly negative on $\mathcal{T}^J$; for $J\in\{1,2\}$, the Sturm–Liouville reduction shows its second eigenvalue is negative.

For $\mathcal{T}^0$, we proved negativity of $L_f$ on $\ker \mathrm{div}_f \cap \Ric^\perp$, and it also holds on $\mathrm{div}_f^\ast(TM)\cap(\nabla^2f)^\perp$.
Indeed, a radially symmetric $\mathcal{L}_X g$ implies $X=\nabla u+Y$ with $u=u(r)$ and $\mathrm{div}_f(Y)=0$, and $L_f$ preserves these spaces \cite[Lemma 2.8]{CM}.

On radially symmetric gradients, $L_f(\mathcal{L}_{\nabla u}g)=\lambda(\mathcal{L}_{\nabla u}g)$ if and only if $\Delta_f u=(\lambda-1)u$; see \eqref{eq:Lf-Hessians}.
Since the first two eigenvalues of the Sturm-Liouville operator $\Delta_f$ on radial functions are $0$ (via the nonvanishing $u = 1$) and $-1$ (via the once-vanishing $u = r^2 -2$), we have $\lambda -1<-1 \implies \lambda < 0$.

For radially symmetric $Y$ with $\mathrm{div}_f(Y)=0$, $Y$ is orthogonal to $\partial_r$ (as every $v(r) \partial_r$ is gradient) and can be written $Y=u_i e^i$ ($u_i=u_i(r)$) or as a linear combination of Killing fields (the right-invariant counterparts of the $X_i$). 
In either form, one may verify that $\mathcal{L}_Yg$ is orthogonal to $\mathbf{b}_0,\mathbf{b}_1$. It follows from Proposition \ref{prop:Lambda-function-comps} that if $L_f \mathcal{L}_Y g = \lambda \mathcal{L}_Y g$, then $\lambda  < 0$. 
\end{remark}

\newpage

\appendix 
\part{Additional Computational Results for the FIK Shrinker}\label{part:appendix}

For many of the computations done in this paper, the authors have written some (Mathematica) code verifying the results \cite{NO2}. 

\section{The connection and curvatures of the FIK shrinker}\label{app:1}

In this section, we compute the Levi-Civita connection, the curvatures, the soliton potential, and the connection on forms for the FIK metric. 

The following results are straightforward computations that can be verified using the expression for the metric and the definitions of $s, C, F, f$. For brevity, we omit the computations, but indicate the key formulas used to obtain the desired identities. 

\subsection{The connection computations}

\begin{lemma}\label{lem:lie-bracket}
The Lie brackets of the orthonormal frame are given by 
\begin{align*}
[e_0, e_1] &= - \frac{s'}{2r} e_1, &[e_2, e_3] &=  -\frac{s}{r^2} e_1, \\
[e_0, e_2] &= -\frac{s}{2r^2} e_2, &[e_1, e_2] &= - \frac{1}{s} e_3,   \\
 [e_0, e_3] &= - \frac{s}{2r^2} e_3,  & [e_1, e_3] &= \frac{1}{s} e_2.
\end{align*}
\end{lemma}

\begin{proof}
    This follows from \eqref{eq:frame} and the identities $[\partial_r, X_i] = 0$, $[X_i, X_j] = -2X_k$ for $i, j, k$ a cyclic permutation of $1,2,3$.
\end{proof}

\begin{lemma}\label{lem:levi-civita}
The Levi-Civita connection is given by 
\begin{align*}
\nabla_{e_0} e_0 &= 0, & \nabla_{e_0} e_1 &=0 , & \nabla_{e_0} e_2 &= 0 , & \nabla_{e_0} e_3 &= 0 , \\
\nabla_{e_1} e_0 &= \frac{s'}{2r} e_1, & \nabla_{e_1} e_1 &= - \frac{s'}{2r} e_0, & \nabla_{e_1} e_2 &= -\frac{1}{s}\Big(1 -\frac{s^2}{2r^2}\Big)e_3, & \nabla_{e_1} e_3 &= \frac{1}{s}\Big(1 -\frac{s^2}{2r^2}\Big)e_2, \\
\nabla_{e_2} e_0 &=  \frac{s}{2r^2}e_2, & \nabla_{e_2} e_1 &=\frac{s}{2r^2} e_3 , & \nabla_{e_2} e_2 &= -\frac{s}{2r^2} e_0, & \nabla_{e_2} e_3 &= - \frac{s}{2r^2} e_1 , \\
\nabla_{e_3} e_0 &=  \frac{s}{2r^2}e_3 , & \nabla_{e_3} e_1 &= -\frac{s}{2r^2} e_2 , & \nabla_{e_3} e_2 &= \frac{s}{2r^2} e_1, & \nabla_{e_3} e_3 &= -\frac{s}{2r^2} e_0. \\
\end{align*}
\end{lemma}
\begin{proof}
    This follows from the Koszul formula for an orthonormal frame,
    \[
2g(\nabla_{e_i} e_j, e_k) = g([e_i, e_j], e_k) + g([e_k, e_i], e_j) - g([e_j, e_k], e_i),
    \]
    and the preceding lemma. 
\end{proof}

Before proceeding, we note that $g(\nabla_{e_i} e^j, e^k) = g(\nabla_{e_i} e_j, e_k)$. So, except for raising indices, identical formulas to those in Lemma \ref{lem:levi-civita} also hold for the coframe. For switching from formulas involving $s$ to formulas involving $F$, it is useful to note 
\begin{equation}\label{eq:sder-to-Fder}
\frac{s'}{r} =  \frac{rF'+2F}{2s}.
\end{equation}

\begin{proposition}\label{prop:levi-civita-2-forms}
On the basis of self-dual 2-forms, the Levi-Civita connection acts by the identities:
\begin{align*}
\nabla_{e_0} \omega_1^+ & = 0,& \nabla_{e_0} \omega_2^+ & = 0,& \nabla_{e_0} \omega_3^+ & = 0, \\
\nabla_{e_1} \omega_1^+ & = 0,& \nabla_{e_1} \omega_2^+ & =\frac{rF' +4F -4}{4s}\omega_3^+  ,& \nabla_{e_1} \omega_3^+ & =   -\frac{rF' +4F -4}{4s}\omega_2^+ , \\
\nabla_{e_2} \omega_1^+ & = 0,& \nabla_{e_2} \omega_2^+ & = 0 ,& \nabla_{e_2} \omega_3^+ & = 0  , \\
\nabla_{e_3} \omega_1^+ & = 0,& \nabla_{e_3} \omega_2^+ & = 0,& \nabla_{e_3} \omega_3^+ & = 0. \\
\end{align*}
On anti-self-dual 2-forms:
\begin{align*}
\nabla_{e_0} \omega_1^- & = 0,& \nabla_{e_0} \omega_2^- & = 0,& \nabla_{e_0} \omega_3^- & = 0, \\
\nabla_{e_1} \omega_1^- & = 0,& \nabla_{e_1} \omega_2^- & =- \frac{rF' +4}{4s}\omega_3^-  ,& \nabla_{e_1} \omega_3^- & =   \frac{rF' +4}{4s}\omega_2^- , \\
\nabla_{e_2} \omega_1^- & = \frac{F}{s} \omega_3^-,& \nabla_{e_2} \omega_2^- & =  0 ,& \nabla_{e_2} \omega_3^- & =  - \frac{F}{s} \omega_1^-, \\
\nabla_{e_3} \omega_1^- & = -\frac{F}{s} \omega_2^-,& \nabla_{e_3} \omega_2^- & = \frac{F}{s} \omega_1^-,& \nabla_{e_3} \omega_3^- & = 0. \\
\end{align*}
\end{proposition}

\begin{proof}
    Note that $\nabla$ preserves (anti-)self-duality and that $\nabla_{e_0} \omega_a^{\pm} = 0$. Additionally, $\omega_a^+, \omega_b^-$ are orthonormal, so for instance $g(\nabla_{e_i} \omega_a^+, \omega_b^+) = - g(\omega_a^+, \nabla_{e_i} \omega_b^+)$. Lastly, by $U(2)$-invariance all formulas should be anti-invariant under the index permutation $(2,3)\to(3,2)$. With these simplifications, the proposition now readily follows from Lemma \ref{lem:levi-civita} and the definitions of the self-dual forms. For instance, 
    \begin{align*}
        \nabla_{e_1} \omega_2^{\pm} & = \nabla_{e_1}(e^0 \wedge e^2 \pm e^3 \wedge e^1) \\
        & = \frac{s'}{2r} e^1 \wedge e^2 - \frac{2-F}{2s} e^0 \wedge e^3 \mp \frac{2-F}{2s} e^1 \wedge e^2 \pm \frac{s'}{2r} e^0 \wedge e^3 \\
        & = \pm\left(\frac{rF'+(2\pm2)F \mp 4}{4s}\right)\omega_3^{\pm}.
    \end{align*}
\end{proof}

If we define $
\Gamma_{i(ab)}^{\pm} := g\big(\nabla_i \omega_a^\pm, \omega_b^\pm \big)$, then the last proposition implies there are only a few nonzero terms. For brevity let us label them. A plot of these functions can be found in Figure \ref{fig:3}.

\begin{figure}
\includegraphics[scale=0.7]{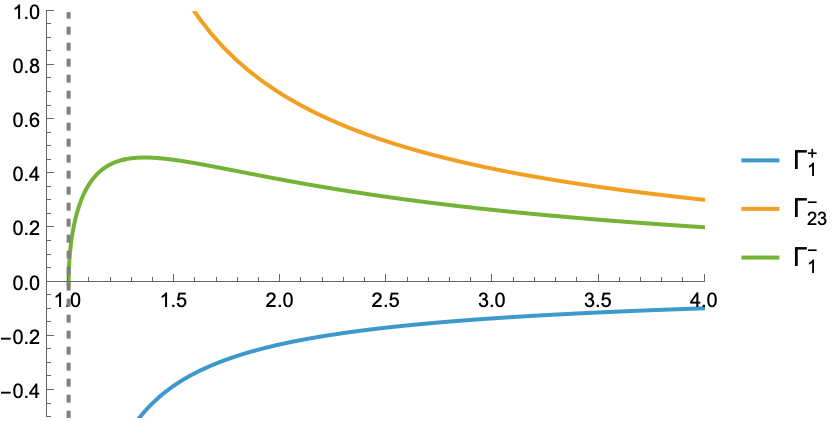}
\caption{Plot of the $\Gamma$-functions.}\label{fig:3}
\end{figure}

\begin{definition}\label{def:levi-civita-omega-coef-functions}
The $\Gamma$-functions are:
\begin{align*}
\Gamma_1^+ &:= \Gamma^+_{1(23)} = - \Gamma^+_{1(32)} = \frac{rF' + 4F -4}{4s}. \\ 
\Gamma_{23}^- &:=- \Gamma_{1(23)}^- =  \Gamma_{1(32)}^- = \frac{rF' +4}{4s}.  \\
\Gamma_1^- &:= \Gamma_{2(13)}^- = -\Gamma_{2(31)}^- = \Gamma_{3(21)}^- = - \Gamma_{3(12)}^- =\frac{F}{s}.
\end{align*}
\end{definition}

\begin{proposition}\label{prop:Gamma-function-comps}
As functions of the radial coordinate $r$, the $\Gamma$-functions obtained by the action of $\nabla$ on the duality basis of 2-forms are given by 
\begin{align*}
\Gamma_1^+ &= -\frac{1}{s}\left(\frac{c_0}{2r^2}+\frac{c_0}{\sqrt{2}} \right) <0,\\
\Gamma_{23}^- & = \frac{1}{s}\Big(1 +\frac{c_0(r^2 + \sqrt{2})}{2r^4}\Big) > 0, \\
\Gamma_1^- &=\frac{1}{s}\left(\frac{(r^2-1)(r^2+c_0)}{\sqrt{2}r^4} \right) > 0. 
\end{align*}
Moreover, near $r = 1$, we have the asymptotics
$$\Gamma_1^+ = -\frac{1}{2\sqrt{2}}\frac{1}{\sqrt{r-1}}+o(1),\qquad \Gamma_{23}^+ = \frac{3}{2\sqrt{2}}\frac{1}{\sqrt{r-1}}+o(1), \qquad \Gamma_1^- = {\sqrt{2}}{\sqrt{r-1}}+o(1).$$
\end{proposition}

\begin{proof}
    Recalling \eqref{eq:Fder1} and \eqref{eq:Fder2}, we observe that 
    \begin{align*}
    rF' + (2\pm 2)F \mp 4 &= \left(\frac{2c_0}{r^2}+\frac{2\sqrt{2}c_0}{r^4}\right)+(2\pm2)\left(\frac{1}{\sqrt{2}} -\frac{c_0}{r^2} - \frac{c_0}{\sqrt{2}r^4}\right)\mp 4 \\
    & = \frac{(1\mp1)\sqrt{2}c_0 }{r^4}\mp \frac{2c_0}{r^2} +\sqrt{2}\pm \sqrt{2} \mp 4.
    \end{align*}
    From this, the first two formulas follow. The last formula is just a restatement of \eqref{eq:Fder1}. The asymptotics can be readily deduced from the asymptotics of $s = \sqrt{2} \sqrt{r-1} + O((r-1)^{\frac{3}{2}})$ near $r = 1$, see Appendix \ref{app:3} below. 
\end{proof}

\subsection{The curvature and soliton computations}

\begin{proposition}\label{prop:riemann-curvatures}
The nonzero sectional curvatures in the given frame (up to the usual symmetries) are given by 
\begin{align*}
R_{0101} & = -\frac{F''}{8} - \frac{3F'}{8r}, &  R_{0202} & =-\frac{F'}{8r},  & R_{0303} & = -\frac{F'}{8r}  \\
R_{2323} &= \frac{1-F}{r^2} , & R_{1212} &=-\frac{F'}{8r} , &  R_{1313} &=-\frac{F'}{8r}, \\
R_{0213} &= -\frac{F'}{8r}, &  R_{0312} &= - \left(- \frac{F'}{8r}\right)  & R_{0123} &= 2\left(-\frac{F'}{8r}\right). \\
\end{align*}
Additionally, the sectional curvatures have signs given by  
\begin{align*}
  R_{1010} = \left( -\frac{F''}{8} - \frac{3F'}{8r}\right) &= \frac{c_0}{\sqrt{2}r^6} > 0 ,\\
  R_{0202} = R_{0303} = R_{1212} = R_{1313} = \left( - \frac{F'}{8r}\right) &= -\frac{c_0}{2\sqrt{2} r^6} - \frac{c_0}{4r^4} < 0,\\
 R_{2323} =   \left(\frac{1-F}{r^2} \right)&= \frac{c_0}{\sqrt{2}r^6} + \frac{c_0}{r^4} + \frac{c_0}{\sqrt{2} r^2} > 0 .
\end{align*}
\end{proposition}

\begin{proof}
    We use the definition
    \[
    R_{ijkl} = g\big((\nabla_{e_j} \nabla_{e_i} - \nabla_{e_i} \nabla_{e_j} -\nabla_{[e_j,e_i]}) e_k, e_l),
    \]
    along with the relations 
    \[
    F = \frac{s^2}{r^2}, \qquad rF' = \frac{2(rss'- s^2)}{r^2}, \qquad r^2F'' = \frac{2r^2s'' + 2r^2(s')^2 -8r ss' + 6 s^2}{r^2},
    \]
    to obtain the expressions for $R_{ijkl}$ asserted. The expression of the curvatures as functions of $r$ follow readily from the expressions for $F, F', F''$ (see \eqref{eq:Fder1}, \eqref{eq:Fder2}, \eqref{eq:Fder3}). 
\end{proof}

From the previous proposition, the following corollary is immediate. 

\begin{corollary}\label{cor:ricci-scalar}
The Ricci curvatures are given by the formulas 
\begin{align*}
\mathrm{Ric}_{00} &= -\frac{F''}{8} - \frac{5F'}{8r}, & \mathrm{Ric}_{11} & = -\frac{F''}{8} - \frac{5F'}{8r}, \\
\mathrm{Ric}_{22} &= \frac{1-F}{r^2}-\frac{F'}{4r}, & \mathrm{Ric}_{33} & =  \frac{1-F}{r^2}-\frac{F'}{4r}.
\end{align*}
In particular, 
\[
\mathrm{Ric}_{00} = \mathrm{Ric}_{11} = -\frac{c_0}{2r^4} < 0 , \qquad\qquad  \mathrm{Ric}_{22} = \mathrm{Ric}_{33} = \frac{c_0}{2r^4} + \frac{c_0 \sqrt{2}}{2r^2} >0 .
\]
The scalar curvature is given by the formula 
\begin{align*}
\mathrm{scal} &= -\frac{F''}{4} - \frac{7 F'}{4r} + \frac{2(1-F)}{r^2} \\
&= \frac{c_0\sqrt{2}}{r^2} > 0.
\end{align*}
\end{corollary}

\begin{proposition}\label{prop:hessian-f}
The gradient of the function $f$ is 
\begin{align*}
\nabla f &= \sqrt{2} s \, e_0, & df &= \sqrt{2} s\, e^0
\end{align*}
The Hessian of the function $f$ is given by 
\begin{align*}
(\nabla^2 f)_{00} &=\frac{F'}{2\sqrt{2}} r + \frac{F}{\sqrt{2}}, & (\nabla^2 f)_{11} & =\frac{F'}{2\sqrt{2}} r + \frac{F}{\sqrt{2}}, \\
(\nabla^2 f)_{22} &=  \frac{F}{\sqrt{2}}, &(\nabla^2 f)_{33} & =   \frac{F}{\sqrt{2}}.
\end{align*}
In particular, 
\[
(\nabla^2 f)_{00} = (\nabla^2 f)_{11} = \frac{1}{2}+\frac{c_0}{2r^4}, \qquad\qquad  (\nabla^2 f)_{22} = (\nabla^2f)_{33} = \frac{1}{2} -\frac{c_0}{2r^4} - \frac{c_0 \sqrt{2}}{2r^2}.
\]
The Laplacian of the potential is given by the formula 
\begin{align*}
\Delta f &= \frac{rF'}{\sqrt{2}} + 2 \sqrt{2} F\\
& = 2 -\frac{\sqrt{2}c_0}{r^2}
\end{align*}
\end{proposition}

\begin{proof}
    Recall $f = \sqrt{2} (r^2  -1) - \log(2c_0)$. Hence $\nabla f = e_0(f) e_0 = \frac{s}{2r} f' e_0 = \sqrt{2} s e_0.$ To obtain the other formulas, one uses that 
    \[
    (\nabla^2 f)(e_i, e_j) = g(\nabla_{e_i} \nabla f, e_j),
    \]
    along with the formulas from Lemma \ref{lem:levi-civita}.
\end{proof}

\section{The basis of $(0, 2)$-tensors and the actions of $\mathrm{div}_f, L_f$}\label{app:2}

The main computational results of this section are Proposition \ref{prop:divf-on-basis}, Corollary \ref{cor:Lf-action-on-basis}, and Proposition \ref{prop:Lambda-function-comps}, which together give formulas for the action of the operators $\mathrm{div}_f$ and $L_f$ on the our basis of $2$-tensors. Recall our choice of basis $\mathbf{B}$ is
\begin{align*}
 \mathbf{b}_0 &= \omega_1^+\circ \omega_1^+, & \mathbf{b}_1 &=\omega_1^- \circ \omega_1^+, 
 \end{align*}
 and
 \begin{align*}
\mathbf{b}_2 &= \frac{1}{\sqrt{2}}(\omega_2^- + \omega_3^-)\circ \omega_1^+,  &  \mathbf{b}_3 &= \frac{1}{\sqrt{2}}(\omega_2^- - \omega_3^-)\circ \omega_1^+, \\
 \mathbf{b}_4 &= \frac{1}{\sqrt{2}}\omega_1^- \circ (\omega_2^+ + \omega_3^+), &  \mathbf{b}_5 &=\frac{1}{\sqrt{2}}\omega_1^- \circ (\omega_2^+ - \omega_3^+),
 \end{align*}
 and
 \begin{align*}
\mathbf{b}_6 &= \frac{1}{\sqrt{2}}(\omega_2^- \circ \omega_2^+ + \omega_3^- \circ \omega_3^+),& \mathbf{b}_7 &= \frac{1}{\sqrt{2}}( \omega_3^- \circ \omega_2^+- \omega_2^- \circ \omega_3^+ ), \\
\mathbf{b}_8 &= \frac{1}{\sqrt{2}}(\omega_3^-\circ \omega_2^+ +  \omega_2^-\circ\omega_3^+),   & \mathbf{b}_9 &= \frac{1}{\sqrt{2}}(\omega_2^- \circ \omega_2^+ - \omega_3^- \circ \omega_3^+).
\end{align*}
Recalling that $e^{ij} = e^i \otimes e^j$ and $e^{\pm} = \frac{1}{\sqrt{2}} (e^2 \pm e^3)$, we record here for reference that 
\begin{align*}
 \mathbf{b}_0 &= \frac{1}{4} (e^{00} + e^{11} + e^{22} + e^{33}), & \mathbf{b}_1 &=\frac{1}{4} (e^{00} + e^{11} - e^{22} - e^{33}),\\
 &= \frac{1}{4} (e^{00} + e^{11} + e^{++} + e^{--}), & &=\frac{1}{4} (e^{00} + e^{11} - e^{++} - e^{--}),
 \end{align*}
and
 \begin{align*}
 \mathbf{b}_2 &= \frac{1}{4\sqrt{2}}(e^{12} + e^{21} - e^{03} -e^{30} + e^{31} + e^{13}+ e^{02} + e^{20}) =  \frac{1}{4}(e^{1+} + e^{+1} + e^{0-} + e^{-0}),   \\
 \mathbf{b}_3 &= \frac{1}{4\sqrt{2}}(e^{12} + e^{21} - e^{03} -e^{30} - e^{31} - e^{13}- e^{02} - e^{20}) =  \frac{1}{4}(e^{1-} + e^{-1} -e^{0+} - e^{+0}), \\
\mathbf{b}_4 &= \frac{1}{4\sqrt{2}}(e^{12} + e^{21} + e^{03} +e^{30} + e^{31} + e^{13}- e^{02} - e^{20}) =  \frac{1}{4}(e^{1+} + e^{+1} - e^{0-} - e^{-0}), \\
\mathbf{b}_5 &=\frac{1}{4\sqrt{2}}(e^{12} + e^{21} + e^{03} +e^{30} - e^{31} - e^{13} + e^{02} + e^{20}) = \frac{1}{4}(e^{1-} + e^{-1}  + e^{0+} + e^{+0}),
\end{align*}
and
\begin{align*}
\mathbf{b}_6 &= \frac{1}{2\sqrt{2}}(e^{00} - e^{11}),& \mathbf{b}_8 &= \frac{1}{2\sqrt{2}}(e^{23} + e^{32}) = \frac{1}{2\sqrt{2}}(e^{++} - e^{--}) , \\
 \mathbf{b}_7 &= -\frac{1}{2\sqrt{2}}(e^{01} + e^{10}),   & \mathbf{b}_9 &= \frac{1}{2\sqrt{2}}(e^{22} - e^{33}) = \frac{1}{2\sqrt{2}}(e^{+-} + e^{-+}).
\end{align*}
This basis is chosen so as to nearly diagonalize the operator $L_f$, as we shall see below. 

Here, as before, we omit some computations that are more or less straightforward, with an aim to provide the important details in obtaining the main computational results required to prove stability. It is helpful to note that many formulas exhibit an \textit{anti-invariant} symmetry that arises by $U(2)$-invariance of the metric and our choice of orientation.  

\subsection{The action of $\nabla$ on the basis elements of symmetric $2$-tensors}

Using formulas for the Levi-Civita connection the duality basis, we summarize the action of the connection on the basis of symmetric $2$-tensors in the following proposition.  

\begin{proposition}\label{prop:levi-civita-on-2-tensors}
    The Levi-Civita connection acts on the basis of symmetric $2$-tensors by the formulas: 
   \begin{align*}
        \nabla \mathbf{b}_0 & = 0, \\
        \nabla \mathbf{b}_1 & = \Gamma_1^- e^- \otimes \mathbf{b}_2 + \Gamma_1^- e^+\otimes (- \mathbf{b}_3), \\
        \nabla \mathbf{b}_2& =\Gamma_{23}^- e^1 \otimes \mathbf{b}_3 + \Gamma_1^-e^- \otimes (-\mathbf{b}_1), \\
        \nabla \mathbf{b}_3 & =\Gamma_{23}^- e^1 \otimes (-\mathbf{b}_2) + \Gamma_1^-e^+ \otimes \mathbf{b}_1,\\
        \nabla \mathbf{b}_4 & = \Gamma_1^+ \, e^1 \otimes (-\mathbf{b}_5) + \frac{1}{\sqrt{2}} \Gamma_1^- e^- \otimes (\mathbf{b}_6+ \mathbf{b}_8)+ \frac{1}{\sqrt{2}} \Gamma_1^- e^+ \otimes (\mathbf{b}_7 - \mathbf{b}_9),\\
        \nabla \mathbf{b}_5 & =  \Gamma_1^+ \, e^1 \otimes \mathbf{b}_4 + \frac{1}{\sqrt{2}} \Gamma_1^- e^- \otimes (  \mathbf{b}_7+\mathbf{b}_9)+ \frac{1}{\sqrt{2}} \Gamma_1^- e^+ \otimes (-\mathbf{b}_6 + \mathbf{b}_8) , \\
        \nabla \mathbf{b}_6 & = (\Gamma_1^+ + \Gamma_{23}^-) \,e^1 \otimes (-\mathbf{b}_7) + \frac{1}{\sqrt{2}} \Gamma_1^- e^-\otimes(- \mathbf{b}_4) + \frac{1}{\sqrt{2}} \Gamma_1^- e^+\otimes  \mathbf{b}_5 ,\\
        \nabla \mathbf{b}_7 & = (\Gamma_1^+ + \Gamma_{23}^-) \,e^1 \otimes \mathbf{b}_6 + \frac{1}{\sqrt{2}} \Gamma_1^- e^-\otimes  (-\mathbf{b}_5) + \frac{1}{\sqrt{2}} \Gamma_1^- e^+\otimes (-\mathbf{b}_4),\\
        \nabla \mathbf{b}_8 & = (\Gamma_1^+ - \Gamma_{23}^-)\, e^1 \otimes (-\mathbf{b}_9) +\frac{1}{\sqrt{2}} \Gamma_1^- e^- \otimes (-\mathbf{b}_4) + \frac{1}{\sqrt{2}} \Gamma_1^-  e^+\otimes (-\mathbf{b}_5),\\
         \nabla \mathbf{b}_9 & = (\Gamma_1^+ - \Gamma_{23}^-)\, e^1 \otimes \mathbf{b}_8 + \frac{1}{\sqrt{2}} \Gamma_1^- e^- \otimes (-\mathbf{b}_5) + \frac{1}{\sqrt{2}} \Gamma_1^- e^+ \otimes \mathbf{b}_4.
    \end{align*}
    In particular, $\nabla_{e_0} \mathbf{b}_p = 0$ for all $p$. 
\end{proposition}

\begin{proof}
    This is a computational corollary of the definitions of the basis elements in terms of the duality basis $\omega_a^{\pm}$ and Proposition \ref{prop:levi-civita-2-forms}. The formulas are made more compact by using $e^{\pm}$ instead of $e^2, e^3$. 
\end{proof}

\subsection{The action of $\mathrm{div}_f$ on basis elements of symmetric $2$-tensors} 

For putting general $2$-tensors on the FIK shrinking soliton in gauge, it will be useful to have the following formulas for the action of $\mathrm{div}_f$ on the basis of $2$-tensors. 

\begin{proposition}\label{prop:divf-on-basis}
The weighted divergence acts on the basis of $2$-tensors by the following formulas:
\begin{align*}
    \mathrm{div}_f(\mathbf{b}_0) & = -\frac{s}{8r^2}\left(2\sqrt{2}r^2\right) e^0, &\mathrm{div}_f(\mathbf{b}_1)  &= \frac{s}{8r^2}\left(4 - 2\sqrt{2} r^2\right) e^0,
\end{align*}
and
\begin{align*}
    \mathrm{div}_f(\mathbf{b}_2) & = \frac{s}{8r^2}\left(\frac{4 - r^2}{F}- \sqrt{2} r^2\right)e^-, & \mathrm{div}_f(\mathbf{b}_6) & =\frac{\sqrt{2}s}{8r^2}\left(\frac{4-2r^2}{F}\right) e^0, \\
    \mathrm{div}_f(\mathbf{b}_3) & = -\frac{s}{8r^2}\left(\frac{4 - r^2}{F}- \sqrt{2} r^2\right)e^+,  &  \mathrm{div}_f(\mathbf{b}_7) & =  -\frac{\sqrt{2}s}{8r^2}\left(\frac{4-2r^2}{F}\right) e^1, \\
    \mathrm{div}_f(\mathbf{b}_4) & = -\frac{s}{8r^2}\left(\frac{4F-  r^2}{F} - \sqrt{2} r^2\right)e^-, &\mathrm{div}_f(\mathbf{b}_8) & = 0, \\
    \mathrm{div}_f(\mathbf{b}_5) & = \frac{s}{8r^2}\left(\frac{4F-  r^2}{F} - \sqrt{2} r^2\right)e^+, &\mathrm{div}_f(\mathbf{b}_9) & =0.\\
\end{align*}
\end{proposition}

\begin{proof}
  For any basis $2$-tensor $\mathbf{b} \in \mathbf{B}$, using $\nabla_{e_0} e^i = 0$, we have  
  \begin{align*}
  \mathrm{div}_f(\mathbf{b}) &= \sum_i (\nabla_{e_i} \mathbf{b})(e_i, \cdot) - \mathbf{b}(\nabla f, \cdot ) \\
  & = (\nabla_{e_1} \mathbf{b}) (e_1, \cdot) + (\nabla_{e_2} \mathbf{b}) (e_2, \cdot) + (\nabla_{e_3} \mathbf{b}) (e_3, \cdot) - \sqrt{2} s \, \mathbf{b}(e_0, \cdot).
  \end{align*}
  Now the formulas asserted follow from Proposition \ref{prop:levi-civita-on-2-tensors}, Definition \ref{def:levi-civita-omega-coef-functions}, Proposition \ref{prop:Gamma-function-comps}, and some simplification. 
\end{proof}

\subsection{The actions of $\Delta_f$ and $R$ on basis elements of symmetric $2$-tensors}

Most of the derivations that follow here are more or less straightforward using the identities for the connection and curvature from the previous section. With a little effort, they can be checked directly by hand, or otherwise straightforwardly verified by computer. For brevity, we omit most of the computational details below. 

\begin{lemma}\label{lem:lap-eis}
\begin{align*}
\Delta_f e^0 & = -\left(\frac{(s')^2}{4r^2}+ \frac{s^2}{2r^4}\right)e^0, & \Delta_f e^1 & =  -\left(\frac{(s')^2}{4r^2}+ \frac{s^2}{2r^4}\right)e^1, \\ 
\Delta_f e^2 & = -\left(\frac{1}{s^2}\Big(1- \frac{s^2}{2r^2}\Big)^2 + \frac{s^2}{2r^4}\right) e^2, & \Delta_f e^3 & = -\left(\frac{1}{s^2}\Big(1 - \frac{s^2}{2r^2}\Big)^2 + \frac{s^2}{2r^4}\right) e^3.
\end{align*}
\end{lemma}

\begin{proof}
     This is a straightforward computation using Lemma \ref{lem:levi-civita}. Additionally, in view of the fact that $\nabla_{e_0} e^i = 0$, it is helpful to note that $\Delta_f e^i = \Delta e^i = (\nabla_{e_1} \nabla_{e_1} + \nabla_{e_2} \nabla_{e_2} + \nabla_{e_3} \nabla_{e_3})e^i$.
\end{proof}

\begin{corollary}\label{cor:lap-eiejs}
On diagonal elements $e^{ii}$, the weighted Laplacian acts by 
\begin{align*}
    \Delta_f(e^{00}) & = \frac{(s')^2}{2r^2}(e^{11} - e^{00}) + \frac{s^2}{2r^4}(-2e^{00} + e^{22} + e^{33}), \\
     \Delta_f(e^{11}) & = \frac{(s')^2}{2r^2}(e^{00} -e^{11})+  \frac{s^2}{2r^4}(-2e^{11} + e^{22} + e^{33}), \\
     \Delta_f(e^{22}) & = \frac{2}{s^2}\left(1- \frac{s^2}{2r^2}\right)^2 (e^{33}-e^{22}) + \frac{s^2}{2r^4}(e^{00} + e^{11}-2e^{22}), \\
    \Delta_f(e^{33}) & = \frac{2}{s^2}\left(1 - \frac{s^2}{2r^2}\right)^2 (e^{22} -e^{33})+ \frac{s^2}{2r^4}(e^{00} + e^{11}-2e^{33}).
\end{align*}
On off diagonal elements $e^{ij} + e^{ji}$, the weighted Laplacian acts by 
\begin{align*}
\Delta_f(e^{01} +e^{10}) & = -\left(\frac{(s')^2}{r^2}+ \frac{s^2}{r^4}\right)(e^{01} + e^{10}), \\
\Delta_f(e^{23} +e^{32}) & = -\left(\frac{4}{s^2} - \frac{4}{r^2} + \frac{2s^2}{r^4}\right)(e^{23} + e^{32}),
\end{align*}
and
\begin{align*}
\Delta_f(e^{02} +e^{20}) & = -\left(\frac{1}{s^2} -\frac{1}{r^2} + \frac{7s^2}{4r^4} + \frac{(s')^2}{4r^2}\right)(e^{02} + e^{20}) + \left(\frac{s^2}{2r^4} - \frac{s'}{sr}\left(1 - \frac{s^2}{2r^2}\right)\right)(e^{31} + e^{13}),\\
\Delta_f(e^{31} +e^{13}) & = -\left(\frac{1}{s^2} -\frac{1}{r^2} + \frac{7s^2}{4r^4} + \frac{(s')^2}{4r^2}\right)(e^{31} + e^{13}) + \left(\frac{s^2}{2r^4} - \frac{s'}{sr}\left(1 - \frac{s^2}{2r^2}\right)\right)(e^{02} + e^{20}), \\
\Delta_f(e^{03} +e^{30}) & = -\left(\frac{1}{s^2} -\frac{1}{r^2} + \frac{7s^2}{4r^4} + \frac{(s')^2}{4r^2}\right)(e^{03} + e^{30}) - \left(\frac{s^2}{2r^4} - \frac{s'}{sr}\left(1 - \frac{s^2}{2r^2}\right)\right)(e^{12} + e^{21}), \\
\Delta_f(e^{12} +e^{21}) & =  -\left(\frac{1}{s^2} -\frac{1}{r^2} + \frac{7s^2}{4r^4} + \frac{(s')^2}{4r^2}\right)(e^{12} + e^{21}) - \left(\frac{s^2}{2r^4} - \frac{s'}{sr}\left(1 - \frac{s^2}{2r^2}\right)\right)(e^{03} + e^{30}). \\
\end{align*}
\end{corollary}
\begin{proof}
    After some simplifications, this is a straightforward computational consequence of Lemma \ref{lem:levi-civita} and Lemma \ref{lem:lap-eis}.
\end{proof}

Our choice of basis was essentially  determined by trying to diagonalize the action of $\Delta_f$ above. For switching from formulas involving $s$ to formulas involving $F$ in the following corollary, it is helpful to use (in addition to \eqref{eq:sder-to-Fder}) the identities 
\begin{equation}\label{eq:sder-to-Fder-2}
\frac{s^2}{r^2} = F, \qquad \frac{1}{s^2} = \frac{1}{Fr^2}, \qquad \frac{s'}{sr} = \frac{1}{r^2} + \frac{F'}{2rF}, \qquad \frac{(s')^2}{r^2} = \frac{F}{r^2} -\frac{F'}{r} + \frac{(F')^2}{4F}.
\end{equation}

\begin{corollary}\label{cor:lap-02basis}
    The weighted Laplacian acts on the basis of $2$-tensors by the following formulas:
\begin{align*}
    \Delta_f\mathbf{b}_0  &= 0, & \Delta_f\mathbf{b}_1  &= -\frac{2F}{r^2} \mathbf{b}_1,
\end{align*}
and
\begin{align*}
    \Delta_f\mathbf{b}_2 & =  -\frac{1}{F}\left(  \frac{(F')^2}{16} + \frac{F^2}{r^2} +\frac{F'}{2r} + \frac{1}{r^2}  \right)\mathbf{b}_2, & \Delta_f\mathbf{b}_6 & = -\frac{1}{F}\left(\frac{(F')^2}{4} + \frac{FF'}{r} +\frac{2F^2}{r^2} \right) \mathbf{b}_6,\\
    \Delta_f\mathbf{b}_3 & = -\frac{1}{F}\left(  \frac{(F')^2}{16} + \frac{F^2}{r^2} +\frac{F'}{2r} + \frac{1}{r^2}  \right) \mathbf{b}_3, & \Delta_f\mathbf{b}_7 & =  -\frac{1}{F}\left(\frac{(F')^2}{4} + \frac{FF'}{r} +\frac{2F^2}{r^2} \right)\mathbf{b}_7,\\
    \Delta_f\mathbf{b}_4 & = -\frac{1}{F}\left(\frac{(F')^2}{16}+ \frac{FF'}{2r} +  \frac{3F^2}{r^2} - \frac{F'}{2r}+\frac{1-2F}{r^2}  \right) \mathbf{b}_4, & \Delta_f\mathbf{b}_8 & =-\frac{1}{F}\left(  \frac{2F^2}{r^2} -\frac{4F}{r^2} + \frac{4}{r^2} \right) \mathbf{b}_8,\\
    \Delta_f\mathbf{b}_5 & =  -\frac{1}{F}\left(\frac{(F')^2}{16}+ \frac{FF'}{2r} +  \frac{3F^2}{r^2} - \frac{F'}{2r}+\frac{1-2F}{r^2}  \right)\mathbf{b}_5, & \Delta_f\mathbf{b}_9 & =-\frac{1}{F}\left(  \frac{2F^2}{r^2} -\frac{4F}{r^2} + \frac{4}{r^2} \right) \mathbf{b}_9.
\end{align*}
\end{corollary}

\begin{proof}
The first formula is immediate since $\Delta_f g = 0$. Here are the details to prove the second formula.
Recall
\[
\mathbf{b}_1 = \omega_1^- \circ \omega_1^+ = \frac{1}{4}(e^{00}+e^{11}-e^{22} - e^{33}). 
\]
Using Corollary \ref{cor:lap-eiejs} (or directly using Lemma \ref{lem:lap-eis} and Lemma \ref{lem:levi-civita}), we have
\[
\Delta_f (e^{00})= -\left(\frac{(s')^2}{2r^2}+ \frac{s^2}{r^4}\right)e^{00}  +\frac{(s')^2}{2r^2} e^{11} + \frac{s^2}{2r^4}(e^{22} + e^{33}). 
\]
An analogous formula for $\Delta_f (e^{11})$  yields 
\[
\Delta_f (e^{00} + e^{11} )=- \frac{F}{r^2} (e^{00} + e^{11}) + \frac{F}{r^2}(e^{22} + e^{33}).
\]
Similarly
\[
\Delta_f(e^{22}) = -\left(\frac{2}{s^2}\Big(1 - \frac{s^2}{2r^2}\Big)^2 + \frac{s^2}{r^4}\right)  e^{22}+ \frac{2}{s^2}\Big(1 - \frac{s^2}{2r^2}\Big)^2e^{33} + \frac{s^2}{2r^2}(e^{00} + e^{11}).
\]
A similar formula holds for $\Delta_f(e^{33})$, and so
\[
\Delta_f (e^{22} + e^{33} )=- \frac{F}{r^2} (e^{22} + e^{33}) + \frac{F}{r^2}(e^{00} + e^{11}).
\]
These identities imply the formula asserted for $\Delta_f\mathbf{b}_1$. After some simplifications and switching from $s$ to $F$, the remaining formulas are similarly straightforward computational consequences of Corollary \ref{cor:lap-eiejs} and the formulas above expressing the $\mathbf{b}_p$ in terms of the $e^{ij}$. 
\end{proof}

\begin{figure}[H]
\includegraphics[scale=0.7]{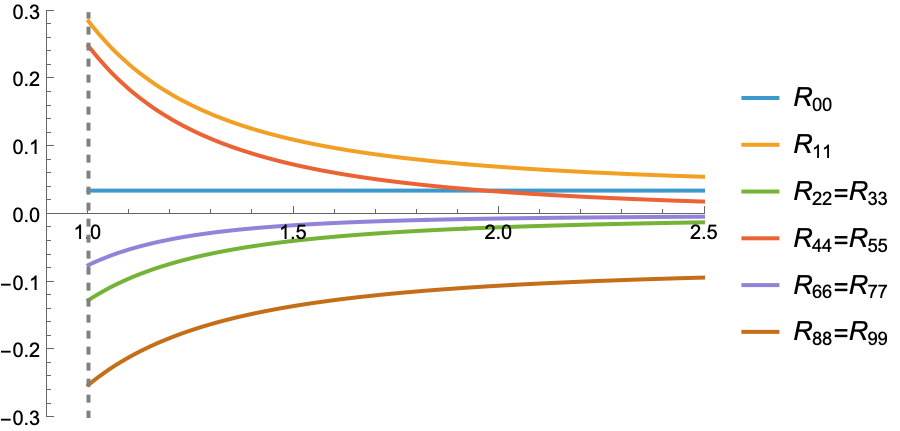}
\caption{Plot of the curvature acting on the basis $2$-tensors, where $R_{pp} := g(R(\mathbf{b}_p), \mathbf{b}_p)$.}\label{fig:curvatures}
\end{figure}

\begin{proposition} \label{prop:riemann-02basis}
    The curvature acts on the basis of 2-tensors by the following formulas:
\begin{align*}
    R(\mathbf{b}_0) & = \left(-\frac{F''}{16} - \frac{7F'}{16r} +\frac{1-F}{2r^2} \right) \mathbf{b}_0+ \left(-\frac{F''}{16} - \frac{3F'}{16r} -\frac{1-F}{2r^2}\right) \mathbf{b}_1,\\
    R(\mathbf{b}_1) &= \left(-\frac{F''}{16} - \frac{3F'}{16r}- \frac{1-F}{2r^2}\right)\mathbf{b}_0 +\left(-\frac{F''}{16} + \frac{F'}{16r}+ \frac{1-F}{2r^2}\right)\mathbf{b}_1.
\end{align*}
and
\begin{align*}
    R(\mathbf{b}_2) &= - \frac{F'}{4r}\mathbf{b}_2,& R(\mathbf{b}_6) &= \left( \frac{F''}{8}  +\frac{3F'}{8r}\right) \mathbf{b}_6,\\
    R(\mathbf{b}_3) &=- \frac{F'}{4r}\mathbf{b}_3, &  R(\mathbf{b}_7) &= \left( \frac{F''}{8}  +\frac{3F'}{8r}\right) \mathbf{b}_7,  \\
    R(\mathbf{b}_4) &= \frac{F'}{2r}\mathbf{b}_4, &  R(\mathbf{b}_8) &= \left( \frac{F-1}{r^2}\right)\mathbf{b}_8,\\
    R(\mathbf{b}_5) &=  \frac{F'}{2r}\mathbf{b}_5,&  R(\mathbf{b}_9) &= \left( \frac{F-1}{r^2}\right) \mathbf{b}_9.
\end{align*} 
\end{proposition}

\begin{proof}
The first formula follows from the fact that $R(\mathbf{b}_0) = R(\frac{1}{4}g) = \frac{1}{2} \mathrm{Ric}$ along with Corollary \ref{cor:ricci-scalar}. Here are the details to prove the second formula. Recalling $\mathbf{b}_1 = \frac{1}{4}(e^{00} + e^{11} -e^{22} - e^{33})$, we have 
\begin{align*}
    R(\mathbf{b}_1) &= \frac{1}{4} \big(R_{0p0q}+ R_{1p1q} - R_{2p2q} - R_{3p3q}\big) e^{pq} \\
    & = \frac{1}{4}(R_{0101}- R_{0202} - R_{0303}) e^{00}  + \frac{1}{4}(R_{0101} - R_{1212} - R_{1313}) e^{11} \\
    & \qquad +\frac{1}{4}(R_{0202} + R_{1212} - R_{2323}) e^{22} + \frac{1}{4}(R_{0303} + R_{1313} - R_{2323}) e^{33}.
\end{align*}
Using the formulas from Proposition \ref{prop:riemann-curvatures}, this gives 
\begin{align*}
    R(\mathbf{b}_1) &=  \frac{1}{4}\left(-\frac{F''}{8} - \frac{F'}{8r}\right)(e^{00} + e^{11}) +\frac{1}{4}\left(-\frac{F'}{4r}- \frac{1-F}{r^2}\right) (e^{22} +e^{33})\\
    & =  \left(-\frac{F''}{16} - \frac{3F'}{16r}- \frac{1-F}{2r^2}\right)\mathbf{b}_0 +\left(-\frac{F''}{16} + \frac{F'}{16r}+ \frac{1-F}{2r^2}\right)\mathbf{b}_1.
\end{align*}
    After some minor simplifications, the remaining formulas are similarly straightforward computational consequences of Proposition \ref{prop:riemann-curvatures} and the formulas  that express the $\mathbf{b}_p$ in terms of the $e^{ij}$. 
\end{proof}

\begin{corollary}\label{cor:riemann-02-basis}
Let $R_{pq} = g(R(\mathbf{b}_p), \mathbf{b}_q)$. As functions of the radial coordinate $r$, the nonzero curvatures are given by
    \begin{align*}
R_{00}&
=\frac{c_0\sqrt2}{16\,r^2}, & R_{11}&=
\frac{c_0\sqrt2}{16\,r^2}+ \frac{c_0\big(\,r^2+\sqrt2\big)}{4r^6},& R_{01} = R_{10} & =  -\frac{c_0( r^2 +\sqrt{2})}{8\sqrt{2}\,r^4}, 
\end{align*}
and 
\begin{align*}
R_{22} = R_{33} &
=-\frac{c_0\big(r^2+\sqrt2\big)}{8\,r^6}
& 
R_{66}= R_{77} &
=-\frac{c_0\sqrt2}{8\,r^6}, \\
R_{44} = R_{55} &
=\frac{c_0\big(r^2+\sqrt2\big)}{4r^6},
& 
R_{88}= R_{99} &
=-\frac{c_0\sqrt{2}}{8\,r^6} - \frac{c_0\sqrt{2}(r^2+\sqrt{2})}{8r^4}.
\end{align*}
In particular, $R_{22}, R_{33}, R_{66}, R_{77}, R_{88}, R_{99} < 0$, and $R_{00}, R_{11}, R_{44}, R_{55} > 0$. 
\end{corollary}
\begin{proof} 
These formulas follow straightforwardly from \eqref{eq:Fder1}, \eqref{eq:Fder2}, and \eqref{eq:Fder3}. 
\end{proof}

\subsection{The action of $L_f$ on basis elements of symmetric $2$-tensors}

We now compute the action of $L_f$ on our basis of $(0, 2)$-tensors.

\begin{definition}\label{def:Lf-coef-functions} The $\Lambda$-functions are:
    \begin{align*}
        \Lambda_{11}^{++} &:= -\frac{F''}{8} - \frac{7F'}{8r} -\frac{F}{r^2} +  \frac{1}{r^2} . \\
        \Lambda_{11}^{--} &:= -\frac{F''}{8} + \frac{F'}{8r} - \frac{3F}{r^2} + \frac{1}{r^2}. \\
        \Lambda_{11}^{\pm} &:= -\frac{F''}{8} - \frac{3F'}{8r} +\frac{F}{r^2}-\frac{1}{r^2}  . \\
        \Lambda_{1+} &:= \frac{1}{F}\left(  -\frac{(F')^2}{16}-  \frac{FF'}{2r}  -\frac{F^2}{r^2} -\frac{F'}{2r} -\frac{1}{r^2}  \right). \\
        \Lambda_{1-} &:=\frac{1}{F}\left( -\frac{(F')^2}{16}+ \frac{FF'}{2r}- \frac{3F^2}{r^2} + \frac{F'}{2r}+\frac{2F}{r^2}- \frac{1}{r^2}   \right). \\
        \Lambda_{01} &:= \frac{1}{F}\left( \frac{FF''}{4} -\frac{FF'}{4r} -\frac{(F')^2}{4} -\frac{2F^2}{r^2} \right).\\
        \Lambda_{23} &:= \frac{1}{F}\left(  \frac{2F}{r^2} - \frac{4}{r^2} \right). 
    \end{align*}
\end{definition}

\begin{figure}
  \centering
  \includegraphics[scale=0.7]{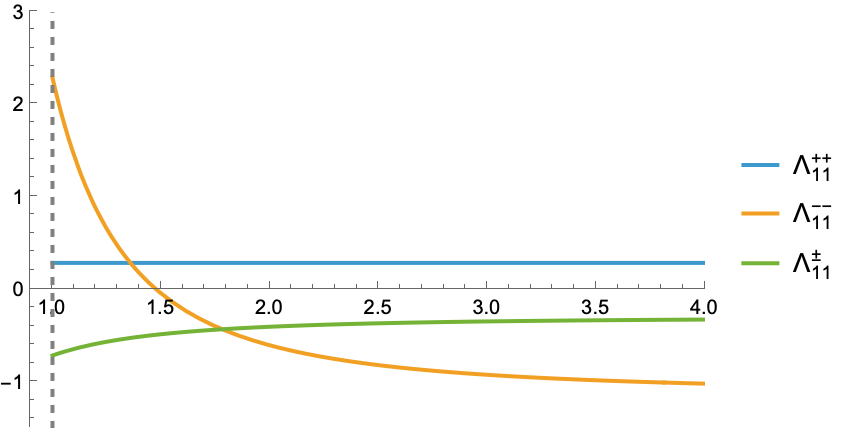}
  \caption{Plots of $r^2\Lambda(r)$ for $\Lambda \in \{\Lambda_{11}^{++}, \Lambda_{11}^{--}, \Lambda_{11}^{\pm}\}$. }\label{fig:Lambda1}
\end{figure}

\begin{figure}
  \centering
  \includegraphics[scale=0.7]{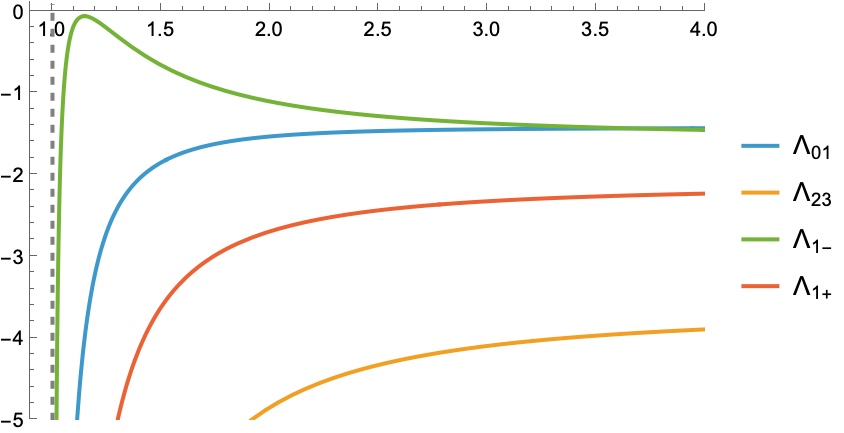}
  \caption{Plots of $r^2\Lambda(r)$ for $\Lambda \in \{\Lambda_{01}, \Lambda_{23}, \Lambda_{1-}, \Lambda_{1+}\}$. }\label{fig:Lambda2}
\end{figure}

Combining the expressions obtained in Corollary \ref{cor:lap-02basis} and Proposition \ref{prop:riemann-02basis}, we obtain the following corollary.

\begin{corollary}\label{cor:Lf-action-on-basis}
The action of $L_f$ on basis $2$-tensors is given by
\begin{equation*}
L_f \begin{bmatrix} \mathbf{b}_0 \\ \mathbf{b}_1 \\\mathbf{b}_2 \\ \mathbf{b}_3 \\\mathbf{b}_4 \\ \mathbf{b}_5 \\ \mathbf{b}_6 \\ \mathbf{b}_7 \\ \mathbf{b}_8 \\ \mathbf{b}_9 \end{bmatrix} = \begin{bmatrix} 
\Lambda_{11}^{++} & \Lambda_{11}^{\pm} & 0 & 0 & 0 & 0 & 0 & 0 & 0 & 0\\
\Lambda_{11}^{\pm} & \Lambda_{11}^{--} & 0 & 0 & 0 & 0 & 0 & 0 & 0 & 0 \\
0 & 0 & \Lambda_{1+} & 0 & 0 & 0 & 0 & 0 & 0 & 0 \\
0 & 0 & 0& \Lambda_{1+}  & 0 & 0 & 0 & 0 & 0 & 0 \\
0 & 0 & 0 & 0 & \Lambda_{1-} & 0 & 0 & 0 & 0 & 0 \\
0 & 0 & 0 & 0 & 0 & \Lambda_{1-} & 0 & 0 & 0 & 0 \\
0 & 0 & 0 & 0 & 0 & 0 & \Lambda_{01} & 0 & 0 & 0 \\
0 & 0 & 0 & 0 & 0 & 0 & 0 & \Lambda_{01} & 0 & 0 \\
0 & 0 & 0 & 0 & 0 & 0 & 0 & 0 & \Lambda_{23} & 0 \\
0 & 0 & 0 & 0 & 0 & 0 & 0 & 0 & 0 & \Lambda_{23}
\end{bmatrix}\begin{bmatrix} \mathbf{b}_0 \\ \mathbf{b}_1 \\\mathbf{b}_2 \\ \mathbf{b}_3 \\\mathbf{b}_4 \\ \mathbf{b}_5 \\ \mathbf{b}_6 \\ \mathbf{b}_7 \\ \mathbf{b}_8 \\ \mathbf{b}_9 \end{bmatrix}.
\end{equation*}
\end{corollary}

It remains to compute the $L_f$  component functions $\Lambda = \Lambda(r)$ as functions of $r$ using the definition of $F$. We are mostly interested in the sign of these $\Lambda$-functions. Numerically obtained plots of the functions can be see in Figures \ref{fig:Lambda1} and \ref{fig:Lambda2}.  We summarize our results in the following proposition. 
\begin{proposition}\label{prop:Lambda-function-comps}
As functions of the radial coordinate $r$, the $\Lambda$-functions obtained by the action of $L_f$ on the basis of $2$-tensors are given by 
\begin{align*}
    \Lambda_{11}^{++} &= \frac{\sqrt{2}}{r^2} \left(\frac{c_0}{2}\right) > 0, \\
    \Lambda_{11}^{--} &= \frac{\sqrt{2}}{r^2}\left( \frac{3c_0}{r^4} + \frac{2c_0\sqrt{2}}{r^2} - \frac{3c_0+1}{2\sqrt{2}}\right) , \\
    \Lambda_{11}^{\pm} &= \frac{\sqrt{2}}{r^2}\left( -\frac{c_0^2}{\sqrt{2}r^2} - \frac{c_0}{2} \right)< 0,
\end{align*}
    and
\begin{align*}
    \Lambda_{1+} &= \frac{1}{Fr^2}\left( -\frac{4\sqrt{2}c_0+c_0^2}{4r^4} -\frac{c_0^2\sqrt{2}}{2r^2}  - \frac{3}{2}  \right) < 0,\\
    \Lambda_{1-} &= \frac{1}{Fr^2}\left( -\frac{3c_0^2}{r^8}  - \frac{5c_0^2\sqrt{2}}{r^6} +\frac{16c_0-17c_0^2}{4r^4}  + \frac{(7\sqrt{2}-2)c_0}{2r^2} +c_0 -\frac{3}{2} \right)< 0, \\
    \Lambda_{01} &= \frac{1}{Fr^2}\left(-\frac{c_0\sqrt{2}}{r^4} + \frac{c_0\sqrt{2}}{r^2}-1\right) < 0 ,\\
    \Lambda_{23} &= \frac{1}{Fr^2}\left(-\frac{c_0\sqrt{2}}{r^4} - \frac{2c_0}{r^2}+\sqrt{2}-4\right) < 0.
\end{align*}
\end{proposition}

\begin{proof}
    These formulas follow straightforwardly from \eqref{eq:Fder1}, \eqref{eq:Fder2}, \eqref{eq:Fder3}
    and Definition \ref{def:Lf-coef-functions}. The signs of $\Lambda_{11}^{++}, \Lambda_{11}^{\pm}, \Lambda_{1+}$, and $\Lambda_{23}$ follow simply because each coefficient of $r^{-k}$ has the asserted sign. Note that $\Lambda_{01} =  \frac{1}{Fr^2}\big((2-\sqrt{2})\frac{r^2-1}{r^4} -1\big)$, so from the estimate $0 \leq \frac{r^2-1}{r^4} \leq \frac{1}{4}$ for $r \in [1, \infty)$, we see that $\Lambda_{01} < 0$. The negativity of $\Lambda_{1-}$ is more subtle than the other $\Lambda$-functions, so we defer a proof of this fact to Lemma \ref{lem:Lambda1m-negativity}.
\end{proof}

\begin{lemma}\label{lem:Lambda1m-negativity}
For $r \geq 1$, $\Lambda_{1-}(r) < 0$.
\end{lemma}
\begin{proof}
The negativity of $\Lambda_{1-}(r)$ is equivalent to the negativity of the polynomial 
\[
p_{1-}(Y) := 4\Lambda_{1-}(\sqrt{Y})F(\sqrt{Y})Y^{5},
\]
where, for convenience, we have introduced the variable $Y = r^2 \in [1, \infty)$. In view the formula derived for $\Lambda_{1-}$ above, $p_{1-}(Y)$ is a quartic polynomial given by 
\[
p_{1-}(Y) = (-12c_0^2) + (- 20c_0^2\sqrt{2})Y +(16c_0-17c_0^2)Y^2  + ((14\sqrt{2}-4)c_0) Y^3 +(4c_0 -6)Y^4.
\]
By writing out the coefficients using $c_0 = \sqrt{2}-1$ and a little algebra, we have
\[
p_{1-}(Y) = -(36 - 24\sqrt{2}) - (-80 + 60 \sqrt{2}) Y + (-67 + 50 \sqrt{2}) Y^2 + (32 - 
    18 \sqrt{2}) Y^3 - (10 - 4 \sqrt{2}) Y^4.
\]

Now we claim that: 
\[
p_{1-}(0) < 0, \qquad p_{1-}(-1/2) > 0, \qquad \text{and} \quad p_{1-}(Y) \to -\infty, \quad \text{as} \quad Y \to -\infty.
\]
For the first assertion, we simply see  $p_{1-}(0) =-12(3-2\sqrt{2})< 0$. Additionally, as $10-4\sqrt{2} > 0$, the third assertion is clear. For the second assertion, we compute that
\begin{align*}
p_{1-}(-0.5)& = -(36 - 24\sqrt{2}) + (-40 + 30 \sqrt{2})+ (-\frac{67}{4} + \frac{25}{2} \sqrt{2}) - (4 - 
    \frac{9}{4} \sqrt{2}) - (\frac{5}{8} - \frac{1}{4} \sqrt{2})\\
& =\frac{1}{8}( -779 +552\sqrt{2}) > 0.
\end{align*}
By the three assertions above, we conclude that $p_{1-}(Y)$ has at least 2 negative real roots (by the intermediate value theorem). It follows that $p_{1-}(Y)$ must have either 2 real roots or 4 real roots. 

On the other hand, with just the coefficients of $p_{1-}(Y)$, one can compute (easily with a computer, or with significant effort by hand) that the discriminant of $p_{1-}(Y)$ is given by 
\[
\mathrm{discriminant}(p_{1-}) = -64 (-22969196 + 16241696\sqrt{2}) < 0.
\]
Numerically, $\mathrm{disc}(p_{1-})\approx -1968.595$ (see the remark below). The discriminant of a quartic polynomial is negative only if the quartic has exactly 2 real roots (and hence 2 complex roots). But we have already seen that $p_{1-}(Y)$ has two negative roots. Therefore $p_{1-}(Y) < 0$ for $Y \in (0, \infty)$ and in particular also for $Y \in (1, \infty)$. It follows that $\Lambda_{1-}(r) < 0$ for $r \geq 1$. 
\end{proof}

\begin{remark}
    For estimating the discriminant numerically, using that $\sqrt{2} \approx 1.414$, one has
    \[
    p_{1-}(Y) \approx -2.06 -4.84Y +3.70 Y^2 +6.55 Y^3 -4.34Y^4. 
    \]
\end{remark}

\subsection{Expressing $\delta^2 \nu_g$ in the basis}\label{sec:matrices}

The following matrices capture the actions of $L_f$ and $\nabla$ on basis symmetric $2$-tensors. See Lemma \ref{lem:ibp-for-H1} below.

For the $J_1^+$-invariant part of a symmetric $2$-tensor, we define:
\begin{align*}
    \mathcal{M}_\Lambda & := \begin{bmatrix} \Lambda_{11}^{++} & \Lambda_{11}^{\pm} & 0 & 0 \\
    \Lambda_{11}^{\pm} & \Lambda_{11}^{--} & 0 & 0 \\
    0 & 0 & \Lambda_{1+} & 0 \\
    0 & 0 & 0 & \Lambda_{1+}\end{bmatrix}, & \mathcal{M}_\Gamma^1& := 2\Gamma_{23}^-\begin{bmatrix} 
    0 & 0 & 0 & 0 \\
    0 & 0 & 0 & 0 \\
    0 & 0 & 0 & -1\\
    0 & 0 & 1 & 0
    \end{bmatrix}, \\
    \mathcal{M}_\Gamma^-& := 2\Gamma_1^-\begin{bmatrix} 
    0 & 0 & 0 & 0 \\
    0 & 0 & -1 & 0 \\
    0 & 1 & 0 & 0 \\
    0 & 0 & 0 & 0
    \end{bmatrix}, & \mathcal{M}_\Gamma^+& := 2\Gamma_1^-\begin{bmatrix} 
    0 & 0 & 0 & 0 \\
    0 & 0 & 0 & 1 \\
    0 & 0 & 0 & 0 \\
    0 & -1 & 0 & 0
    \end{bmatrix}.
\end{align*}
For the $J_1^+$-anti-invariant part of a symmetric $2$-tensor, we define:
\begin{align*}
    \mathcal{N}_\Lambda & := \begin{bmatrix} 
    \Lambda_{1-} & 0 & 0 & 0 & 0 & 0 \\
    0 & \Lambda_{1-} & 0 & 0 & 0 & 0\\
    0 & 0 & \Lambda_{01} & 0 & 0 & 0 \\
    0 & 0 & 0 & \Lambda_{01} & 0 & 0\\
    0 & 0 & 0 & 0 & \Lambda_{23} & 0\\
    0 & 0 & 0 & 0 & 0 & \Lambda_{23}
    \end{bmatrix}
    \end{align*}
    \begin{align*}
      \mathcal{N}_\Gamma^1& := 2\Gamma_1^+\begin{bmatrix} 
    0 & 1& 0 & 0 & 0 & 0 \\
    -1 & 0 & 0 & 0 & 0 & 0\\
    0 & 0 & 0 & 1 & 0 & 0 \\
    0 & 0 & -1 & 0 & 0 & 0\\
    0 & 0 & 0 & 0 & 0 & 1\\
    0 & 0 & 0 & 0 & -1& 0
    \end{bmatrix} + \Gamma_{23}^-\begin{bmatrix} 
    0 & 0& 0 & 0 & 0 & 0 \\
    0 & 0 & 0 & 0 & 0 & 0\\
    0 & 0 & 0 & 1 & 0 & 0 \\
    0 & 0 & -1 & 0 & 0 & 0\\
    0 & 0 & 0 & 0 & 0 & -1\\
    0 & 0 & 0 & 0 & 1& 0
    \end{bmatrix}, 
    \end{align*}
    \begin{align*}
     \mathcal{N}_\Gamma^-& := \sqrt{2} \Gamma_1^- \begin{bmatrix} 
    0 & 0 & -1 & 0 & -1 & 0 \\
    0 & 0 & 0 & -1 & 0 & -1\\
    1 & 0 & 0 & 0 & 0 & 0 \\
    0 & 1 & 0 & 0 & 0 & 0\\
    1 & 0 & 0 & 0 & 0 & 0\\
    0 & 1 & 0 & 0 & 0 & 0
    \end{bmatrix}, \qquad  \mathcal{N}_\Gamma^+ :=\sqrt{2} \Gamma_1^- \begin{bmatrix} 
    0 & 0 & 0 & -1 & 0 & 1\\
    0 & 0 & 1 & 0 & -1 & 0\\
    0 & -1 & 0 & 0 & 0 & 0 \\
    1 & 0 & 0 & 0 & 0 & 0\\
    0 & 1 & 0 & 0 & 0 & 0\\
    -1 & 0 & 0 & 0 & 0 & 0
    \end{bmatrix},
\end{align*}

We now derive a formula for principal part of $\delta^2 \nu_g(h)$ with respect to the basis of symmetric $2$-tensors. The important take-away here is there is a decoupling between the $J_1^+$-invariant and the $J_1^+$-anti-invariant parts of a deformation. The lemma is essentially a consequence of Proposition \ref{prop:levi-civita-on-2-tensors} and Corollary \ref{cor:Lf-action-on-basis}. To express our formulas succinctly we make use of the vector notation $\vec{h}_I : M \to \mathbb{R}^4$ and $\vec{h}_A : M \to \mathbb{R}^6$ given by 
\begin{align}\label {eq:vector-representations}
    \vec{h}_I &:=  (h_0, h_1, h_2, h_3), & \vec{h}_A &:= (h_4, h_5, h_6, h_7, h_8, h_9). 
\end{align}
 We will use ``$\cdot$" to denote the inner product on $\mathbb{R}^4, \mathbb{R}^6$. As a shorthand, we let $|\nabla \vec{h}_I|^2 =|\nabla h_0|^2 + |\nabla h_1|^2+|\nabla h_2|^2+|\nabla h_3|^2$, and similarly define $|\nabla \vec{h}_A|^2$.

\begin{lemma}\label{lem:ibp-for-H1}
Suppose $h \in C^{\infty}(M) \cap H^1_f(M)$  with component functions defined by the identity \eqref{eq:h-in-basis}. Then 
\begin{align*}
       & \int_M (2 R(h, h) - |\nabla h|^2) \, e^{-f} d\mu_g   \\
       & =  \int_M (2 R(h_I, h_I) - |\nabla h_I|^2) \, e^{-f} d\mu_g + \int_M (2 R(h_A, h_A) - |\nabla h_A|^2) \, e^{-f} d\mu_g,
\end{align*}
with 
\begin{align}
\nonumber    &\int_M (2 R(h_I, h_I) - |\nabla h_I|^2) \, e^{-f} d\mu_g \\
    & = \frac{1}{4} \int_M\left( \vec{h}_I \cdot \mathcal{M}_{\Lambda} \vec{h}_I + \vec{h}_I\cdot\Big(\mathcal{M}_{\Gamma}^1e_1+ \mathcal{M}_{\Gamma}^-e_-+\mathcal{M}_{\Gamma}^+e_+\Big)\vec{h}_I - |\nabla \vec{h}_I|^2 \right) e^{-f} d\mu_g,
\end{align}
and 
\begin{align}
  \nonumber &  \int_M (2 R(h_A, h_A) - |\nabla h_A|^2) \, e^{-f} d\mu_g\\
  & = \frac{1}{4} \int_M \left( \vec{h}_A \cdot \mathcal{N}_{\Lambda} \vec{h}_A + \vec{h}_A\cdot\Big(\mathcal{N}_{\Gamma}^1e_1+ \mathcal{N}_{\Gamma}^-e_-+\mathcal{N}_{\Gamma}^+e_+\Big)\vec{h}_A - |\nabla \vec{h}_A|^2 \right) e^{-f} d\mu_g.
\end{align}
where $\mathcal{M}_{\Lambda},\mathcal{M}_{\Gamma}^1,\mathcal{M}_{\Gamma}^-,\mathcal{M}_{\Gamma}^+$ and $\mathcal{N}_{\Lambda},\mathcal{N}_{\Gamma}^1,\mathcal{N}_{\Gamma}^-,\mathcal{N}_{\Gamma}^+$ are the matrices defined above.
\end{lemma}
\begin{proof}
We observe that 
\begin{align*}
    2R(h, h) - |\nabla h|^2 & = \langle L_f h, h \rangle - \mathrm{div}_f(\langle \nabla h ,h \rangle)
\end{align*}
while 
\begin{align*}
L_f h &= \sum_{p=0}^9\Big( (\Delta_f h_p)\mathbf{b}_p + h_p L_f \mathbf{b}_p\Big) + 2 \sum_{i=0}^3 e_i(h_p) \nabla_{e_i} \mathbf{b}_p \\
& = \sum_{p=0}^9\Big( (\Delta_f h_p)\mathbf{b}_p + h_p L_f \mathbf{b}_p\Big) \\
& \qquad + 2 \sum_{p=0}^9 e_1(h_p) \nabla_{e_1} \mathbf{b}_p  +  e_-(h_p) \nabla_{e_-} \mathbf{b}_p  +  e_+(h_p) \nabla_{e_+} \mathbf{b}_p.
\end{align*}
Here we have used that $\nabla_{e_0} \mathbf{b}_p = 0$. Using Corollary \ref{cor:Lf-action-on-basis} and $\langle \mathbf{b}_p, \mathbf{b}_q \rangle = \frac{1}{4} \delta_{pq}$, it follows that 
\begin{align*} 
4\Big\langle\sum_{p=0}^9\Big( (\Delta_f h_p)\mathbf{b}_p + h_p L_f \mathbf{b}_p\Big), h \Big \rangle & = \sum_{p=0}^9 \left( h_p \Delta_f h_p \right) + h_0(\Lambda_{11}^{++} h_0 + \Lambda_{11}^{\pm} h_1)+ h_1(\Lambda_{11}^{--} h_1 + \Lambda_{11}^{\pm} h_0)\\
&  + \Lambda_{1+}(h_2^2 + h_3^2) + \Lambda_{1-} (h_4^2 + h_5^2) + \Lambda_{01} (h_6^2 + h_7^2) + \Lambda_{23} (h_8^2 + h_9^2)
\end{align*}
Note that 
\begin{align*}
\sum_{p=0}^9 (h_p \Delta_f h_p) - 4\mathrm{div}_f(\langle \nabla h, h \rangle) &= \sum_{p=0}^9\mathrm{div}_f(h_p \nabla h_p) -\sum_{p = 0}^9 \mathrm{div}_f( \,h_p \nabla h_p + 4h_p \langle \nabla \mathbf{b}_p, h \rangle) \\
& \qquad - \sum_{p=0}^9 |\nabla h_p|^2 \\
& =4 \sum_{p=0}^9 \mathrm{div}_f(h_p \langle \nabla \mathbf{b}_p, h \rangle) -\sum_{p=0}^9 |\nabla h_p|^2. 
\end{align*}
Therefore,
\begin{align*}
  4\Big(  2R(h, h) - |\nabla h|^2\Big) & = 4\Big\langle\sum_{p=0}^9\Big( (\Delta_f h_p)\mathbf{b}_p + h_p L_f \mathbf{b}_p\Big), h \Big \rangle  - 4\mathrm{div}_f(\langle \nabla h ,h \rangle) \\
  &  +  8 \Big\langle \sum_{p=0}^9 e_1(h_p) \nabla_{e_1} \mathbf{b}_p, h \Big \rangle  +  8\Big\langle \sum_{p=0}^9 e_-(h_p) \nabla_{e_-} \mathbf{b}_p, h \Big \rangle  +  8 \Big\langle \sum_{p=0}^9 e_+(h_p) \nabla_{e_+} \mathbf{b}_p, h \Big \rangle  \\
  & = 4 \sum_{p=0}^9 \mathrm{div}_f(h_p \langle \nabla \mathbf{b}_p, h \rangle) -\sum_{p=0}^9 |\nabla h_p|^2+ h_0(\Lambda_{11}^{++} h_0 + \Lambda_{11}^{\pm} h_1)+ h_1(\Lambda_{11}^{--} h_1 + \Lambda_{11}^{\pm} h_0)\\
&  + \Lambda_{1+}(h_2^2 + h_3^2) + \Lambda_{1-} (h_4^2 + h_5^2) + \Lambda_{01} (h_6^2 + h_7^2) + \Lambda_{23} (h_8^2 + h_9^2) \\
& +  8 \Big\langle \sum_{p=0}^9 e_1(h_p) \nabla_{e_1} \mathbf{b}_p, h \Big \rangle  +  8 \Big\langle \sum_{p=0}^9 e_-(h_p) \nabla_{e_-} \mathbf{b}_p, h \Big \rangle  +  8 \Big\langle \sum_{p=0}^9 e_+(h_p) \nabla_{e_+} \mathbf{b}_p, h \Big \rangle
\end{align*}
Using Proposition \ref{prop:levi-civita-on-2-tensors}, we have
\begin{align*}
   4 \Big\langle \sum_{p=0}^9 e_1(h_p) \nabla_{e_1} \mathbf{b}_p, h \Big \rangle  & = \Gamma_{23}^- \Big(e_1(h_2)h_3 -  e_1(h_3) h_2\Big)  + \Gamma_1^+\Big( e_1(h_5)h_4-e_1(h_4) h_5  \Big)\\
   & \qquad + (\Gamma_1^+ + \Gamma_{23}^-) \Big( e_1(h_7) h_6-e_1(h_6) h_7\Big)\\ 
   &\qquad + (\Gamma_1^+ - \Gamma_{23}^-)\Big(e_1(h_9) h_8 - e_1(h_8) h_9\Big);
\end{align*}
and
\begin{align*}
   4 \Big\langle \sum_{p=0}^9 e_-(h_p) \nabla_{e_-} \mathbf{b}_p, h \Big \rangle  & = \Gamma_1^- \Big(e_-(h_1)h_2 -  e_-(h_2) h_1\Big)\\
   & \qquad + \frac{1}{\sqrt{2}}\Gamma_1^-\Big(e_-(h_4)h_6- e_-(h_6)h_4  + e_-(h_4)h_8- e_-(h_8)h_4 \Big)\\
   & \qquad + \frac{1}{\sqrt{2}}\Gamma_1^-\Big(   e_-(h_5)h_7 - e_-(h_7)h_5+ e_-(h_5) h_9  - e_-(h_9)h_5   \Big);
\end{align*}
and
\begin{align*}
   4 \Big\langle \sum_{p=0}^9 e_+(h_p) \nabla_{e_+} \mathbf{b}_p, h \Big \rangle  & = \Gamma_1^- \Big(e_+(h_3)h_1 -  e_+(h_1) h_3\Big)\\
   & \qquad + \frac{1}{\sqrt{2}}\Gamma_1^-\Big(e_+(h_4)h_7 -e_+(h_7)h_4 + e_+(h_9)h_4- e_+(h_4)h_9 \Big)\\
   & \qquad + \frac{1}{\sqrt{2}}\Gamma_1^-\Big(  e_+(h_6)h_5 -  e_+(h_5)h_6 + e_+(h_5) h_8  - e_+(h_8)h_5   \Big).
\end{align*}
Inspecting the definitions of the matrices given in Section \ref{sec:matrices}, we put the equations above together to find
\begin{align*}
    4\Big(2R(h,h) - |\nabla h|^2\Big) &= - |\nabla \vec{h}_I|^2+ \vec{h}_I\cdot \mathcal{M}_\Lambda\vec{h}_I  - |\nabla \vec{h}_A|^2 + \vec{h}_A\cdot \mathcal{N}_\Lambda\vec{h}_A  \\
    & \quad + \vec{h}_I\cdot\Big(\mathcal{M}_{\Gamma}^1e_1+ \mathcal{M}_{\Gamma}^-e_-+\mathcal{M}_{\Gamma}^+e_+\Big)\vec{h}_I+ \vec{h}_A\cdot\Big(\mathcal{N}_{\Gamma}^1e_1+ \mathcal{N}_{\Gamma}^-e_-+\mathcal{N}_{\Gamma}^+e_+\Big)\vec{h}_A \\
    & - 4 \sum_{p=0}^9 \mathrm{div}_f (h_p \langle \nabla \mathbf{b}_p, h \rangle ).
\end{align*}
From this last expression, we readily see the asserted decoupling. We note also that since $h \in H^1_f(M)$, suitable bounds for the basis $2$-tensors imply $h_p \mathrm\langle \nabla \mathbf{b}_p, h \rangle \in H^1_f(M)$ and so by Lemma \ref{lem:integration-by-parts} above the integral of the divergence is zero. 
\end{proof}

\section{Coordinates, Berger spheres, and Wigner functions}\label{app:3}

\subsection{The FIK metric near the tip}\label{sec:extending-FIK-to-tip} 

The coordinate expression for the FIK metric in \eqref{eq:FIK-metric}, which is equivalently written, 
\[
g = \frac{4\sqrt{2}r^4}{(r^2-1)(r^2 + \sqrt{2}-1)} dr^2 + \frac{4(r^2-1)(r^2 + \sqrt{2}-1)}{\sqrt{2}r^2} \eta_1^2 + 4r^2(\eta_2^2 + \eta_3^2)
\]
on $M \setminus \Sigma \cong (1, \infty) \times \mathbb{S}^3$ naturally extends to all of $M$. As $F$ is a monotone function of $r$, we may define a new radial coordinate
\begin{align*}
\rho := \int_1^r \frac{2}{\sqrt{F(t)}}\, dt,
\end{align*}
so that $\rho \in (0, \infty)$, 
\[
g = d\rho^2 + U(\rho)\eta_1^2 + V(\rho)(\eta_2^2 + \eta_3^2), 
\]
where $U(\rho) = C(r(\rho))F(r(\rho))$ and $V(\rho) = C(r(\rho))$. For $r$ near $1$, we have 
\begin{align*}
    F(r) &= 2(r -1) + O((r-1)^2), \\
    \sqrt{F(r)} &= \sqrt{2}(r-1)^{\frac{1}{2}} + O((r-1)^{\frac{3}{2}}), \\
    \frac{2}{\sqrt{F(r)}} &=  \sqrt{2}(r-1)^{-\frac{1}{2}}  + O((r-1)^{\frac{1}{2}}).
\end{align*}
Therefore 
\[
\rho(r) = 2\sqrt{2}(r-1)^{\frac{1}{2}} + O((r-1)^{\frac{3}{2}}), 
\]
which gives
\[
r(\rho) = 1 + \frac{1}{8}\rho^2 + O(\rho^4)
\]
It follows that near $\rho = 0$, we have $U(\rho) = \rho^2 + O(\rho^4)$ and $V(\rho) = 4 + \rho^2 + O(\rho^4)$. 
Thus, near $\rho = 0$, our metric has the form 
\[
g = d\rho^2 + \rho^2 g_{\mathbb{S}^3} + 4 (\eta_2^2 + \eta_3^2) + O(\rho^4)\eta_1^2 + O(\rho^4)(\eta_2^2 + \eta_3^2). 
\]
From these asymptotics, we see the metric readily extends to the round metric on $\mathbb{S}^2$ at $\rho = 0$.

\subsection{Euler angles and frames}

We review the left and right invariant frames on $\mathbb{S}^3$. Endow $\mathbb{R}^4$ with Euclidean coordinates $(x_0, x_1, x_2, x_3)$ define a radial coordinate $\rho = (x_0^2 + x_1^2 + x_2^2 + x_3^2)^{\frac{1}{2}}$. Define vector fields $X_i$ and 1-forms $\eta_i$ by   
\begin{align*}
X_1 &:= x_0 \partial_{x_1} - x_1 \partial_{x_0} + x_2 \partial_{x_3} - x_3 \partial_{x_2} 
& \rho^2 \eta_1 &:=  x_0 dx_1 - x_1 dx_0 + x_2 dx_3 - x_3 dx_2 , \\
X_2 &:= x_0 \partial_{x_2} - x_2 \partial_{x_0} + x_3 \partial_{x_1}  - x_1 \partial_{x_3}, 
& \rho^2 \eta_2 &:=  x_0 dx_2 - x_2 dx_0 + x_3 dx_1 - x_1 dx_3,  \\
X_3 &:= x_0 \partial_{x_3} - x_3 \partial_{x_0} + x_1 \partial_{x_2} - x_2 \partial_{x_1}, 
& \rho^2 \eta_3 &:=  x_0 dx_3 - x_3 dx_0 + x_1 dx_2 - x_2 dx_1.
\end{align*}
We can readily verify that $[X_i, X_j] = -2 X_k$ for $i, j, k$ a cyclic permutation of $1, 2, 3$. 
We may introduce Euler angular coordinates, $\rho \in (0, \infty)$, $\theta \in (0, \pi)$, $|\phi \pm \psi| < 2\pi$ on $\mathbb{R}^4$, 
\begin{align*}
x_0 &= \rho \cos\frac{\theta}{2} \cos \left(\frac{\psi + \phi}{2} \right), & x_1 &= \rho \cos\frac{\theta}{2} \sin \left(\frac{\psi +  \phi}{2} \right), \\
x_2 &= \rho \sin\frac{\theta}{2} \cos \left(\frac{\psi -  \phi}{2} \right), & x_3 &= \rho \sin\frac{\theta}{2} \sin \left(\frac{\psi -  \phi}{2} \right).
\end{align*}
In these coordinates, the left-invariant frame on $\mathbb{S}^3$ is given by 
\begin{align*}
X_1 & = 2 \partial_\psi, & \eta_1 &= \frac{1}{2}( d \psi + \cos\theta\, d\phi),   \\
X_2 & = 2 \cos \psi \, \partial_\theta + 2\frac{\sin\psi}{\sin\theta}\big( \partial_\phi-   \cos \theta \,\partial_\psi \big), & \eta_2 &= \frac{1}{2}( \cos \psi \, d \theta + \sin \theta\,  \sin \psi \, d\phi) ,\\
X_3 & = 2 \sin \psi \, \partial_\theta - 2\frac{\cos\psi}{\sin\theta}\big( \partial_\phi  -  \cos \theta \, \partial_\psi \big), & \eta_3 &= \frac{1}{2}( \sin \psi \, d \theta - \sin \theta\,  \cos \psi \, d\phi) .
\end{align*}  
Note that the (left-invariant) round metric on $\mathbb{S}^3$ is given in these coordinates by
\[
\eta_1^2 + \eta_2^2 + \eta_3^2 = \frac{1}{4}\Big( d\theta^2 + d\psi^2 + d\phi^2+ \cos \theta( d\psi d\phi + d\phi d\psi) \Big),
\]
and the volume form by 
\[
d\mu_{\mathbb{S}^3} = \eta_1 \wedge \eta_2 \wedge \eta_3 = \frac{1}{8} \sin \theta \,   d\theta \wedge d\psi \wedge d\phi. 
\]
To verify this is correct, recall that the volume of $\mathbb{S}^3$ is $2\pi^2$, while $(1/8)\int_0^\pi \sin(\theta) d\theta = 1/4$ and $|\psi \pm \phi| < 2\pi$ draws out square of area $8\pi^2$.
Additionally, note that the spherical Laplacian is
\[
X^2 := X_1^2 + X_2^2 + X_3^2 = 4 \partial_\theta^2 + 4 \frac{\cos \theta}{\sin\theta} \partial_\theta + \frac{4}{\sin^2\theta} \Big( \partial_\psi^2 + \partial_\phi^2 -2 \cos\theta\partial_\psi\partial_\phi\Big).
\]

\subsection{Wigner functions} 
\label{sec:Wigner fct}

The Wigner functions are a complete, $L^2$ orthonormal set of functions on $\mathbb{S}^3$ introduced by Wigner in the early 20th century in the study of angular momentum in quantum mechanics. The are complex-valued eigenfunctions of $\Delta_{\mathbb{S}^3}$ with the additional property of behaving well with respect to first derivatives by the left-invariant (and right-invariant) frame. Recall that the eigenvalues of $\Delta_{\mathbb{S}^3}$ are $\lambda_J := J(J+2)$ and the dimension of the $k$th eigenspace is $\mathrm{dim}_{\mathbb{R}} \mathcal{H}_k = (J+1)^2$. For $J \in \mathbb{N}$, define 
\begin{equation}
\mathcal{J} := \{ -J + 2k : k = 0, 1, \dots, J\} =  \{-J, -J+2, \dots, J -2 , J\}.
\end{equation}
Note that $|\mathcal{J}| = J+1$. 

As above, we parametrize $\mathbb{S}^3$ by Euler angles
\[
\Phi : (0, \pi) \times \{(\phi, \psi): |\phi + \psi|, |\phi - \psi|< 2\pi\} \to \mathbb{S}^3
\]
given by 
\[
\Phi(\psi, \theta, \phi) = \left(\cos \frac{\theta}{2} \cos \left(\frac{\psi + \phi}{2}\right),\; \cos \frac{\theta}{2} \sin \left(\frac{\psi + \phi}{2}\right), \; \sin \frac{\theta}{2} \cos \left(\frac{\psi - \phi}{2}\right), \; \sin \frac{\theta}{2} \sin \left(\frac{\psi - \phi}{2}\right)\right).
\]

We summarize the properties of the Wigner functions in the following proposition.

\begin{proposition}\label{prop:wigner}
For each integer $J \geq 0$ and integers  $M, M' \in \mathcal{J}:=\{-J, -J + 2, \dots, J-2, J\}$, there exist real-valued, smooth functions, $d^J_{M,M'}(\theta)$ on $[0, \pi]$ with the property that the functions 
\[
D^J_{M, M'}(\psi, \theta, \phi)  := e^{-i \frac{M}{2} \psi} d^J_{M,M'}(\theta) e^{-i \frac{M'}{2} \phi} 
\]
satisfy
\begin{align*}
X_1(D^J_{M, M'})  & = - iM D^J_{M, M'}, \\
(X_3 \pm  i X_2) (D^J_{M, M'}) &= -i \sqrt{J(J + 2) - M (M \pm 2)} D^J_{M \pm 2, M'}, \\
X^2(D^J_{M, M'}) &= - J(J+2) D^J_{M, M'}.
\end{align*}
Here  $X^2 := X_1^2 + X_2^2 + X_3^2 
$.
The functions $d^J_{M, M'}$ have the symmetry 
\[
d^J_{M, M'}(\theta)= (-1)^{\frac{M-M'}{2}}d^J_{-M, -M'} (\theta),
\]
and consequently the functions $D^J_{M, M'}$ have the conjugation symmetry
\begin{equation}\label{eq:wigner-conjugation}
\overline{D^J_{M, M'}} = (-1)^{\frac{M-M'}{2}} D^{J}_{-M, -M'}. 
\end{equation}

The differential identities above imply
\begin{align*}
(X_2 \pm i X_3) (D^J_{M, M'}) &= \pm \sqrt{J(J + 2) - M (M \mp 2)} D^J_{M \mp 2, M'}, \\
(X_2^2 + X_3^2 )D^J_{M, M'} &= -\big( J(J+2) - M^2\big) D^J_{M, M'}.
\end{align*}
In particular, on a Berger sphere with Hopf fiber length $\sqrt{F}$, the Laplacian acts on our basis by
\[
\left(\frac{1}{F}X_1^2 + X_2^2 + X_3^2\right)D^J_{M,M'} = -\left(\left(\frac{1}{F}-1\right) M^2  +J(J+2)\right) D^J_{M,M'}.
\]

The functions above are $L^2$ orthogonal on the 3-sphere and satisfy
\[
\int_{\mathbb{S}^3} D^J_{M, M'}\overline{D^{J_0}_{M_0, M_0'} }d\mu_{\mathbb{S}^3} = \frac{2\pi^2}{J + 1} \delta^{JJ_0} \delta_{MM_0} \delta_{M'M_0'}.
\]
Any complex-valued smooth function $u$ on $\mathbb{S}^3$ can be uniquely expressed using complex coefficients $u^J_{M,M'}$ as
\[
u = \sum_{J =0}^{\infty} \sum_{M, M' \in \mathcal{J}} u^J_{M,M'} D^J_{M, M'}
\]
where
\[
u^J_{M, M'} = \frac{J+1}{2\pi^2} \int_{\mathbb{S}^3} u \,\overline{D^J_{M, M'} } \; d\mu_{\mathbb{S}^3}
\]
so that 
\[
\int_{\mathbb{S}^3} |u|^2 d\mathrm{vol}_{\mathbb{S}^3} = \sum_{J =0}^{\infty} \frac{2\pi^2}{J + 1} \sum_{M, M' \in \mathcal{J}}|u^J_{M,M'}|^2. 
\]
When $u$ is real-valued, one must have that $\overline{u^J_{M, M'}} = (-1)^{\frac{M-M'}{2}}u^J_{-M, -M'}$.
\end{proposition}

Observe that in general, we have 
\begin{align*}
& X^2(D^J_{M, M'}) \\
& = \Big[4 \partial_\theta^2 + 4 \frac{\cos \theta}{\sin\theta} \partial_\theta + \frac{4}{\sin^2\theta} \Big( \partial_\psi^2 + \partial_\phi^2 -2 \cos\theta\partial_\psi\partial_\phi\Big)\Big] \Big(e^{-i M \frac{ \psi}{2}} d^J_{M, M'}(\theta) e^{-i M' \frac{ \phi}{2}}\Big)  \\
& = e^{-i M \frac{ \psi}{2}} \left(4(d^J_{M, M'})''(\theta) + 4\frac{\cos \theta}{\sin\theta} (d^J_{M, M'})'(\theta) - \frac{M^2 + M'^2 - 2MM' \cos \theta}{\sin^2 \theta} (d^J_{M, M'})(\theta)  \right)e^{-i M' \frac{ \phi}{2}}. 
\end{align*}
This leads to the ODE 
\begin{align*}
4(d^J_{M, M'})'' + 4\frac{\cos \theta}{\sin\theta} (d^J_{M, M'})' + \frac{J(J+2) \sin^2\theta - M^2 - M'^2 +2MM' \cos \theta}{\sin^2 \theta} d^J_{M, M'}  = 0. 
\end{align*}
Additionally, 
\begin{align*}
\int_{\mathbb{S}^3} |D^J_{M, M'}|^2 d\mathrm{vol}_{\mathbb{S}^3} = \pi^2  \int_0^\pi  (d^J_{M, M'})(\theta)^2 \sin \theta \, d\theta  
\end{align*}
which leads to the normalization 
\[
\int_0^\pi  (d^J_{M, M'})(\theta) ^2 \sin \theta \, d\theta  =  \frac{2}{J + 1}.
\]
For each $J$, there are $(J+ 1)^2$ Wigner functions. 

One can verify that the little Wigner function matrices $d^J= (d^J_{M, M'})_{M, M' \in\mathcal{J}}$ for small values of $J$ are given as follows. For $J = 0$: 
\[
d^0 = \begin{bmatrix} d^{0}_{0, 0} \end{bmatrix} = \begin{bmatrix}\; 1 \;\end{bmatrix}. 
\]
For $J = 1$: 
\[
d^1 = \begin{bmatrix} d^1_{-1, -1} & d^1_{-1, 1} \\ d^1_{1, -1} & d^{1}_{1,1} \end{bmatrix}  = \begin{bmatrix} \cos \frac{\theta}{2} & \sin \frac{\theta}{2} \\ -\sin \frac{\theta}{2} & \cos \frac{\theta}{2} \end{bmatrix} 
\]
For $J = 2$: 
\[
d^2= \begin{bmatrix} 
d^2_{-2, -2} & d^2_{-2, 0} & d^2_{-2, 2}  \\ 
d^2_{0, -2} & d^2_{0, 0} & d^2_{0, 2} \\
d^2_{2, -2} & d^2_{2, 0} & d^2_{2, 2} \end{bmatrix}  = \begin{bmatrix} 
\frac{1}{2}(1 + \cos \theta) & -\frac{1}{\sqrt{2}} \sin \theta &  \frac{1}{2}(1 - \cos \theta)  \\ 
 \frac{1}{\sqrt{2}} \sin \theta& \cos \theta & -\frac{1}{\sqrt{2}} \sin \theta\\
 \frac{1}{2}(1 - \cos \theta) &  \frac{1}{\sqrt{2}} \sin \theta & \frac{1}{2}(1 + \cos \theta)
\end{bmatrix}. 
\]
For $J = 3$: 
\begin{align*}
d^3= \begin{bmatrix}
d^3_{-3, -3} & d^3_{-3, -1} & d^3_{-3, 1} & d^3_{-3, 3} \\
d^3_{-1, -3} & d^3_{-1, -1} & d^3_{-1, 1} & d^3_{-1, 3} \\
d^3_{1, -3} & d^3_{1, -1} & d^3_{1, 1} & d^3_{1, 3} \\
d^3_{3, -3} & d^3_{3, -1} & d^3_{3, 1} & d^3_{3, 3}
\end{bmatrix} 
\end{align*}
\begin{align*}
= \begin{bmatrix}
\cos\left(\frac{\theta}{2}\right)^3 & -\sqrt{3} \sin\left(\frac{\theta}{2}\right) \cos\left(\frac{\theta}{2}\right)^2 & \sqrt{3} \sin\left(\frac{\theta}{2}\right)^2 \cos\left(\frac{\theta}{2}\right) & -\sin\left(\frac{\theta}{2}\right)^3 \\
\sqrt{3} \sin\left(\frac{\theta}{2}\right) \cos\left(\frac{\theta}{2}\right)^2 & \cos\left(\frac{\theta}{2}\right)\Big(3 \cos\left(\frac{\theta}{2}\right)^2 - 2\Big) & \sin\left(\frac{\theta}{2}\right)\Big(3 \sin\left(\frac{\theta}{2}\right)^2 - 2\Big) & \sqrt{3} \sin\left(\frac{\theta}{2}\right)^2 \cos\left(\frac{\theta}{2}\right) \\
\sqrt{3} \sin\left(\frac{\theta}{2}\right)^2 \cos\left(\frac{\theta}{2}\right) & \sin\left(\frac{\theta}{2}\right)\Big(2-3 \sin\left(\frac{\theta}{2}\right)^2 \Big) & \cos\left(\frac{\theta}{2}\right)\Big(3 \cos\left(\frac{\theta}{2}\right)^2 - 2\Big) & -\sqrt{3} \sin\left(\frac{\theta}{2}\right) \cos\left(\frac{\theta}{2}\right)^2 \\
\sin\left(\frac{\theta}{2}\right)^3 & \sqrt{3} \sin\left(\frac{\theta}{2}\right)^2 \cos\left(\frac{\theta}{2}\right) & \sqrt{3} \sin\left(\frac{\theta}{2}\right) \cos\left(\frac{\theta}{2}\right)^2 & \cos\left(\frac{\theta}{2}\right)^3
\end{bmatrix}.
\end{align*}
Note that $\mathrm{det}(d^J) = 1$.

\end{document}